\documentclass[a4paper, 12pt, oneside]{Thesis}  
\graphicspath{}  

\usepackage[square, numbers, comma, sort&compress]{natbib}  
\usepackage{verbatim}  
\usepackage{vector}  

\hypersetup{urlcolor=black, colorlinks=true}  

\usepackage{booktabs} 
\usepackage{textcomp} 

\usepackage{graphicx}
\usepackage{bm}
\usepackage{amsfonts}
\usepackage{color}
\usepackage{wrapfig} 
\usepackage[table]{xcolor}
\usepackage{siunitx}
\usepackage{amsmath}    
\usepackage{amsthm}
\usepackage{epsfig}
\usepackage{lipsum}
\usepackage{subcaption}  
\usepackage{hyperref}   
\usepackage{amssymb}

\usepackage{pdfpages}
\usepackage{tikz}
\usetikzlibrary{shapes}

\usepackage{listings}
\NewDocumentCommand{\codeword}{v}{%
\texttt{\textcolor{black}{#1}}%
}

\newcommand{\Acomment}[1]{\textcolor{black}{#1}}

\newcommand{\stencilpt}[4][]{\node[circle,draw,inner sep=0.1em,minimum size=0.8cm,font=\tiny,#1] at (#2) (#3) {#4}}

\newcommand{\der}{\mathrm{d}}
\newcommand{\dx}{\Delta x}
\newcommand{\dy}{\Delta y}
\newcommand{\dt}{\Delta t}
\newcommand{\vecrhsT}{\mathbf{d}}

\newcommand{\vecx}{\mathbf{x}}
\newcommand{\vecy}{\mathbf{y}}
\newcommand{\vecz}{\mathbf{z}}
\newcommand{\vecb}{\mathbf{b}}
\newcommand{\vecrhs}{\mathbf{f}}
\newcommand{\vecf}{\mathbf{f}}
\newcommand{\vecg}{\mathbf{g}}

\newcommand{\vecfhat}{\widehat{\mathbf{f}}}
\newcommand{\vech}{\mathbf{h}}
\newcommand{\veck}{\mathbf{k}}

\newcommand{\mxA}{\mathbf{A}}
\newcommand{\mxE}{\mathbf{E}}
\newcommand{\mxXhat}{\widehat{\mathbf{X}}}

\newcommand{\mxL}{\mathbf{L}}
\newcommand{\mxR}{\mathbf{R}}
\newcommand{\mxU}{\mathbf{u}}
\newcommand{\mxV}{\mathbf{v}}

\newcommand{\mathd}{\mathrm{d}}
\newcommand{\mathe}{\mathrm{e}}

\newcommand{\mydim}{D}
\newcommand{\vecn}{\bm{n}}

\begin{document}

\pagestyle{empty}
\includepdf[pages=-,pagecommand={},offset=2.55cm -2.55cm]{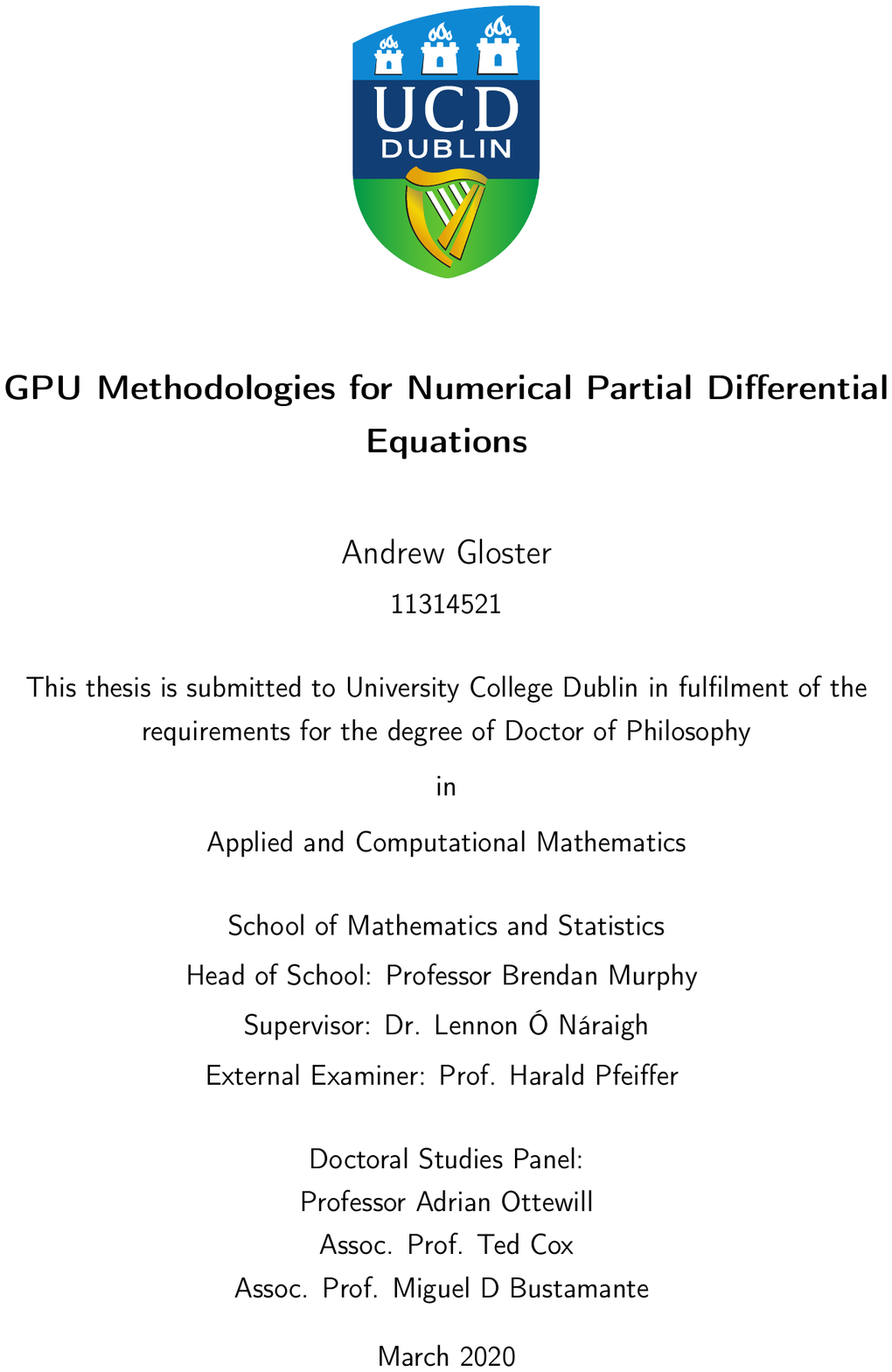}

\frontmatter      
\pagenumbering{roman}



\fancyhead{}  
\rhead{\thepage}  
\lhead{}  


\pagestyle{empty}  

\null\vfill
\textit{``You'll address me by my proper title, you little bollocks!'}

\begin{flushright}
Bishop Leonard Brennan (Father Ted 1996)
\end{flushright}

\vfill\vfill\vfill\vfill\vfill\vfill\null
\clearpage  





\tableofcontents  

\btypeout{Abstract}
\thispagestyle{plain}
\begin{center}{\huge{Abstract} \par}\end{center}
{ 
In this thesis we develop techniques to efficiently solve numerical Partial Differential Equations (PDEs) using Graphical Processing Units (GPUs).
Focus is put on both performance and re--usability of the methods developed, to this end a library, cuSten, for applying finite--difference stencils to numerical grids is presented herein.
On top of this various batched tridiagonal and pentadiagonal matrix solvers are discussed.
\Acomment{These} have been benchmarked against the current state of the art and shown to improve performance in the solution of numerical PDEs.
A variety of other benchmarks and use cases for the GPU methodologies are presented using the Cahn--Hilliard equation as a core example, but it is emphasised the methods are completely general. 
Finally through the application of the GPU methodologies to the Cahn--Hilliard equation new results are presented on the growth rates of the coarsened domains.
In particular a statistical model is built up using batches of simulations run on GPUs from which the growth rates are extracted, it is shown that in a finite domain that the traditionally presented results of $1/3$ scaling is in fact a distribution around this value.
This result is discussed in conjunction with modelling via a stochastic PDE and sheds new light on the behaviour of the Cahn--Hilliard equation in finite domains.
}
\clearpage



\btypeout{Statement of Original Work}
\thispagestyle{plain}
\begin{center}{\huge{Statement of Original Work} \par}\end{center}
{\normalsize 

I hereby certify that the submitted work is my own work, was completed 
while registered as a candidate for the degree stated on the Title Page, and I have not obtained a 
degree elsewhere on the basis of the research presented in this submitted work.

}
\clearpage

\btypeout{Sponsor}
\thispagestyle{plain}
\begin{center}{\huge{Sponsor} \par}\end{center}
{\normalsize 

This work was supported by the University College Dublin Structured Ph.D. Programme in Applied and Computational Mathematics and was funded by the UCD Research Demonstratorship.  
I acknowledge the support of NVIDIA Corporation for the donation of the two Titan X Pascal GPUs used for this research.
In addtion I wish to acknowledge the DJEI/DES/SFI/HEA Irish Centre for High-End Computing (ICHEC) for the provision of computational facilities and support.
\begin{figure}[b!]
\includegraphics[width=0.2\textwidth]{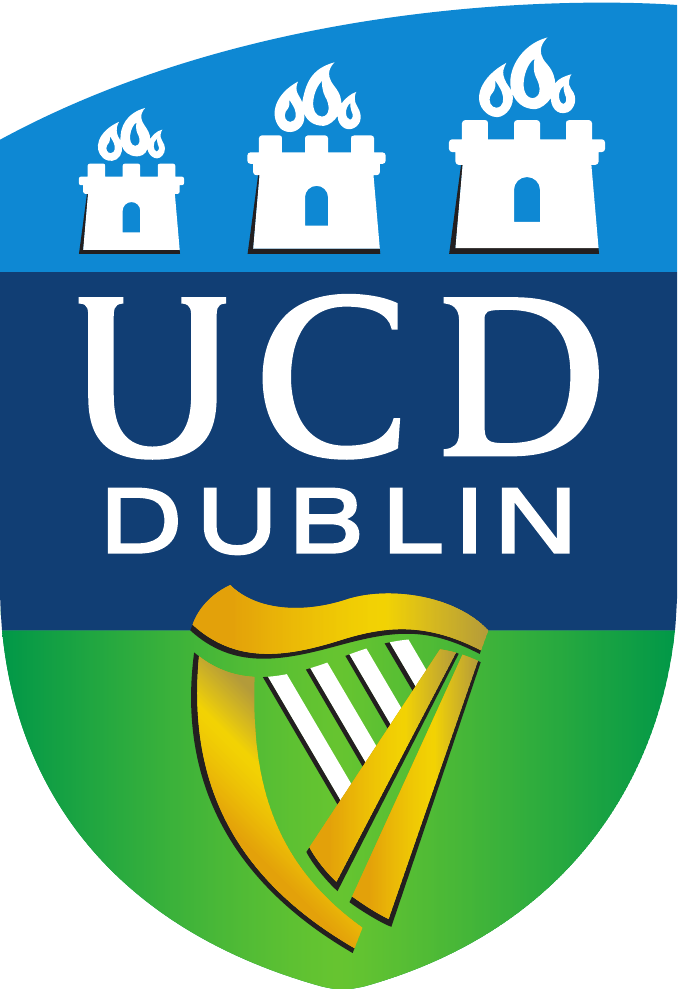}
\end{figure}

}
\clearpage

\btypeout{Collaborations}
\thispagestyle{plain}
\begin{center}{\huge{Collaborations} \par}\end{center}
{

\begin{itemize}
\item \textbf{Lennon {\'O} N{\'a}raigh}. Asst. Prof. Lennon {\'O} N{\'a}raigh served as my Ph.D. supervisor, and as such all the work in this thesis should be viewed as a collaboration with him. 
\item \textbf{Khang Ee Pang}. Wrote the serial C programs necessary for benchmarking in Chapter 3 and contributed to the associated paper as part of an internship with Lennon {\'O} N{\'a}raigh.
\item \textbf{Enda Carroll}. Wrote the tridiagonal versions of the solvers and contributed to the paper forming the basis of Chapter 4.
\item \textbf{Miguel Bustamante}. Supervisor of Enda Carroll with contributions to paper forming the basis of Chapter 4.
\item \textbf{Lung Sheng Chien}. Software engineer at NVIDIA, providied feedback on some of the GPU methodolgies presented in this thesis.
\item \textbf{Harun Bayraktar}. Manager, CUDA Mathematical Software Libraries at NVIDIA, providied feedback on some of the GPU methodolgies presented in this thesis.

\end{itemize}
}
\clearpage





\acknowledgements{
\addtocontents{toc}{\vspace{1em}}  
I would like to thank my supervisor Dr. Lennon {\'O} N{\'a}raigh for all his time, support and help throughout my PhD.
Without his guidance none of this work would have been possible and his ability to put things back together, and instil confidence, even when everything was falling apart was invaluable. 
I also owe a big thank you to Dr. Ted Cox who served as a secondary supervisor for part of this work and whose willingness to work with myself and Lennon allowed me to secure the original funding necessary to undertake this PhD.

The constant support and encouragement throughout all of my education that my parents gave me ultimately brought me to a position from which I was able to undertake a PhD.
Without the fundamental drive and passion for education they gave me I wouldn't have gotten past Junior Cert Maths, never mind a PhD.
I owe them a huge thanks for everything, from funding a year in Imperial College London through to listening to my rants when something was not going right. 
Arguments over how to solve simple quadratic equations in the kitchen at home with my Dad while studying for the Leaving Cert deserve a mention, it seems the perseverance paid off.

Finally to my girlfriend who I met in UCD while doing this PhD and to all my UCD friends, whether the original lunch crew, the Famous Five of the EIRSAT-1 team, Da Office and the rest of the people in the School of Mathematics and Statistics at UCD thank you for all the fun and support along with way.

}
\clearpage  

\lhead{Publications and Conferences}  
\chapter*{Publications and Conferences}

\section*{Contributed talks}
\noindent
- June 2019, IUTAM Symposium 2019, University College Dublin, Ireland: 'GPU methods for Fluid Mechanics.'

\noindent
- September 2018,  UCD Applied \& Computational Mathematics Seminar Series, University College Dublin, Dublin, Ireland: 'cuPentBatch -- A batched pentadiagonal solver for NVIDIA GPUs'

\section*{Publications}
- \textbf{Andrew Gloster, Lennon {\'O} N{\'a}raigh and Khang Ee Pang} 
cuPentBatch -- A batched pentadiagonal solver for NVIDIA GPUs. \\
Computer Physics Communications, Volume 241, Pages 113-121, 2019. \\

- \textbf{Andrew Gloster and Lennon {\'O} N{\'a}raigh}
cuSten -- CUDA Finite Difference and Stencil Library. \\ 
Software X, Volume 10, 2019.

\section*{Preprints}

\noindent
- \textbf{Andrew Gloster, Enda Carroll, Miguel Bustamante \Acomment{and} Lennon {\'O} N{\'a}raigh} Efficient Interleaved Batch Matrix Solvers for CUDA. \\ 
Preprint: arXiv:1909.04539, 2019 

- \textbf{Lennon {\'O} N{\'a}raigh and Andrew Gloster} A large-scale statistical study of the coarsening rate in models of Ostwald-Ripening. \\ 
Preprint: arXiv:1911.03386, 2019 (Submitted for review)

\section*{Source Code}
\Acomment{
cuSten library version used in thesis - \url{https://github.com/munstermonster/cuSten/releases/tag/2.1}}

\Acomment{
cuPentBatch version used in thesis - \url{https://github.com/munstermonster/cuPentBatch/releases/tag/1.0}}

\mainmatter	  
\pagestyle{plain}  


\lhead{\emph{Introduction}}  
\chapter{Introduction}
\label{chapter:intro}
Partial Differential Equations (PDEs) are present in almost all dynamic physical systems, examples include the Navier--Stokes equations for fluid flow \cite{pope2001turbulent, doering1995applied}, Euler equations for shock--waves \cite{fedkiwBook, hesthavenBook}, the Black--Scholes equation for options pricing \cite{black1973pricing, wilmott_howison_dewynne_1995} and Burgers equation, which is \Acomment{used in modelling} gas dynamics and traffic flow \cite{whitham2011linear}. 
In many situations analytic methods for solving a given PDE are not possible due to the complexity of the equations; an alternate approach is to solve the PDE numerically on a computer.
For these numerical simulations to have a high fidelity to the physical problem high resolution grids are needed which in turn require large data--sets and computational resources.
Typically High Performance Computing (HPC) is used to tackle these large simulations, specifically having many individual computers/processors work together, dividing the problem into small parts which are then calculated in parallel at the same time.
Graphics Processing Units (GPUs) are an increasingly popular HPC hardware solution to solve these computing problems in applied mathematics and computational physics.
This is due to the increased levels of parallelisation, speed--up and energy efficiency GPUs offer over standard parallelisation methods such as Open Multi--Processing (OpenMP) and MPI (Message--Passing--Interface) which rely on standard Central Processing Units (CPUs).
Indeed, 5 of the top 10 supercomputers in the TOP500 survey, conducted in June 2019, have GPUs as a core piece of hardware within their architecture for accelerating HPC applications.
The energy efficiency benefits of GPUs are made particularly clear in the Green500 survey, where efficiency is measured in terms of GFlops / watts (FLOPs -- floating point operations per Second).
9 of the top 10 supercomputers in this list make use of GPUs, as of June 2019.
The Green500 survey focuses on energy efficiency as well as computational speed of supercomputers, the more energy efficient the machine the lower the electricity costs are to run it and subsequently the lower the environmental cost of running large parallel HPC applications. 

In this thesis we will focus on methods for the application of GPUs to solve various batches of 1--dimensional (1D) and 2--dimensional (2D) PDEs numerically using NVIDIA GPUs.
Extending the use of GPUs to numerical problems and other areas outside of computer graphics is often refereed to General Purpose Computing on Graphics Processing Units (GPGPU).
NVIDIA is the current leading manufacturer of GPUs for scientific computing purposes as emphasised by the prevalence of their hardware in the TOP500 and Green500 surveys.
The numerical programs in this thesis will be written in CUDA (Compute Unified Device Architecture) which is an application programming interface (API) which extends the ability to execute programs on a NVIDIA GPU to the C/C++ programming languages.
Competing implementations include OpenCL and \Acomment{OpenACC}.
OpenCL is an effort to have a cross--platform programming language which can be applied on various hardware solutions including CPUs, GPUs and field--programmable gate arrays (FPGAs).
OpenACC is another implementation which attempts to implement programs on a GPU in a similar fashion to OpenMP where compiler directives are used to flag areas of code which should be executed in parallel on the GPU.

This thesis will also focus on the application of these GPU methodologies to the Cahn--Hilliard equation~\cite{CH_orig}.
The Cahn--Hilliard equation is a PDE with a fourth order derivative and a \Acomment{non--linear} term which models the phase separation of a binary mixture.
It was chosen as a representative PDE to which to apply GPUs to within the scope of this thesis, yet it is worth emphasising that the GPU methodologies discussed herein can be applied generally to the numerical study of PDEs.
We explore methods of parallelising the solution of the Cahn--Hilliard equation, whether as batches of independent 1D equations or larger 2D simulations.
Using these methods we will examine batches of 1D Cahn--Hilliard equations to examine averaged scaling behaviours and perform a parameter study, varying the given parameters across members of a batch of simulations, to produce flow pattern maps, using a clustering algorithm to group the data. 
We also run batches of 2D simulations of the Cahn--Hilliard equation to examine the application of Lifshitz--Slyozov--Wagner (LSW) theory to look at the growth rates of the separated binary phases.

In order to present and discuss the work in this thesis we first explore separate introductions to the four subtopics necessary to understand the work herein, these topics are presented in Chapter~\ref{chapter:method}.
In addition to this many of the chapters in this thesis are based on papers completed during the course of the PhD that are either published or under review at the time of writing.
Each paper has its own introduction and they all overlap to a large degree, thus these \Acomment{introductory} sections can be considered a unification of these paper introductions, the content of the papers is then restructured slightly in order to refer to them.
Also this layout ensures the discussion contained in this thesis is cohesive and avoids repetition.
The first subtopic is discussed in Section~\ref{sec2:fdmethod}.
\Acomment{We} look at Finite--Difference methods, our chosen method of numerical discretisation of PDEs.
Then in the following section (\ref{sec2:gpuArch}) we look at GPU hardware architectures and how they differ to standard CPU architectures.
We introduce the CUDA API and the key concepts behind its use as the third topic in Section~\ref{sec2:cudaAPI}.
The final topic, covered in Section~\ref{sec2:cahnEq}, will focus on the Cahn--Hilliard equation and the necessary background material relating to it relevant to this thesis.

Having covered the required background material in Chapter~\ref{chapter:method} we can then move to the body of the thesis.
In Chapter~\ref{chapter:cuSten} we present a GPU library, cuSten, developed to apply finite difference stencils to various batched 1D arrays and 2D arrays.
The aim of the library is to simplify the implementation of finite-difference programs on GPUs taking much of the heavy lifting away from the programmer, much like the implementations of cuBLAS and cuSPARSE simplify linear algebra operations, to allow for speedier code production.
Benchmarks and examples of the library application are covered in Chapter~\ref{chapter:cuSten}.
Following this, in Chapter~\ref{chapter:cuPentBatch}, a batched pentadiagonal solver is developed for the GPU.
This solver is benchmarked against the state of the art algorithm from cuSPARSE and general performance speed-ups for solving batches of 1D PDEs on a GPU versus serial and OpenMP implementations.
The solver is extended further in Chapter~\ref{chapter:efficient} to a more data efficient implementation which is suited particularly to PDE parameter studies, this is benchmarked against the solver from the previous chapter.

In Chapter~\ref{chapter:batched} we apply the 1D methodologies to examine averaged results across 1D Cahn--Hilliard equations, in particular looking at the scaling of the separated phases as a function of time.
We also apply the batch solving methodology to batches of forced 1D Cahn--Hilliard equations to look at generating data sets for flow--pattern maps.
\Acomment{The} maps are then created using a clustering algorithm, this is an attempt to automate a tedious process of clustering that usually has to be performed by hand in the absence of a metric that can classify solutions. 
The Alternating Direction Implicit (ADI) scheme for solving the Cahn--Hilliard equation developed in Chapter~\ref{chapter:cuSten} is then applied to solve batches of independent 2D simulations from which we extract $\beta$, the growth rate of the separated regions in the simulations.
\Acomment{The} batches are then used to build a statistical picture of $\beta$, this work is presented in Chapter~\ref{chapter:statistical}.
These statistics are also discussed in the context of LSW theory, \Acomment{the} underlying theory for the ripening/coarsening phenomena seen in the Cahn--Hilliard equation.
Finally in Chapter~\ref{chapter:con} a conclusion and discussion of the work carried out as part of this thesis is presented. 
Discussion of possible future work directions are also included in this final chapter.


\lhead{\emph{Chapter 2}}  
\chapter{Theory and Methodology}
\label{chapter:method}
In this chapter we present the four primary topic areas that are covered in this thesis and the theory/methodologies behind them. 
This is to provide the background for the content discussed in later chapters and to remove overlap of discussion from the papers that resulted from the work completed in this thesis.

\section{Finite--Difference Methods}
\label{sec2:fdmethod}
In this thesis we will be making use of one of the standard methods for the numerical discretisation of PDEs, the Finite--Difference Method.
It is a method for approximating derivatives using weighted differences between various points in a numerical grid.
We break the discussion into three subsections here where we first describe finite--differencing stencils in Section~\ref{subsec2:fdStencil} which are used to discretise the domain. 
We then present a discussion in Section~\ref{subsec2:numStab} on numerical stability where we discuss the conditions one must enforce in order to ensure numerical stability.
Finally a discussion, presented in Section~\ref{subsec2:numAcc}, included on methods for showing that the expected accuracy of a given numerical scheme has been achieved.

\subsection{Finite--Difference Stencils}
\label{subsec2:fdStencil}
We begin by defining the Taylor series of a well--behaved function (smooth, differentiable and continuous) at $x + h$ where $h$ is a small increment in space

\begin{equation}
f(x + h) = f(x) + \frac{f^{\prime}(x)}{1!}h + \frac{f^{(2)}(x)}{2!}h^2 + \dots
\end{equation}

We then set $x = a$ as this is the explicit point we wish to approximate our derivative at, we also divide across by $h$ to yield

\Acomment{
\begin{equation}
\frac{f(a + h)}{h} = \frac{f(a)}{h} + f^{\prime}(a) + \frac{f^{(2)}(a)}{2!}h + \dots
\end{equation} 
}

Now rearranging for the derivative we have

\Acomment{
\begin{equation}
f^{\prime}(a) = \frac{f(a + h) - f(a)}{h} - \frac{f^{(2)}(a)}{2!}h + \dots
\end{equation}
}
We now neglect terms of $O(h)$ to yield a final approximation for the first derivative

\begin{equation}
f^{\prime}(a) = \frac{f(a + h) - f(a)}{h} + O(h)
\label{eq2:firstorderacc}
\end{equation}

thus we have a scheme that is first order accurate.
Accuracy in this sense is given by the order of $h$ of the largest truncation term, here we have truncated at the term $\frac{f^{(2)}(x)}{2!}h$ where $h$ has a power of $1$, and so the scheme is first order accurate.
A clear analogy can be seen between this approximation and the standard definition of differentiation from first principles, if we let $h \rightarrow 0$ in Equation~\eqref{eq2:firstorderacc} \Acomment{we then recover} the standard definition of the first derivative of a function.

We now adopt the common notation for finite--difference methods.
Let $\dx = h$ be the spacing in the $x$--direction of a \Acomment{uniform} numerical grid.
Position within this grid can be determined by $x = i \dx$ where $i = 1, 2\dots, N$ \Acomment{and} $N$ is the number of points used to discretise the finite domain $\Omega$.
$\Delta x$ can be recovered directly by dividing the domain length $L$ by \Acomment{$N + 1$}, so we have \Acomment{$\dx = L / (N + 1)$}. 
Similar indexing with $j$ is adopted for $y$--direction derivatives and $n$ for time derivatives.
\Acomment{Typically spatial} derivative indexing is subscripted with time indexing superscripted. 
So adopting this notation changes equation~\eqref{eq2:firstorderacc} into 

\begin{equation}
\left(\frac{\der f}{\der x}\right)_i = \frac{f_{i+1} - f_i}{\dx}
\label{eq2:forDiff}
\end{equation}

The above expression is a forward difference as we're differencing using the current point $i$ and the point in front $i+1$, similarly the backwards difference can be defined by differencing using $i$ and $i - 1$ to give 

\begin{equation}
\left(\frac{\der f}{\der x}\right)_i = \frac{f_{i} - f_{i - 1}}{\dx}
\label{eq2:backDiff}
\end{equation}

In order to extract higher derivatives we can simply apply the above expressions together by applying a backward difference to a forward difference

\begin{equation}
\left(\frac{\der^2f}{\der x^2}\right)_i = \frac{ \frac{f_{i + 1} - f_{i}}{\Delta x}  -  \frac{f_{i} - f_{i - 1}}{\dx}    }{\Delta x}
\end{equation}

which simplifies to the classic second order accurate central difference approximation to the second derivative

\begin{equation}
\left(\frac{\der^2f}{\der x^2}\right)_i = \frac{f_{i + 1} - 2f_i + f_{i - 1}}{\dx^2}
\label{eq2:centreSecond}
\end{equation}

Central difference is the terminology used whenever there is a symmetric weighting of values around position $i$ such as the case above where we have used the points $i - 1$ and $i + 1$.
For completeness we show this expression to be second order accurate, first generalising the coefficients yields

\begin{equation}
\left(\frac{\der^2f}{\der x^2}\right)_i = \frac{\alpha f_{i + 1} + \beta f_i + \gamma f_{i - 1}}{\dx^2}
\label{eq2:generalSecond}
\end{equation}

We then Taylor expand the terms $f_{i + 1}$ and $f_{i - 1}$ as follows

\begin{align}
f_{i + 1} & = f_i + \dx \left(\frac{\der f}{\der x}\right)_i + \frac{\dx^2}{2} \left(\frac{\der^2 f}{\der x^2}\right)_i + \frac{\dx^3}{3!} \left(\frac{\der^3 f}{\der x^3}\right)_i + \dots \\
f_{i - 1} & =  f_i - \dx \left(\frac{\der f}{\der x}\right)_i + \frac{\dx^2}{2} \left(\frac{\der^2 f}{\der x^2}\right)_i - \frac{\dx^3}{3!} \left(\frac{\der^3 f}{\der x^3}\right)_i + \dots
\end{align}

which can then be substituted into equation~\eqref{eq2:generalSecond} to give

\Acomment{
\begin{align}
\left(\frac{\der^2f}{\der x^2}\right)_i  = & \frac{\alpha + \beta + \gamma}{\dx^2}f_i + \frac{\alpha - \gamma}{\dx}\left(\frac{\der f}{\der x}\right)_i + \\ &  \frac{\alpha + \gamma}{2}\left(\frac{\der^2 f}{\der x^2}\right)_i + \frac{\alpha - \gamma}{6}\Delta x\left(\frac{\der^3 f}{\der x^3}\right)_i + O(\dx^2)
\end{align}
}

This system can be solved by setting $\alpha = \gamma = 1$ and $\beta = -2$, thus recovering Equation~\eqref{eq2:centreSecond} and showing that it is indeed a second order accurate approximation to the second derivative. 
Later in this thesis we will also require a difference for the fourth derivative which we give here as a second order accurate expression

\begin{equation}
\frac{\der^4 f}{\der x^4} = \frac{f_{i + 2} - 4 f_{i + 1} + 4 f_i - 4 f_{i - 1} + f_{i + 2}}{\dx^4}
\label{eq2:centreFourth}
\end{equation}

Extensions of the Finite--Difference Method to more sophisticated schemes such as Essentially--Non--Oscillatory (ENO)~\cite{shuENO, shu88, shu89, fedkiwBook, hesthavenBook} and Weighted--Essentially--Non--Oscillatory (WENO)~\cite{jcpWENO, shuENO, fedkiwBook, hesthavenBook} are available in order to deal with taking numerical derivatives across discontinuities such as shocks.
These schemes account for the direction shocks are travelling in a system, which if not properly dealt with can lead to instabilities and numerical artefacts, a typical feature that appears is an oscillation near the shock fronts known as Gibbs--Phenomena~\cite{hesthavenBook}.
For expressions that deal with non--uniform grids and retain accuracy with respect to truncation errors one must resort to coordinate transformations from the physical non--uniform space to a uniform space, carry out the differencing there, followed by transforming back~\cite{chung2010computational}.
Typically if one resorts to simple differencing on non--uniform grids by making $\dx$ variable, then no better than fist order accuracy will be achieved.


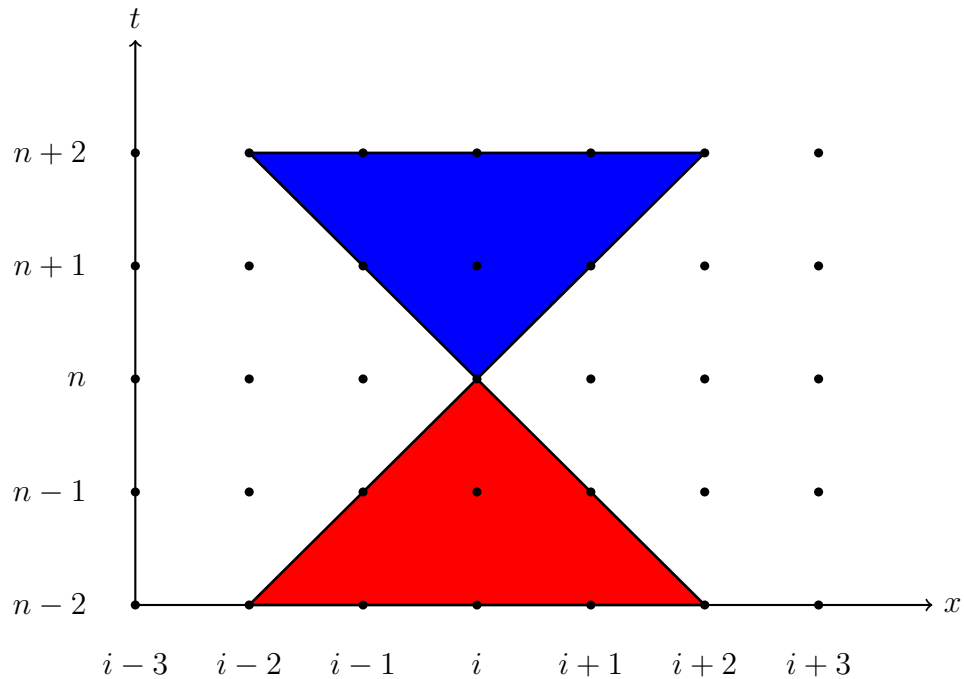
\begin{figure}
	\centering
    \begin{tikzpicture}[scale=3]
    \draw [<->,thick] (0,2.5) node (yaxis) [above] {$t$}
        |- (3.5,0) node (xaxis) [right] {$x$};

    \draw (0.5, 0.0) coordinate (a1) -- (2.5, 2.0) coordinate (a2);
    \draw (2.5, 0.0) coordinate (b1) -- (0.5, 2.0) coordinate (b2);
	\coordinate (c) at (intersection of a1--a2 and b1--b2);
	\filldraw[draw=black, fill=red, line width=1pt] (a1) -- (b1) -- (c) -- cycle; 
	\filldraw[draw=black, fill=blue, line width=1pt] (a2) -- (b2) -- (c) -- cycle;

	\foreach \x in {0,0.5,...,3.0} {
        \foreach \y in {0,0.5,...,2} {
            \fill[color=black] (\x,\y) circle (0.02);
        }
    }

    \node at (0.0, 0.0) [below = 0.5cm] {$i - 3$};
    \node at (0.5, 0.0) [below = 0.5cm] {$i - 2$};
    \node at (1.0, 0.0) [below = 0.5cm] {$i - 1$};
    \node at (1.5, 0.0) [below = 0.5cm] {$i$};
    \node at (2.0, 0.0) [below = 0.5cm] {$i + 1$};
    \node at (2.5, 0.0) [below = 0.5cm] {$i + 2$};
    \node at (3.0, 0.0) [below = 0.5cm] {$i + 3$};

    \node at (0.0, 0.0) [left = 0.5cm] 	{$n - 2$};
    \node at (0.0, 0.5) [left = 0.5cm] 	{$n - 1$};
    \node at (0.0, 1.0) [left = 0.5cm] 	{$n$};
    \node at (0.0, 1.5) [left = 0.5cm] 	{$n + 1$};
    \node at (0.0, 2.0) [left = 0.5cm] 	{$n + 2$};

    \end{tikzpicture}
    \caption{Diagram showing the domain of dependence in red and the domain of influence in blue.}
    \label{fig:dependence}
\end{figure}

\subsection{Numerical Stability}
\label{subsec2:numStab}
In order \Acomment{to evolve a system numerically in time there} are conditions on the relationship between the time--step size $\dt$ and the spacing $\dx$ that must be satisfied in order to ensure that no incorrect information is propagated through the domain.
This criteria can be described heuristically through the concept of the domain of dependence.
\Acomment{Imagining a physical wave travelling through a discrete computational space, the speed of the computation $\Delta x / \Delta t$ must be at least as fast as the physical wave speed to ensure the correct information is propagated throughout the computation.
Should the computational speed not be fast enough then incorrect, un-physical data will be included in the computation, typically leading to the development of errors and ultimately numerical infinities.}
Thus we must have that

\begin{equation}
u \leq C \frac{\dx}{\dt}
\end{equation}

Here $C$ is a dimensionless number know as the Courant number, it is also know as the Courant-Friedrichs-Lewy (CFL) condition~\cite{hesthavenBook}.
This equation now allows us to define the domain of dependence as the area shown in red in Figure~\ref{fig:dependence}, in order for the point $i$ to be calculated all of the information necessary must be contained in the past space--time cone. 
The slopes of the cone are given by $\pm u$.
The points which rely on this point $i$ at future times must fall within the blue region, this region is  known as the domain of influence.

While the domain of dependence provides a heuristic for how to ensure numerical stability it does not provide us with a usable expression to determine the necessary relationship between $\dx$ and $\dt$, for this we must turn to Von Neumann stability analysis which is based on Fourier series.
\Acomment{For the purposes of this discussion we provide an example numerical scheme which evolves the diffusion equation using an explicit time-stepping scheme.
\begin{equation}
\frac{\partial T}{\partial t} = \alpha \frac{\partial^2 T}{\partial x^2}
\label{eq2:diffusionPDE}
\end{equation}
We then discretise Equation~\eqref{eq2:diffusionPDE} using a first order forward time stencil along with a second order spatial central difference stencil to yield
\begin{equation}
T^{n + 1}_i = \sigma T^n_{i + 1} +  (1 - 2 \sigma) T^n_i + \sigma T^n_{i - 1}
\label{eq2:heatFTCS}
\end{equation}
where we have defined $\sigma = \alpha \dt / \dx^2$.}
In order to find the stability criteria using Von Neumann analysis we take a typical term of a Fourier series solution at time step $n$ given by 

\begin{equation}
T^n(x) =  \xi(k)^n e^{ik_xx}
\end{equation}

and substitute it into equation~\eqref{eq2:heatFTCS} to yield

\begin{equation}
\xi(k)^{n + 1} e^{ik_xx} = \sigma \xi(k)^n e^{ik_x(x + \dx)} + (1 - 2 \sigma) \xi(k)^n e^{ik_xx} +  \sigma \xi(k)^n e^{ik_x(x - \dx)}
\end{equation}

Simplifying and solving for the amplitude $\xi(k)$ we have

\begin{equation}
\xi(k) = 1 - 4\sigma \sin^2\left(\frac{k_x \dx}{2}\right)
\end{equation}

From here we can say there is a sufficient condition for stability if the absolute value of this expression is bounded above by $1$, in other words

\begin{equation}
\left|1 - 4\sigma \sin^2\left(\frac{k_x \dx}{2}\right)\right| \leq 1
\end{equation}

It is easily seen that the term $4\sigma \sin^2\left(\frac{k_x \dx}{2}\right)$ is always positive and so a sufficient condition becomes

\begin{equation}
4\sigma \sin^2\left(\frac{k_x \dx}{2}\right) \leq 2
\end{equation}

which, when one considers that $\sin^2(x) \leq 1$ and is strictly positive, can be used to give a restriction on $\sigma$ given by

\begin{equation}
\sigma = \frac{\alpha \dt}{\dx^2} \leq \frac{1}{2}
\end{equation}

So now we have our desired relationship between $\dt$ and $\dx$, for a given spacing $\dx$ we must have that 

\begin{equation}
\dt \leq \frac{\dx^2}{2\alpha}
\label{eq2:heatCFL}
\end{equation}

\Acomment{
While this example is purely instructional it provides us with a clear picture of the relationship between $\Delta x$ and $\Delta t$ and how this relationship affects the stability of a numerical scheme.
Should $\Delta t$ be too large then this scheme will become numerically unstable and lead to numerical infinities within the domain.
It can also be seen in this example that the choice of numerical discretisation and differential order of the PDE itself impact on the stability of the scheme.
This is made obvious in the appearance of $\sigma$ in the stability condition, a quantity that comes from how we chose to discretise the operators in Equation~\eqref{eq2:diffusionPDE}.
In particular the choice of an explicit first order accurate stencil for the time differential contributes to the stability criteria, should we have chosen an implicit scheme such as Crank--Nicolson then unconditional stability could have been achieved.
Unconditional stability eliminates the restriction of the relationship between time-step size and spatial discretisation allowing for larger time-steps to be taken when computing, reducing overall compute time, a desirable feature when computing high resolution systems.
We explore the application of a Crank--Nicolson scheme later in Chapter~\ref{chapter:cuPentBatch} to the hyperdiffusion  equation where unconditional stability is achieved.
}

In reality for many schemes it is only useful, or indeed possible, to apply Von Neumann stability analysis when dealing with linear equations.
Typically with non--linear equations one must extract a condition on $\dt$ via trial and error, this is commonly \Acomment{achieved} by varying the Courant number $C$ above, taking into account also the maximum wave--speed in the domain $u$ where possible, this can easily be done in advection equations where the wave--speed is an analytic expression but can be harder in more involved systems.
\Acomment{Should $C$ be chosen to be too large then the numerical scheme will give un-physical results and ultimately lead to numerical infinities.
Other methods to improve stability of a given scheme include resorting to implicit time--stepping, operator splitting, which allows for application of schemes to individual terms which are then brought back together, and projection methods, where the system is solved in another space that is easier to discretise in and then transformed back to physical coordinates.}


\begin{figure}
    \centering
        \includegraphics[width=0.6\textwidth]{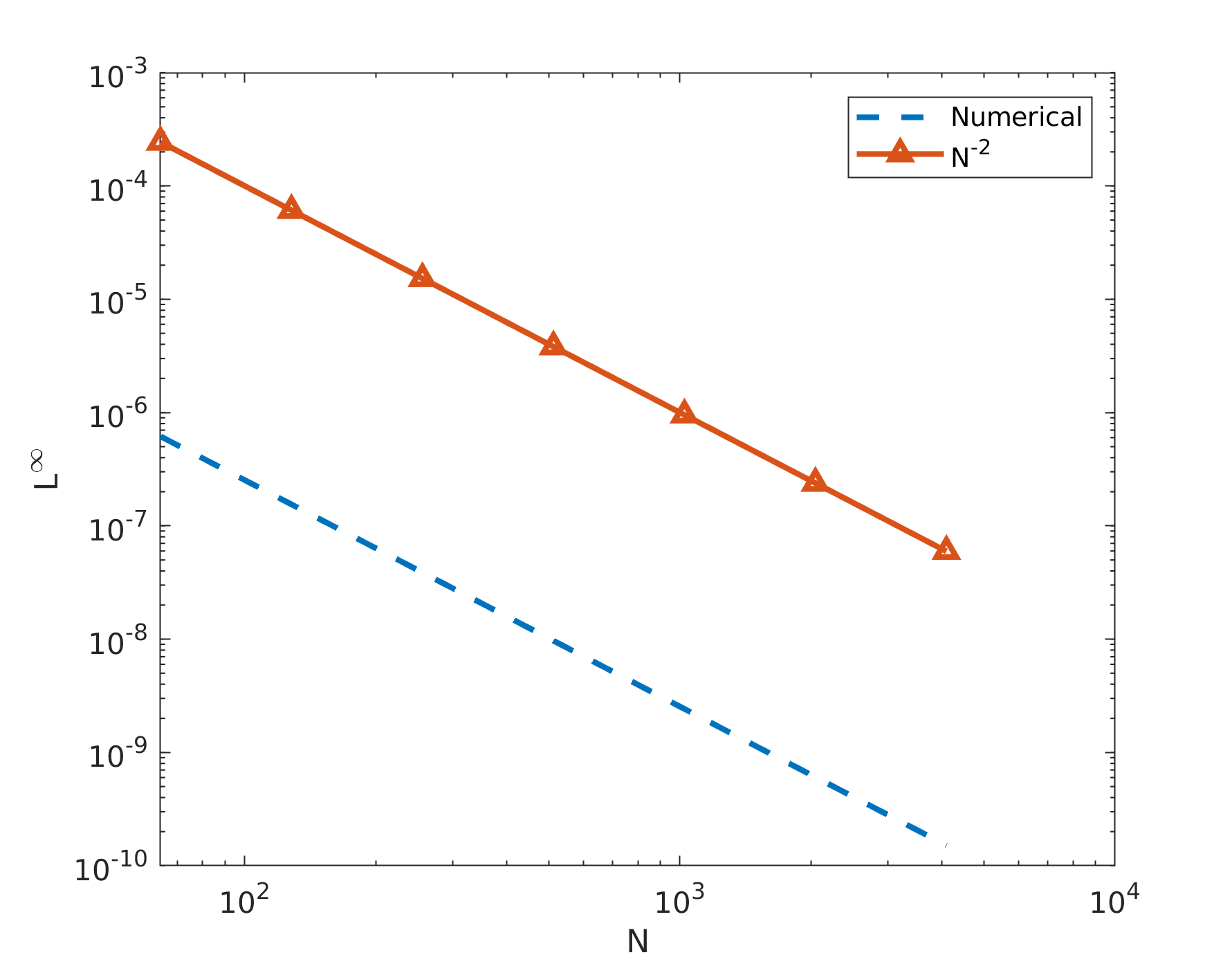}
        \caption{\Acomment{Plot of $L^{\infty}$ norm of the error relative to the analytic solution of the heat equation versus number of grid points showing a $O(\dx^2)$ accuracy.}}
    \label{fig2:convergenceHeatAnalytical}
\end{figure}

\subsection{Numerical Accuracy}
\label{subsec2:numAcc}
After implementing a given numerical scheme one must benchmark it to ensure that the expected order of accuracy has been achieved and that the solution is behaving as expected.
As an example in this section we make use of the scheme we previously discussed in Equation~\eqref{eq2:heatFTCS}, this is implemented in MATLAB with periodic boundary conditions.
We choose periodic boundary conditions as this allow us to produce an analytic solution for the heat equation.
We take an $L^\infty$--norm between the analytic and numerical solutions and norm should then converge with increasing $N$ with the same spatial order of accuracy as the scheme.

For the convergence study we take an initial condition of $T(x, t = 0) = A_0 \sin(kx + \phi)$, where $A_0$ is some initial amplitude, this coupled with the periodic boundary gives us a solution to benchmark against

\begin{equation}
T(x, t) = A_0 e^{- \alpha k^2 t} \sin(kx + \phi)
\end{equation} 

We simulate up to $t = 10$ with $\alpha = 1.0$, $k = 1.0$ and $\phi = 0$.
The domain is set to a size of $2\pi$.
In order to ensure stability and satisfy equation~\eqref{eq2:heatCFL} we set $\dt = 0.9 \frac{\dx^2}{2\dt}$.
The results of this study are presented in Figure~\ref{fig2:convergenceHeatAnalytical}.

\begin{figure}
    \centering
        \includegraphics[width=0.6\textwidth]{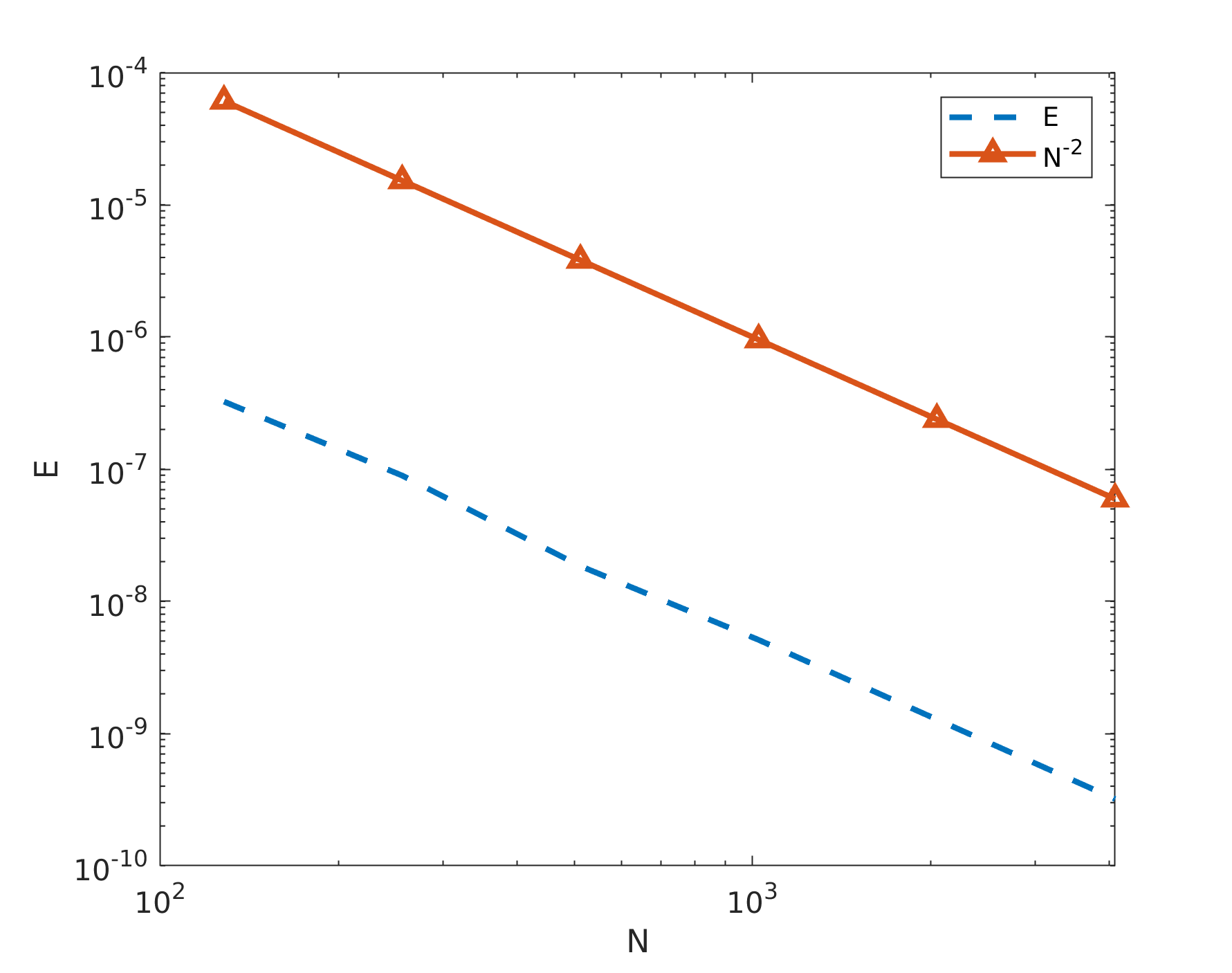}
        \caption{Plot of numerical convergence of the heat equation versus number of grid points showing a $O(\dx^2)$ accuracy.}
    \label{fig2:convergenceHeatNumerical}
\end{figure}

This process is not always possible due to many cases where there is no known analytic solution to benchmark against, thus we must resort to a numerical benchmark.
The numerical benchmark methodology we will follow in this thesis is that presented in~\cite{convergeCite}.
This benchmark is performed by successively refining grids with the same initial condition and comparing matching points, matching points are easily achieved by repeatedly doubling the total number of points in the domain. 
Thus the quantity we are looking to compute is given by
\begin{equation}
E_N(\cdot, t) = \frac{1}{\Omega} \int_\Omega |(\cdot)_N({\bf{x}}, t) - (\cdot)_{N/2}({\bf{x}}, t)|dV
\end{equation}
which can then be approximated by a sum over the domain
\begin{equation}
E_N(\cdot, t) \approx \frac{1}{\Omega} \sum_{i = 1}^{N/2} |(\cdot)_N(x_{2i-1}, t) - (\cdot)_{N/2}(x_i, t)|\Delta x_{N/2},
\label{eq2:converge1D}
\end{equation}

In Table~\ref{table2:converge1DHeat} the results are presented for this convergence study along with a plot in Figure~\ref{fig2:convergenceHeatNumerical}, both of which show clear $O(\dx^2)$ accuracy.
Comparing the analytic and numerical methods it is clear that both are effective measures \Acomment{of how} the scheme performs under grid convergence.
The numerical methodology can be extended to 2D by using

\begin{multline}
E_N(\cdot, t) \approx \frac{1}{\Omega} \sum^{N/2}_{i = 1, j = 1} | ((\cdot)_N(\vecx_{2i - 1, 2j - 1}, t) + (\cdot)_N(\vecx_{2i - 1, 2j}, t) + \\  (\cdot)_N(\vecx_{2i, 2j - 1}, t) + (\cdot)_N(\vecx_{2i, 2j}, t)) / 4 - (\cdot)_{N/2}(\vecx_{i,j}, t)| \Delta x^2_{N/2}
\label{eq2:converge2D}
\end{multline}

This equation will be applied to 2D convergence studies later in this thesis.

\begin{table}
\centering
\begin{tabular}{c|c|c} 
\hline
    {$N_x$} & {$E_N$} & {$\log_2(E_N / E_{2N})$} \\
    \hline
    128   & $3.2479 \times 10^{-7}$      & 1.8598   \\
    256   & $8.9486 \times 10^{-8}$      & 2.2749    \\
    512   & $1.8491 \times 10^{-8}$    & 1.8544    \\
    1024  & $5.1138 \times 10^{-9}$    & 1.9999    \\
    2048  & $1.2786 \times 10^{-9}$    & 2.0000    \\
    4096  & $3.1965 \times 10^{-10}$    &    \\
\end{tabular}
\caption{Details of the numerical convergence study for the diffusion equation solved with forward time central difference numerical scheme.}
\label{table2:converge1DHeat}
\end{table}


\begin{figure}
    \centering
        \includegraphics[width=0.8\textwidth]{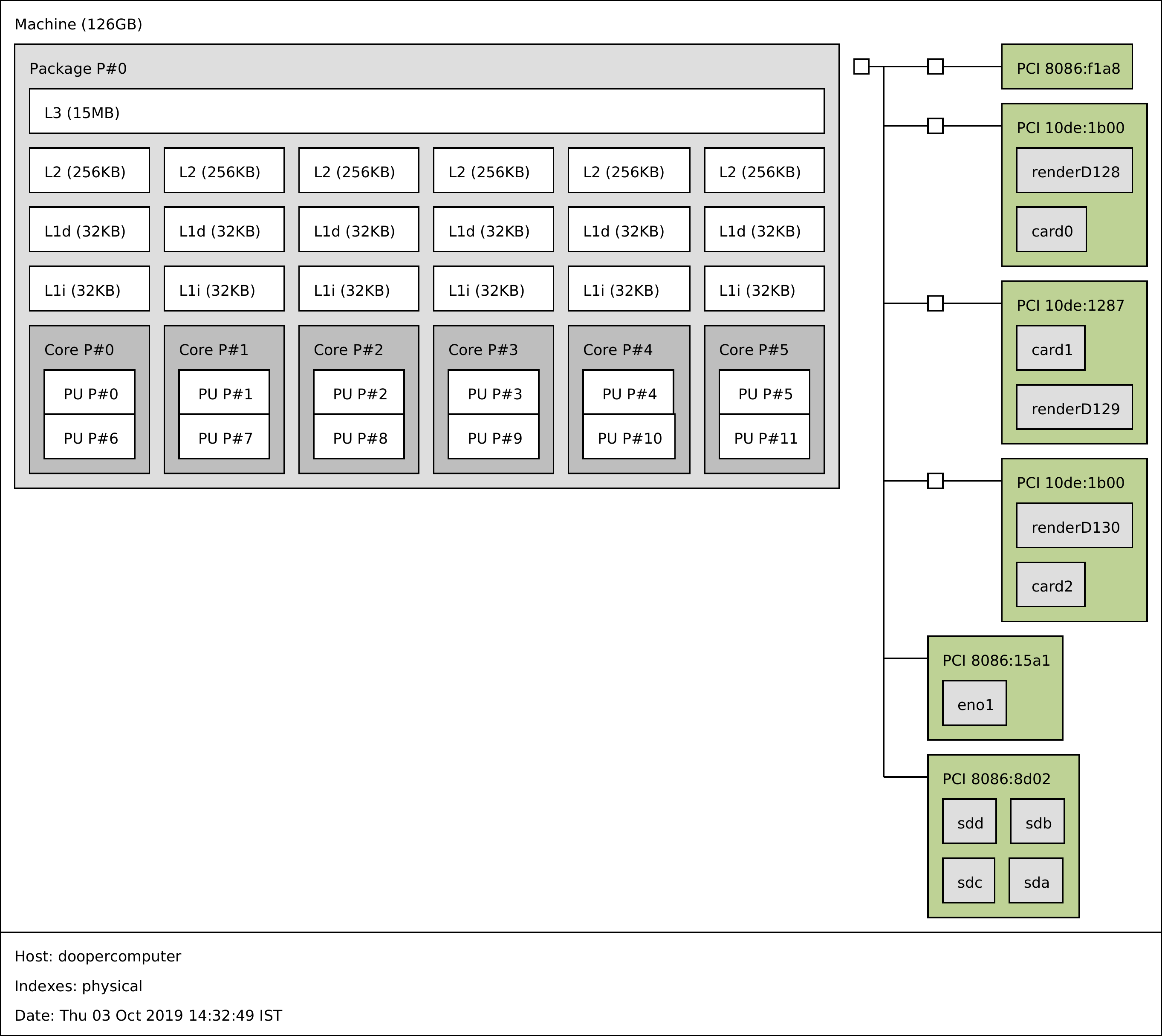}
        \caption{Diagram showing hardware layout of a standard 6--core, hyper--threaded computer with shared memory. This diagram was produced using hwloc~\cite{hwloc}, a tool for extracting the topology of computer architectures.}
    \label{fig2:dooper}
\end{figure}

\section{Parallel Architectures: CPU vs. GPU}
\label{sec2:gpuArch}
In this section we discuss traditional methods for the parallelisation of computational problems and then introduce GPUs, comparing their differences and similarities to the traditional approaches. 
We also touch on the limits of parallel computing due to the laws of physics and the law of diminishing returns.
A brief discussion on methodologies for parallelising computational problems is also presented.

\subsection{Distributed Memory Architecture}
Traditional architectures for implementing parallel applications rely on single standard CPU cores equipped with either shared or distributed memory~\cite{hager2010introduction}.
In the case of shared memory, many single CPUs share the same block of RAM, while with distributed memory each CPU has its own RAM.
An example of a shared memory computer can be seen in Figure~\ref{fig2:dooper} where 6 CPU cores, which are hyper--threaded, share one bank of 128GB RAM.
The levels of cache in the CPU can also be seen in this diagram along with the PCI slots, we will discuss the PCI slots later in the context of GPUs.
Commonly to parallelise tasks on such a machine one uses OpenMP, where many CPUs have access to one block of memory in a program.
It is clear that a bottleneck exists in terms of the total number of CPUs that can be put on a single machine and the amount of RAM that can be made available, thus limiting the applicability of such machines to smaller computational problems.

For large scale computation most systems rely on distributed memory.
At a basic level distributed memory systems can be viewed as a collection of standard independent computers, equipped with their own RAM and CPUs, connected together via Ethernet cables (or Infiniband where performance is a major concern) to a network switch.
The machines can then talk to each--other over the switch using a message passing standard such as MPI.
OpenMPI~\cite{openMPI} and MPICH~\cite{mpich2} are the dominant implementations of this standard with vendors such as Cray and Intel also supplying their own versions for specific machines.
The MPI standard is available in a broad range of languages including C/C++, FORTRAN and Python among others.
A number of switches can then be linked together to form an even larger machine, typically know as a supercomputer, examples of world leading supercomputers at the time of writing include Summit, a US machine hosted at Oak Ridge and Sunway TaihuLight, a Chinese machine hosted at the National Supercomputing Centre in Wuxi. 
Thus parallelisation on these machines is achieved via many CPUs, each performing their own portion of the program in serial, but simultaneously and in parallel with all other CPUs on the system.
Summit augments its performance with GPUs allowing for hybrid GPU/CPU programs.

Typical performance measures of supercomputers rely on the total peak number of FLOPs achieved by the machine, the standard benchmark tool is LINPACK where a large linear algebra problem is solved on the machine in a distributed manner and the necessary metrics are extracted.
Current leading machines operate on petaFLOPs levels with projects under--way to achieve exaFLOPs.
In a typical HPC problem there are three fundamental limitations to performance latency, bandwidth and Amdahl's Law which we now discuss.


\subsection{Latency and Bandwidth}
Fundamentally all signals, whether travelling on a chip or through a cable such as fibre optic or Ethernet, are limited in terms of speed by how fast an electromagnetic wave can travel in that medium.
Thus both the physical distance between CPUs in a supercomputer and the materials out of which they are made play a role.
In particular, the impact of distance on how long it takes a wave to travel, gives a clear picture of why heavily distributed computers (separate countries for example) can be inefficient for HPC applications.
This limitation due to the speed of light in computing terms is called latency and it is defined as the amount of time to send a zero byte message~\cite{hager2010introduction}.

Bandwidth is the amount of data that can be transferred per second between two locations, it is analogous to the flow rate of water through a pipe and is, like latency, hardware dependent.
Together latency and bandwidth can be combined to yield a time to transfer a given block of data between two nodes on a supercomputer.

\begin{equation}
\mathrm{Time} = \mathrm{Latency} + \frac{\mathrm{DataSize}}{\mathrm{Bandwidth}}
\end{equation}

This expression presents us with a clear picture of the importance of the hardware connections present on a given machine.
The layout of the network connections also has a large impact on moving data, this area falls under the topic of network topology, beyond the scope of this thesis~\cite{groth2005network}.

Other limits presented by physics on computational times include the number of transistors that can be placed within a given chip area, this is limited by Quantum Mechanics.
While the clock speed of CPUs, how fast computations are completed, is limited by the ability to cool the CPU efficiently.
The higher the clock speed the greater the amount of heat produced, this can lead to hardware melting if not dealt with correctly.


\subsection{Amdahl's Law}
Beyond limits due to the physical properties of the supercomputer there is also a limit on the amount of parallelisation that can be used to speed--up a particular algorithm.
There is a diminishing return on adding extra CPUs to solve a given problem, this restriction is known as Amdahl's law.
We derive it here, we begin by defining the time take to complete a computational task.

\begin{equation}
T_s = s + p
\end{equation}

Here $s$ denotes the amount of time that the serial portion of the task takes and $p$ denotes the amount of time the parallelisable portion of the task takes.
Thus we can define the amount of time this task would take using $N$ CPUs in parallel as

\begin{equation}
T_p = s + \frac{p}{N}
\end{equation}

We can then define the speed--up, $S$, of an application due to parallelisation as the quotient of these two quantities 

\begin{equation}
S = \frac{T_s}{T_p} = \frac{s + p}{s + \frac{p}{N}}
\end{equation}

Normalising using $T_s = 1$ allows us to make the substitution $p = 1 - s$ yielding a final expression

\begin{equation}
S = \frac{T_s}{T_p} = \frac{1}{s + \frac{1 - s}{N}}
\end{equation}

Thus it is easy to see that as $N \rightarrow \infty$ we get that $S$ is bound by $1 / s$, so all parallelisable programs have their execution time theoretically bound above by the serial portion of the execution.
The diminishing returns of adding too many processes to solve a problem can be seen in Figure~\ref{fig2:amdhal} where the curve flattens out as $N$ gets large, adding more when already at the plateau will simply \Acomment{waste} electricity and not improve performance.

\begin{figure}
    \centering
        \includegraphics[width=0.8\textwidth]{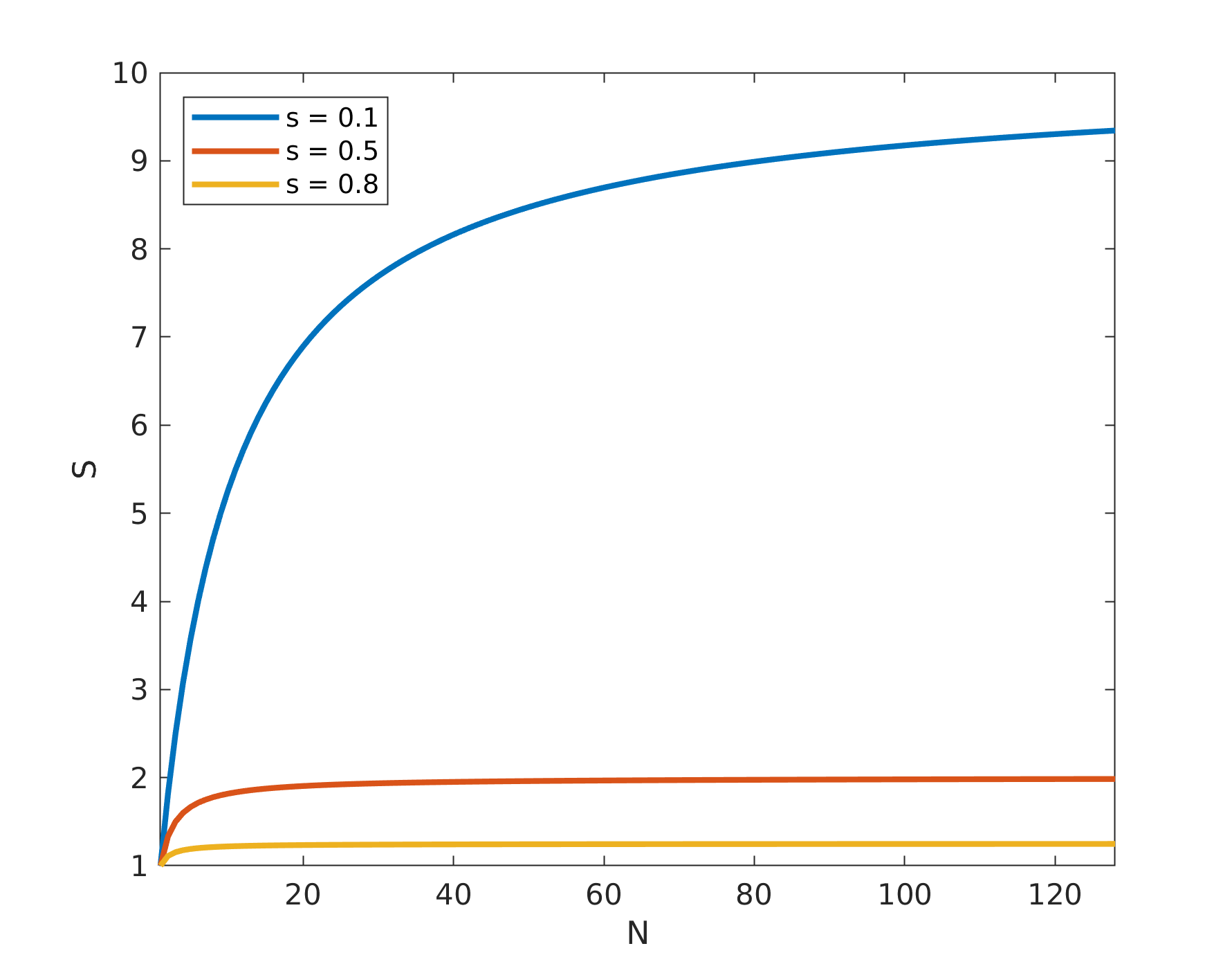}
        \caption{Diagram showing Amdhal's law for various values of $s$.}
    \label{fig2:amdhal}
\end{figure}


\subsection{GPU Architecture}
In this thesis we will focus on NVIDIA GPUs~\cite{cudaDoc}, these are composed of an array of multi--threaded Streaming Multiprocessors (SMs) on which groups of threads are executed in parallel.
The fundamental group size of threads on a NVIDIA GPU is called a warp and they each consist of 32 threads.
A key aspect of these warps and how they execute instructions is that all threads in a warp execute the same instruction at the same time in parallel on a given SM.
This is where GPUs fundamentally differ from a standard CPU, on a standard CPU a thread can complete different instructions to another without impacting the execution of \Acomment{each other}.
While on a GPU threads are not independent.
For example if an \codeword{if} statement is being computed on a GPU then threads with data not satisfying that \codeword{if} statement simply do nothing while the threads in the same warp satisfying the statement are computed.
This type of architecture is called Single--Instruction, Multiple--Thread (SIMT) while the behaviour of some threads doing nothing while other execute is called divergence.
For optimal performance it is desirable that there is no divergence within a GPU program, ensuring that all threads are performing computation when possible.

A program making use of a NVIDIA GPU invokes a kernel on the GPU which then distributes warps to each of the SMs, with multiple warps assigned to each SM.
This allows the SM to hide instruction latency in the execution of warps by rotating through each warp assigned to it, computing as instructions and data become available.
In order to maximise the instruction throughput on a GPU it is also desirable that the threads within a warp access neighbouring memory locations, coalesced memory, in RAM.
This is due to the fact that data is collected from the RAM on the GPU (known as global memory) in chunks of at least 32 bytes, thus memory accesses which are not coalesced leads to the retrieval of unused data and extra memory fetch overhead in order to retrieve the data stored elsewhere.
For memory which is reused by the same warp or group of warps, known as a block, on the same SM the programmer can make use of shared memory.
Shared memory is faster than global memory as it is located on the SM rather than the global memory which is shared among all SMs on the GPU device.

Finally we note that the GPU itself is connected to the computer over a standard PCIe lane on the motherboard and the SMs only access directly the RAM available on the GPU, thus there is a memory retrieval overhead when going between the standard CPU RAM (host RAM) and the GPU RAM (device RAM).
This introduces a performance limitation due to the available bandwidth and latency between the two sets of RAM.
In many cases, while a task may be parallelisable, one should consider the cost in moving to and from the GPU.
Should this cost be too great for the GPU to offer any significant speed--up over a standard CPU parallelisation then one should not use the GPU.
Also due to the limited size of the device RAM, data needs to be loaded and unloaded from host RAM as needed for computation.
Various strategies are available to overcome this, the data stream can be pipelined asynchronously such that data is always being loaded/unloaded while computation is taking place, thus overlapping the necessary tasks and ensuring the data throughput on the GPU is kept high.
MPI implementations are also available to make use of multiple GPUs on a distributed memory machine to further divide large problems, these also allow for direct transfer of data between GPUs within these machines.
Indeed \Acomment{hybrid} implementations of GPU programs exist, making use of both GPUs and the CPUs present on the compute nodes, maximising performance of a machine equipped with both. 
An example hybrid problem could have GPUs performing computation while CPUs handle IO in parallel.

Now that we have covered GPU architecture in brief we move to discuss briefly how problems are parallelised in HPC applications.
The CUDA API is discussed in the section following this where we examine in more depth how a GPU is used from the programmer's point of view.


\subsection{Problem Parallelisation}
In order to take advantage of parallel processing on either CPUs or GPUs one most frame the problem they wish to solve into a set of independent tasks that can be then computed in parallel.
A variety of approaches are available \Acomment{to} parallelise problems, we discuss a few here in brief.

The most straightforward of class of problem to parallelise are referred to as 'embarrassingly parallel', in other words problems that are naturally made up of independent tasks with no interdependence between said tasks other than collection of the final answers.
The most common of these problems in physics/applied mathematics is Monte Carlo analysis.
The sampling of the probability space and subsequent computation can be done independently for each sample with the final answer then collected at the end.
Beyond physics/applied mathematics problems including ray tracing and brute force password cracking can be considered embarrassingly parallel.
In this thesis we will be exploring an embarrassingly parallel problem, namely solving batches of independent 1D equations, this will be discussed further in Chapters~\ref{chapter:cuPentBatch},~\ref{chapter:efficient} and~\ref{chapter:batched}.

Commonly problems in physics/applied mathematics are not embarrassingly parallel due to interdependency between various pieces of the problem, thus a decomposition methodology is required.
Various approaches are available, the most common involves employing a domain decomposition.
The simplest domain decomposition is to physically divide the domain into smaller chunks, for example a simple division of a 2D grid into squares/rectangles, which can then each be handled by a separate CPU/GPU process.
Each process computes its portion of the solution and then the necessary halo cell information is communicated between processes using MPI at each time--step.
More sophisticated physical decomposition techniques involve using graph theory algorithms~\cite{metis} such as the greedy algorithm to divide points in the domain among processes, this method is commonly used in situations where the grid is non--uniform and thus the distribution of work is less obvious.
In this thesis we will focus on the \Acomment{Alternating Direction Implicit} (ADI) method which will allow us to decompose a 2D problem into a set of independent 1D linear equations which can then be solved in parallel.
This is discussed later in Chapters~\ref{chapter:cuSten} and~\ref{chapter:statistical}.


\section{CUDA API}
\label{sec2:cudaAPI}
In order to make use of NVIDIA GPUs the most common approach is to use their API for the C/C++ programming language, Compute Unified Device Architecture (CUDA) and is compiled with the \codeword{nvcc} compiler, itself built on top of \codeword{gcc}.
Implementations are also available for Python and FORTAN with wrappers also available for other languages.
We discuss the CUDA API briefly here to cover the terminology required to read this thesis, for a complete background the recommended resource is the CUDA documentation~\cite{cudaDoc}.

The CUDA programming model relies on the grouping of threads into blocks.
These thread blocks at the hardware level are then implemented by the compiler as a group of warps, so it is always important to keep warps in mind when programming in CUDA.
The compiler automatically assigns each block to a SM on the GPU.
It can be advantageous to make blocks multiples of 32 threads to ensure a thread block is implemented as a complete integer number of warps and no threads within a warp are wasted.

The programmer then implements a CUDA kernel, in other words a function to be executed on a GPU, which makes use of the CUDA indexing API for the thread blocks to apply operations on a grid.
This API allows for indexing threads in up to three dimensions.
The API additionally provides a means for dividing a given grid of points into a suitable number of blocks with the user only needing to specify the number of threads per block.
The returned values of this API are then given to the kernel as part of an argument before the standard function arguments, contained in a \codeword{<<< ...., ..... >>>} execution configuration syntax.

In this thesis we will concern ourselves with two primary types of memory, global and shared.
Global memory is the standard RAM on the GPU device that can be accessed by any thread in a given kernel and thus is the slowest memory type.
It is primarily used as the primary location for data storage on the GPU and is where memory can be copied to and from the host RAM, we will discuss the process of copying later.
Shared memory is located on the SM and belongs to the block of threads executing on that SM.
Due to its location it is significantly faster than Global memory but comes with a significant size restriction, it is commonly used for memory which will be \Acomment{repeatedly accessed} by a block of threads.
An example of this would be the stencil data for a finite difference scheme, \Acomment{the data} will need to be accessed multiple times during a computation and so should be copied from global memory to shared memory before computation takes place.

Traditionally the data used in GPU computations required two explicit copies of data, one allocated on the host (allocated using \codeword{malloc}) and the other on the device {allocated using \codeword{cudaMalloc}}.
The programmer would then have to explicitly copy data to and from the device for computation using the API.
The copying of data and execution of this data could then be pipelined using CUDA Streams, streams are used for managing the execution of multiple kernels on a device and data transfer to/from that device.
By assigning a kernel or copy operation to a stream the user is able to specify more clearly when particular operations should happen, giving more control over kernels which are treated asynchronously otherwise. 
More recent implementations of CUDA now include Unified Memory, this unifies the pointer address space for the host and device, thus eliminating the need for two copies of data making a program tidier (memory is allocated once using \codeword{cudaMallocManaged}).
When called on the host or device the data is automatically transferred to where it is needed, data transfer can still be handled explicitly using streams where desired to avoid page faults and increase performance.

The CUDA API also provides other useful features such as dynamic parallelism which allows for kernels to call kernels, this is useful for recursive programs such as those involving grid refinement.
Cooperative groups allow for various levels of synchronisation across a kernel from warp level to grid level barriers. 
Useful functions including thread voting and data shuffles are also implemented using cooperative groups, voting can be used to check that a certain condition has been met by all threads in a warp while shuffles can be used for optimal reduction algorithms.
CUDA is also equipped with GPU versions of many standard numerical libraries including BLAS and FFT, these carry the same names prefixed with 'cu' so cuBLAS and cuFFT.
The cuSPARSE library, used for solving sparse matrix systems, is explored in this thesis in Chapters~\ref{chapter:cuPentBatch} and~\ref{chapter:efficient} where we develop our own version of a pentadiagonal solver and show it outperforms that the solver in cuSPARSE.


\section{Cahn--Hilliard Equation}
\label{sec2:cahnEq}
In this section we discuss the Cahn--Hilliard equation as this is the chosen equation in this thesis for applying the GPU methodologies that we have developed.
It is emphasised here though that the Cahn--Hilliard was chosen as a representative equation and that the GPU methodologies are indeed general.
The Cahn--Hilliard equation was chosen as it had a representative mix of time--dependence, diffusion and non--linear terms, indeed in our examination of it in Chapter~\ref{chapter:batched} we also include an advection term, thus it presents many of the challenges present in more general PDEs.
In particular the diffusion term is fourth order so this will allow us to test large stencil sizes. 
We also discuss the equation briefly here in a general sense to provide context for work in later chapters and to avoid repetition of discussion relating to it.

The Cahn--Hilliard equation itself models phase separation in a binary liquid; when a binary fluid in which both components are initially well mixed undergoes rapid cooling below a critical temperature, both phases spontaneously separate to form domains rich in the fluid's component parts~\cite{CH_orig}. 
The domains expand over time in a phenomenon known as coarsening~\cite{Bray_advphys}.   
The equation is used as a model in polymer physics~\cite{Hashimoto} and inter--facial flows~\cite{ding2007diffuse}.  
The equation can be stated as follows
\begin{equation}
\frac{\partial C}{\partial t} = D\nabla^2(C^3-C-\gamma\nabla^2C),\qquad \vecx\in \Omega,\qquad t>0,
\label{eq2:ch}
\end{equation}
$C$ is a volume fraction tracking the abundance of the different binary fluid components, with $C=\pm 1$ corresponding to the pure phases.
Here $D$ is the diffusion coefficient and $\sqrt{\gamma}$ characterises the length scale of the regions joining the separated binary fluids.
We can also identify $\mu = C^3 - C - \gamma\nabla^2 C$ as the chemical potential thus we can view the Cahn--Hilliard equation also in conservative form as

\begin{equation}
\frac{\partial C}{\partial t} = \nabla \cdot \mathbf{j}(C)
\end{equation}

where $\mathbf{j}(C) = D \nabla \mu$.
Under suitable boundary conditions on $\partial\Omega$, the Cahn--Hilliard equation~\eqref{eq2:ch} reduces the following free energy:

\begin{subequations}
\begin{align}
F= & \int_{\Omega}\left[\tfrac{1}{2}\left(C^2-1\right)^2+\tfrac{1}{2}\gamma|\nabla C|^2\right]\mathd^n x \\
\frac{\mathd F}{\mathd t}= & -\int_{\Omega}\left|\nabla (C^3-C-\gamma\nabla^2C)\right|^2\mathd^n x.
\label{eq2:dFdt_ch}
\end{align}
\end{subequations}

Here, $n$ is the dimension of the space, which in our investigations in this thesis, will be set equal to either $n = 1$ or $n = 2$, as required.
Solutions of Equation~\eqref{eq2:ch} are characterized by a rapid relaxation to $C=\pm 1$ locally, in domains, followed by slow domain growth -- this evolution is driven by the energy-minimization~\eqref{eq2:dFdt_ch} and the conservation law $(\mathd/\mathd t)\int_\Omega C\,\mathd^D x=0$, the latter being a further consequence of the structure of Equation~\eqref{eq2:ch} and the assumed boundary conditions.

In this thesis we will examine various features of the Cahn--Hilliard equation and, for simplicity, we will focus on periodic boundary conditions.
Focus will be put on using large numbers of individual 1D simulations to examine the averaged growth across simulations of the separated regions of the fluids.
The 1D equation will then be extended to include a travelling--wave term to present a methodology on how GPUs can be used to produce a large dataset of a given parameter space.
Then using this dataset automate the production of flow--pattern maps, a methodology it is hoped can be applied more generally to two phase flows.
Finally batches of 2D simulations on GPUs will also be used to produce a dataset of $\beta$, the power of $t$ to which the separated regions of fluid grow in 2D.
Here new results will be presented showing how in a finite sized domain the value of $\beta$ does not have a constant value of $1/3$ as per the literature but is a random variable that can be drawn from a statistical distribution around $1/3$.
These results will also be discussed in the context of LSW theory for Ostwald Ripening and extended to the Cahn--Hilliard--Cooke equation.

\lhead{\emph{Chapter 3}} 
\chapter{cuSten -- CUDA Finite Difference and Stencil Library}
\label{chapter:cuSten}
\Acomment{This chapter presents the content of the cuSten paper 'cuSten -- CUDA Finite Difference and Stencil Library'~\cite{gloster2019custen}. 
The cuSten library is developed to help a programmer deal with implementing a finite--difference stencil using CUDA, hiding much of the work needed in the back--end.
It is designed to treat stencils much like cuBLAS and cuSPARSE treat linear algebra algorithms, simplifying the interface and eliminating the re--writing of code.}

\section{Introduction}
\label{sec3:intro}
To discretise any PDE system numerically, several standard approaches exist, including the finite-difference, finite-volume, and finite-element methods.   
For definiteness, this thesis focuses on the Finite--Difference Method as presented in Section~\ref{sec2:fdmethod}, however, it can be applied in any situation requiring stencil-based operations.

There is a large field of literature associated with the implementation Finite--Difference Methods using CUDA, a few examples include~\cite{Micikevicius093dfinite,waveFD,elsStencil,CUDAthesis}.
This literature commonly explains how to approach the problem of implementing a finite--difference scheme using CUDA but yet the authors of the relevant books/papers/articles provide no publicly available library or code with their papers that a reader readily use in their own project, thus requiring the reader to rewrite code that repeats work already done elsewhere.
Libraries providing PDE solvers and other stencil--based computations exist, such as \cite{libGeo} and indeed some approaches that can generate code for the programmer~\cite{autoGenZhang,autoGenHolewinski}, but these libraries and approaches can be limiting due to investment cost in learning essentially a full software package or new method.
Indeed the PETSc library~\cite{petscwebpage,petscuserref,petscefficient}, which supports finite--difference methods, also now provides a GPU implementation but this limits the program to be written mostly using that library's API (thus limiting flexibility), and requires the programmer to also develop knowledge of cuTHRUST~\cite{cudaDoc} to implement the GPU aspects of the library effectively. 
It is noted that the PETSc web-page~\cite{petscwebpage} documents some difficulties associated with using GPUs effectively in PETSc. 
As such, we present cuSten as a computational framework complementary to PETSc, readily deployable by a programmer interested in Physics applications, with relatively low overhead in terms of learning to implement large complex libraries.

Common problems at development time of finite--difference programs include readjusting boundaries when changing finite--difference schemes or ensuring the correct data has been loaded onto the GPU at the time of computation, both of these are dealt with by cuSten.
It aims to overcome these difficulties along with addressing the problems with the above problems by providing a new software tool, introducing a simple set of four functions (three in many cases) for the programmer to implement their finite difference solver.
These functions are accessed much like cuBLAS or cuSPARSE giving freedom to the programmer to build the program as they choose but eliminating the need to worry about the finite difference implementation specifics.
This tool allows a programmer to simply input their desired finite-difference stencil and the direction in which it should be applied and then the rest of the implementation, including the domain decomposition, boundary positioning and data handling are wrapped into functions which are easily called.
This approach reduces the development time necessary for implementing new systems/solvers and provides a robust framework that does not involve a black-box-approach to the solution from the programmer.  
Furthermore, the approach does not require a major overhead of time to invest in learning/implementing a new tool.

It is not intended that the code produced by this tool be the most efficient implementation of a given scheme versus a dedicated code for a specific problem, but it is intended that the development time of a code is drastically cut by removing the need for the programmer to do unnecessary work at development time.
We include a comprehensive example of the application of this tool in Section~\ref{sec3:ADI}.
Here the ease of implementation of the cuSten library to solve the 2D Cahn--Hilliard equation is highlighted.
A benchmark of cuSten versus a serial implementation of the same is also included to highlight the improvements in performance due to parallelisation on the GPU. 
2D problems are the main focus of this new tool; 2D problems provide a test-bed for the development numerical algorithms which can then be extended to 3D, where debugging, testing, and validation are more time-consuming.  
The extension of the present method to 3D is discussed in Section~\ref{sec3:con} below.
In terms of floating point precision this library focuses on the use of the double floating point type as in most application it is desirable to have 64 bit precision when solving PDEs, the source code is easily modified using a standard text editor with find and replace to change to other data types if so desired (this is discussed also in Section~\ref{sec3:con} below).

In Section~\ref{sec3:arch} we introduce the underlying architecture of the cuSten library, including how it uses streams and events for optimal memory management. 
Then in Sections~\ref{sec3:fun} and~\ref{sec3:ex} we talk the reader through the cuSten API along with examples and where to find all the source code within the library should they wish to edit it, the API is further explained in the Doxygen documentation included with the library itself.
Section~\ref{sec3:ADI} presents the implementation of a full 2D Cahn--Hilliard solver along with a GPU versus CPU benchmark of the cuSten library and finally concluding discussions with potential future work are presented in Section~\ref{sec3:con}.


\section{Software Architecture}
\label{sec3:arch}
The library in this chapter makes use of the CUDA API for C/C++ discussed in Section~\ref{sec2:cudaAPI} of this thesis.
The tool is built on two main sets of code, one handling the creation and destruction of the \codeword{cuSten_t} data type which handles all of the programmer's inputs (found in \codeword{/src/struct}) and the other handling the compute kernels (found in \codeword{/src/kernels}). 

At the top level will be the main program solving whatever PDE is of concern to the programmer and the library is called through the header \codeword{cuSten.h}.
The programmer provides the necessary memory to the library using Unified Memory along with the stencil details, these will be detailed in Section~\ref{sec3:fun}. 
Unified Memory was chosen as it simplifies the handling of memory in the library and interfacing with the rest of programmer's code. 
The ability to address data beyond the device memory limit is also useful in cases where not all the data required for a given program fits in device memory, the movement of memory on and off the device is handled by the cuSten library as explained in the following paragraph.

To take advantage of Unified Memory the library allows the programmer to divide their domain into `tiles' such that each tile will fit into the device RAM.
Each tile is a chunk of the total domain in the y direction to ensure the memory is contiguous.
The tiles are loaded onto the GPU in time for the kernel to be launched such that there are no GPU page faults.
The programmer also has the option to unload the tiles onto host RAM after the computation is completed on a given tile, this can be for IO or if the programmer needed to free device RAM for the next tile or a new task.
This system ensures that loading/unloading data and computation is implemented as a pipeline using separate streams for data loading/unloading and kernel launches ensuring that everything overlaps and ensuring that as little time is wasted retrieving memory over the PCIe bus which is a bottleneck to a memory bound program.
Finite difference programs, such as the ones discussed in this article, are typically memory bound as only a few computations are required per point in the array yet the memory overhead can be quite large when several variables need to have stencils applied to them. 
Events are used to ensure the data has been loaded prior to the launch of a kernel. 

The programmer has the choice of supplying a standard linear stencil or a function pointer with additional input coefficients to the library, examples of which are discussed in Sections~\ref{sec3:weights} and~\ref{sec3:pointer} respectively.
Within the compute kernel blocks of data with suitable boundary halos are loaded into shared memory.
The stencil or function is then applied to the block with each thread calculating the output for its position.
When this has completed the data is then output as blocks into the memory provided by the programmer for output, the same memory cannot be used for both as the blocks require overlapping data and thus cannot use already output values.

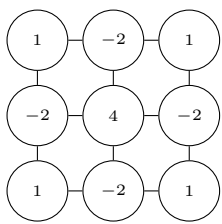
\begin{figure}
	\centering
\begin{tikzpicture}
  \stencilpt{-1,-1}{i-5}{$1$};
  \stencilpt{0,-1}{i-4}{$-2$};
  \stencilpt{1,-1}{i-3}{$1$};
  \stencilpt{-1,0}{i-2}{$-2$};
  \stencilpt{0,0}{i-1}{$4$};
  \stencilpt{1,0}{i}{$-2$};
  \stencilpt{-1,1}{i+1}{$1$};
  \stencilpt{0,1}{i+2}{$-2$};
  \stencilpt{1,1}{i+3}{$1$};
  \draw (i-5) -- (i-4)
        (i-4) -- (i-3)
        (i-5)   -- (i-2)
        (i-4) -- (i-1)
        (i-3) -- (i)
        (i-1) -- (i)
        (i)   -- (i+3)
        (i+1) -- (i+2)
        (i-2) -- (i-1)
        (i+3) -- (i+2)
        (i-1) -- (i+2)
        (i-2) -- (i+1);
\end{tikzpicture}
\caption{Typical stencil for a second order accurate cross derivative $\frac{\partial^4}{\partial x^2 \partial y^2}$.}
\label{fig3:stencil}
\end{figure}


\section{Software API}
\label{sec3:fun}
The programmer can use up to four functions for the application of any given finite difference stencil, in most cases only three are required.
The possible stencil directions include x, y and xy, where xy allows for cross derivatives which require that diagonal/off-diagonal information is available for the stencil to be completed, the stencil size is not limited in any direction and can be any desired shape, for example the stencil can be a $5 \times 3$ in dimension and use every point within that area.
Indeed the area for the stencil need not be centred at $(i, j)$ and it can extend in any direction more than another as necessary, this can be done be specifying non symmetric quantities for the number of points required left/right or top/bottom of $(i, j)$ in the stencil.
A typical stencil for a second order accurate cross derivative $\frac{\partial^4}{\partial x^2 \partial y^2}$ is shown in Figure~\ref{fig3:stencil}, this stencil also appears in the linear biharmonic term for the Cahn--Hilliard solver presented in section~\ref{sec3:adiApply}.

Each direction then comes with a periodic and non-periodic boundary option along with a choice between supplying just a set of weights (example in \ref{sec3:weights}) which are applied linearly or a function pointer (examples in Section~\ref{sec3:pointer} and~\ref{sec3:adiApply}) that can be used to apply more sophisticated schemes.
In order to apply non--periodic boundary conditions the programmer will need to write their own boundary condition kernel, this was done to keep the library flexible to the programmer's desired numerical scheme which may require more sophisticated boundaries than simple Neumann/Dirichlet conditions. 
The cuSten library simply leaves the data in the boundary cells untouched when performing a non--periodic computation.
The naming convention for the functions available in the library is 
\[\codeword{cuSten[Create/Destroy/Swap/Compute]2D[X/Y/XY][p/np][BLANK/Fun]}\]
The descriptions for the options are as follows:

\codeword{Create}: This will take the programmer inputs such as the stencil size, weights, number of tiles to use etc. and return the \codeword{cuSten_t} ready for use later in the code.

\codeword{Destroy}: This will undo everything in create, freeing pointers and streams etc. To be used when the programmer has finished using the current stencil, for example at the end of a program.

\codeword{Swap}: This will swap all relevant pointers, in other words swap the input and output data pointers around so the stencil can be applied to the updated stencil after time-stepping. The need for this function is generally dependent on the overall numerical scheme a programmer is using, it is not needed in all situations.

\codeword{Compute}: This will run the computation applying the stencil to the input data and outputting it to the appropriate output pointer.

\codeword{X}: Apply the stencil in the x direction.

\codeword{Y}: Apply the stencil in the y direction.

\codeword{XY}: Apply the stencil in the xy direction simultaneously (for situations with cross derivatives etc.). The library will account for corner halo data in this situation.

\codeword{p}: Apply the stencil with periodic boundary conditions.

\codeword{np}: Apply the stencil with non-periodic boundary conditions, this leaves suitable boundary cells untouched for the programmer to then apply their own boundary conditions.

\codeword{Fun}: Version of the function to be used if supplying a function pointer, otherwise leave blank.

The functions are then called in order of \codeword{Create}, \codeword{Compute}, \codeword{Swap} (if necessary) and then \codeword{Destroy} at the end of the program. 
Complete usage examples are found in the next section with further examples found in \codeword{examples/src}. 
The complete API can be found in the Doxygen documentation, see \codeword{README} on how to generate this.


\section{Examples}
\label{sec3:ex}
In this section we provide an overview of using library.
We present three examples.  
The first is an implementation using linear stencil weights.  
The second involves a function pointer instead.  
The third example is at the level of a detailed physics problem (advection in Fluid Mechanics), and is included here to demonstrate to the user how to modify the source code as necessary.
These three examples (and more) can be found in \codeword{examples/src}. 
The \codeword{README} provides compilation details.
In all examples in the repository we take derivatives of various trigonometric functions as these are easy to benchmark against in periodic and non-periodic domains.

\subsection{Standard Weights}
\label{sec3:weights}
We present here the example \codeword{2d_x_np.cu}, it is recommended to have this example open in a text editor to follow along. 
In this example we implement an 8th order accurate central difference approximation to the second derivative of $\sin(x)$ in the $x$ direction.
The domain has 1024 points in $x$ and 512 points in $y$, set by nx and ny respectively with the domain size lx set to $2 \pi$.

Unified memory is allocated with \codeword{dataOld} set to the input $\sin(x)$ and answer set to $-\sin(x)$, \codeword{dataNew} is zeroed to ensure correct output.
We choose to implement this scheme on compute device 0 by setting \codeword{deviceNum} and implement the scheme using a single tile, setting \codeword{numTiles} to 1.
The stencil is then implemented by setting the parameters \codeword{numSten}, \codeword{numStenLeft} and \codeword{numStenRight} along with providing an array of the stencil weights the same length as \codeword{numSten}.
\codeword{numSten} is the total number of points in the stencil, in this case 9, while \codeword{numStenLeft/Right} are the number of points in the left and right of the stencil, both 4 in this case. 
A \codeword{cuSten_t} named \codeword{xDirCompute} is then declared and fed along with the above parameters into custenCreate2DXnp, this then equips \codeword{cuSten_t} with the necessary information. 
The ordering of parameters to be fed into \codeword{cuStenCreate2DXnp} can be found in both the Doxygen documentation and \codeword{cuSten/src/struct/cuSten_struct_functions.h}

The computation is  run using \codeword{cuStenCompute2DXnp(&xDirCompute, HOST)} where the \codeword{HOST} indicates we wish to load the data back to the host memory after the computation is completed, \codeword{DEVICE} if you wish to leave it in device memory.
Finally the result is output along with the expected answer to \codeword{stdout}, the 4 cells on either side in the x direction will be $0.0$ due to the boundary, these would then be set by the programmer using suitable boundary conditions in a full solver.
Then the \codeword{cuStenDestroy2DXnp} function is called to destroy the \codeword{cuSten_t}.
Memory is then freed in the usual manner.

\subsection{Function Pointer}
\label{sec3:pointer}
Now we present the function pointer version of the above example, \codeword{2d_x_np_fun.cu}, again is is recommended to have a text editor open with the code to follow along.
Many of the parameters are the same as before except this time we remove the weights and replace them with coefficients that are then fed into the function pointer by the library.

The function pointer in this case implements a standard second-order accurate central-difference approximation to the second derivative of $\sin(x)$. 
We supply \codeword{numSten}, \codeword{numStenLeft} and \codeword{numStenRight} as before but now we also need \codeword{numCoe} to specify how many coefficients we need in our function pointer.

Our function pointer is of type \codeword{devArg1X}, where the 1 indicates how many input data sets are required. 
Each thread in a block will call the function and it returns the desired output value for that thread, each index in the array has one thread assigned to it.
The inputs are pointers to the input data, the coefficients and the index location in the stencil
\[
\codeword{CentralDifference(double* data, double* coe, int loc)}
\]
The central-difference scheme is implemented in a standard way with indexing done relative to loc, the coefficient in this case is set to $1.0 / \Delta x ^2$ as is standard.
A key point to notice, is that the programmer must allocate memory for the function pointer on the device, this can be seen on line 131 and 132 of the example code prior to calling the \codeword{Create} function.

The rest of the access to the API is then the same as before except some of the inputs change and there is 
a \codeword{Fun} at the end of each function name, for example \codeword{cuStenCreate2DXnpFun}. 
We will see later in Section~\ref{sec3:ADI} how function pointers provide us with a powerful tool to apply stencils to non-linear quantities, in particular we will see this with the cubic term of the Cahn--Hilliard equation to which we wish to apply a Laplacian. 

\subsection{Advection}
The library also comes with an extra variant of the above functions \codeword{2d_xyADVWENO_p} in which a 2D periodic advection WENO scheme has been implemented by modifying the \codeword{2DXYp} source code. 
This is included as an example to show the user how to modify the source code as necessary to more specific needs or in situations where the function pointer may not meet requirements, for example in this situation where extra data needed to be input in the form of $u$ and $v$ velocities. 
The files can be found in the cuSten/src folder with how its called in \codeword{examples/src/2d_xyWENOADV_p.cu}.

A brief overview of the modifications made to the \codeword{2DXYp} code are as follows:

\begin{itemize}
\item The stencil dimensions are now set automatically when the creation function is called.
\item The $u$ and $v$ velocities were linked to the cuSten type with appropriate tiling.
\item Additional asynchronous memory copies were included in the memory loading portion of the code to ensure the velocities are present on the device at the required time.
\item The corner data copying to shared memory blocks was removed from the kernel as it is no longer required.
\item The standard stencil compute was removed and replaced with a device function call to a WENO solver, details of the solver can be found in~\cite{fedkiwBook}.
\end{itemize}


\section{cuPentCahnADI}
\label{sec3:ADI}
In this section we show how the cuSten library can be used as part of a larger program developed using the cuPentBatch~\cite{gloster2019cupentbatch} solver, a batched pentadiagonal matrix solver, details of which can be found in Chapter~\ref{chapter:cuPentBatch}.
We also provide a benchmark at the end of the section to show how cuSten performs versus a serial implementation.
A discussion of the Cahn--Hilliard equation can be found in Section~\ref{sec2:cahnEq}, here we focus on its implementation in CUDA using the cuSten library.

\subsection{Discretisation}
For simplicity, we focus on the case where $\Omega=(0,2 \pi)^\mydim$, with periodic boundary conditions applied in each of the $\mydim$ spatial dimensions, (in this thesis we will focus on $\mydim = 1, 2$).  The method of solution we choose is based on the ADI method presented in~\cite{stableHyper} for the linear hyperdiffusion equation -- we extend that scheme here and apply it to the non-linear Cahn--Hilliard equation as follows:

\begin{subequations}
\begin{align}
{\bf{L_x}}w = -\frac{2}{3}(C^{n} - C^{n-1}) - \frac{2}{3} \Delta t \nabla^4 \bar{C}^{n+1} + \frac{2}{3}D \Delta t \nabla^2(C^3 - C)^n \\
{\bf{L_y}}v = w \\
C^{n + 1} = \bar{C}^{n+1} + v,
\end{align}%
\label{eq3:ch_adi}%
\end{subequations}%
Where ${\bf{L_x}} = {\bf{I}} + \frac{2}{3} D \gamma \Delta t \partial_{xxxx}$ and similarly for ${\bf{L_y}}$.
\Acomment{
In Equation~\eqref{eq3:ch_adi} we invert the batches of ${\bf{L_x}}$ and ${\bf{L_y}}$ matrices in parallel using cuPentBatch with one CUDA thread per system.
More details of this methology are presented in \cite{gloster2019cupentbatch}, this paper is also presented in Chapter~\ref{chapter:cuPentBatch} as part of the work completed during this thesis.
We transpose the matrix when changing from the x direction to y direction sweep to ensure the data is in the proper interleaved format.
}
To deal with the periodic element of the inversion the method is the same as in Reference~\cite{gloster2019cupentbatch,navon_pent}. 
To recover the initial $n-1$ time step we simply set this to the initial condition and appropiately update and time steps there after. 
The derivatives are discretised using standard second order accurate central differences discussed in Section~\ref{subsec2:fdStencil} and $\dx = \dy$ for the a uniform grid.

\subsection{Application of cuSten}
\label{sec3:adiApply}
The code for the example can be found with the repository in the \codeword{cuCahnPentADI} folder, supplied also in this folder is a \codeword{Makefile} to compile the files and a Python script to analyse the results which we present in Section~\ref{sec3:results}.
cuSten is applied for all of the finite-difference elements of the code excluding the matrix inversion where we use cuPentBatch. 
Between lines 148 and 190 we can see an application of a more sophisticated function pointer than presented previously in Section~\ref{sec3:pointer}, here we apply the Laplacian to the right-hand-side (RHS) non-linear term $C^3 - C$.
The coefficients are declared between lines 481 and 516, the non--linear term is a $5 \times 5$ stencil.
This shows a clear example of ease of use of the function pointers and the easy swap in/out of values. 
Note how the indexing starts from the top left of the stencil and sweeps left to right in i, row by row in j for indexing. 

The linear terms for the RHS are implemented using standard weighted schemes, the scheme uses a $5 \times 5$ stencil, this and the non--linear term highlight one of the key features of the library with the easy change in stencil size and the boundaries are dealt with automatically (in this case periodic).
The additional static functions at the start of the file apply the time stepping parts of the algorithm and combination of terms to set the full RHS. 
Output is done using the standard HDF5 library, this is required for the cuPentBatchADI program but not the cuSten library itself.

\begin{figure}
    \centering
        \includegraphics[width=0.6\textwidth]{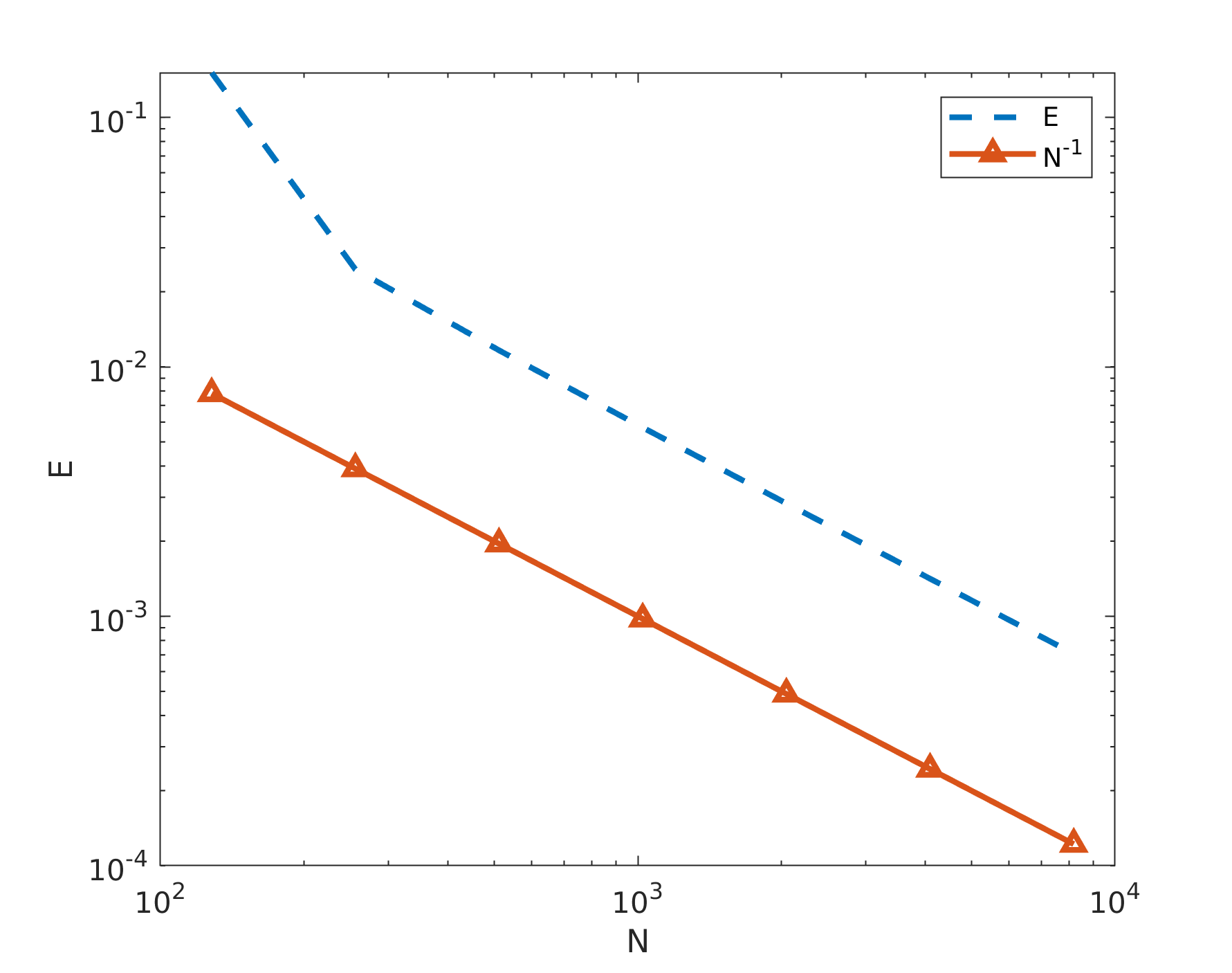}
        \caption{Plot showing the convergence of the ADI numerical scheme for the Cahn--Hilliard equation.}
    \label{fig3:convergence2D}
\end{figure}

\subsection{Convergence of Numerical Scheme}
We first begin by performing a convergence study of the numerical scheme using the methodology presented in Section~\ref{subsec2:numAcc} for 2D solutions.
For the convergence study the initial condition is set as

\begin{equation}
C(\vecx, 0) = \epsilon \tanh(r - \pi)
\end{equation}

where $r = \sqrt{x^2 + y^2}$ to ensure that one bubble will form in domain and $\epsilon = 1 \times 10^{-6}$. 
Other parameters are set as follows $\Omega = [2\pi, 2\pi]$, $D = 1$ and $\gamma = 0.01$ with a final simulation time of $T = 10$.
We choose a time step of $\Delta t = 0.1 \Delta x$ where $\Delta x$ is uniform grid spacing in the $x$ and $y$ directions.

\begin{table}
\centering
\begin{tabular}{c|c|c} 
\hline
    {$N_x$} & {$E_N$} & {$log_2 (E_N / E_{2N})$} \\
    \hline
    128   & 0.1510  & 2.6187 \\
    256   & 0.0246  & 1.0785 \\
    512   & 0.0116 & 1.0281 \\
    1024  & 0.0057 & 1.0107 \\
    2048  & 0.0028 & 1.0022 \\
    4096  & 0.0014 & 0.9996  \\
    8192  & 0.0007 &      \\

\end{tabular}
\caption{Table showing rates of convergence for simulations of the 2D Cahn--Hilliard equation using the numerical scheme in equation~\ref{eq3:ch_adi}.}
\label{table:converge2D}
\end{table}

We can see from the results presented in Figure~\ref{fig3:convergence2D} and Table~\ref{table:converge2D} that this scheme does indeed converge and is a first order accurate scheme. 
The hyperdiffusion version of the scheme is second order accurate but the addition of the non-linear Cahn-Hilliard term reduces it to first order accuracy.
\Acomment{This loss of accuracy is likely due to the addition of an explicit non-linear term to an implicit linear method.}

\begin{figure}[ht]
  \centering
    \includegraphics[width=0.7\textwidth]{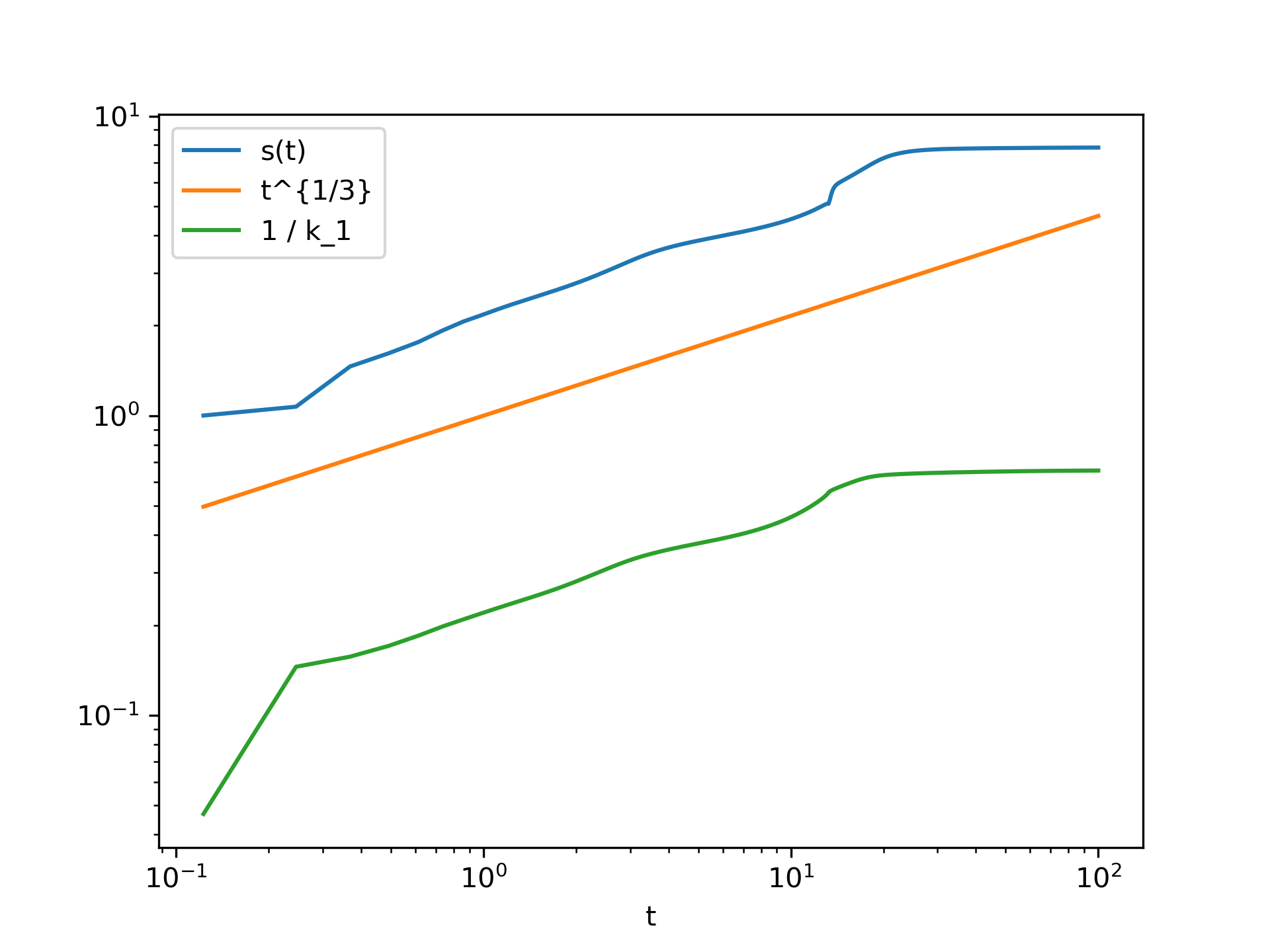}
    \caption{Plot showing $s(t)$ and $k_1$ as functions of $t$. We can see the clear $t^{1/3}$ behaviour as expected in each.}
  \label{fig3:scaling}
\end{figure} 

\subsection{Numerical Results}
\label{sec3:results}
In order to analyse the performance of the code we use two standard tests to quantify the coarsening rate \cite{LennonAurore}. 
First we have the quantity $s(t)$ which can be defined as 
\begin{equation}
s(t) = \frac{1}{1 - \langle C^2\rangle}
\end{equation} 
\Acomment{where} $\langle \cdot\rangle$ denotes the spatial average, which we calculate by a simple integration over the domain using Simpsons's rule.
Secondly we plot $1 / k_1(t)$, which also captures the growth in length scales, where $k_1$ can be defined as 
\begin{equation}
k_1(t) = \frac{\int d^nk |\hat{C}|^2}{\int d^nk |{\bf{k}}|^{-1} |\hat{C}|^2}
\end{equation}
with the hat denoting the Fourier Transform. 
We run the simulation to a final time $T=100$ with $n_x = n_y = 512$ points, the time--step size is set at $\Delta t = 0.1 \Delta x$.
The initial conditions are a random uniform distribution of values between $-0.1$ and $0.1$, we have set the coefficients $D$ and $\gamma$ to $1.0$ and $0.01$ respectively.  
The initial condition is chosen to mimic a `deep quench', where the system is cooled suddenly below the critical temperature, which allows for phase separation to occur spontaneously~\cite{Zhu_numerics}.
The quantities $s(t)$ and $1/k_1(t)$ are plotted in Figure~\ref{fig3:scaling} as a function of $t$ with a reference line of $t^{1/3}$ included as both should scale proportionally to this.
We can see clear match between our two quantities and $t^{1/3}$. 
Finite-size effects spoil the comparison between numerics and theory towards the end of the computation, as by that time the $(C=\pm 1)$-regions fill out the computational domain.  
Figure~\ref{fig3:contour_all} \Acomment{illustrates} the behaviour of the solution in space and time: the system clearly evolves into extended regions where $C=\pm 1$, which grow over time, consistent with Figure~\ref{fig3:scaling} and the established theory~\cite{LS}, in particular we can see the development of the finite size effects.
\Acomment{Indeed we will explore the physics of these finite size effects further in Chapter~\ref{chapter:statistical}, in particular presenting new results focusing on statistics relating of the growth rates of the separated regions which on average grow in proportion to $t^{1/3}$.}  
\begin{figure}
    \centering
    \begin{subfigure}[b]{0.49\textwidth}
        \centering
        \includegraphics[width=1.0\textwidth]{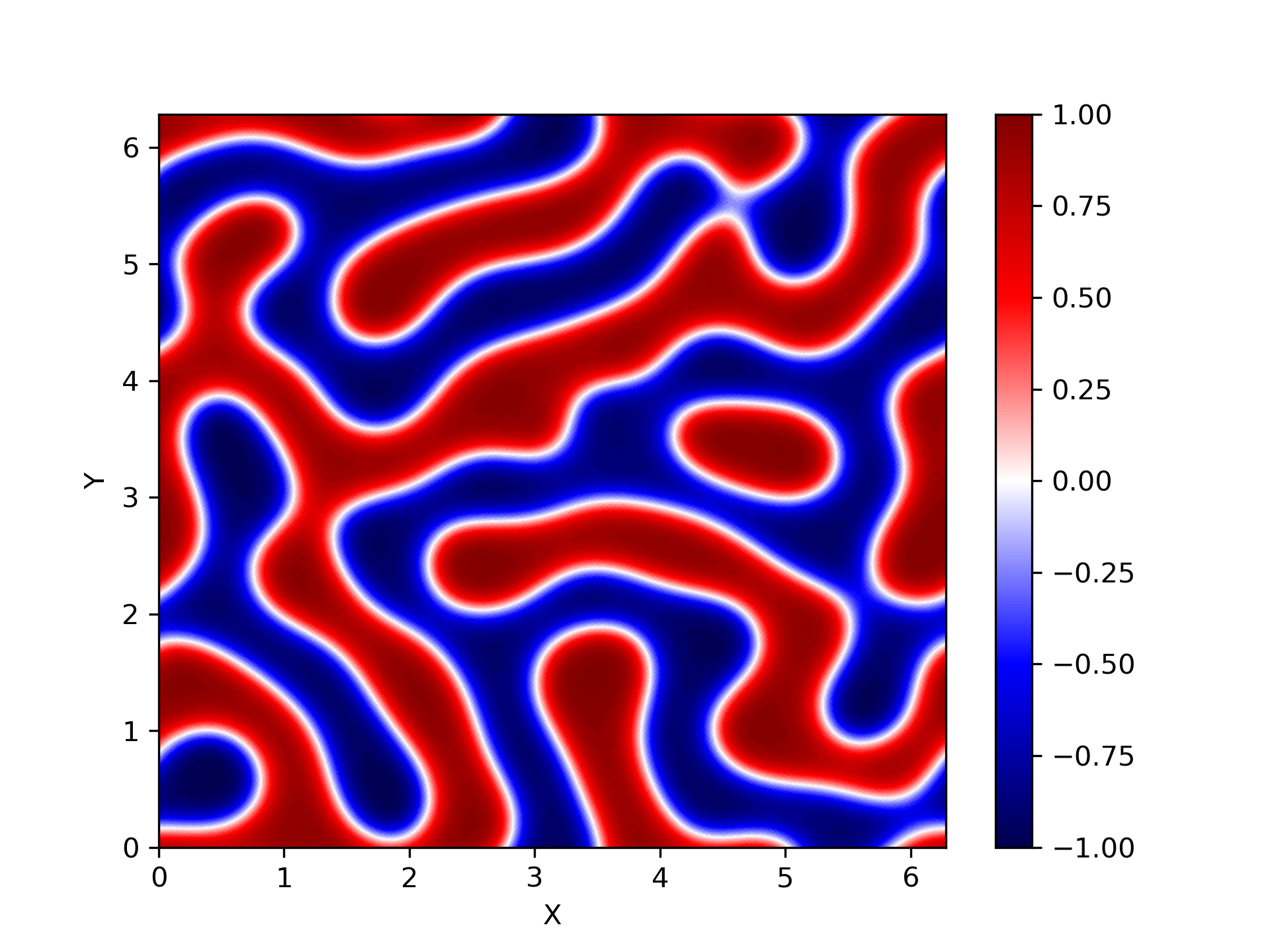}
    \caption{$t = 0.7363107782$}
    \end{subfigure}
    \hfill
    \begin{subfigure}[b]{0.49\textwidth}
        \centering
        \includegraphics[width=1.0\textwidth]{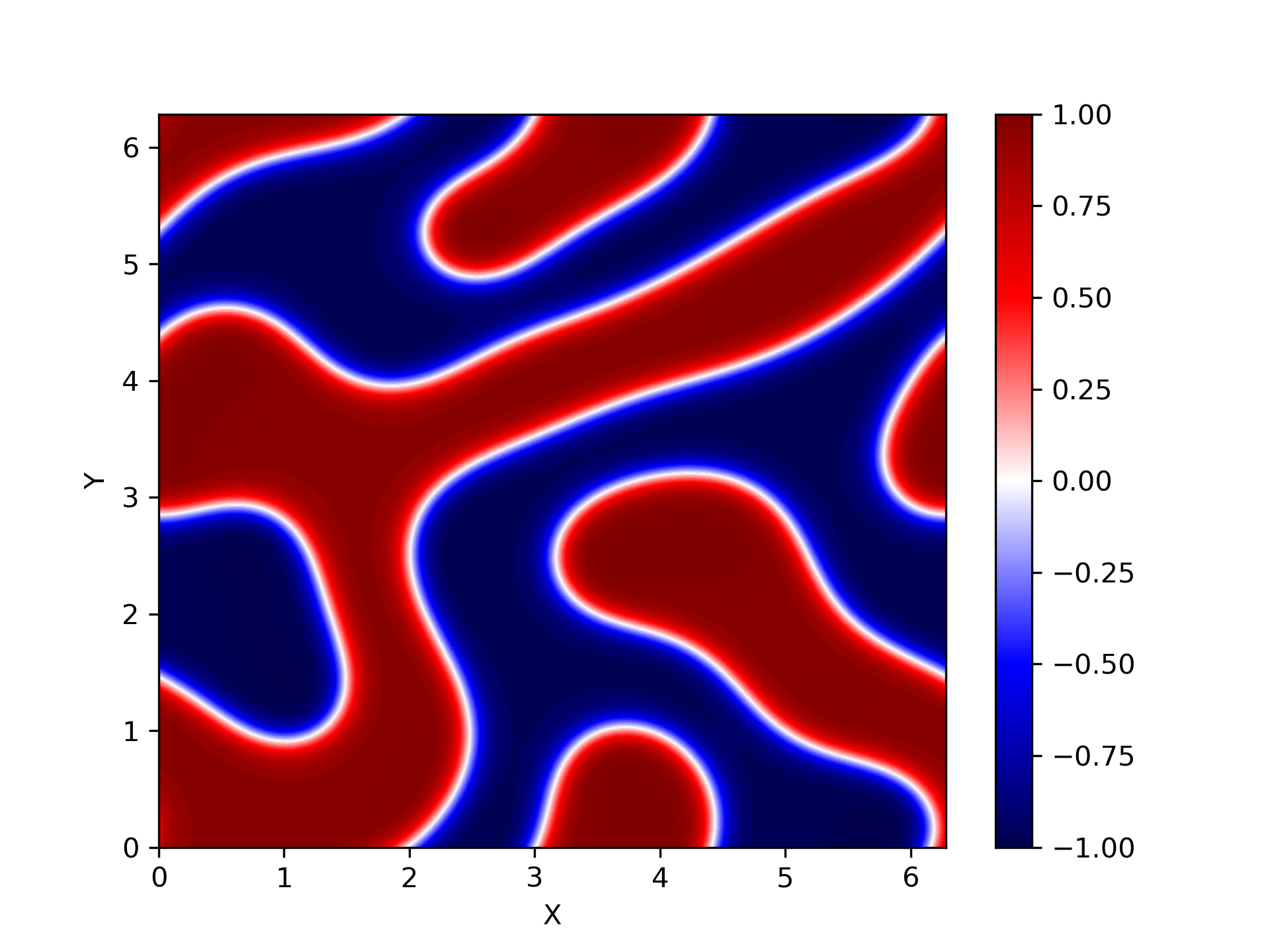}
    \caption{$t = 3.1906800388$}
    \end{subfigure}
    \vskip\baselineskip
    \begin{subfigure}[b]{0.49\textwidth}
        \centering
        \includegraphics[width=1.0\textwidth]{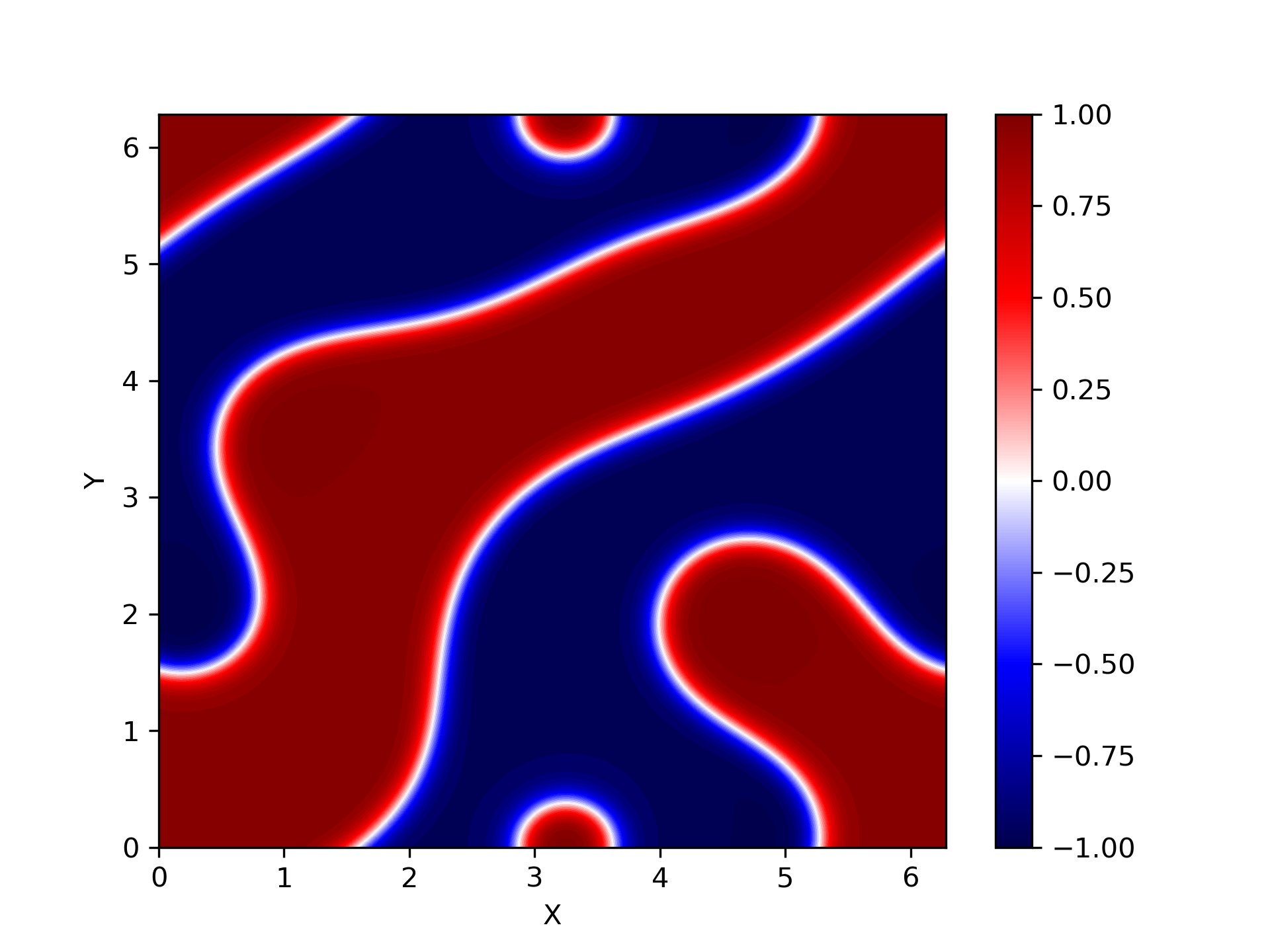}
    \caption{$t = 12.3945647661$}
    \end{subfigure}
    \begin{subfigure}[b]{0.49\textwidth}
        \centering
        \includegraphics[width=1.0\textwidth]{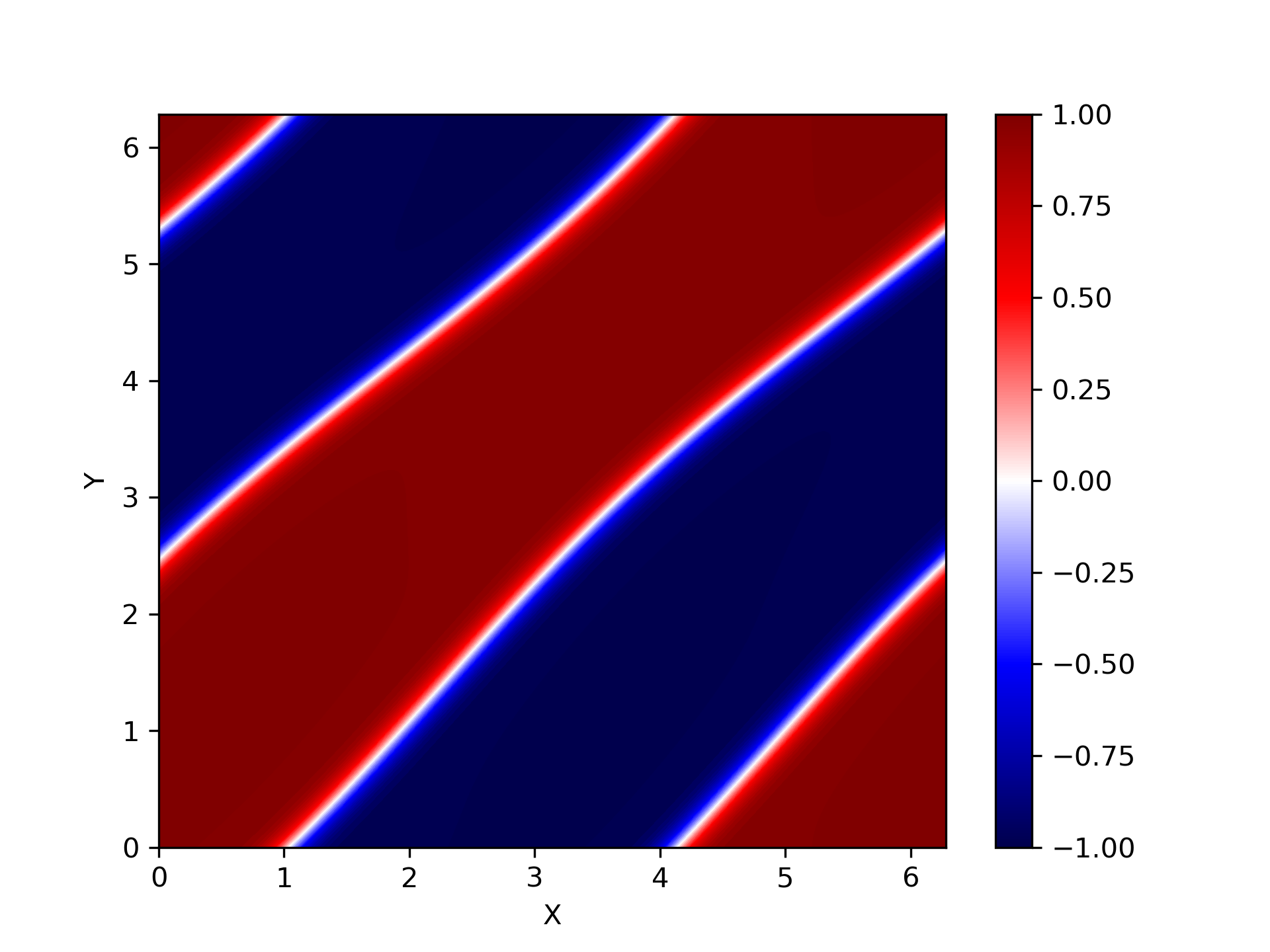}
    \caption{$t = 73.7537962816$}
    \end{subfigure}
        \caption{\Acomment{Contour plots showing the Cahn--Hilliard equation with a symmetric mixture at various time--steps.}}
    \label{fig3:contour_all}
\end{figure}

\subsection{Benchmark of cuSten}
\Acomment{
In this section we benchmark the cuPentCahnADI program, which uses cuSten and cuPentBatch, against a serial version of the program running on a CPU.
The GPU used in this benchmark is an NVIDIA Titan X Pascal and the CPU is an Intel i7--6850K which has 6 hyper-threaded cores operating at $3.6 GHz$.
The benchmark is performed by measuring the time to time--step the simulation to a final time of $T = 10$, scaling $N$ where $N \times N$ is the total size of the domain.
As the number of time-steps in the simulation is proportional to $N$ the entire computation is $O(N^3)$, in other words $O(N)$ time-steps each with $O(N^2)$ computations to be performed on the computational domain. 
All start--up overheads and program--finish overheads are excluded from the timing, this is to ensure a fair benchmark of only the numerical computation, $T = 10$ was chosen to ensure any effects of background processes due to the operating system are averaged out.
No IO steps were included in either code.
The same parameters and initial conditions were used as in the previous section, the serial code and version of cuPentCahnADI which outputs times rather than simulation data can be found in the folder} \codeword{cuPentSpeedUp}.

\Acomment{
In Figure~\ref{fig3:timeScale} we present the scaling in time as a function of $N$ for the serial and GPU codes, superimposed are the lines for $N^1$ and $N^3$ for comparison.
It can be seen clearly from these plots that the CPU code scales in time as $N^3$, in keeping with the above analysis.
While the GPU code initially scales with $N^1$ increasing to $N^3$ as $N$ increases.
This behaviour for the GPU scaling be attributed to the fact that initially at small $N$ the GPU is able to perform the $O(N^2)$ computations at each time-step completely in parallel, thus eliminating them from the scaling and so the dominant scale is the $O(N)$ time-steps being performed.
As $N$ increases the GPU becomes saturated with work, not all of the $O(N^2)$ computations can be performed at once and so the $O(N^3)$ scaling becomes dominant again.
}

\begin{figure}[ht]
  \centering
    \includegraphics[width=0.8\textwidth]{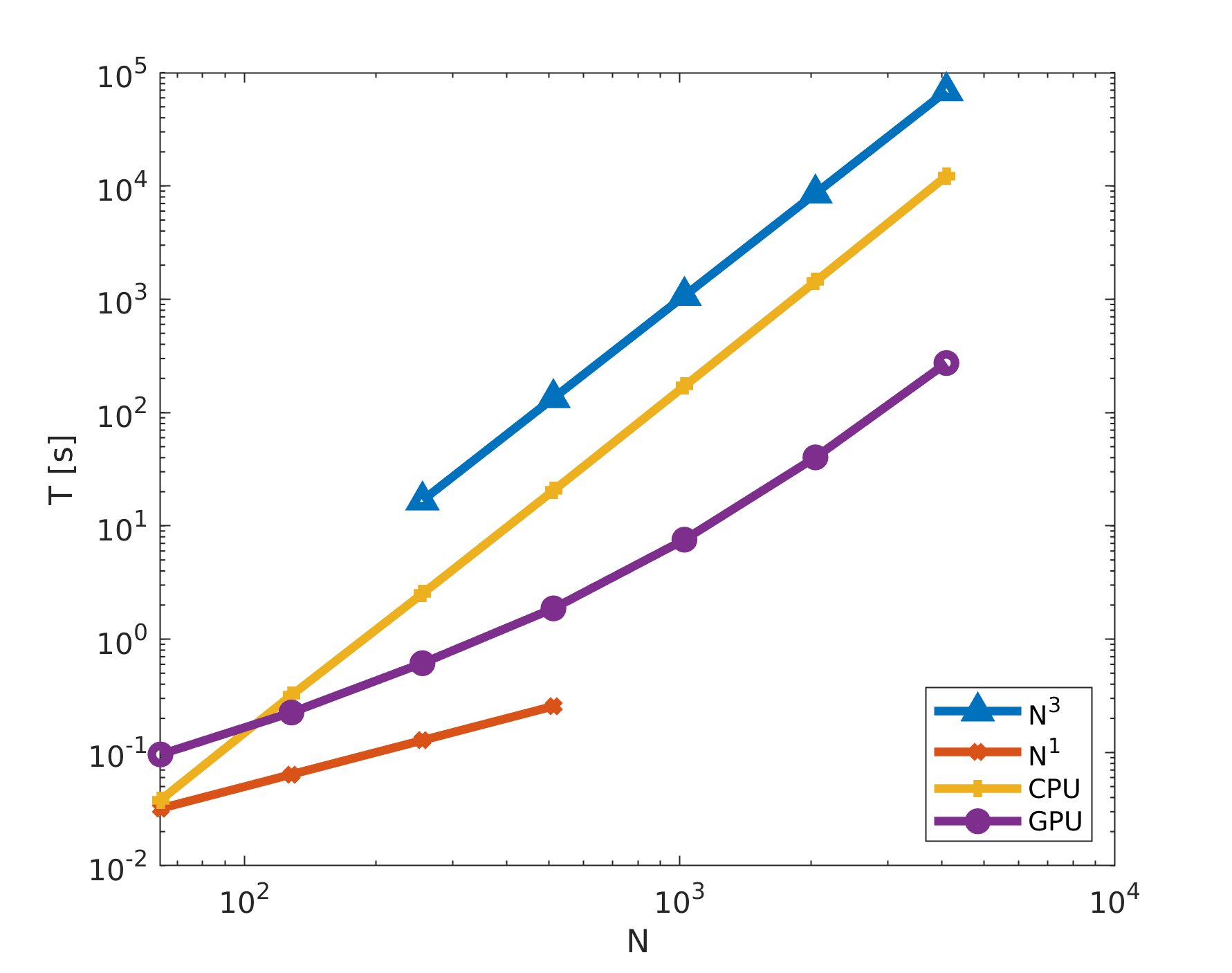}
    \caption{Plot showing how the CPU and GPU times scale with $N$.}
  \label{fig3:timeScale}
\end{figure} 

Further evidence of this can be seen in Figure~\ref{fig3:speedup} as the curve begins to level off for $N$ large. 
In this plot though we can see the clear advantage of parallelising this 2D solver on a GPU versus the serial CPU code, the speed-up is on the $O(10)$ for all reasonable grid resolutions, indeed the speed-up gets to $40\times$ faster for large $N$, a significant performance increase.
This performance would be increased further on newer GPUs such as the V100.
Thus significant performance can be gotten by using a GPU with the cuSten library for 2D computations.
The advantages of GPUs for the speed-up of batches of 1D problems has already been discussed in~\cite{gloster2019cupentbatch}, the results presented in that paper used an earlier version of the cuSten library.

\begin{figure}[ht]
  \centering
    \includegraphics[width=0.8\textwidth]{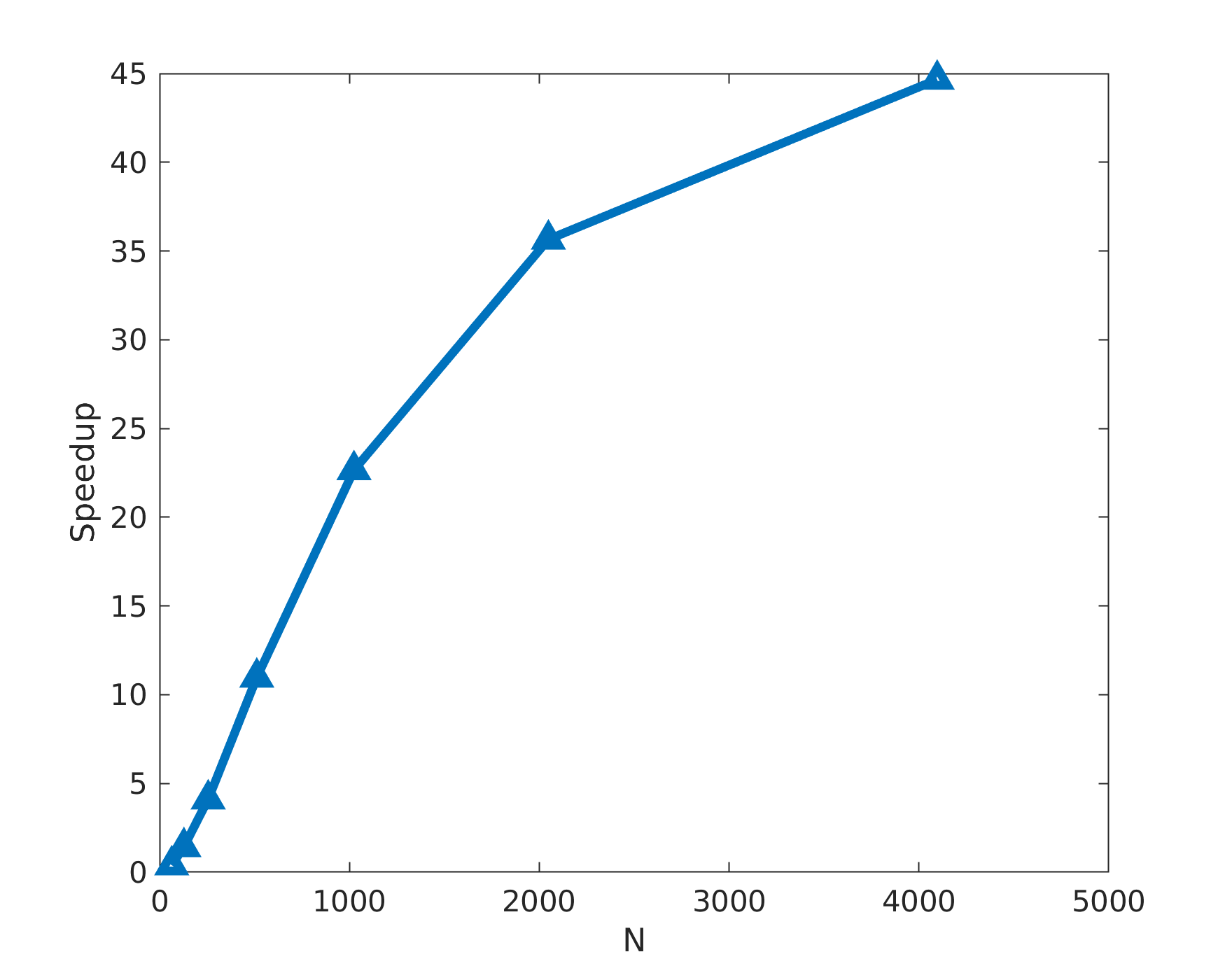}
    \caption{Plot showing speed-up of the GPU code versus the CPU code as a function of $N$.}
  \label{fig3:speedup}
\end{figure} 


\section{Discussion and Conclusions}
\label{sec3:con}
\subsection{Possible Future Extensions to the Library}
As previously mentioned the current library is limited to 2D uniform grids with double precision. 
Future areas of expansion could include moving the current library functions into C++ templates, this would allow for easier generalisation to other data types without the current need for find and replace to be done manually. 
Expansion to 3D and non-uniform grids is less trivial. 
3D would require a different approach to loading data than currently implemented as data will not be contiguous in RAM in the z direction, a more sophisticated loading scheme with pointers would be required.
For non-uniform grids additional data would need to be loaded into memory, it is likely in this situation that a hybrid of modifying the code such as in the WENO example to have extra data available ($u$ and $v$ velocities in the case of WENO, coordinate transformations in the case of a non-uniform grid) and using function pointers would be the best approach to make to the existing source. 

\subsection{MPI}
The design of the library lends itself to an MPI domain decomposition to be used in a hybrid code with the cuSten library.
Each MPI process could be assigned to a GPU using the deviceNum parameter, then the user would apply the non periodic versions of the stencils along with using MPI to swap the boundary halos.
Memory exchange is simplified in MPI due to the use of Unified Memory, the required data will be copied directly between GPU devices. 
This allows for the application of this library in much larger solvers which require more than just a single GPU.

\subsection{Concluding Remarks}
In this chapter we have shown how cuSten can be used to simplify the implementation of finite difference programs in CUDA compared with other state of the art libraries such as PETSc.
cuSten has a lightweight interface with a minimal learning curve required to implement the functions as part of a wider project. 
The library has been benchmarked against a serial code using a Cahn--Hilliard solver and numerous examples are provided to show potential users how to use the functionality provided. 
It has wide ranging applications in finite-difference solver development and in further areas requiring stencil--based operations such as image processing and optimisation problems. 
\lhead{\emph{Chapter 4}}  
\chapter{cuPentBatch -- A batched pentadiagonal solver for NVIDIA GPUs}
\label{chapter:cuPentBatch}
\Acomment{
In this chapter cuPentBatch is introduced, a batched pentadiagonal solver for NVIDIA GPUs.
The work in this chapter was primarily the subject of the paper 'cuPentBatch -- A batched pentadiagonal solver for NVIDIA GPUs~\cite{gloster2019cupentbatch}.
The development of cuPentBatch was motivated by applications involving batched numerical solutions of parabolic partial differential equations, necessitated by parameter studies of various 1D physical models and methods such as ADI in 2D which relies on batched matrix inversions.  
We demonstrate that cuPentBatch outperforms the NVIDIA standard pentadiagonal batch solver gpsvInterleavedBatch for the class of physically--relevant computational problems encountered herein.
}


\section{Introduction}
\label{sec4:intro}
This chapter considers the solving of batches of pentadiagonal linear systems, systems with five matrix diagonals, using GPU computing.   
Batched solutions of $\mxA \vecx = \vecb$, in particular where $\mxA$ is a tridiagonal or pentadiagonal matrix, are becoming increasingly prevalent methods for tackling a variety of problems on GPUs which offer a high level of parallelism~\cite{zhang2010fast, kim2011scalable, cuThomasBatch, gloster2019cupentbatch}. 
In the field of gravitational wave data analysis, much work is done simplifying complex waveforms and analyses to achieve results in a physically reasonable and desirable amount of time. 
A large portion of this work involves using cubic splines for interpolation~\cite{splineIEEE}, a highly parallelizable process that has shown promising results using GPUs in terms of accelerating established analysis procedures, as well as allowing for new methods with the extreme increase in computational speed.  
Fluid mechanics has also seen a broad application of GPUs where solutions of Poisson's equation are commonly required~\cite{hockney1964fast, valeroPoisson}, \Acomment{tsunami modelling and simulation~\cite{reguly2018volna}}, numerical linear algebra \cite{laraVariableBatched, HaidarSmallMatrices}, batch solving of 1D partial differential equations~\cite{gloster2019cupentbatch} and ADI methods for 2D simulations \cite{ADIGPU, gloster2019custen}.
Other examples of areas using GPUs include image in-painting \cite{imageInpainting}, simulations of the human brain~\cite{cuHines, BrainsOilGas}, and animation~\cite{pixar}.
This is a computationally-intensive task, and justifies the deployment of GPU computing. 
Batched computations are becoming increasingly prevalent as methods for tackling large numbers of individual problems to generate data sets for parameter studies or for solving loosely coupled systems of equations. 
At the library level various implementations/functions have been developed for batched solves, examples include the Batched BLAS project \cite{dongarra2017optimized, dongarra2017design}, as part of the CUDA libraries cuBLAS, cuSPARSE and cuSOLVER, and the Intel MKL library.

CUDA already contains a library for solving pentadiagonal problems in batch mode, and the current state of the art algorithm is gpsvInterleavedBatch, which comes as part of the cuSPARSE library in CUDA.  
The problem to be solved may be written in abstract terms as $\mxA \vecx_i=\vecrhs_i$, where the index $i$ labels the various pentadiagonal problems to be solved.  
The application we have in mind is a parametric study, in which the vector $\vecrhs$ may depend on a physical parameter (or parameters); hence, the index $i$ labels different values taken by the parameter.  
As such, the vectors $\vecx_i$ and $\vecrhs_i$ change as the index $i$ changes, but the matrix $\mxA$ is the same in each case.  
In this context, use of gpsvInterleavedBatch is not appropriate, as gpsvInterleavedBatch  updates the entries of $\mxA$ for each instance of the linear problem, which leads to superfluous memory access and unnecessary computational overhead.  
As such, the key point of the present work is to develop and test a new pentadiagonal batch solver (cuPentBatch) which leaves the matrix $\mxA$ intact for each instance of the pentadiagonal solver, thereby enhancing computational performance -- in short, we demonstrate that our newly developed cuPentBatch outperforms gpsvInterleavedBatch for the computational problems encountered herein.

The starting-point for developing the batched pentadiagonal solver is an existing batched tridiagonal solver called cuThomasBatch~\cite{cuThomasBatch, cuThomasVBatch}, based on the Thomas Algorithm, and now part of the CUDA library as \codeword{gtsvInterleavedBatch}. 
We herein extend cuThomasBatch to accommodate pentadiagonal problems.  
We also provide several examples from Computational Physics and Applied Mathematics where pentadiagonal problems naturally arise -- not only as computational problems to be solved on a one-off basis, but in the context of parametric studies, where solutions in batch mode are essential.  
The pentadiagonal problems we exemplify are symmetric positive definite -- this justifies the use of an extended Thomas algorithm, which is numerically stable in precisely this setting.

The chapter is organised as follows.  
In Section~\ref{sec4:model} we outline the numerical PDE-based models that provide the motivation for the development of the pentadiagonal solver.  
We describe the pentadiagonal system to be solved and outline the algorithm for its solution.  
We also present some sample numerical results with validation.  
We introduce the parallel pentadiagonal solver in Section~\ref{sec4:perf} and present the results of a  performance analysis -- we show how the present algorithm has superior performance to the existing in-house CUDA library (gpsvInterleavedBatch).  
Concluding remarks are presented in Section~\ref{sec4:conc}.


\section{Physical and Computational model}
\label{sec4:model}
As we are motivated by key physical problems from applied mathematics and computational physics, in this section we develop the new pentadiagonal solver in the context of physical models, namely the hyperdiffusion equation.  
We choose the hyperdiffusion equation as it is a simplicifcation of the full Cahn--Hilliard equation, namely limiting the equation to the fourth--order derivative term, which we will study in later chapters in this thesis. 
We also limit our study to 1D equations as the methodology is the same when extended to 2D ADI methods.
At the same time, we emphasise that the algorithms developed herein are generic and carry over to arbitrary pentadiagonal systems.


\subsection{The hyperdiffusion equation}
\label{sec4:hyper}
Based on the motivation given above, we focus on the following hyperdiffusion equation in one spatial dimension:
\begin{equation}
\frac{\partial C}{\partial t}=-\gamma D \frac{\partial^4 C}{\partial x^4},\qquad t>0,\qquad x\in (0,L),
\label{eq:hyperdiff}
\end{equation}
with periodic boundary condition $C(x+L,t)=C(x)$ and initial condition $C(x,t=0)=f(x)$, valid on $[0,L]$.  
We henceforth rescale the space and time variables; this is equivalent to setting $\gamma=D=L=1$.  
We \Acomment{discretise} Equation~\eqref{eq:hyperdiff} in space using centred differences and in time using the Crank--Nicolson method.  
We use standard notation for the discretisation, with
\[
C_i^n = C(x=i\Delta x ,t=n\Delta t),
\]
where $\Delta x$ is the grid spacing in the $x$-direction.  
The grid spacing, the problem domain length $L$ and the number of unknowns $N$ are related through $\Delta x=L/N$.    
In this way, the discretised version of Equation~\eqref{eq:hyperdiff} is written as
\begin{multline}
\frac{C_i^{n+1}-C_i}{\Delta t}=
-\tfrac{1}{2}\Delta x^{-4}\left[C_{i+2}^{n+1}-4C_{i+1}^{n+1}+6C_{i}^{n+1}-4C_{i-1}^{n+1}+C_{i-2}^{n+1}\right]\\
-\tfrac{1}{2}\Delta x^{-4}\left[C_{i+2}^{n}  -4C_{i+1}^{n}  +6C_{i}^{n}  -4C_{i-1}^{n}+C_{i-2}^{n}\right].
\label{eq:disc}
\end{multline}
Upon rearranging terms, Equation~\eqref{eq:disc} can be written more compactly as follows:
\begin{multline}
\sigma_x C_{i - 2}^{n + 1} - 4 \sigma_x C_{i - 1}^{n+1} + (1 + 6 \sigma_x)C_{i}^{n+1} - 4 \sigma_x C_{i+1}^{n+1} + \sigma_x C_{i+2}^{n+1} \\ 
=  - \sigma_x C_{i - 2}^{n} + 4 \sigma_x C_{i - 1}^{n} + (1 - 6 \sigma_x)C_{i}^{n} + 4 \sigma_x C_{i+1}^{n} - \sigma_x C_{i+2}^{n},
\label{eq:1d_hyper_scheme}
\end{multline} 
where $\sigma_x = \Delta t / 2\Delta x^4$. 
\begin{remark}
With the Crank--Nicolson temporal discretisation and the centred spatial discretisation, the truncation error in the hyperdiffusion equation~\eqref{eq:1d_hyper_scheme} is $O(\Delta t^2, \Delta x^2)$. 
It can also be shown that this discretisation is unconditionally stable using Von Neumann stability analysis.
\end{remark}
\begin{remark}
A finite-difference approximation of the Heat Equation $\partial_t C=\partial_{xx} C$ with Crank--Nicolson temporal discretisation and \Acomment{fourth-order} accurate spatial discretisation (specifically, involving nearest neighbours and next-nearest-neighbours on the spatial grid) also produces a pentadiagonal problem that can solved with the methods developed herein.
\end{remark}
We conclude this section by emphasising that both the Cahn--Hilliard and hyperdiffusion equations fall into the category of fourth-order parabolic PDEs~\cite{Elliott_Zheng}, as the highest-order derivative term appears in \Acomment{a linear fashion} (specifically, through the appearance of the operator $\mathcal{L}=-\gamma D\partial_{xxxx}$).    
The linear operator $\mathcal{L}$ satisfies the generalised parabolic property
\[
\langle \phi ,\mathcal{L}\phi\rangle=\int_0^L \phi \left(\mathcal{L}\phi\right)\,\mathd x
=-\gamma D\int_0^L |\partial_{xx}\phi|^2\,\mathd x\leq 0,
\] 
i.e. $\langle \phi,\mathcal{L}\phi\rangle\leq 0$ for all non-zero smooth real-valued $L$-periodic functions $\phi(x)$.


\subsection{The pentadiagonal matrix system}
Equation~\eqref{eq:1d_hyper_scheme} can be rewritten as a pentadiagonal matrix system, modulo some off-diagonal terms to deal with the periodic boundary conditions:
\begin{subequations}
\begin{equation}
\underbrace{
\begin{pmatrix}
c & d & e & 0 & \cdots &  0 & a & b \\
b & c & d & e & 0 & \cdots &   0 & a \\
a & b & c & d & e & 0 & \cdots  & 0\\
0 & \ddots & \ddots & \ddots & \ddots & \ddots & \ddots &   \vdots\\
\vdots& \ddots & \ddots & \ddots & \ddots & \ddots & \ddots & 0 \\
0 & \cdots  & 0  &  a  & b  & c & d & e \\
e & 0 &  & 0   & a  & b & c & d \\
d & e & 0  &   \cdots & 0 &  a & b & c
\end{pmatrix}
}_{=\mxA}
\underbrace{
\begin{pmatrix}
x_{1} \\
x_{2} \\
\vdots \\
\vdots \\
\vdots \\
x_{N-2} \\
x_{N-1} \\
x_{N}
\end{pmatrix}
}_{=\vecx}
=
\underbrace{
\begin{pmatrix}
f_{1} \\
f_{2} \\
\vdots \\
\vdots \\
\vdots \\
f_{N-2} \\
f_{N-1} \\
f_{N}
\end{pmatrix}
}_{=\vecrhs}.
\label{eq:matrix_sys1}
\end{equation}
Here, the coefficients of the matrix in Equation~\eqref{eq:matrix_sys1} have the following meaning:
\begin{equation}
a  = \sigma_x,\qquad b = - 4 \sigma_x,\qquad
c = 1 + 6\sigma_x, \qquad
d= - 4 \sigma_x, \qquad e = \sigma_x.
\end{equation}%
Similarly,
\begin{equation}
f_i = - \sigma_x C_{i - 2}^{n} + 4 \sigma_x C_{i - 1}^{n} + (1 - 6 \sigma_x)C_{i}^{n} + 4 \sigma_x C_{i+1}^{n} - \sigma_x C_{i+2}^{n}.
\end{equation}%
\label{eq:matrix_sys}%
\end{subequations}%
As such, by inverting the matrix~\eqref{eq:matrix_sys}, the solution of the hyperdiffusion equation is advanced from time step $n$ to time step $n+1$.  
Here, information concerning $C$ at time step $n$ is contained in the vector $\vecrhs$, from which $C$ at time step  $n+1$ is extracted via the vector $\vecx$.

The matrix~\eqref{eq:matrix_sys} can be inverted using any standard method but our focus is now on using a specific pentadiagonal solver~\cite{numalgC}. 
This though requires us to re-examine the matrix above which has terms that lie off the diagonal, thus an additional step will be required to remove these terms.  
For this purpose, we use the algorithm of Navon in Reference~\cite{navon_pent}.   
As such, the matrix $\mxA$ is decomposed such that the last two rows and last two columns are eliminated, this then reduces the matrix to a pure pentadiagonal form that can be solved along with an additional smaller solve to deal with the eliminated points in the matrix. 
Therefore, following the discussion in Reference~\cite{navon_pent} we introduce the matrix $\mxE$ which is simply an $(N - 2) \times (N-2)$ reduced version of $\mxA$, removing the last two rows and columns:
\begin{equation}
\mxE = 
\begin{pmatrix}
c & d & e & 0 & \cdots &  0 & \cdots & 0 \\
b & c & d & e & 0 & \cdots & \cdots   & \vdots \\
a & b & c & d & e & 0 & \cdots  & 0\\
0 & \ddots & \ddots & \ddots & \ddots & \ddots & \ddots &   \vdots\\
\vdots& \ddots & \ddots & \ddots & \ddots & \ddots & \ddots & 0 \\
0 & \cdots  & 0  &  a  & b  & c & d & e \\
0 & \cdots & \cdots  & 0   & a  & b & c & d \\
0 & \cdots & \cdots  &   \cdots & 0 &  a & b & c
\end{pmatrix}
\end{equation}
We define the following vectors based on these eliminations, all of row dimension $(N-2)$:
\begin{equation}
\mxXhat = 
\begin{pmatrix}
x_{1}^{n+1} \\
x_{2}^{n+1} \\
\vdots \\
\vdots \\
\vdots \\
x_{N-2}^{n+1}
\end{pmatrix},
\qquad
\vech =
\begin{pmatrix}
e & d \\
0 & e\\
0 & 0 \\
\vdots \\
a & 0 \\
b & a
\end{pmatrix},
\qquad
\vecfhat=
\begin{pmatrix}
f_{1} \\
f_{2} \\
\vdots \\
\vdots \\
\vdots \\
f_{N-2}
\end{pmatrix},
\qquad
\veck=
\begin{pmatrix}
a & b \\
0 & a\\
0 & 0 \\
\vdots \\
e & 0 \\
d & e
\end{pmatrix}.
\end{equation}
We have therefore reduced our system~\eqref{eq:matrix_sys} to two coupled simultaneous equations that can be written as follows:
\begin{subequations}
\begin{align}
\mxE\mxXhat + \veck 
\begin{pmatrix}
x_{N - 1} \\
x_{N}
\end{pmatrix}
= \vecfhat
\\
\vech^T \mxXhat + 
\begin{pmatrix}
c & d \\
b & c
\end{pmatrix}
\begin{pmatrix}
x_{N - 1} \\
x_{N}
\end{pmatrix}
=
\begin{pmatrix}
f_{N - 1} \\
f_{N}
\end{pmatrix}.
\label{eq:xhat}
\end{align}%
\end{subequations}%
We can solve for $\mxXhat$ in the first simultaneous equation through a pentadiagonal inversion of E and obtain:
\begin{equation}
\mxXhat = \mxE^{-1} \left[\vecfhat - \veck 
\begin{pmatrix}
x_{N - 1} \\
x_{N}
\end{pmatrix}
\right]
\label{eq:solve}
\end{equation}
Equation~\eqref{eq:solve} is substituted into Equation~\eqref{eq:xhat}.  After some rearrangement of terms, these operations yield an expression for the final two unknowns:
\begin{equation}
\begin{pmatrix}
x_{N - 1} \\
x_{N}
\end{pmatrix} = 
\left[
\begin{pmatrix}
c & d \\
b & c
\end{pmatrix} 
- \vech^T \mxE^{-1}\veck\right]^{-1}
\left[
\begin{pmatrix}
f_{N - 1} \\
f_{N}
\end{pmatrix}
- \vech^T\mxE^{-1}\vecfhat
\right].
\label{eq:first_two}
\end{equation}
As such, we solve for the final two unknowns first (via Equation~\eqref{eq:first_two}).  
We then substitute the result for $(x_{N-1},x_N)^T$) into Equation~\eqref{eq:solve} and then invert to yield the entire vector $\vecx$.  
Computationally the expressions for the inverted matrix in~\eqref{eq:first_two} and the $\vech^T\mxE^{-1}$ can be computed and stored at the start of any code to be reused as required. 
This eliminates much of the overhead for each time step of the hyperdiffusion algorithm.  
In particular we make use of the fact that $(\mxE^{-1})^T\vech = (\vech^T\mxE^{-1})^T$ to simplify this process further.


\subsection{Solution of the pentadiagonal system}
In this section we describe a standard numerical method~\cite{numalgC} for solving a pentadiagonal problem $\mxA\vecx=\vecrhs$.  
We present the algorithm in a general context (in particular, independent of the earlier discussion on finite-difference solutions of PDEs).  
As such,  in this section we assume that $\mxA$ is strictly pentadiagonal  with arbitrary nonzero entries, such that
 \begin{equation*}
\mxA = 
\begin{pmatrix}
c_1 & d_1 & e_1 & 0 & \cdots &  0 & \cdots & 0 \\
b_2 & c_2 & d_2 & e_2 & 0 & \cdots & \cdots   & \vdots \\
a_3 & b_3 & c_3 & d_3 & e_3 & 0 & \cdots  & 0\\
0 & \ddots & \ddots & \ddots & \ddots & \ddots & \ddots &   \vdots\\
\vdots& \ddots & \ddots & \ddots & \ddots & \ddots & \ddots & 0 \\
0 & \cdots  & 0  &  a_{N - 2}  & b_{N - 2}  & c_{N - 2} & d_{N - 2} & e_{N - 2} \\
0 & \cdots & \cdots  & 0   & a_{N - 1}  & b_{N - 1} & c_{N - 1} & d _{N - 1}\\
0 & \cdots & \cdots  &   \cdots & 0 &  a_{N} & b_{N}  & c_{N} 
\end{pmatrix}.
\end{equation*}

Three steps are required to solve the system:
\begin{enumerate}
\item Factor $\mxA = \mxL\mxR$  to obtain $\mxL$ and $\mxR$.
\item Find $\vecg$ from $\vecf = \mxL\vecg$
\item Back-substitute to find $\vecx$ from $\mxR\vecx = \vecg$
\end{enumerate}
Here, $\mxL$, $\mxR$ and $\vecg$ are given by the following equations:
\begin{subequations}
\begin{equation}
\mxL = 
\begin{pmatrix}
\alpha_1 &  &  &  &  &   &   \\
\beta_2 & \alpha_2 &  &  &  &  &     \\
\epsilon_3 & \beta_3 & \alpha_3 &  &  &  &   \\
 & \ddots & \ddots & \ddots &  &  &     \\
 &  &  \epsilon_{N - 1}  & \beta_{N - 1}  & \alpha_{N - 2} &    \\
&  &    & \epsilon_{N - 1}  & \beta_{N - 1} & \alpha_{N }\\
\end{pmatrix},
\qquad
\vecg = \begin{pmatrix}
g_{1} \\
g_{2} \\
\vdots \\
\vdots \\
g_{N-1} \\
g_{N}
\end{pmatrix},
\end{equation}
\begin{equation}
\mxR = 
\begin{pmatrix}
1 & \gamma_1  & \delta_1  &  &  &   &   \\
 & 1 & \gamma_2  & \delta_2  &  &  &     \\
 &  & \ddots & \ddots & \ddots  &  &   \\
 & & & 1  & \gamma_{N-2}  & \delta_{N-2}    \\
 &  &  &  & 1 & \gamma_{N-1}    \\
&  &    & &  & 1\\
\end{pmatrix}
\end{equation}%
\end{subequations}%
(the other entries in $\mxL$ and $\mxR$ are zero).  
The explicit factorisation steps for the factorisation $\mxA=\mxL\mxR$ are as follows:
\begin{enumerate}
\item $\alpha_1 = c_1$
\item $\gamma_1 = \frac{d_1}{\alpha_1}$
\item $\delta_1 = \frac{e_1}{\alpha_1}$
\item $\beta_2 = b_2$
\item $\alpha_2 = c_2 - \beta_2\gamma_1$
\item $\gamma_2 = \frac{d_2 - \beta_2 \delta_1}{\alpha_2}$
\item $\delta_2 = \frac{e_2}{\alpha_2}$
\item For each $i = 3, \cdots, N-2$
\begin{enumerate}
\item $\beta_i = b_i - a_i \gamma_{i-2}$
\item $\alpha_i = c_i - a_i\delta_{i-2} - \beta_i \gamma_{i-1}$
\item $\gamma_i = \frac{d_i - \beta_i \delta_{i-1}}{\alpha_i}$
\item $\delta_i = \frac{e_i}{\alpha_i}$
\end{enumerate}
\item $\beta_{N-1} = b_{N-1} - a_{N - 1}\gamma_{N-3}$
\item $\alpha_{N - 1} =  c_{N-1} - a_{N-1}\delta_{N-3} - \beta_{N-1}\gamma_{N-2}$
\item $\gamma_{N-1} = \frac{d_{N-1}-\beta_{N-1}\delta_{N-2}}{\alpha_{N-1}}$
\item $\beta_{N} = b_{N} - a_{N }\gamma_{N-2}$
\item $\alpha_{N} =  c_{N}- a_{N}\delta_{N-2} - \beta_{N}\gamma_{N-1}$
\item $\epsilon_i = a_i, \quad \forall i$
\end{enumerate}

The steps to find $\vecg$ are as follows:

\begin{enumerate}
\item $g_1 = \frac{f_1}{ \alpha_1}$
\item $g_2 = \frac{f_2 - \beta_2 g_1}{\alpha_2}$
\item  $g_i = \frac{f_i - \epsilon_i g_{i-2} - \beta_i g_{i - 1}}{\alpha_i} \quad \forall i = 3 \cdots N$
\end{enumerate}

Finally, the back-substitution steps  find $\vecx$ are as follows:

\begin{enumerate}
\item $x_N = g_N$
\item $x_{N-1} = g_{N-1} - \gamma_{N-1}x_N$
\item  $x_i = g_i - \gamma_i x_{i+1} - \delta_{i}x_{i+2} \quad \forall i = (N-2) \cdots 1$
\end{enumerate}

In this work, we implement this algorithm in serial and parallel batch. 
It can be easily seen that only six vectors are required to implement this algorithm: five for the left-hand side and one for the right-hand side.  
In the initial factorisation step $\mxA=\mxL\mxR$ we overwrite the input matrix $\mxA$ with the factorised matrices $\mxL$ and $\mxR$ which can then be used for the inversion steps later, this is done to minimise memory usage.
It should be noted that this method is $O(N)$ and each system of equations in the batch must be solved serially by a thread.


\subsection{Validation of Scheme}

\begin{figure}
	\centering
		\includegraphics[width=0.6\textwidth]{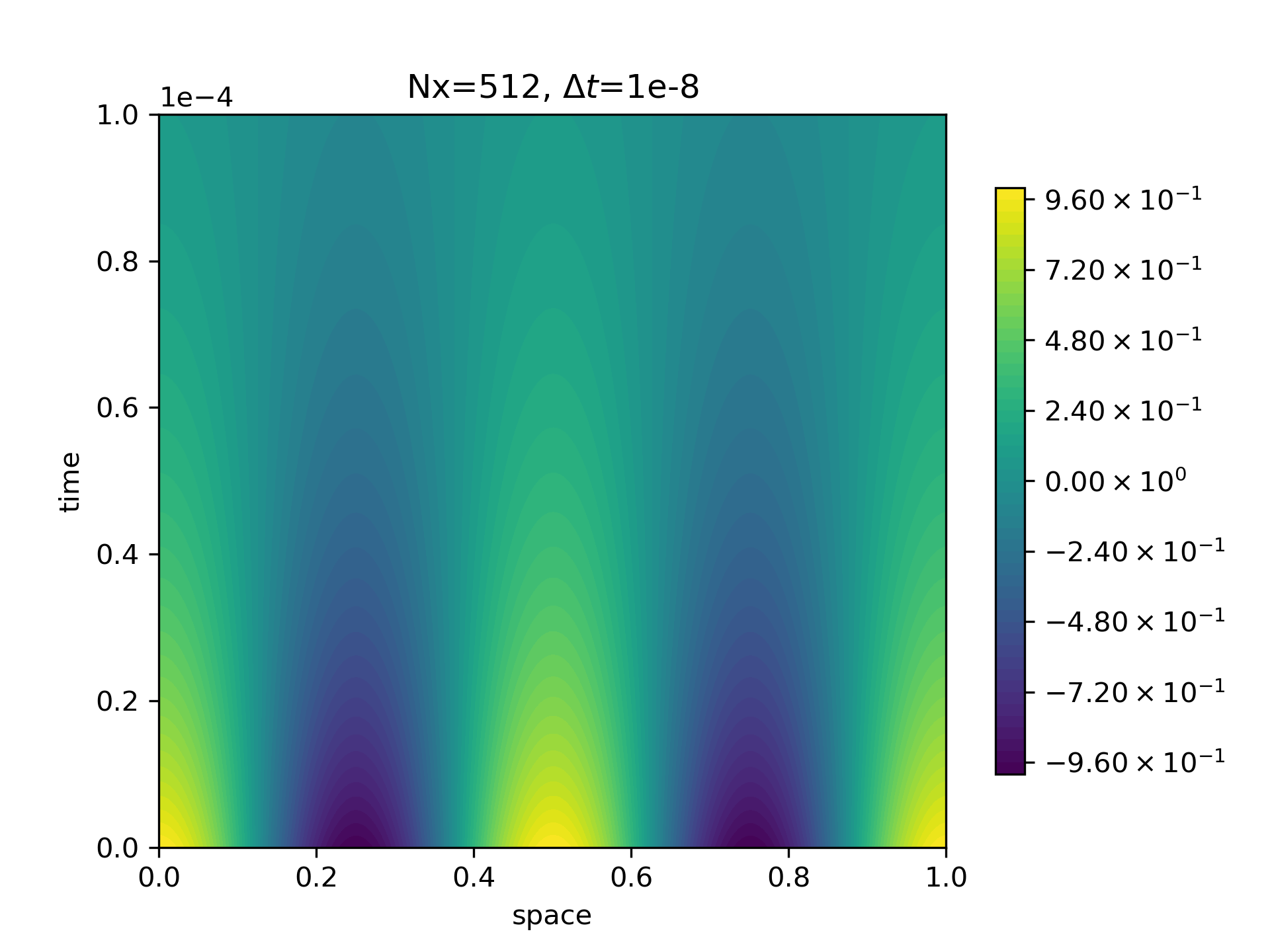}
		\caption{Space-time plot of the solution of the hyperdiffusion equation for a final time $T=10^{-4}$.  The model parameters and the initial condition are given in the main text.}
	\label{fig4:1d_hyperdiffusion_sol}
\end{figure}
We have validated the implicit finite-difference method~\eqref{eq:disc}--\eqref{eq:1d_hyper_scheme} for the hyperdiffusion equation.  We use the pentadiagonal solver developed above.  As a first validation step, we have implemented the numerical algorithm in a serial C code.  This serves as a base case against which to compare the performance of the GPU code in what follows.
An advantage of performing validation tests with the hyperdiffusion equation is that the hyperdiffusion equation admits exact solutions.  As such, a harmonic initial condition $C(x,t=0)= A\cos(kx+\varphi)$ (with constant amplitude $A$, wavenumber $k=(2\pi/L)n$ and phase $\varphi$ evolves into an exponentially-damped harmonic solution for $t>0$, 
\begin{equation}
C(x,t)=A\mathe^{-\lambda t}\cos(kx+\varphi),\qquad t>0.
\label{eq:exact}
\end{equation}  
Here, $n$ is a positive integer, and $\lambda=\gamma D k^4$ is the known analytical decay rate.  In this section we work with $\gamma=D=L=1$.  We also take $A=1$, $\varphi=0$, and $n=2$. 

Based on this numerical setup, a spacetime plot of the numerical solution $C(x,t)$ is shown in Figure~\ref{fig4:1d_hyperdiffusion_sol}, starting at $t=0$, and ending at the final time $T=10^{-4}$.  The amplitude numerical solution exhibits a rapid decay in time, consistent with the exact solution~\eqref{eq:exact}. 
We further examine the $L^2$ norm of the absolute error $\epsilon_N(t)$, given here in an obvious notation by
\begin{equation}
\epsilon_N(t)=\bigg\{\frac{1}{N}\sum_{i=1}^N \left[C_{\mathrm{numerical}}(i\Delta x,t)
-C_{\mathrm{analytical}}(i\Delta x,t)\right]^2\bigg\}^{1/2}.
\label{eq:myerr}
\end{equation}
Here, the dependency of the error on the number of grid points is indicated by the subscript $N$.  We examine this dependency by taking $t=T$ and investigating the functional relationship between $\epsilon_N(T)$ and $N$ in Figure~\ref{fig4:1d_hyperdiffusion_convergence}.  The error decreases as $\epsilon_N(T)\sim N^{-2}$, consistent with the fact that that our chosen spatial discretisation of the fourth-order derivative in the hyperdiffusion equation is $O(\Delta x^2)$ (i.e., $O(N^{-2}$)).
\begin{figure}
	\centering
		\includegraphics[width=0.6\textwidth]{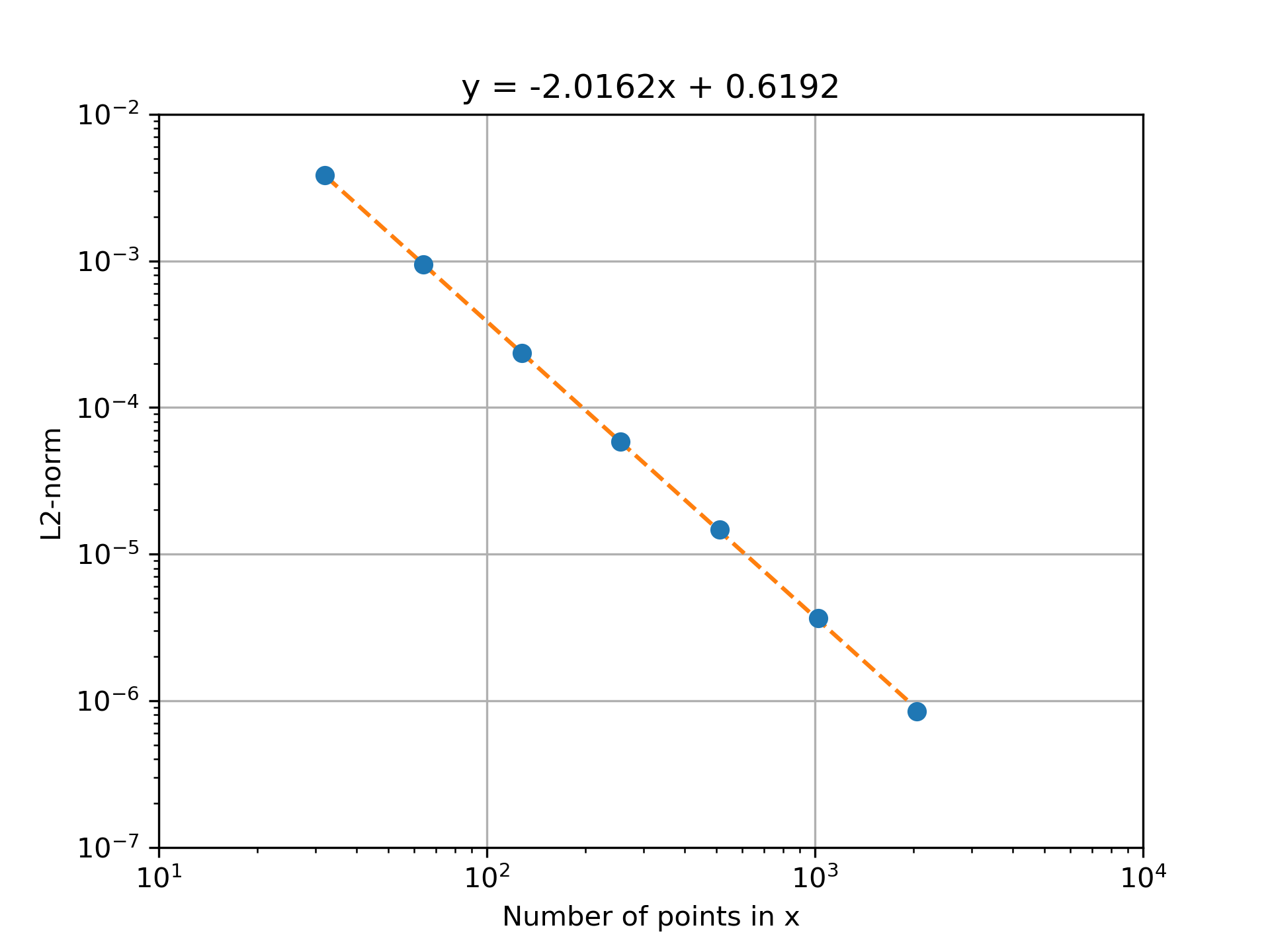}
		\caption{A plot of $\epsilon_N(T)$ as a function of $N$.  The final time is $T=10^{-4}$ and the time step is $\Delta t$ is $10^{-8}$.  The model parameters and the initial condition are given in the main text.  The line of best fit (on a log-log scale) is fitted and yields $\epsilon_N(T)\propto N^{-2.0162}$, compared to the theoretical value  $\epsilon_N(T)\propto N^{-2}$. }
	\label{fig4:1d_hyperdiffusion_convergence}
\end{figure}


\subsection{Implementation on GPU}
\label{sec4:implementation}
In order to solve the above scheme in batched form on a GPU we follow the methodology of cuThomasBatch \cite{cuThomasBatch, cuThomasVBatch} with some modifications. 
We retain the key aspect of interleaved data layout, this means that the first row of the batch data will contain the first entry in each linear system $\mxA\vecx_i=\vecf_i$ (the subscript $i$ labels the different systems in the batch), the second row the second entry and so on. 
The scheme is then implemented as in the serial case, but with one thread per system. 
This allows the GPU threads to access the global memory with coalesced memory accesses and prevents the need to worry about the physical size limits of shared memory. 
Where our implementation differs, apart from the change in type of matrix, is the splitting of the initial factorisation steps from the solve steps. 
This allows a user who wishes to use a constant matrix repeatedly to avoid factorising at every function call, and, we will show, a user who requires a new matrix at every call is not unjustly penalised versus using the existing gpsvInterleavedBatch. 
Indeed in many cases we see an improvement in performance even when refactoring the matrix at every time step. 
Finally, we use the library cuSten~\cite{gloster2019custen} which was discussed in the previous chapter to generate the right-hand side of each linear system in the batch.

The gpsvInterleavedBatch function relies on QR factorisation to solve the system of equations with householder reflection \cite{matComp}. 
It also relies on an interleaved data layout, thus making global data access performance identical to that of cuPentBatch. 
While QR factorisation is numerically stable a priori when compared to cuPentBatch it requires a greater number of operations. 
We note this as a flaw in cuPentBatch but we will show that for systems where the stability of the inversion is not a concern, such as in our example problem discussed in the following section where the matrix is symmetric positive definite, that cuPentBatch is a more efficient and faster algorithm. 
It should be noted that diagonally dominant is also a valid criterion for stability when solving with cuPentBatch.

\begin{remark}
The function gpsvInterleavedBatch uses dense QR factorisation with a zero fill pattern to accommodate the 5 diagonals while cuPentBatch is an LU factorisation without pivoting for 5 diagonals. 
Thus gpsvInterleavedBatch has a higher operation count than cuPentBatch. 
The performance benefit of this reduction is shown in section~\ref{sec4:paraversus}.
\end{remark}


\section{Performance Analysis}

For the purpose of performance analysis, we  solve a benchmark problem comprising  a series of identical one-dimensional hyperdiffusion simulations, as outlined in section~\ref{sec4:model}. 
To fix the emphasis on the performance analysis, each system in the batch has the same initial conditions and parameters.
Furthermore, we run each simulation for 250 time steps to average out any small variations in execution time by the computer due to scheduling, OS overhead etc. 
The measured time also omits any start up costs, setting initial conditions etc. 
The calculations are performed on an NVIDIA Titan X Pascal with 12GB of GDDR5 global memory and an Intel i7-6850K with 6 hyper-threaded cores. 
The system is running Ubuntu 16.04 LTS, CUDA v9.2.88, gcc 5.4 and has 128GB of RAM. 
Compiler flags used were \codeword{-O3} \codeword{-lineinfo} \codeword{--cudart=static} \codeword{-arch=compute_61} \codeword{-code=compute_61} \codeword{-std=c++11} \codeword{-lcusparse} \codeword{-lcublas}. 
Also it should be noted that these benchmarks are for 64 bit doubles, so the \codeword{cusparseDgpsvInterleavedBatch} is the variety of the cuSPARSE function used, this choice was made as when solving numerical PDEs higher floating point accuracy is generally desirable. 

In benchmarking we have measured the following three quantities:
\begin{enumerate} 
\item The time it takes to solve a batch of hyperdiffusion equations using gpsvInterleavedBatch. We shall refer to this method as simply gpsv from now on.
\item The time it takes to solve a batch of hyperdiffusion equation using cuPentBatch, factorising the matrix once at the beginning and repeatedly solving. We shall refer to this method as cuPentBatchConstant from now on.
\item The time it takes to solve a batch of hyperdiffusion equations using cuPentBatch, resetting and factorising the matrix repeatedly at every time step. This is to examine the performance in cases where the user will want to reset the matrix at every time step.   We shall refer to this method as cuPentBatchRewrite from now on.

Even in the present context of solving parabolic numerical PDEs in batch mode, it is conceivable that the matrix $\mxA$ may change at each time step -- for instance, in situations involving mesh refinement, adaptive time stepping or where the diffusion coefficient $D$ is no longer constant.
\end{enumerate} 
\label{sec4:perf}
Based on these measurements, we quantify the performance of our cuPentBatch using the following speedup ratios:
\begin{subequations}
\begin{equation}
\text{Speedup}=
\frac{\text{Time taken by gpsv}}{\text{Time taken by cuPentBatchConstant}},
\end{equation}
or
\begin{equation}
\text{Speedup}=
\frac{\text{Time taken by gpsv}}{\text{Time taken by cuPentBatchRewrite}},
\end{equation}%
\end{subequations}
depending on the context.  Hence, if $\text{Speedup}>1$, our in-house methods
are outperforming the standard gpsv.

\begin{figure}[ht]
	\centering
		\includegraphics[width=0.7\textwidth]{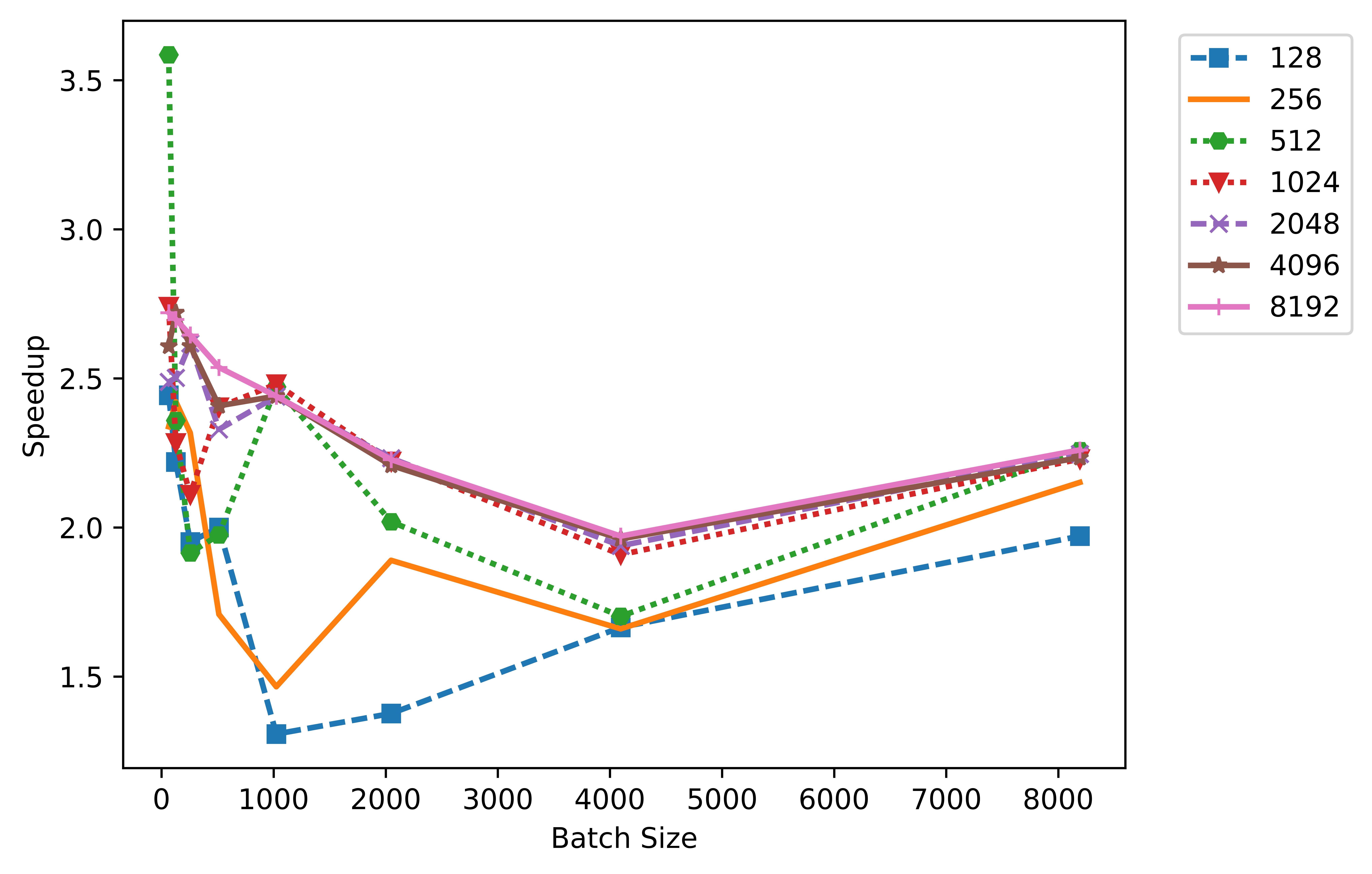}
		\caption{Speedup of cuPentBatchConstant versus gpsv. The number of unknowns for each is shown in the legend.}
	\label{fig4:fixConstantNFULL}
\end{figure} 

\begin{figure}[ht]
	\centering
		\includegraphics[width=0.7\textwidth]{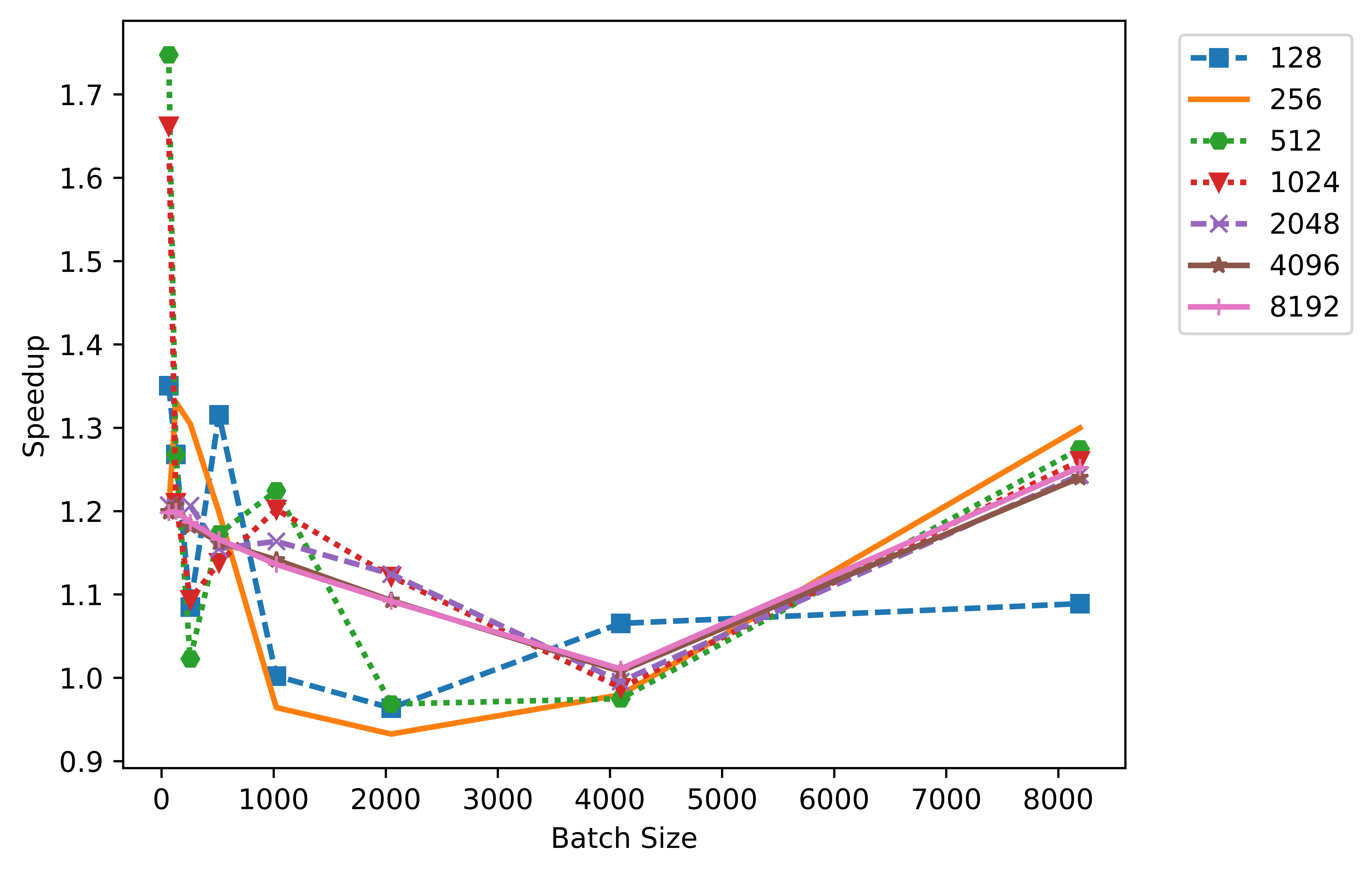}
		\caption{Speedup of cuPentBatchRewrite versus gpsv. The number of unknowns for each is shown in the legend.}
	\label{fig4:fixRewriteNFULL}
\end{figure}


\subsection{cuPentBatch vs. gpsvInterleavedBatch}
\label{sec4:paraversus}
\Acomment{We begin by fixing the number of unknowns and varying the batch size. 
The scaling is always $O(N)$ regardless of the whether we're benchmarking the Rewrite or Constant method.}
In Figure~\ref{fig4:fixConstantNFULL} we can see clear speedup for all cases of cuPentBatchConstant, generally over $2\times$ better performance for batches with high numbers of unknowns.  
Here, we see the clear advantage of the single factorisation and multiple solve over the multiple rewrites and factorisations that are required by gpsv. 
These batches are also small enough that they easily fit on the GPU memory thus the benchmark is free of any memory transfer penalties, only the run time of the algorithms is being compared.

In Figure~\ref{fig4:fixRewriteNFULL} we can see the speedup of cuPentBatchRewrite versus the gpsv algorithm. 
As the matrix is now being treated as non-constant between time steps the performance is much closer to that of gpsv. 
Nevertheless, due to the reduced number of operations required by cuPentBatch compared to gpsv, an increase in performance can be seen.  
\Acomment{This is most visible where there are larger numbers of unknowns, thus the serial aspect of the pentadiagonal inversion dominates, where there is an increase in performance such that $\text{Speedup} = 1.2 - 1.3$.}

\begin{figure}[ht]
	\centering
		\includegraphics[width=0.7\textwidth]{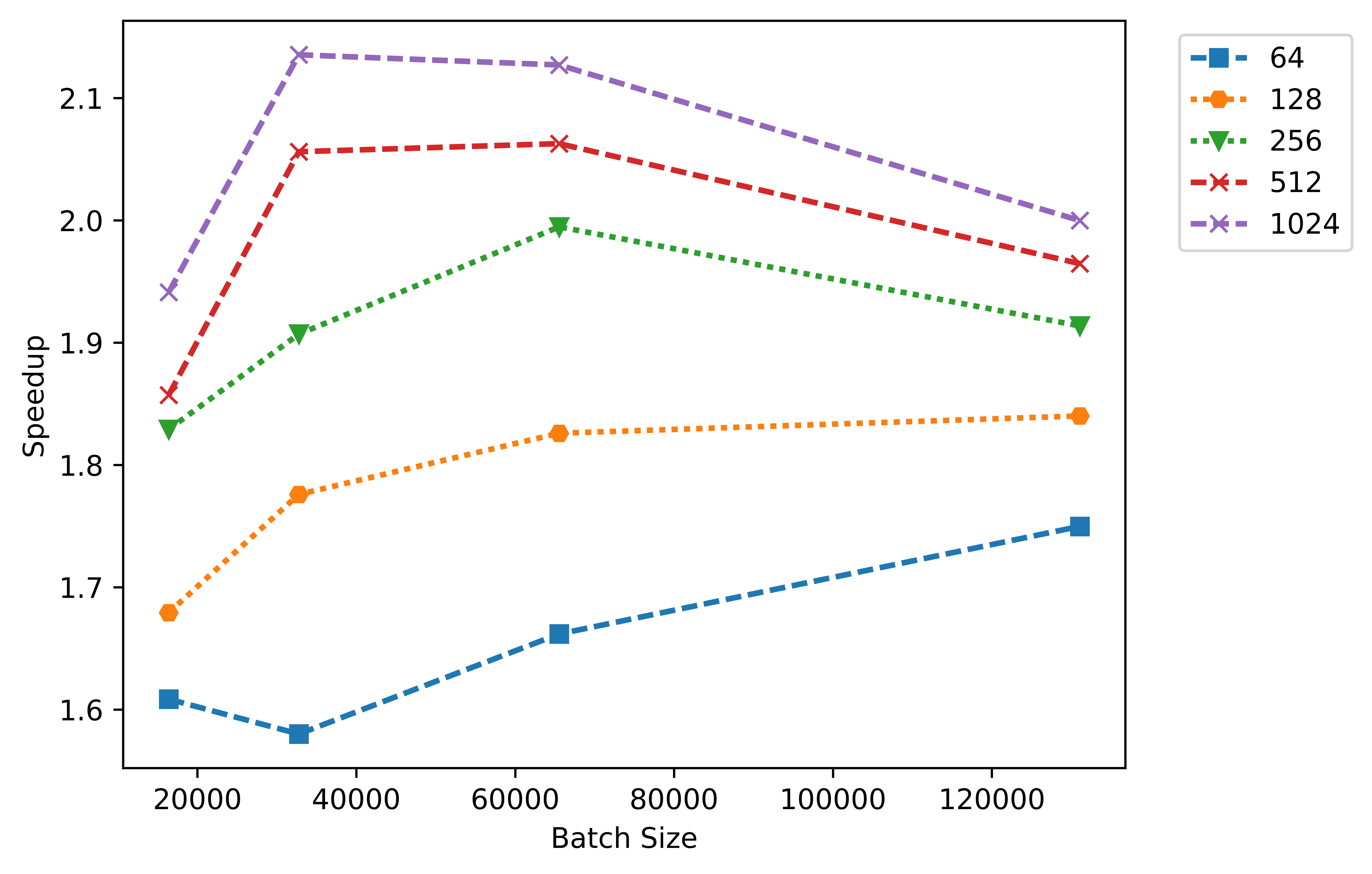}
		\caption{Speedup of cuPentBatchConstant versus gpsv for larger batch sizes $O(10^4-10^5)$. The number of unknowns for each is shown in the legend.}
	\label{fig4:fixLNEQConstantNFULL}
\end{figure} 

\begin{figure}[ht]
	\centering
		\includegraphics[width=0.7\textwidth]{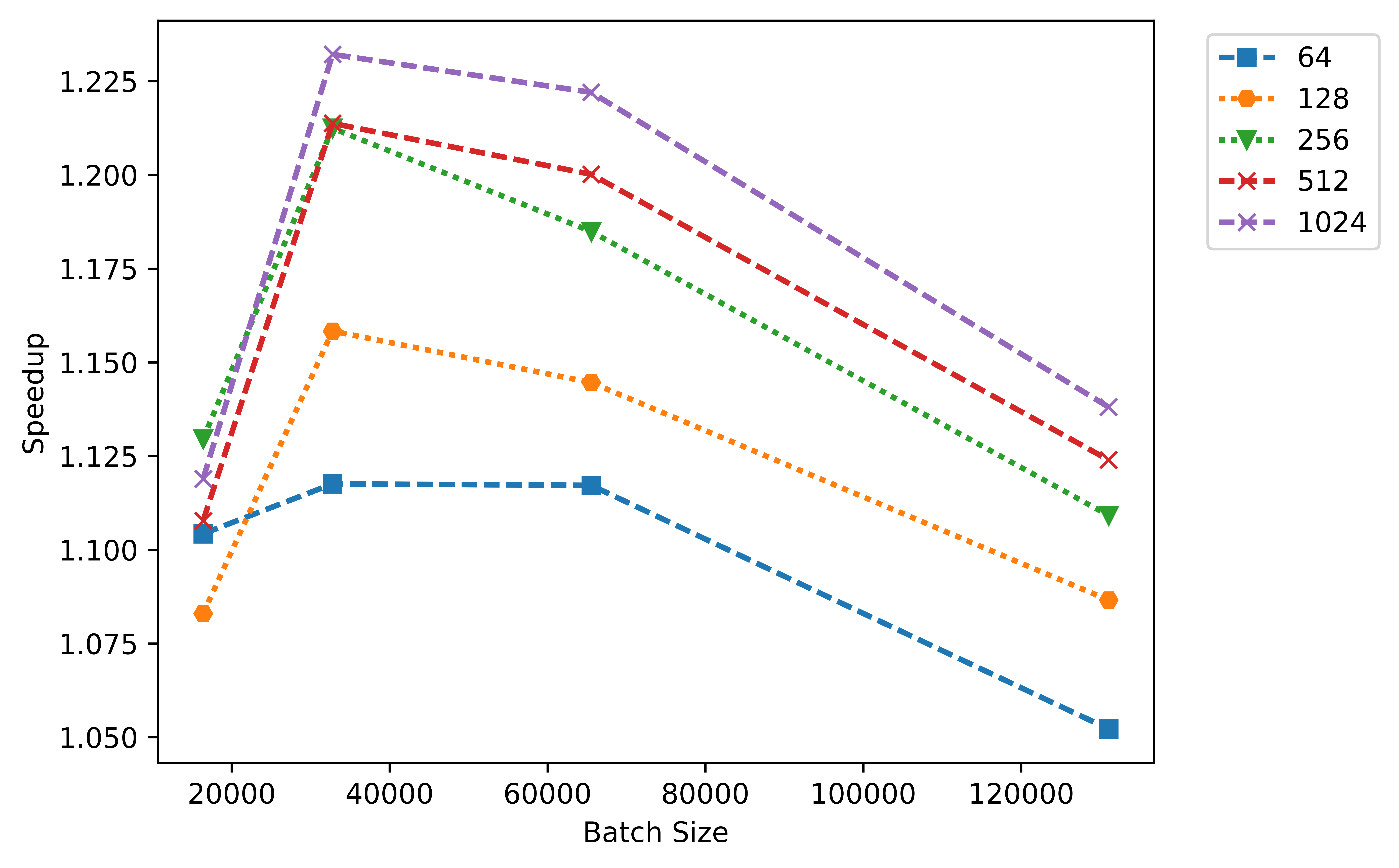}
		\caption{Speedup of cuPentBatchRewrite versus gpsv for larger batch sizes $O(10^4-10^5)$. The number of unknowns for each is shown in the legend.}
	\label{fig4:fixLNEQRewriteNFULL}
\end{figure} 

Taking the batch size to an extreme $O(10^4-10^5)$ we can see the performance comparisons in Figures~\ref{fig4:fixLNEQConstantNFULL} and~\ref{fig4:fixLNEQRewriteNFULL}. 
In both we can see that the improvement drops away slightly as batch size increases but cuPentBatch is still faster in both cases, particularly for the higher unknown sizes of 512 and 1024. 
\Acomment{This slight reduction can be attributed to the large amount of data that the GPU needs to access from global memory on the GPU which is the slowest memory access location
There is very little reuse of memory as the forward and backwards sweeps in the pentadiagonal system are solved.}
Thus it is clear for almost all batch sizes cuPentBatch is the better performer regardless of fixing a constant matrix or using a new one for every time step. 

\begin{figure}[ht]
	\centering
		\includegraphics[width=0.7\textwidth]{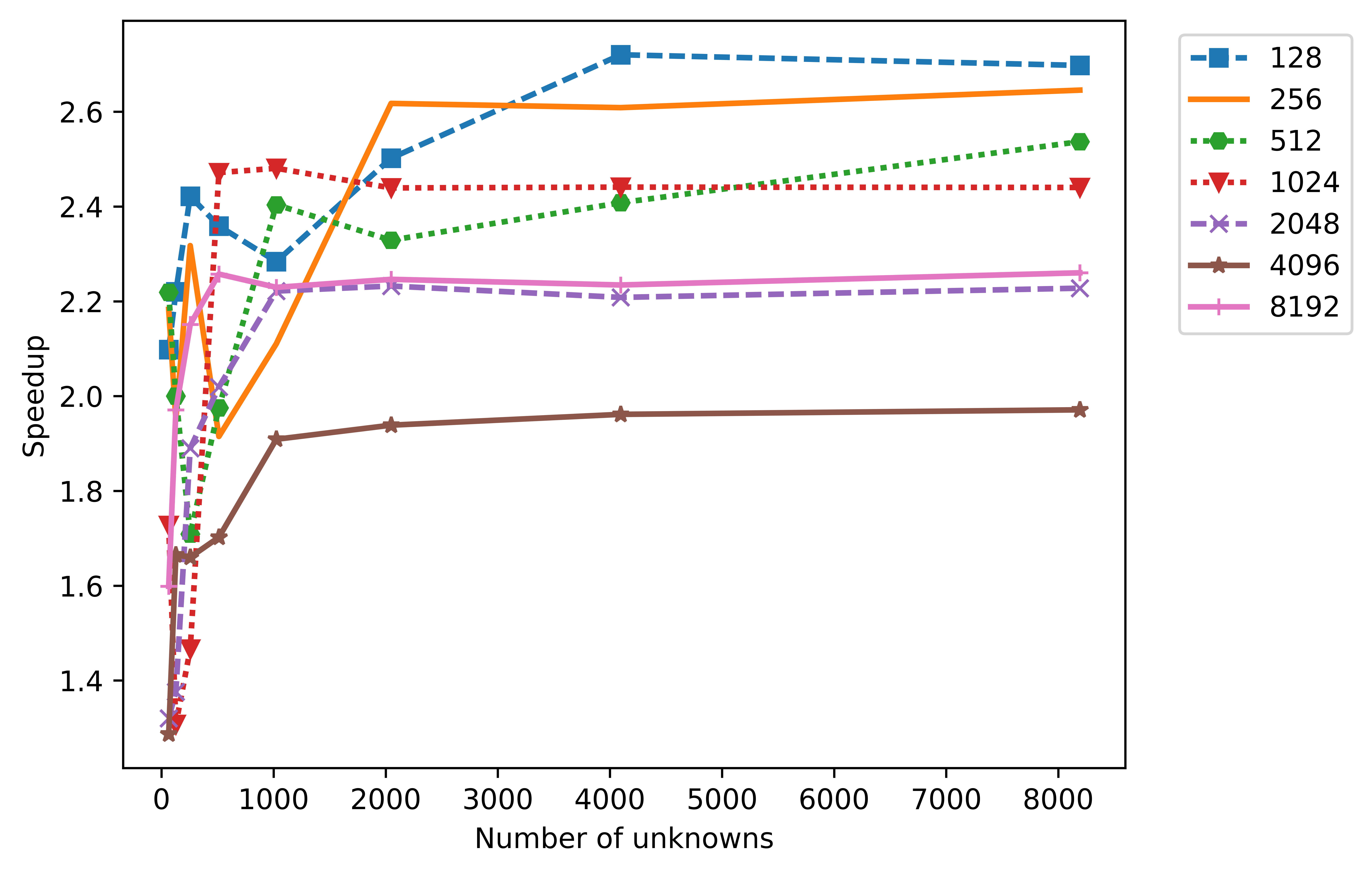}
		\caption{Speedup of cuPentBatchConstant versus gpsv. The batch size for each is shown in the legend.}
	\label{fig4:fixConstantNEQ}
\end{figure} 

\begin{figure}[ht]
	\centering
		\includegraphics[width=0.7\textwidth]{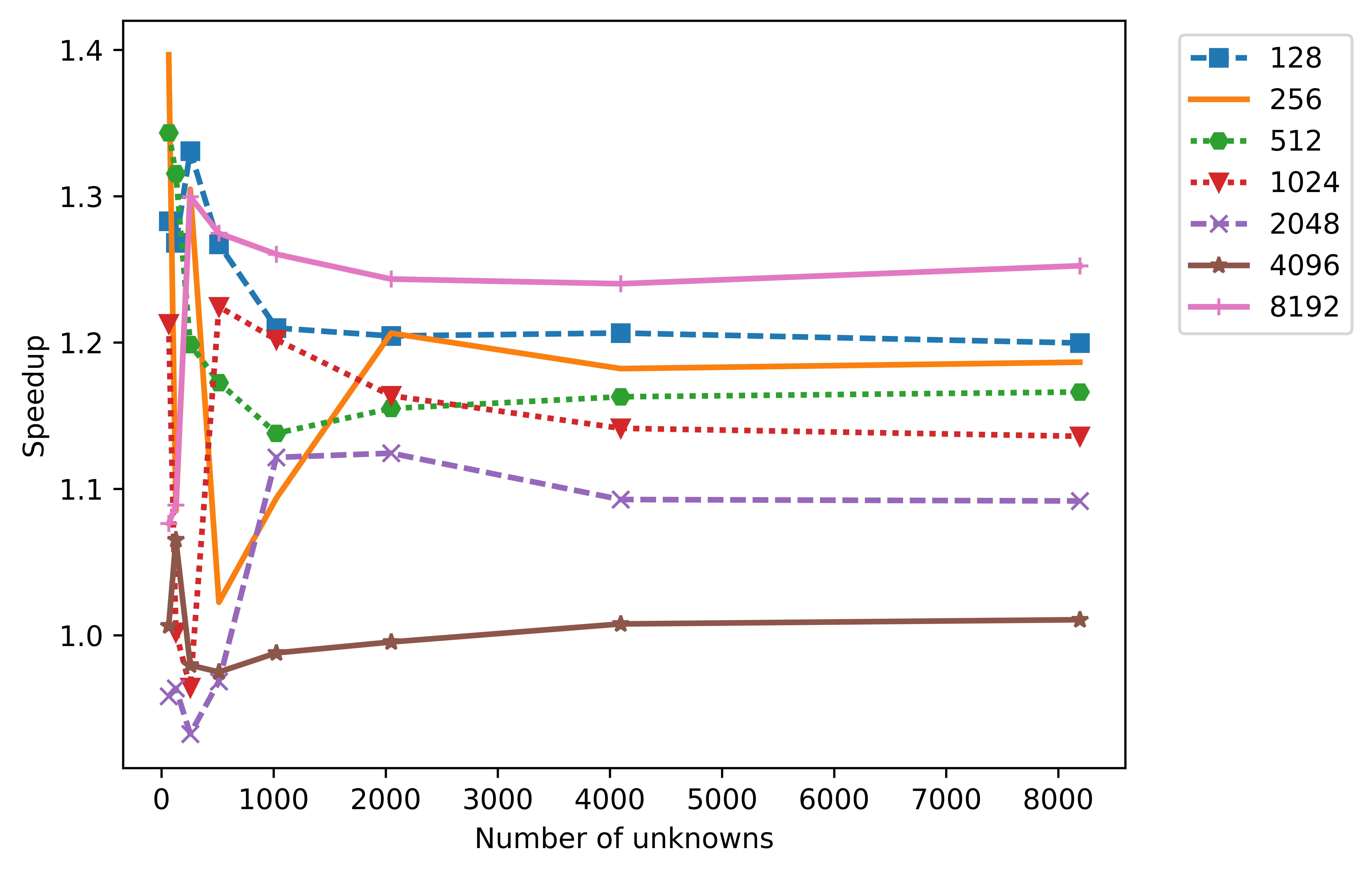}
		\caption{Speedup of cuPentBatchRewrite versus gpsv. The batch size for each is shown in the legend.}
	\label{fig4:fixRewriteNEQ}
\end{figure} 

We now perform a further analysis where we keep the size of the batch fixed and vary the number of unknowns. 
For cuPentBatchConstant in Figure~\ref{fig4:fixConstantNEQ} we again see significant speedup, especially at higher numbers of unknowns where speed up is well over $2\times$. 
Similarly in Figure~\ref{fig4:fixRewriteNEQ} we see better performance.   
It is clear that at high numbers of unknowns cuPentBatch performs significantly better than gpsv. 
\Acomment{In both figures the asymptotic behaviour can be attributed to a saturation of the GPU's ability to solve members of the batch in parallel and both methods are then bounded by the serial portion of their respective algorithms.
At large numbers of unknowns the speedup is further evident in Figures~\ref{fig4:fixLNFULLConstantNEQ} and~\ref{fig4:fixLNFULLRewriteNEQ} where the resolution of each hyperdiffusion equation is highly resolved with a moderate number of equation being solved in the batch.}
Summarising, cuPentBatch outperforms gpsvDInterleavedBatch in terms of scaling the number of unknowns in a system.

\begin{figure}[ht]
	\centering
		\includegraphics[width=0.7\textwidth]{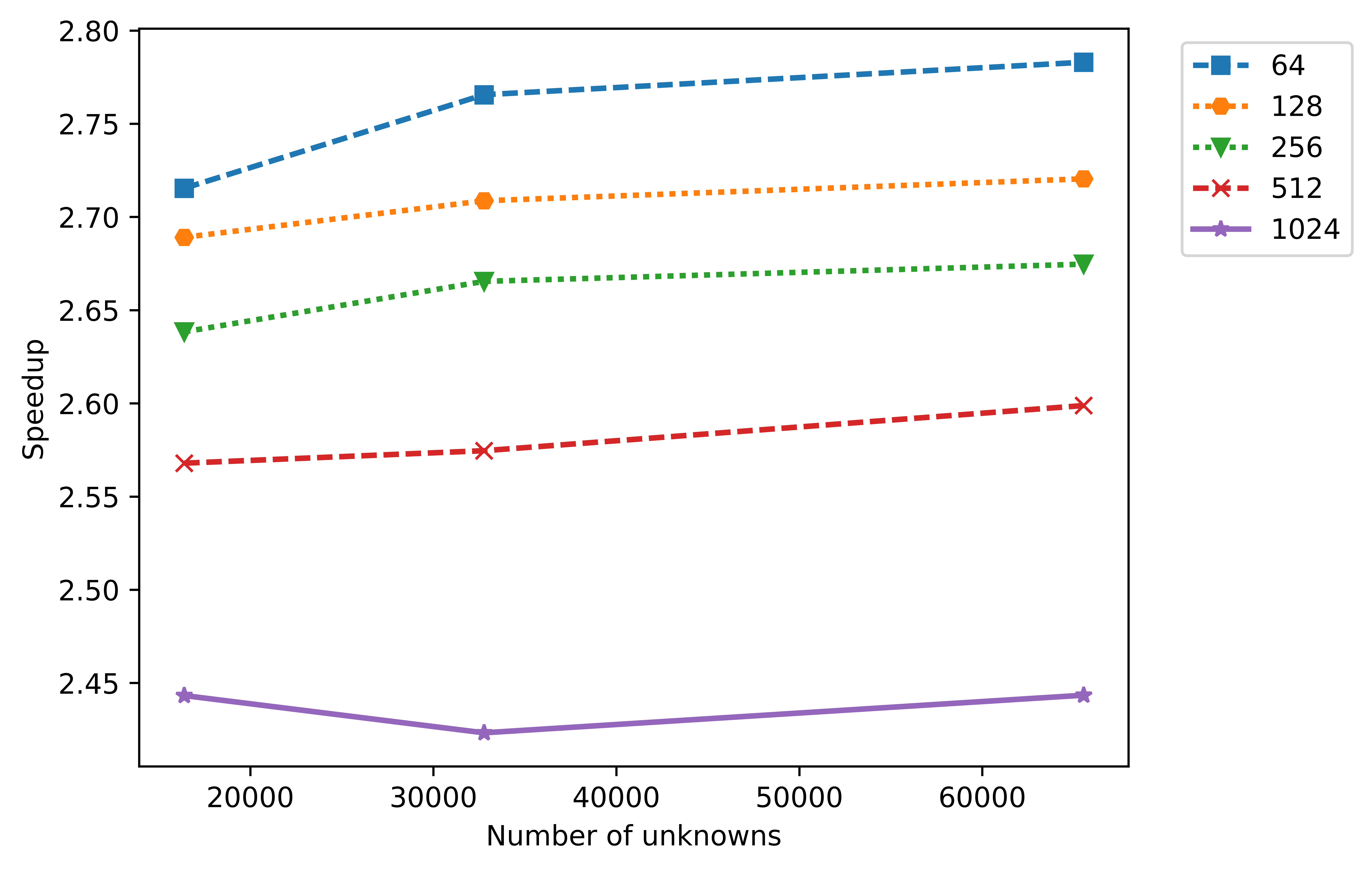}
		\caption{Speedup of cuPentBatchConstant versus gpsv  for large numbers of unknowns $O(10^4)$. The batch size for each is shown in the legend.}
	\label{fig4:fixLNFULLConstantNEQ}
\end{figure} 

\begin{figure}[ht]
	\centering
		\includegraphics[width=0.7\textwidth]{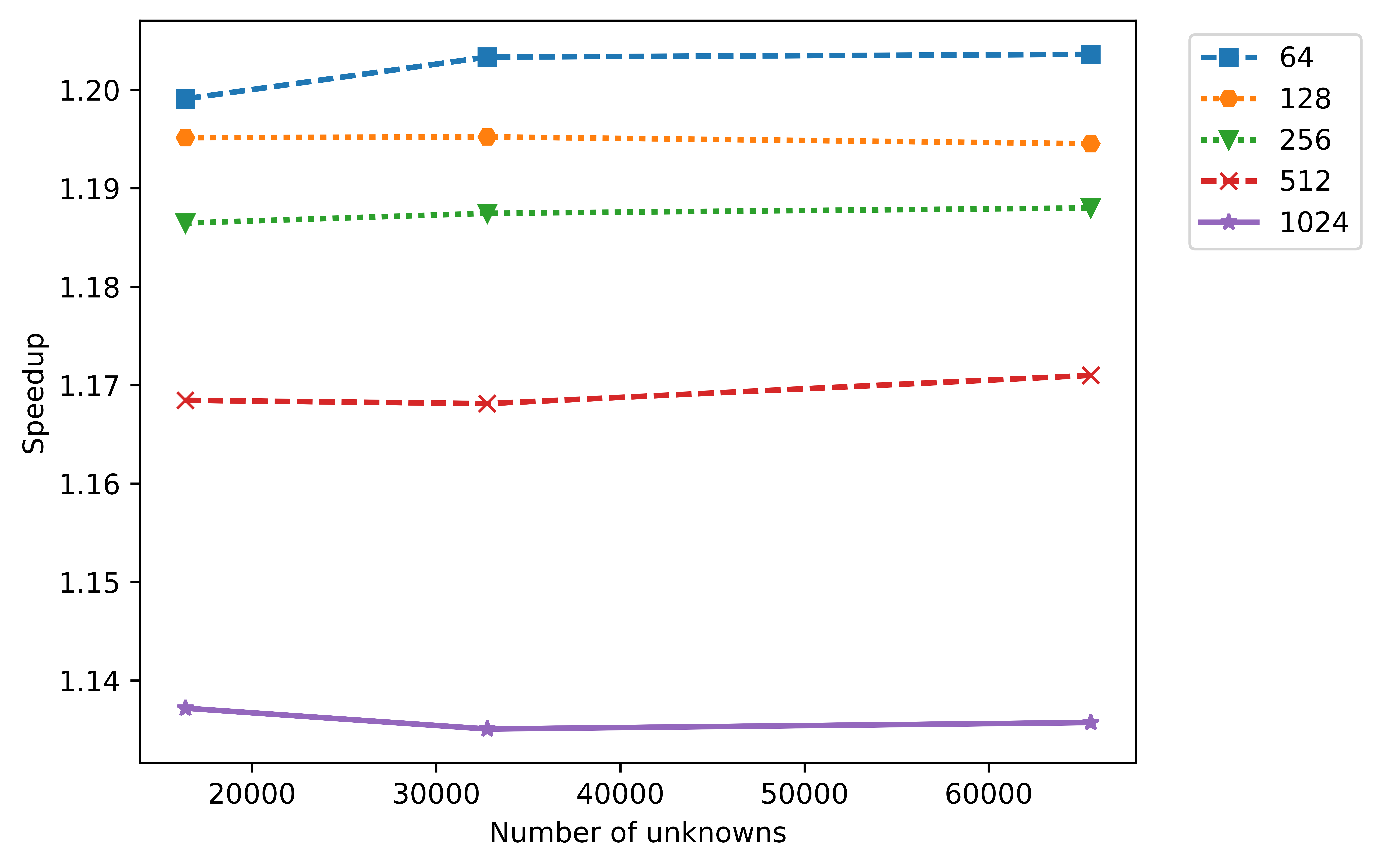}
		\caption{Speedup of cuPentBatchRewrite versus gpsv for large numbers of unknowns $O(10^4)$. The batch size for each is shown in the legend.}
	\label{fig4:fixLNFULLRewriteNEQ}
\end{figure} 

\subsection{cuPentBatch vs. Serial}
Given that we have established the speedup available to us over gpsvDInterleavedBatch we now show that for solving batches of pentadiagonal systems cuPentBatch is far superior to doing the same calculation in serial. 
\Acomment{The serial benchmark was run on the CPU of the same machine as the GPU benchmark with similar compiler optimisations turned on.}
The data is laid out in a standard format, not interleaved. 
This is to allow the memory to be accessed in C's preferred row major format, one hyperdiffusion system per row. 

\begin{figure}[ht]
	\centering
		\includegraphics[width=0.7\textwidth]{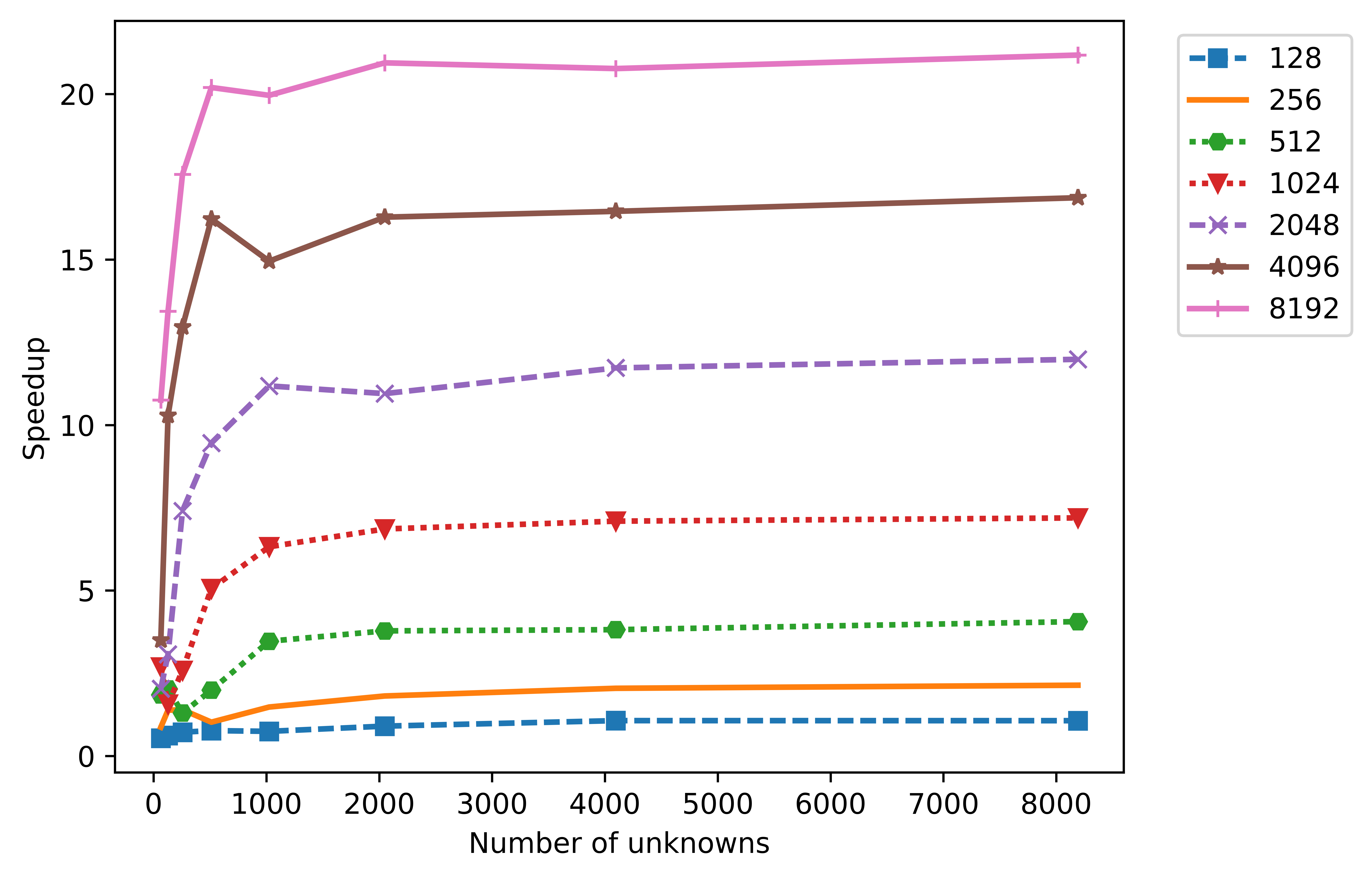}
		\caption{Speedup of cuPentBatchConstant versus serial. The batch size for each is shown in the legend.}
	\label{fig4:serialConstantNEQ}
\end{figure} 

\begin{figure}[ht]
	\centering
		\includegraphics[width=0.7\textwidth]{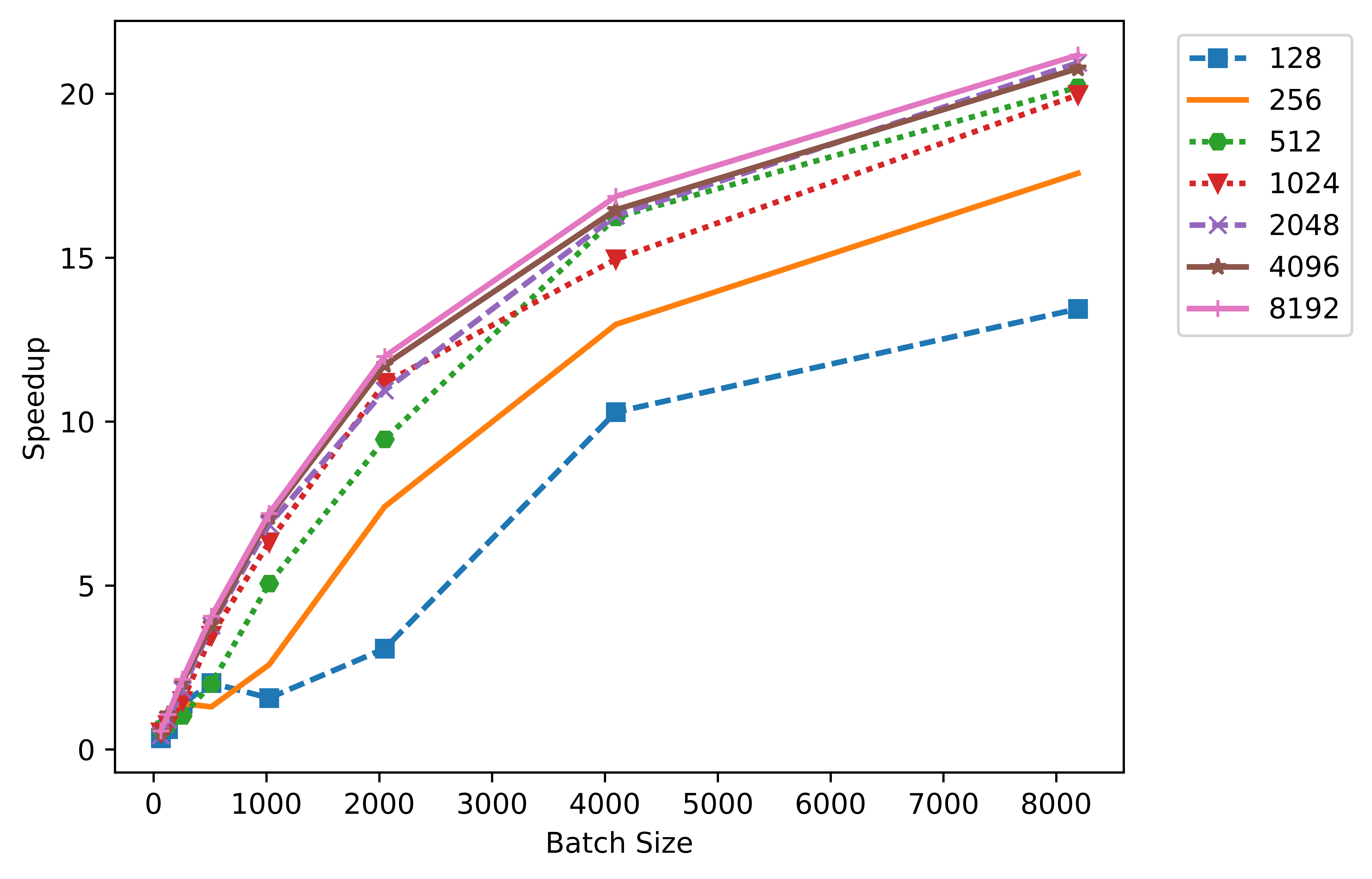}
		\caption{Speedup of cuPentBatchConstant versus serial. The number of unknowns for each is shown in the legend.}
	\label{fig4:serialConstantNFULL}
\end{figure} 

\begin{figure}[ht]
	\centering
		\includegraphics[width=0.7\textwidth]{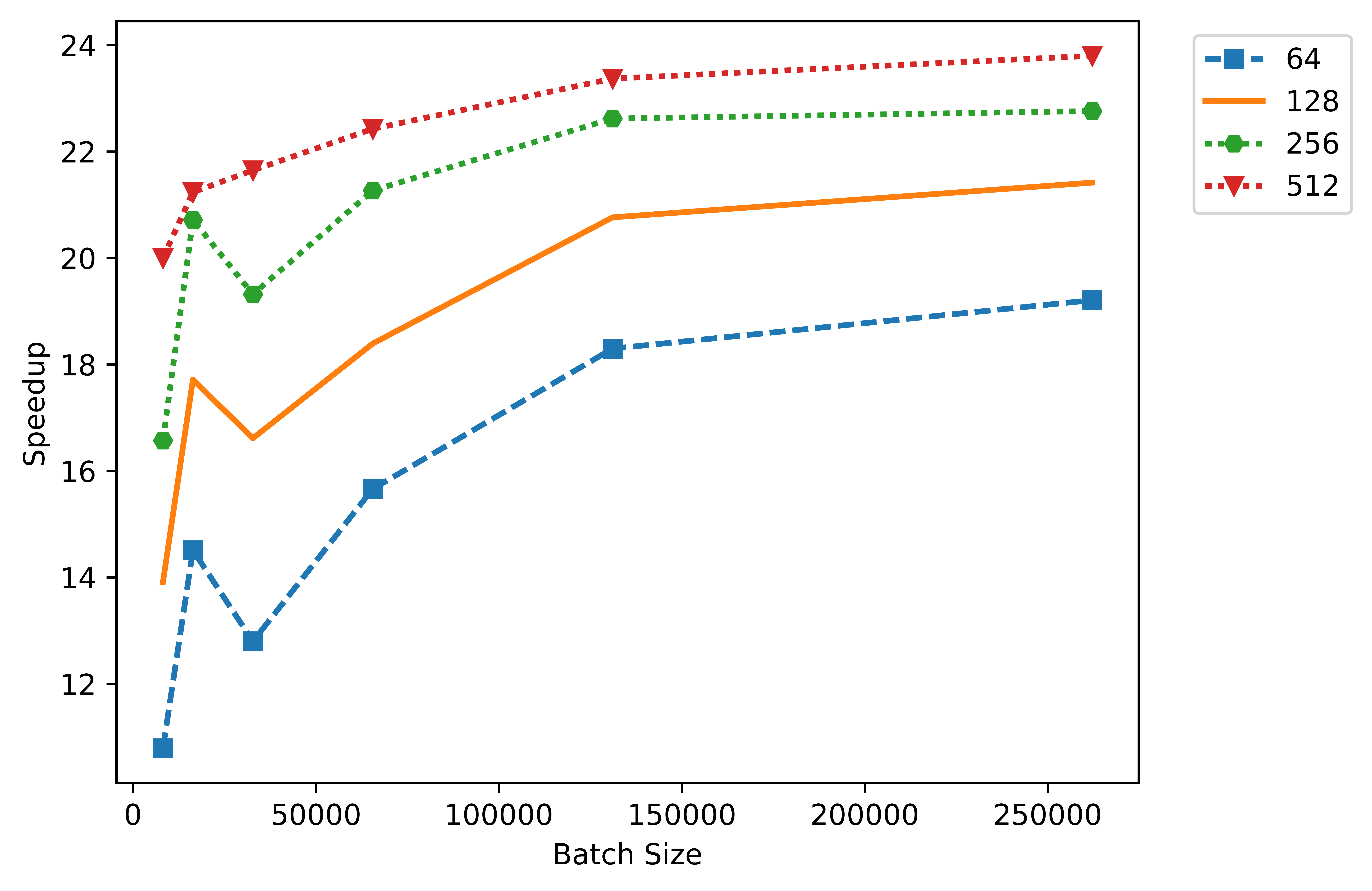}
		\caption{Speedup of cuPentBatchConstant versus serial with large batch size. The number of unknowns for each is shown in the legend.}
	\label{fig4:serialExtConstantNFULL}
\end{figure} 

\Acomment{
In Figure~\ref{fig4:serialConstantNEQ} we see a speedup comparison of cuPentBatchConstant and the serial version of the code keeping the batch size constant and varying the number of unknowns in a system. 
Again we take the time taken to execute the serial code and divide this by the time taken by the GPU code. 
At low batch numbers the speedup is minimal as the serial aspect of the pentadiagonal inversion dominates. 
As the size of the batch increases so does the speed up, with over $20\times$ faster performance for systems with a batch size of 8192. 
This finding is reinforced in Figure~\ref{fig4:serialConstantNFULL} where the number of unknowns is kept constant and the batch size is varied. 
Significant speedup can only be seen at higher batch sizes with over $10\times$ faster for most systems with a batch size $>2048$. 
Taking the batch number to an extreme in Figure~\ref{fig4:serialExtConstantNFULL} we can further see how cuPentBatch scales well in terms of increasing batch size.
}

\Acomment{
We see the clear presence of Amdahl's Law in these graphs, particularly in Figure~\ref{fig4:serialConstantNEQ} and Figure~\ref{fig4:serialConstantNFULL}. 
We see the benefits of parallelising the code until the serial aspect of the pentadiagonal solve begins to dominate. 
The performance increases then level off at this point and no increased speedup can be obtained from the system. 
As the problem is fundamentally memory bound with relatively few computations per grid point the speedup of the batch solution over serial is fundamentally bound by the number of SMs on the GPU and the possible memory bandwidth to access the data in global memory. 
We can see that saturation of the number of tasks, in other words how many members of the batch, the GPU can perform in parallel in Figure~\ref{fig4:serialExtConstantNFULL} as the performance increases level off at large batch sizes.
Further improved performance will be seen on better GPUs such as the Telsa V100 which have a better architecture for computing double precision numbers that the GPUs used in this study and also have additional SMs available.
}


\subsection{cuPentBatch vs. OpenMP}
The OpenMP implementation is the same as the serial code except we have parallelised the loop over the batches,  the speedup is measured as the time taken for the OpenMP version divided by cuPentBatchConstant. 
For the OpenMP benchmark we ran the batch solver with 512 unknowns and varied with high batch numbers, we choose this method for comparison as we are most concerned with scaling at high batch numbers as such problems benefit most from parallelisation. 
For consistency the same system was used as for the previous computations. 
The number of threads was set at 8 as this was the highest power of 2 available. The results in Figure~\ref{fig4:ompCompare} show a speedup of $5\times$ to $6\times$ in every case, thus demonstrating a substantial improvement in performance.  

\begin{figure}[ht]
	\centering
		\includegraphics[width=0.7\textwidth]{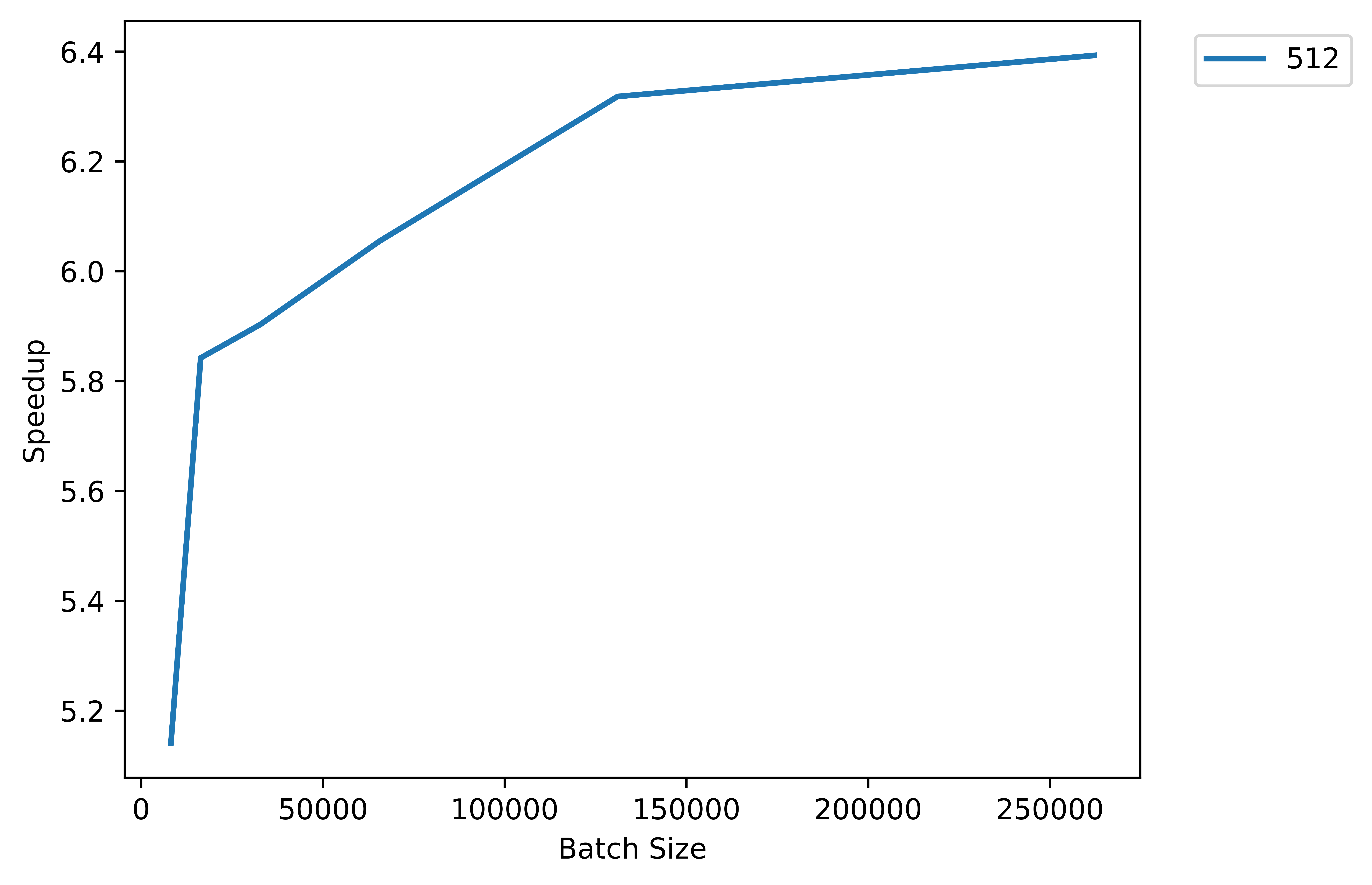}
		\caption{Performance of cuPentBatch versus an OpenMP version of the serial code, the number of unknowns was fixed at 512 and 8 threads were used for the OpenMP benchmark.}
	\label{fig4:ompCompare}
\end{figure}


\section{Discussion and Conclusions}
\label{sec4:conc}
Summarising, this chapter has introduced a new pentadiagonal solver (cuPentBatch) for implementation on NVIDIA GPUs, specifically aimed at solving large numbers of pentadiagonal problems in parallel, in batch mode.
It has been shown that the solver cuPentBatch is superior in terms of performance and efficiency to that of the standard existing NVIDIA pentadiagonal solver, gpsvInterleavedBatch. 
Our method further exhibits substantial performance speedup when compared with serial and OpenMP implementations.
Our method is particularly useful for solving parabolic numerical PDEs, where the matrix to be solved at each time step is constant and symmetric positive definite.
We have demonstrated a potential application of our method in the context of parameter studies, whereby the pertinent PDE possesses parameters which may be varied over different simulations to produce different solution types. 
By solving multiple instances of the PDE in batch mode, our method can speed up such parameter studies.

A further application of our method may in future be found in solving parabolic numerical PDEs in two and three dimensions -- here the pertinent parabolic PDE is typically solved using using an implicit temporal discretisation and a standard finite-difference spatial discretisation.  
The resulting matrix to be inverted at each time step can be reduced to a series of one-dimensional problems using the alternating-direction-implicit (ADI) technique~\cite{douglas1955numerical, ADIGPU}, see Chapter~\ref{chapter:cuSten} for an application of this method to the 2D Cahn--Hilliard equation.  
In this scenario, the present pentadiagonal batch solver may prove useful for parallelising this wide variety of numerical algorithms.

\lhead{\emph{Chapter 5}}  
\chapter{Extending cuPentBatch -- Efficient Interleaved Batch Matrix Solvers for CUDA}
\label{chapter:efficient}
In this chapter we extend the cuPentBatch solver to a more efficient implementation with a single LHS matrix shared amongst threads on the GPU, the work here was the subject of the paper "Efficient Interleaved Batch Matrix Solvers for CUDA"~\cite{gloster2019efficient}.

\section{Introduction}
\label{sec5:intro}
In this chapter we will focus on the development of tridiagonal and pentadiagonal batch solvers with single LHS, multiple RHS, matrices in CUDA for application in solving batches of 1D problems and 2D ADI methods.
In particular we develop solvers in this work which are more efficient in terms of data storage than the state of the art, this saving of data usage is due to the fact we have only one global copy of the LHS matrix rather than one for each thread/system.
While the primary improvement is in data storage reduction the new functions also provide increased speedup, due to better memory access patterns, when compared to the existing state of the art.
This increase in efficiency is also beneficial as it leads to further savings in resources, both in terms of run--time and electricity usage.

To date all existing batch pentadiagonal solvers in CUDA require that each system in the batch being solved have its own copy of the LHS matrix entries, leading to a sub-optimal use of GPU RAM and increased memory access overhead when only one global copy is actually needed. 
We also note that, while there are options for single LHS, multiple RHS, tridiagonal matrices in cuSPARSE these all rely on Cyclic Reduction or pivoting algorithms.
In the context of numerically solving PDEs on well behaved uniform grids these solution methods have unnecessary computational overhead when compared to the standard Thomas Algorithm for solving tridiagonal systems.
Thus we propose the implementation of two solvers, one implementing the Thomas Algorithm for tridiagonal matrices and the other implementing the pentadiagonal equivalent as presented in~\cite{gloster2019cupentbatch, numalgC}, each with a single global copy of the LHS matrix and multiple RHS matrices.
We will benchmark these implementations against cuThomasBatch \cite{cuThomasBatch} (implemented as gtsvInterleavedBatch in the cuSPARSE library) and existing work by the authors cuPentBatch \cite{gloster2019cupentbatch}, for tridiagonal and pentadiagonal matrices respectively, which are the existing state of the art for the algorithms we are interested in.
We point the user to the papers \cite{cuThomasBatch} and \cite{gloster2019cupentbatch} for existing benchmark comparisons of these algorithms with multiple LHS matrices with the cuSPARSE library and comparisons with serial/OpenMP implementations.
Thus we can compare our new implementations to the existing state of the art, cuThomasBatch and cuPentBatch, from which relative performance compared to the rest of cuSPARSE can be interpolated by the reader. 

The chapter is laid out as follows, in Section~\ref{sec5:method} we outline the modified methodology of interleaved data layout for RHS matrics and a single global copy for the LHS matrix.
In Sections~\ref{sec5:tri} and~\ref{sec5:pent} we outline the specifics for tridiagonal and pentadiagonal matrices respectively including also benchmarks against the state of the art algorithms to show performance is at least as good, if not better in most cases of batch size and matrix size. 
Finally we present our conclusions in Section~\ref{sec5:con}.


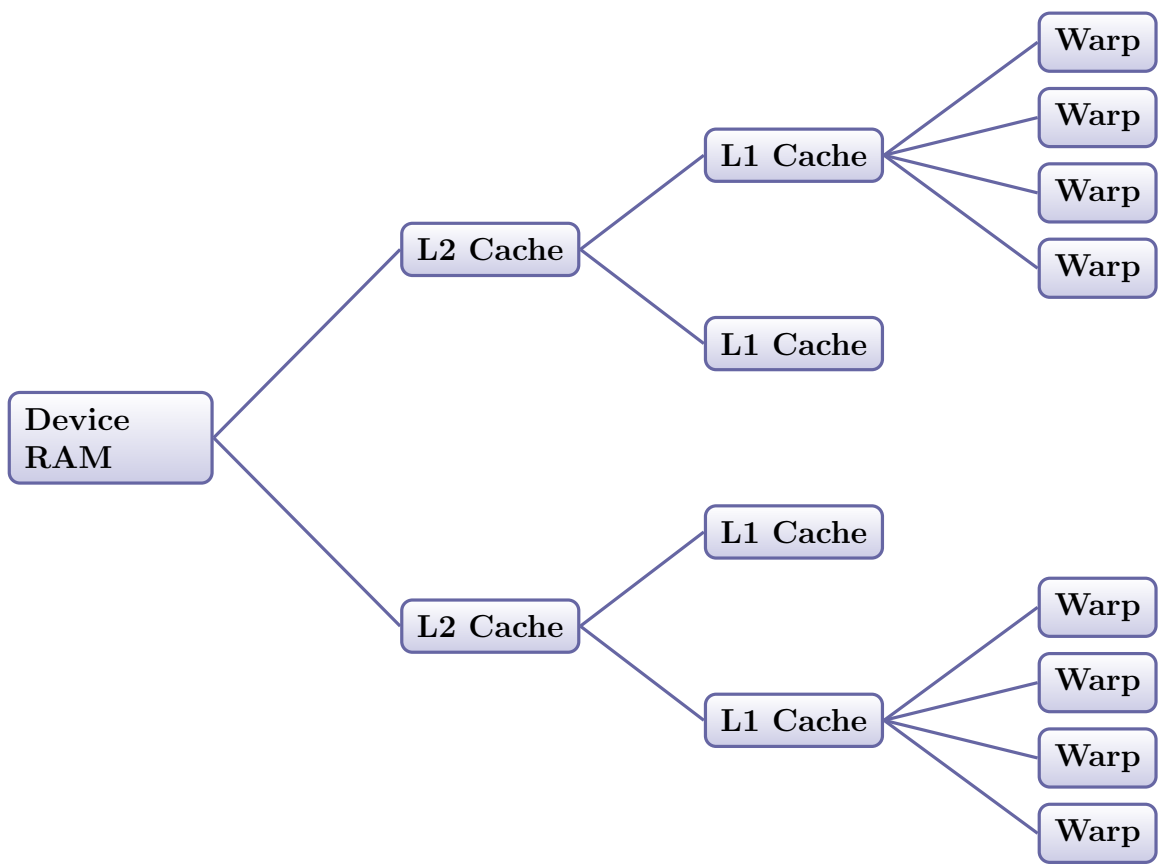
\begin{figure}
\begin{tikzpicture}[
    grow=right,
    level 1/.style={sibling distance=5.0cm,level distance=5.0cm},
    level 2/.style={sibling distance=2.5cm, level distance=4cm},
    level 3/.style={sibling distance=1.0cm, level distance=4cm},
    edge from parent/.style={very thick,draw=blue!40!black!60},
    edge from parent path={(\tikzparentnode.east) -- (\tikzchildnode.west)},
    kant/.style={text width=2cm, text centered, sloped},
    every node/.style={text ragged, inner sep=2mm},
    punkt/.style={rectangle, rounded corners, shade, top color=white,
    bottom color=blue!50!black!20, draw=blue!40!black!60, very
    thick }
    ]

\node[punkt, text width=5.5em] {\textbf{Device RAM}}
    child 
    {
        node [punkt]{\textbf{L2 Cache} \nodepart{second}}
        child 
	    {
	        node [punkt]{\textbf{L1 Cache} \nodepart{second}}
	        child 
		    {
		        node [punkt]{\textbf{Warp} \nodepart{second}}
		        edge from parent
		        node[kant, below, pos=.6] {}
		    }
		    child 
		    {
		        node [punkt]{\textbf{Warp} \nodepart{second}}
		        edge from parent
		        node[kant, below, pos=.6] {}
		    }
		    child 
		    {
		        node [punkt]{\textbf{Warp} \nodepart{second}}
		        edge from parent
		        node[kant, below, pos=.6] {}
		    }
		    child 
		    {
		        node [punkt]{\textbf{Warp} \nodepart{second}}
		        edge from parent
		        node[kant, below, pos=.6] {}
		    }
	        edge from parent
	        node[kant, below, pos=.6] {}
	    }
        child 
	    {
	        node [punkt]{\textbf{L1 Cache} \nodepart{second}}
	        edge from parent
	        node[kant, below, pos=.6] {}
	    }
        edge from parent
        node[kant, below, pos=.6] {}
    }
    child 
    {
        node [punkt]{\textbf{L2 Cache} \nodepart{second}}
        child 
	    {
	        node [punkt]{\textbf{L1 Cache} \nodepart{second}}
	        edge from parent
	        node[kant, below, pos=.6] {}
	    }
        child 
	    {
	        node [punkt]{\textbf{L1 Cache} \nodepart{second}}
	        child 
		    {
		        node [punkt]{\textbf{Warp} \nodepart{second}}
		        edge from parent
		        node[kant, below, pos=.6] {}
		    }
		    child 
		    {
		        node [punkt]{\textbf{Warp} \nodepart{second}}
		        edge from parent
		        node[kant, below, pos=.6] {}
		    }
		    child 
		    {
		        node [punkt]{\textbf{Warp} \nodepart{second}}
		        edge from parent
		        node[kant, below, pos=.6] {}
		    }
		    child 
		    {
		        node [punkt]{\textbf{Warp} \nodepart{second}}
		        edge from parent
		        node[kant, below, pos=.6] {}
		    }
	        edge from parent
	        node[kant, below, pos=.6] {}
	    }    
        edge from parent
        node[kant, below, pos=.6] {}
    };
\end{tikzpicture}
\caption{Diagram of memory hierarchy in an NVIDIA GPU. We omit nodes for the purposes of the diagram, the top and bottom branches are representative of the memory routing from Device Ram to warp. When every thread is accessing the same memory location warps which share the same SM will receive an L1 cache hit, SMs sharing an L2 cache will receive an L2 cache hit and if the data is not present in the L2 cache it will be retrieved from the Device RAM. Thus it can be seen here how the first warp to request a location in memory in a given group will be the slowest but accesses for following warps become faster and faster as the memory propagates down the tree into various caches.}
\label{fig5:memoryaccess}
\end{figure}

\section{Matrix Solver Methodology}
\label{sec5:method}
The methodologies implementing the tridiagonal and pentadiagonal algorithms we are concerned with in this chapter have to date relied on using interleaved data layouts in global memory.
Each thread handles a particular system solution and the system is solved by a forward sweep followed by a backward sweep.  
The interleaved data format is to ensure coalescence of data accesses during theses sweeps, with each thread retrieving neighbouring memory locations in global memory for various LHS and RHS matrix entries~\cite{cuThomasBatch}.
This maximises bandwidth but is limiting to the overall memory budget as each system requires its own copy of the LHS matrix.
Instead we propose to store globally just a single copy of the LHS matrix entries with all threads accessing the desired value at the same time, the RHS matrices will still be stored in the same interleaved data access pattern as before.
This will reduce the overall bandwidth of memory accesses to the LHS but this will not harm the solver's performance as we will show in later sections.
Critically also this approach of using just a single LHS matrix will save drastically on the amount of memory used in batch solvers allowing for larger batch numbers and greater matrix sizes to be solved on a single GPU, increasing hardware usage efficiency.

We now discuss the memory access pattern for the single LHS matrix.
Global memory access in GPUs is achieved via the L2 cache.
When a piece of memory is requested for by a warp this memory is retrieved from RAM and copied to the L2 cache followed by the L1 cache on the SM before then being used by the relevant warps being computed on that SM.
Each SM has their own pipe to the L2 cache. 
So when warps from different SMs all request the same memory they will all get a cache hit in at least the L2 cache except the first one to arrive which will be the warp to retrieve the data from the RAM on the device.
We can guarantee that this will occur as each warp will be computing the same instruction in either the forward or backward sweep of the solver but warps on different SMs will not arrive to the L2 cache at the same time.
In addition to this warps which share L1 caches will get cache hits here when they're not the first to request the piece of memory. 
A diagram of this memory access pattern discussion can also be seen in Figure~\ref{fig5:memoryaccess}.
Thus it is clear that given every piece of memory must follow the above path most of the threads solving a batched problem will benefit from speed--ups due to cache hits as they no longer need to retrieve their own copy of the data.

In the following two sections we present specific details for each of the tridiagonal and pentadiagonal schemes along with the previously discussed benchmarks.
The benchmarks are performed on 16GB Tesla V100 GPUs with the pre-factorisation step for the single LHS performed on the CPU.
We shall refer to the new versions of the algorithms as cuThomasConstantBatch and cuPentConstantBatch for tridiagonal and pentadiagonal respectively.
We use the nomenclature `Constant' denoting the fact that all systems have the same LHS. 
The RHS will be stored as usual in an interleaved data format and is computed using cuSten~\cite{gloster2019custen}.
For notational purposes we will use $N$ to describe the number of unknowns in our systems and $M$ to describe the batch size (number of systems being solved).


\section{Tridiagonal Systems}
\label{sec5:tri}
In this section we first present the Thomas Algorithm~\cite{numalgC} and how it is modified for batch solutions with multiple RHS matrices and with a single LHS. 
We then present our chosen benchmark problem of periodic diffusion equations and then the results.

\subsection{Tridiagonal Inversion Algorithm}
We begin with a generalised tridiagonal matrix system $\mxA \vecx=\vecrhsT$ with $\mxA$ given by
\begin{equation*}
\mxA = 
\begin{pmatrix}
b_1 & c_1 & 0 & \cdots &   \cdots & 0 \\
a_2 & b_2 & c_2 & 0 & \cdots   & \vdots \\
0 & a_3 & b_3 & c_3 & 0 &  \vdots \\
\vdots & \ddots & \ddots & \ddots & \ddots & 0  \\
\vdots & \cdots  & 0   & a_{N - 1}  & b_{N - 1} & c_{N - 1} \\
0 & \cdots  &   \cdots & 0 & a_{N}  & b_{N} 
\end{pmatrix}.
\end{equation*}
We then solve this system using a pre-factorisation step followed by a forwards and backwards sweep. 
The pre-factorisation is given by
\begin{equation}
\hat{c}_1 = \frac{c_1}{b_1}
\end{equation}
And then for $i = 2, \dots, N$
\begin{equation}
\hat{c}_i = \frac{c_i}{b_i - a_i \hat{c}_{i - 1}}
\end{equation}
For the forwards sweep we have
\begin{equation}
\hat{d}_1 = \frac{d_1}{b_1}
\end{equation}
\begin{equation}
\hat{d}_i = \frac{d_i - a_i \hat{d}_{i - 1}}{b_i - a_i \hat{c}_{i - 1}}
\end{equation}
While for the backwards sweep we have 
\begin{equation}
x_N = \hat{d}_N
\end{equation}
And for $i = N-1, \dots, 1$
\begin{equation}
x_i = \hat{d}_i - a_i \hat{c}_i
\end{equation}

Previous applications of this algorithm~\cite{cuThomasBatch} required that each thread had access to its own copy of $4$ vectors, the $3$ diagonals $a_i$, $b_i$ and $c_i$ along with the the RHS $d_i$.
These would then be overwritten in the pre-factorisation and solve steps to save memory, thus the total memory usage here is $O(4 \times M \times N)$.
We now limit the LHS to a single global case that will be accessed simultaneously using all threads as discussed in Section~\ref{sec5:method} and retain the individual RHS in interleaved format for each thread $f_i$.
This reduces the data storage to $O(3 \times N + M \times N)$, an approximate 75\% reduction.
We present benchmark methodology and results for this method in the following subsections.

\subsection{Benchmark Problem}
For a benchmark problem we solve the diffusion equation, a standard model equation in Computational Science and Engineering, its presence can be seen in most systems involving heat and mass transfer.
The solution as $t \rightarrow \infty$ is also a solution of a Poisson equation.
The equation in one dimension is given as 
\begin{equation}
\frac{\partial C}{\partial t} = \alpha \frac{\partial^2 C}{\partial x^2} 
\label{eq5:diffusion}
\end{equation} 
where $\alpha$ is the diffusion coefficient.
We solve this equation on a periodic domain of length $L$ such that $C(x + L) = C(x)$ with an initial condition $C(x, t = 0) = f(x)$ valid on the domain.
We rescale by setting $\alpha = 1$ and $L = 1$ and integrate Eq.~\eqref{eq5:diffusion} in time using a standard Crank--Nicholson scheme with central differences for space which is unconditionally stable.
Finite differencing is done using standard notation
\begin{equation}
C_i^n = C(x=i\Delta x ,t=n\Delta t)
\label{eq5:stdDiff}
\end{equation}
where $\Delta x = L / N$ and $i = 1 \dots N$.
Thus our numerical scheme can be written as 
\begin{equation}
- \sigma_x C_{i - 1}^{n+1} + (1 + 2 \sigma_x) C_{i}^{n+1} - \sigma_x C_{i+1}^{n+1}
=  \sigma_x C_{i - 1}^{n} + (1 - 2 \sigma_x)C_{i}^{n} + \sigma_x C_{i+1}^{n} 
\label{eq5:1ddiffscheme}
\end{equation} 
where
\begin{equation}
\sigma_x = \frac{\Delta t}{2 \Delta x^2}
\end{equation}
Thus we can relate these coefficients to matrix entries by
\begin{equation}
a_i  = - \sigma_x,\qquad b_i = 1 + 2 \sigma_x,\qquad
c_i = - \sigma_x
\end{equation}%

In the following section we present our method to deal with the periodicity of the matrix and then after that present the benchmark comparison with the existing state of the art.

\subsection{Periodic Tridiagonal Matrix}
As the system is periodic two extra entries will appear in the matrix, one in the top right corner and another in the bottom left, thus our matrix is now given by
\begin{equation*}
\mxA = 
\begin{pmatrix}
b & c & 0 & \cdots &  \cdots & 0 & a \\
a & b & c & 0 & \cdots & \cdots   & 0 \\
0 & a & b & c & 0 & \cdots  & \vdots \\
\vdots & \ddots & \ddots & \ddots & \ddots & \ddots &   \vdots\\
\vdots & \cdots  & 0  &  a  & b  & c & 0  \\
0 & \cdots & \cdots  & 0   & a  & b & c \\
c & 0 & \cdots  &   \cdots & 0 & a  & b 
\end{pmatrix}.
\end{equation*}
In order to deal with these we use the Sherman-Morrison formula.
We rewrite our system as 
\begin{equation}
\mxA\vecx  = (\mxA' + \mxU \otimes \mxV) \vecx = \vecrhsT 
\end{equation} 
Where 
\begin{subequations}
\begin{equation}
\mxU = 
\begin{pmatrix}
-b \\
0 \\
\vdots \\
\vdots \\
0 \\
c\\
\end{pmatrix},
\qquad
\mxV = 
\begin{pmatrix}
1 \\
0 \\
\vdots \\
\vdots \\
0 \\
- a / b\\
\end{pmatrix},
\qquad
\end{equation}
\Acomment{
\begin{equation}
\mxA' = 
\begin{pmatrix}
2b & c & 0 & \cdots &  \cdots & 0 & 0 \\
a & b & c & 0 & \cdots & \cdots   & 0 \\
0 & a & b & c & 0 & \cdots  & \vdots \\
\vdots & \ddots & \ddots & \ddots & \ddots & \ddots &   \vdots\\
\vdots & \cdots  & 0  &  a  & b  & c & 0  \\
0 & \cdots & \cdots  & 0   & a  & b & c \\
0 & 0 & \cdots  &   \cdots & 0 & a  & b + a c / b
\end{pmatrix}
\end{equation}%
}
\end{subequations}%
Thus two tridiagonal systems must now be solved
\begin{equation}
\mxA' \vecy = \vecrhsT  \qquad \mxA' \vecx = \mxU
\end{equation}
the second of which need only be performed once at the beginning of a given simulation.
Finally to recover $\vecx$ we substitute these results into 

\begin{equation}
\vecx = \vecy - \left(\frac{\mxV \cdot \vecy}{1 + (\mxV \cdot \vecz)}\right)\vecz
\end{equation}

\begin{figure}[h]
	\centering
		\includegraphics[width=0.8\textwidth]{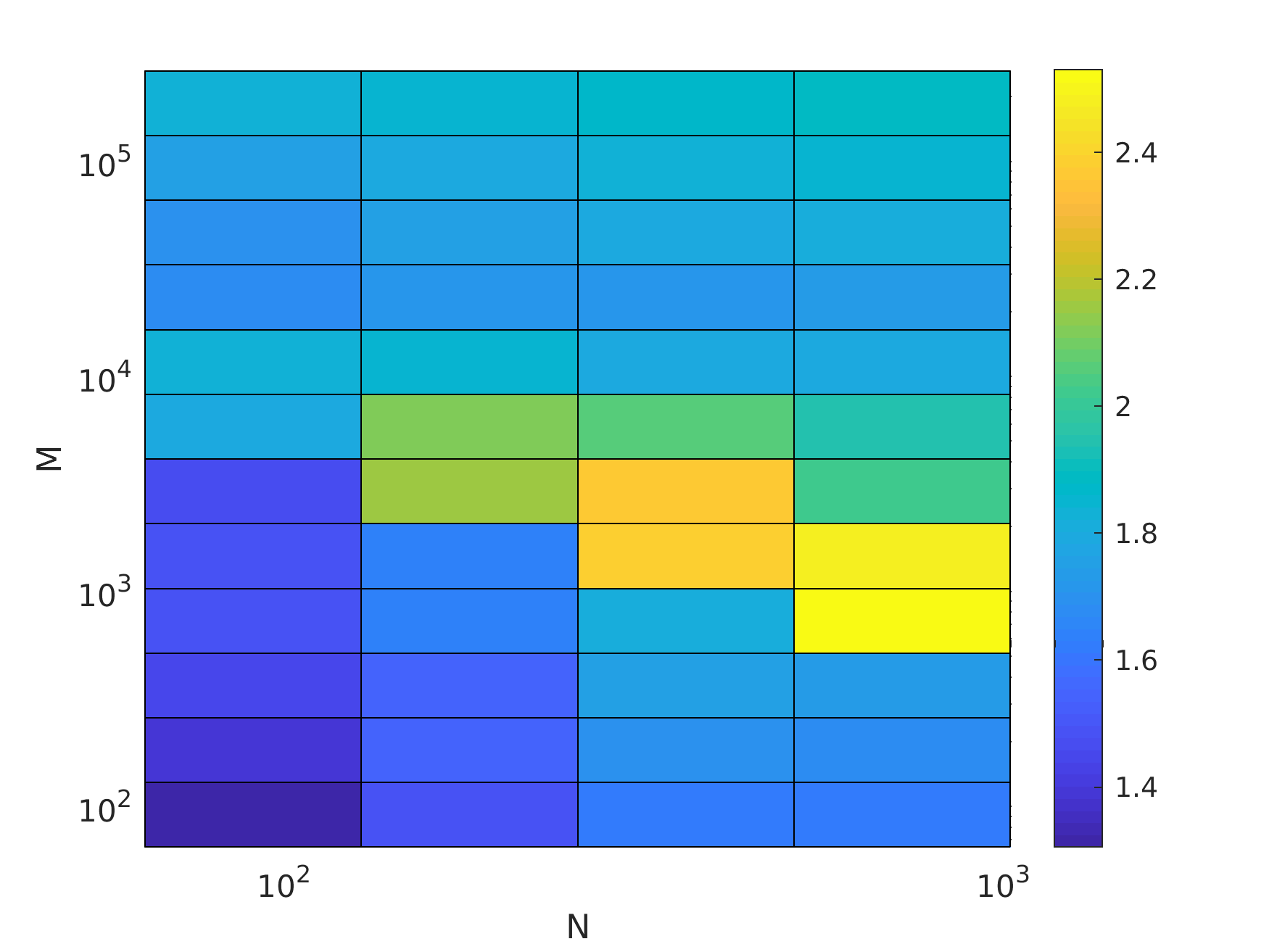}
		\caption{Speedup of cuThomasConstantBatch versus cuThomasBatch (gtsvInterleavedBatch).}
	\label{fig5:speedupInterleaved}
\end{figure}

\subsection{Benchmark Results}
The diffusion equation benchmark problem presented above is solved for 1000 time--steps in order to extract the relevant timing statistics and average out any effects of background processes on the benchmark.
Both codes are timed for just the time--stepping sections of the process, we omit start up costs such as memory allocation, setting of parameters etc.
The speed--up results can be seen in Figure~\ref{fig5:speedupInterleaved}, in all situations there is a clear speed--up over the existing state of the art cuThomasBatch. 
The largest speed--ups are for $N$ large and moderate $M$.
Significant speed--up is available for all of large $M$ which is the domain where GPUs would typically be deployed to solve the problem.

It should be noted that some of the speed--up seen here can be attributed to the pre--factorisation step which the cuThomasBatch implementation does not have. 
cuThomasBatch requires that the LHS matrix be reset at every time--step as the pre--factorisation and solve steps are carried out in the one function, leading to an overwrite of the data and making it unusable for repeated time--stepping.
This overwriting feature has been seen in previous studies~\cite{gloster2019cupentbatch} where the authors also carried out a rewrite to make the benchmarking conditions between their work and the state of the art.
It was shown that not all of the speed--up can be attributed to the lack of needing to reset the LHS matrix, thus some of the performance increase we are seeing in Figure~\ref{fig5:speedupInterleaved} can be attributed to the new data layout presented in this chapter.
Our second benchmark for the pentadiagonal case will also show this as there was no resetting of the matrix required.


\section{Pentadiagonal Systems}
\label{sec5:pent}
In this section we first present a modified version of the pentadiagonal inversion algorithm as presented in~\cite{gloster2019cupentbatch} with a single LHS. 
Then we present the benchmark problem and finally this is followed by the results comparing our new implementation with existing state of the art.

\subsection{Pentadiagonal Inversion Algorithm}
In this section we describe a standard numerical method~\cite{numalgC, gloster2019cupentbatch} for solving a pentadiagonal problem $\mxA\vecx=\vecrhs$.  We present the algorithm in a general context so we have a pentadiagonal matrix given by

\begin{equation*}
\mxA = 
\begin{pmatrix}
c_1 & d_1 & e_1 & 0 & \cdots &  0 & \cdots & 0 \\
b_2 & c_2 & d_2 & e_2 & 0 & \cdots & \cdots   & \vdots \\
a_3 & b_3 & c_3 & d_3 & e_3 & 0 & \cdots  & 0\\
0 & \ddots & \ddots & \ddots & \ddots & \ddots & \ddots &   \vdots\\
\vdots& \ddots & \ddots & \ddots & \ddots & \ddots & \ddots & 0 \\
0 & \cdots  & 0  &  a_{N - 2}  & b_{N - 2}  & c_{N - 2} & d_{N - 2} & e_{N - 2} \\
0 & \cdots & \cdots  & 0   & a_{N - 1}  & b_{N - 1} & c_{N - 1} & d _{N - 1}\\
0 & \cdots & \cdots  &   \cdots & 0 &  a_{N} & b_{N}  & c_{N} 
\end{pmatrix}.
\end{equation*}
Three steps are required to solve the system:
\begin{enumerate}
\item Factor $\mxA = \mxL\mxR$  to obtain $\mxL$ and $\mxR$.
\item Find $\vecg$ from $\vecf = \mxL\vecg$
\item Back-substitute to find $\vecx$ from $\mxR\vecx = \vecg$
\end{enumerate}
Here, $\mxL$, $\mxR$ and $\vecg$ are given by the following equations:
\begin{subequations}
\begin{equation}
\mxL = 
\begin{pmatrix}
\alpha_1 &  &  &  &  &   &   \\
\beta_2 & \alpha_2 &  &  &  &  &     \\
\epsilon_3 & \beta_3 & \alpha_3 &  &  &  &   \\
 & \ddots & \ddots & \ddots &  &  &     \\
 &  &  \epsilon_{N - 1}  & \beta_{N - 1}  & \alpha_{N - 2} &    \\
&  &    & \epsilon_{N - 1}  & \beta_{N - 1} & \alpha_{N }\\
\end{pmatrix},
\qquad
\vecg = \begin{pmatrix}
g_{1} \\
g_{2} \\
\vdots \\
\vdots \\
g_{N-1} \\
g_{N}
\end{pmatrix},
\end{equation}
\begin{equation}
\mxR = 
\begin{pmatrix}
1 & \gamma_1  & \delta_1  &  &  &   &   \\
 & 1 & \gamma_2  & \delta_2  &  &  &     \\
 &  & \ddots & \ddots & \ddots  &  &   \\
 & & & 1  & \gamma_{N-2}  & \delta_{N-2}    \\
 &  &  &  & 1 & \gamma_{N-1}    \\
&  &    & &  & 1\\
\end{pmatrix}
\end{equation}%
\end{subequations}%
(the other entries in $\mxL$ and $\mxR$ are zero).  
The explicit factorisation steps for the factorisation $\mxA=\mxL\mxR$ are as follows:
\begin{enumerate}
\item $\alpha_1 = c_1$
\item $\gamma_1 = \frac{d_1}{\alpha_1}$
\item $\delta_1 = \frac{e_1}{\alpha_1}$
\item $\beta_2 = b_2$
\item $\alpha_2 = c_2 - \beta_2\gamma_1$
\item $\gamma_2 = \frac{d_2 - \beta_2 \delta_1}{\alpha_2}$
\item $\delta_2 = \frac{e_2}{\alpha_2}$
\item For each $i = 3, \dots, N-2$
\begin{enumerate}
\item $\beta_i = b_i - a_i \gamma_{i-2}$
\item $\alpha_i = c_i - a_i\delta_{i-2} - \beta_i \gamma_{i-1}$
\item $\gamma_i = \frac{d_i - \beta_i \delta_{i-1}}{\alpha_i}$
\item $\delta_i = \frac{e_i}{\alpha_i}$
\end{enumerate}
\item $\beta_{N-1} = b_{N-1} - a_{N - 1}\gamma_{N-3}$
\item $\alpha_{N - 1} =  c_{N-1} - a_{N-1}\delta_{N-3} - \beta_{N-1}\gamma_{N-2}$
\item $\gamma_{N-1} = \frac{d_{N-1}-\beta_{N-1}\delta_{N-2}}{\alpha_{N-1}}$
\item $\beta_{N} = b_{N} - a_{N }\gamma_{N-2}$
\item $\alpha_{N} =  c_{N}- a_{N}\delta_{N-2} - \beta_{N}\gamma_{N-1}$
\item $\epsilon_i = a_i, \quad \forall i$
\end{enumerate}
The steps to find $\vecg$ are as follows:
\begin{enumerate}
\item $g_1 = \frac{f_1}{ \alpha_1}$
\item $g_2 = \frac{f_2 - \beta_2 g_1}{\alpha_2}$
\item  $g_i = \frac{f_i - \epsilon_i g_{i-2} - \beta_i g_{i - 1}}{\alpha_i} \quad \forall i = 3 \cdots N$
\end{enumerate}
Finally, the back-substitution steps  find $\vecx$ are as follows:
\begin{enumerate}
\item $x_N = g_N$
\item $x_{N-1} = g_{N-1} - \gamma_{N-1}x_N$
\item  $x_i = g_i - \gamma_i x_{i+1} - \delta_{i}x_{i+2} \quad \forall i = (N-2) \cdots 1$
\end{enumerate}

Previous applications of this algorithm~\cite{gloster2019cupentbatch} required that each thread had access to its own copy of $6$ vectors, the $5$ diagonals $a_i$, $b_i$, $c_i$, $d_i$ and $e_i$ along with the RHS $f_i$.
These would then be overwritten in the pre-factorisation and solve steps to save memory, thus the total memory usage here is $O(6 \times M \times N)$.
We now limit the LHS to a single global case that will be accessed simultaneously using all threads as discussed in Section~\ref{sec5:method} and retain the individual RHS in interleaved format for each thread $f_i$.
This reduces the data storage to $O(5 \times N + M \times N)$, this is an approximate 83\% reduction in data usage.
We present benchmark methodology and results for this method in the following subsections.
For completion we present results for an extra implementation where all the entries on respective diagonals are equal eliminating the need to store $a_i$ (which is also $\epsilon_i$), reducing the storage further to $O(4 \times N + M \times N)$, we shall refer to these results as cuPentUniformBatch.

\subsection{Benchmark Problem}
The benchmark system is the same as the one presented in Section~\ref{sec4:hyper}, it had been included here in the full paper but we now omit it as it would be merely repetition.

\begin{figure}[h]
	\centering
		\includegraphics[width=0.8\textwidth]{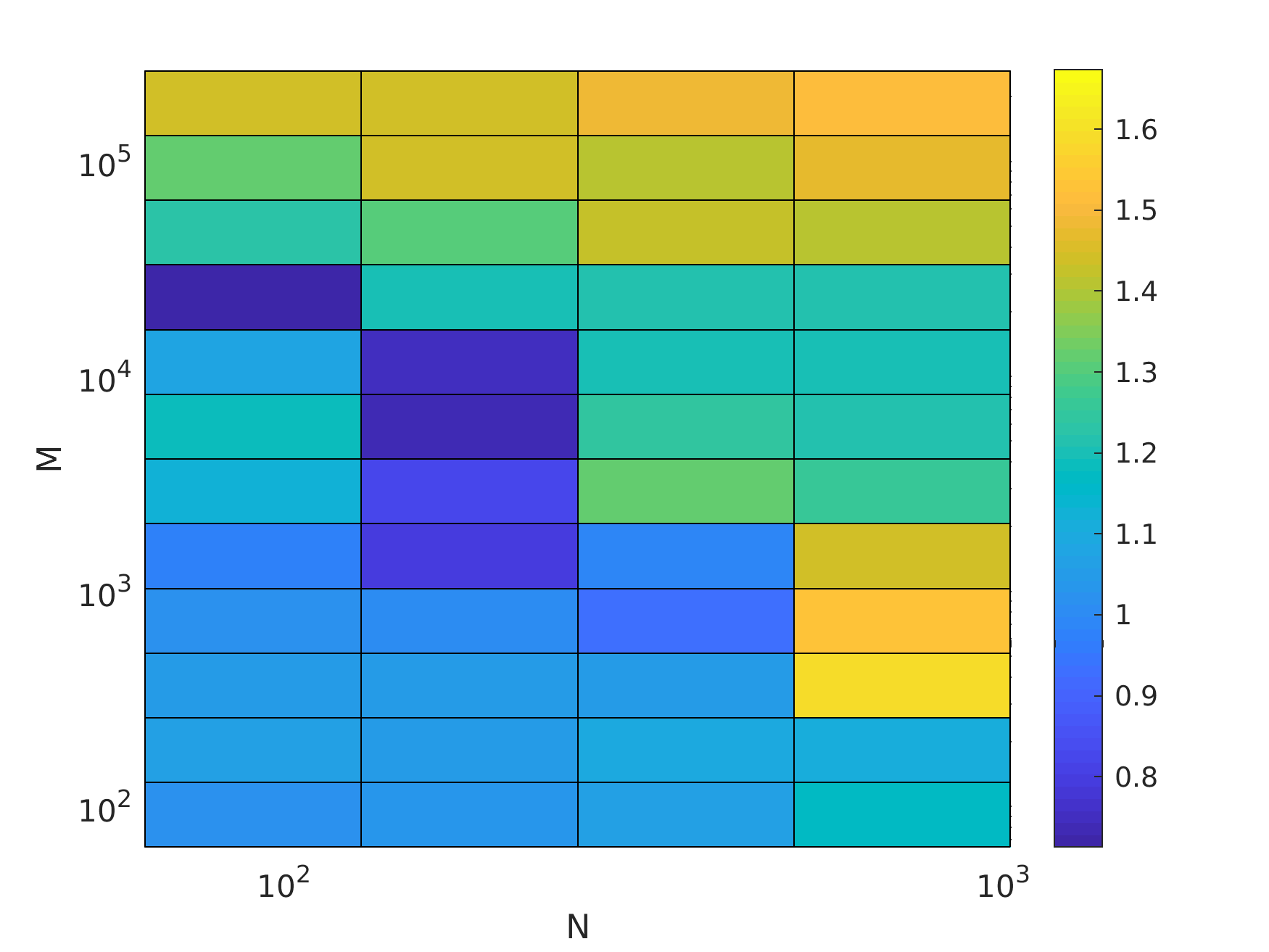}
		\caption{Speedup of cuPentConstantBatch versus cuPentBatch.}
	\label{fig5:pentConstantSpeedUp}
\end{figure}

\subsection{Benchmark Results}
We plot the speed--up of cuPentConstantBatch versus cuPentBatch solving batches of the above method for the hyperdiffusion equation in Figure~\ref{fig5:pentConstantSpeedUp}. 
In order to calculate the values for $\vecrhs$ we use cuSten \cite{gloster2019custen}.
It can be seen in the figure that cuPentBatchConstant outperforms cuPentBatch consistently for high values of both $M$ and $N$. 
At low $M$ and $N$ they are roughly equivalent and there are some areas where cuPentBatch performs better, generally a standard CPU implementation is more desirable at these low numbers when the movement over the PCI lane is taken into account along with other overheads so we can discount these. 
If we push the values out further than those plotted the difference becomes orders of magnitudes as the memory for cuPentBatch rapidly exceeds the available RAM on the GPU while cuPentBatchConstant can still fit.

Thus we conclude in situations where there is uniform LHS matrices with multiple RHS and the algorithm is suitably stable for the problem (symmetric positive definite is enough here) that cuPentBatchConstant is a better choice in terms of both memory usage (allowing for larger values of $M$ and $N$ on one GPU) and speed.
Similar results can be seen in Figure~\ref{fig5:pentUniformSpeedUp} for the case of cuPentUniformBatch with slight improvements in overall speed in certain locations due to the lack of access to the vector storing $\epsilon_i$.
Slight advantages of cuPentUniformBatch over cuPentConstantBatch are apparent where it can be used, if the functions were being repeatedly called enough times in a given simulation its use is certainly warranted as the savings on time will compound with each call. 

\begin{figure}[h]
	\centering
		\includegraphics[width=0.8\textwidth]{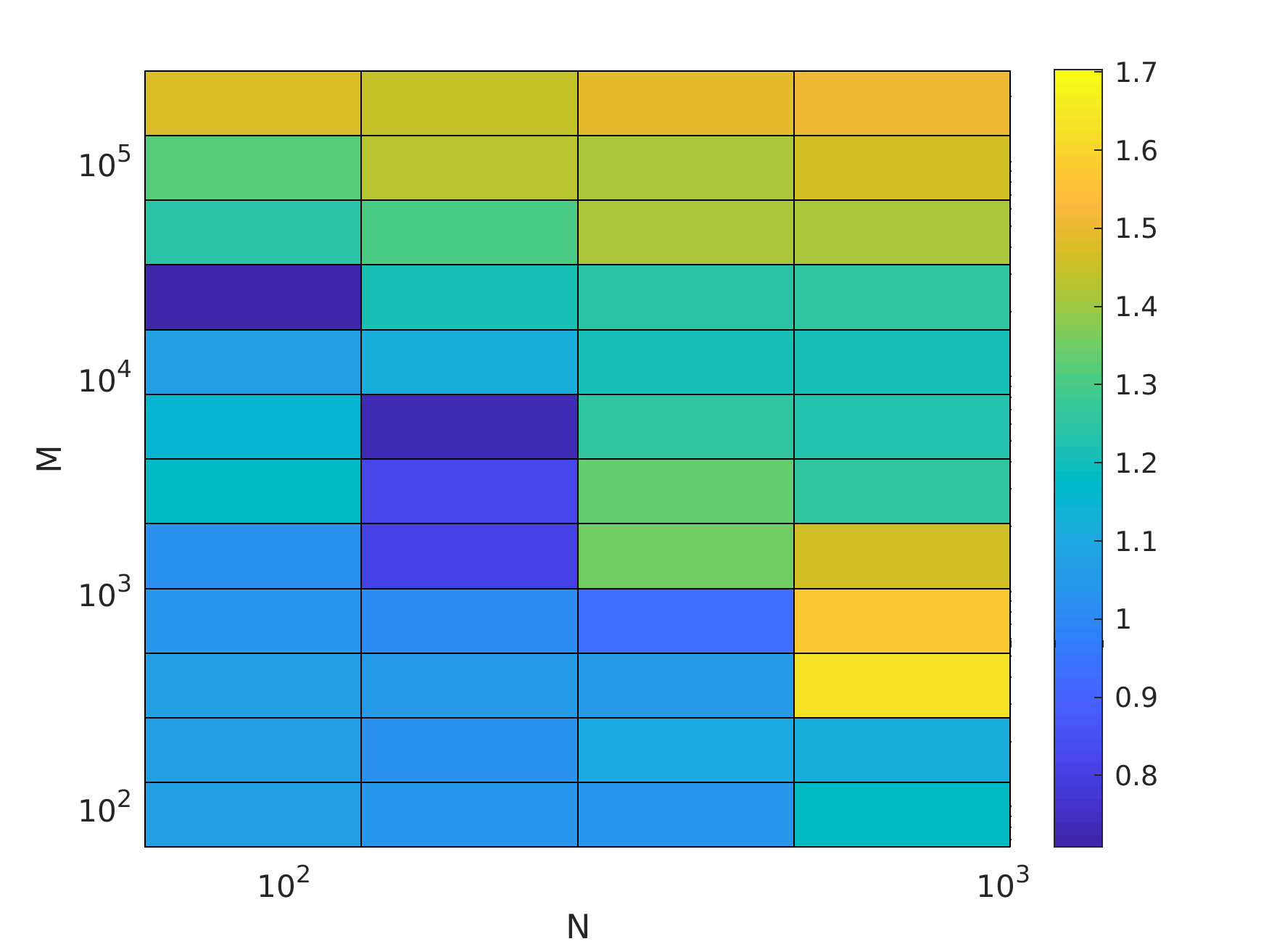}
		\caption{Speed--up of cuPentUnifromBatch versus cuPentBatch.}
	\label{fig5:pentUniformSpeedUp}
\end{figure}


\section{Conclusions}
\label{sec5:con}
We have shown that an interleaved batch of RHS matrices along with a single LHS achieves large reductions in data usage along with a substantial speed--up when compared with the existing state of the art.
The reduction in data usage allows for greater use of hardware resources to be made with significant extra space that can now be devoted to solving more systems of equations simultaneously rather than unnecessarily storing data.
The extra equations now being solved by one GPU along with the speed--ups provided by the new implementation are a significant improvement over the state of the art in applications where only one LHS matrix is required for the solution of all the systems in the batch.
In addition the reduction in the number of GPUs one would need when solving very large batches, coupled with the speed--ups achieved, leads to a reduction in electricity usage.
As HPC moves into the future, energy needs are becoming more and more prevalent, the increases in FLOPs/WATT provided by GPUs have both financial and environmental implications.
Thus there is an increasing need for the use of GPUs in HPC platforms and the need for efficient algorithms such as those presented in this chapter.
\lhead{\emph{Chapter 6}}  
\chapter{Batched Solutions of the 1D Cahn--Hilliard Equation}
\label{chapter:batched}
In this chapter we present results of batched 1D Cahn--Hilliard equations applying the GPU methodologies of the previous chapters.
In particular focusing on reproducing numerically the analytic 1D scaling results of~\cite{argentina2005coarsening} and presenting a methodology that can be used to automate the production of flow pattern maps of large parameter spaces.

\section{Introduction}
In this short chapter we present a methodology for examining a variety of parameter spaces of a given PDE.
In particular we focus on two cases relating to the Cahn--Hilliard equation in 1D which we solve numerically, running batches of simulations using a GPU methodology to fill the parameter space and yield the necessary parallelism.
GPUs offer increasingly improved parallel performance over standard CPU parallelisation methods and also offer improved energy efficiency with increased GFLOPs/Watt performance, thus reducing the environmental impact of large computer simulations.
While traditionally a lot of work has been devoted to parallelising single large problems a increasing trend has focused on decomposing large problems into batches of smaller problems and then solving independently before collecting the results or indeed focusing on solving batches of small problems to form a parameter study.
In this chapter we will present methods which focus on the later of these approaches by applying previous computational works by the same authors~\cite{gloster2019cupentbatch, gloster2019efficient, gloster2019custen}.
We solve batches of independent 1D Cahn--Hilliard equations within a single GPU program varying initial conditions and various forcing parameters to study averaged 1D scaling dynamics and to produce flow--pattern maps in an automated fashion, an improvement over the current by--hand/analytic methodologies.

In Section~\ref{sec6:1D} we present a numerical method for solving multiple 1D Cahn--Hilliard equations in a batch using an implicit scheme, a convergence study for the scheme and then a study of averaged 1D scaling. 
While in Section~\ref{sec6:flowpattern} we focus on the application of GPUs to produce a large parameter space of simulation results and then using k--means clustering extract a flow--pattern map.


\begin{figure}
    \centering
        \includegraphics[width=0.5\textwidth]{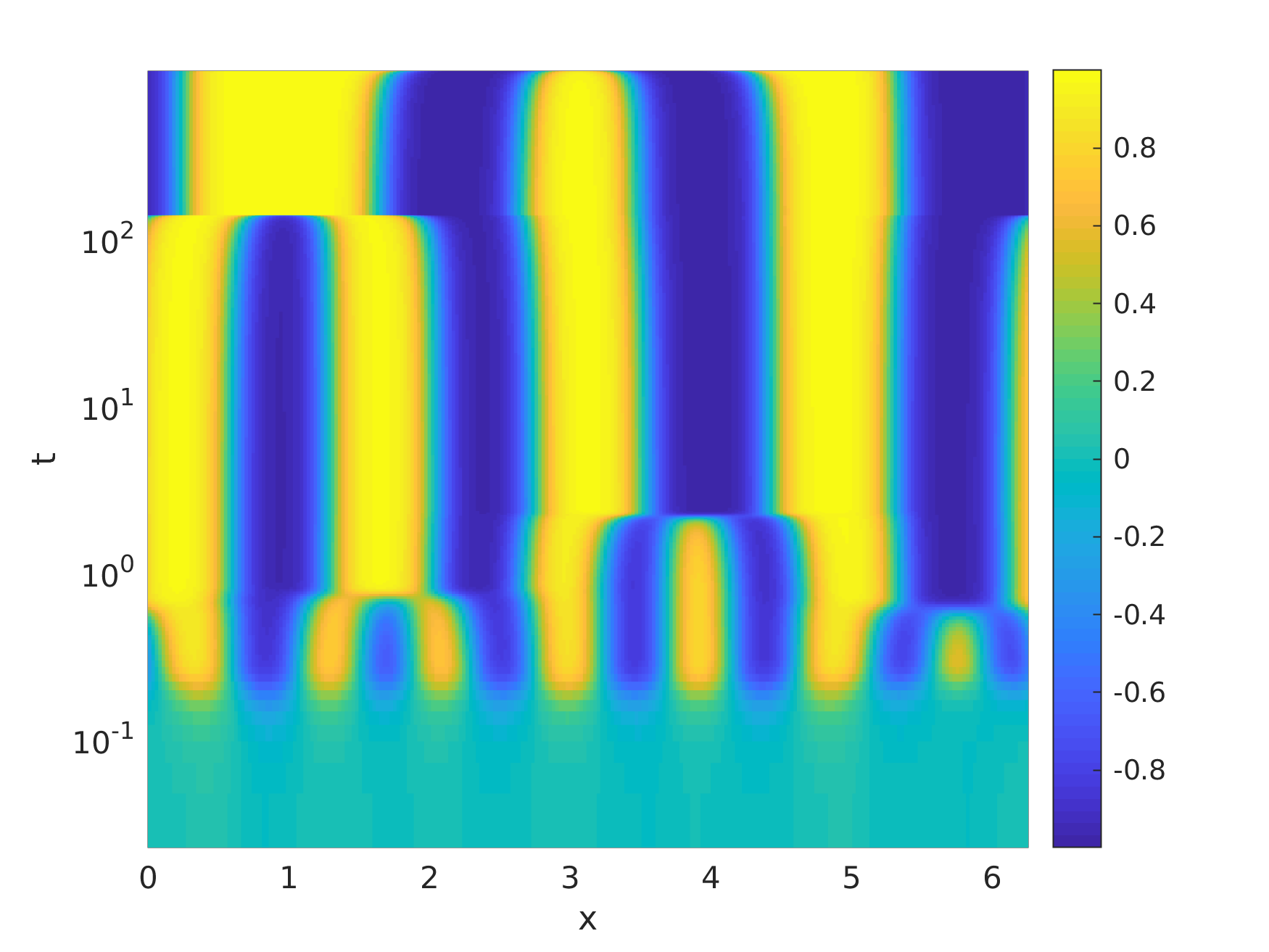}
        \caption{Space time plot showing the evolution of a 1D Cahn--Hilliard equation evolving in time, domain size is $2\pi$ with randomised initial conditions drawn from a uniform distribution.}
    \label{fig6:1Dspacetime}
\end{figure}

\section{1D Cahn--Hilliard Scaling}
\label{sec6:1D}
We now reduce equation~\eqref{eq2:ch} to its 1D form in order to study the averaged time--scale over which coarsening occurs. 
\Acomment{We also set the coefficient $D = 1$ for convenience as it has no effect on the dynamics of the equation, it simply scales the RHS.}

\begin{equation}
\frac{\partial C}{\partial t} = \frac{\partial^2}{\partial x^2} \left(C^3 - C - \gamma \frac{\partial^2 C}{\partial x^2} \right)
\label{eq6:1Dcahn}
\end{equation}

This equation is commonly solved for initial conditions where the binary mixture is well mixed.
There are then three time--scales involved in the problem. 
First is the fastest time--scale dominated by the hyperdiffusion term, the initial condition of the domain rapidly decays to a distribution where $|C(x,t)| < \epsilon$. 
This distribution then dominates the initial times of a given simulation but is unstable and leads to nucleation which occurs on the next fastest time--scale.
Bubbles begin to nucleate in the domain and subsequently grow to form regions of $C(x,t) = \pm 1$ joined together by $\approx \tanh(x)$ curves over a length scale given by $\sqrt{\gamma}$.
Finally there is the slowest time--scale of the problem where neighbouring bubble join together, reducing the number of regions of $C(x,t) = \pm 1$ as time increases, eventually reaching a steady state of two regions, one of $C(x, t) = 1$ for some region of $x$ and the other of $C(x, t) = 1$, these are again joined by $\tanh(x)$ like regions of width $\sqrt{\gamma}$. 
This final state is also know as a finite size effect as the smaller the domain the fewer bubbles and the quicker it occurs, thus how soon this feature emerges is dominated by the finite size of a domain.

It is this final time--scale that we will be studying in this section. 
In particular the coarsening events of individual 1D bubbles does not \Acomment{behave with any known analytic} expression but it has been shown in~\cite{argentina2005coarsening} that, on average, the domain coarsens with $\log(t)$, this is discussed further below in Section~\ref{sec6:1Dscale} where we solve a batch of individual equations to recover this average scaling.
In Figure~\ref{fig6:1Dspacetime} we can see the development of \Acomment{the different time--scales}, at the start of the simulation we can see the decay of the domain to $\approx 0$, followed by nucleation and finally coarsening.
For this simulation we took \Acomment{$\gamma = 0.01$} and ran the simulation to a final time of $T = 1000$ and used the numerical scheme presented in the following sub--section~\ref{sec6:1Dscheme}.
The initial condition is random with the numbers drawn from an uniform distribution of $-0.1 \leq C \leq 0.1$, the domain is $2\pi$ and we take $N = 256$ points to ensure $\Delta x \ll 2\pi \sqrt(\gamma)$ to properly resolve the interfaces.

\subsection{1D Numerical Scheme}
\label{sec6:1Dscheme}
The equation is solved implicitly on a uniform grid, we use $n$ to denote the time--step in the standard fashion.
Taking the hyperdiffusion term to the LHS, and discretising in time we get
\begin{equation}
C^{n+1} + \Delta t \gamma \frac{\partial^4 C^{n+1}}{\partial x^4} = C^n + \Delta t \left(\frac{\partial^2}{\partial x^2} \left(C^3 - C \right)^n \right)
\label{eq6:1Dnumerical}
\end{equation}
Here we have also set $D = 1.0$ as it has no effect on the dynamics of the problem and only scales the RHS.
We apply standard second order accurate differencing to the spatial derivatives
\begin{subequations}
\begin{equation}
\frac{\partial^2 C}{\partial x^2} \approx \delta_x^2 C = \frac{C_{i + 1} - 2 C_i + C_{i-1}}{\Delta x^2}
\end{equation}
\begin{equation}
\frac{\partial^4 C}{\partial x^4} \approx \delta_x^4 C = \frac{C_{i - 2} - 4 C_{i-1} + C_i - 4 C_{i+1} + C_{i + 2}}{\Delta x^4}
\end{equation}
\end{subequations}
This, along with the periodicity of the domain, yields a pentadiagonal matrix system of the form
\begin{subequations}
\begin{equation}
\underbrace{
\begin{pmatrix}
c & d & e & 0 & \cdots &  0 & a & b \\
b & c & d & e & 0 & \cdots &   0 & a \\
a & b & c & d & e & 0 & \cdots  & 0\\
0 & \ddots & \ddots & \ddots & \ddots & \ddots & \ddots &   \vdots\\
\vdots& \ddots & \ddots & \ddots & \ddots & \ddots & \ddots & 0 \\
0 & \cdots  & 0  &  a  & b  & c & d & e \\
e & 0 &  & 0   & a  & b & c & d \\
d & e & 0  &   \cdots & 0 &  a & b & c
\end{pmatrix}
}_{=\mxA}
\underbrace{
\begin{pmatrix}
C_{1} \\
C_{2} \\
\vdots \\
\vdots \\
\vdots \\
C_{N-2} \\
C_{N-1} \\
C_{N}
\end{pmatrix}
}_{=\vecx}
=
\underbrace{
\begin{pmatrix}
f_{1} \\
f_{2} \\
\vdots \\
\vdots \\
\vdots \\
f_{N-2} \\
f_{N-1} \\
f_{N}
\end{pmatrix}
}_{=\vecrhs}.
\label{eq6:matrix_sys1}
\end{equation}
Here, the coefficients of the matrix in Equation~\eqref{eq6:matrix_sys1} have the following meaning:
\begin{equation}
a  = \sigma_x,\qquad b = - 4 \sigma_x,\qquad
c = 1 + 6\sigma_x, \qquad
d= - 4 \sigma_x, \qquad e = \sigma_x
\end{equation}%
Similarly,
\begin{equation}
f_i = \alpha_x N_{i - 1}^{n} - 2 \alpha_xN_{i}^{n} + \alpha_x N_{i+1}^{n}
\end{equation}%
\Acomment{where} we have taken $\sigma_x = \gamma \Delta t / \Delta x^4$, $\alpha = \Delta t / \Delta x ^2$ and $N^n_i  = (C^3 - C)^n_i$ for brevity. 
\Acomment{This system can then be solved following the methodology presented for cyclic pentadiagonal matrices as found in \cite{navon_pent, gloster2019cupentbatch, gloster2019efficient} and in Chapters~\ref{chapter:cuPentBatch} and~\ref{chapter:efficient}, this methodology also influenced our choice of second order accuracy above.}
All of the necessary RHS matrix calculations were carried out using an in--house finite difference library developed by the authors~\cite{gloster2019custen}.
We use a conservative time step size of $\Delta t = 0.1 \Delta x$ to ensure stability of the scheme. 
In the following section we \Acomment{show} the convergence features of the scheme which is clearly second order accurate in space due to the choice of the central differences above.
\label{eq6:matrix_sys}%
\end{subequations}%

\subsection{1D Convergence Study}
\label{sec6:1Dcov}
For the convergence study we set $D = 1$, $\gamma = 0.01$ and the domain length $\Omega$ is set to $2 \pi$.
We set the time--step to a conservative $\Delta t = 0.1 \Delta x$ to ensure stability and we introduce a $\cos$ initial condition given by
\begin{equation}
C(x,0) = \epsilon \cos(5x)
\end{equation} 
with $\epsilon = 1 \times 10^{-6}$. 
It can be seen clearly from this table that the convergence rate is $\approx 2$ in keeping with the scheme's second order accuracy. 
The convergence results are presented in Figure~\ref{fig6:convergence1D} and Table~\ref{table5:converge1D}.
We can see that this scheme for intermediate grid points is second order accurate except for very high resolutions where the scheme converges with first order accuracy.
\Acomment{In all cases considering the diffuse nature of the Cahn--Hilliard equation and the lack of sharp festures to be resolved these orders of convergence are acceptable.}

\begin{figure}
	\centering
		\includegraphics[width=0.5\textwidth]{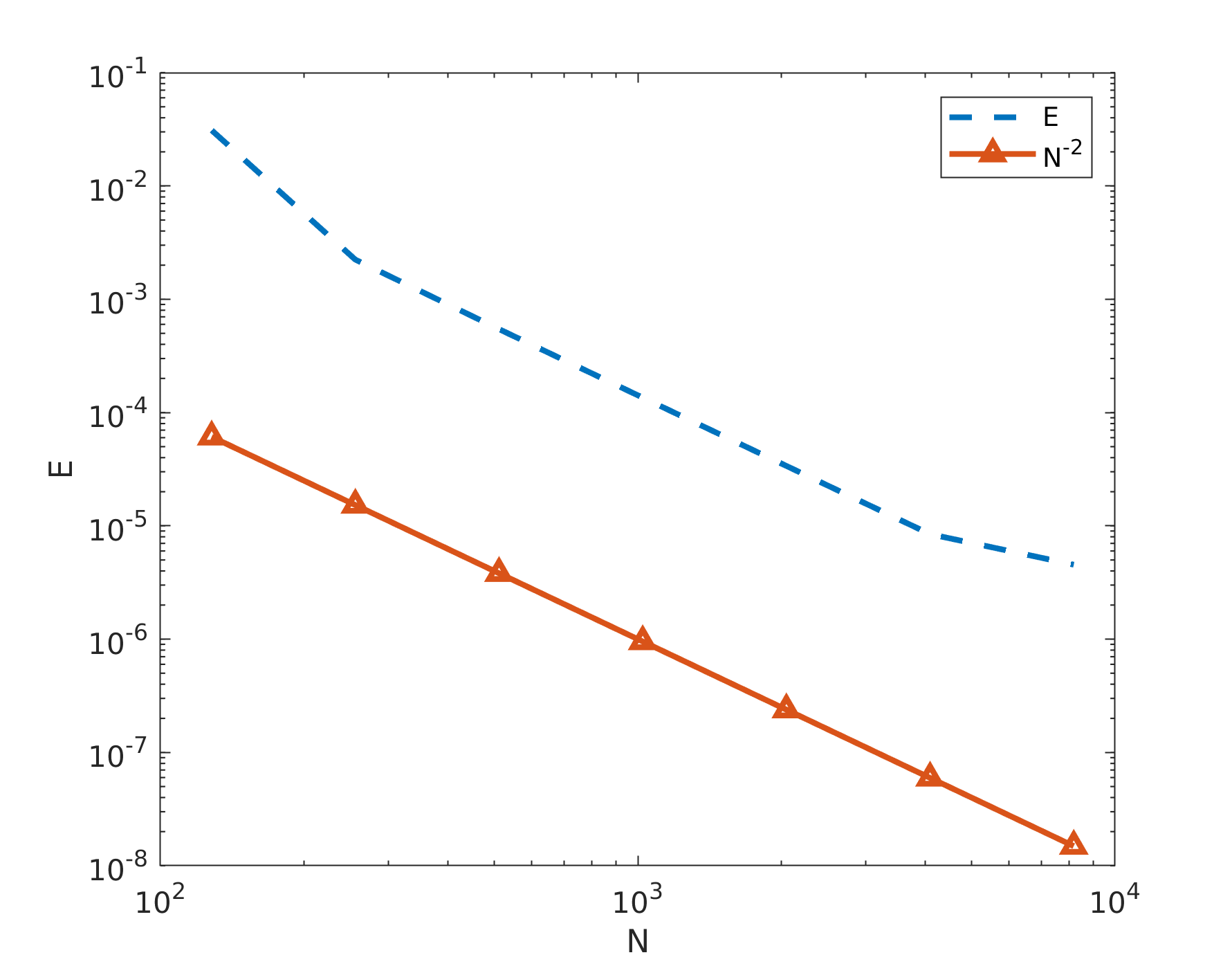}
		\caption{Plot showing the convergence of the scheme in equation~\ref{eq6:1Dnumerical}. The slope is $-2$ matching the spatial accuracy of the scheme.}
	\label{fig6:convergence1D}
\end{figure}

\begin{table}
\centering
\begin{tabular}{c|c|c} 
\hline
    {$N_x$} & {$E_N$} & {$log_2 (E_N / E_{2N})$} \\
    \hline
    128   & 0.0310      & 3.7932    \\
    256   & 0.0022      & 2.0376    \\
    512   & $5.45 \times 10^{-4}$    & 2.0089    \\
    1024  & $1.35 \times 10^{-4}$    & 2.0022    \\
    2048  & $3.38 \times 10^{-5}$    & 2.0005    \\
    4096  & $8.45 \times 10^{-6}$    & 0.8947	\\
    8192  & $4.54 \times 10^{-6}$    &           \\
\end{tabular}
\caption{Details of convergence study of semi-implicit scheme for solving the 1D Cahn--Hilliard equation.}
\label{table5:converge1D}
\end{table}

\subsection{Batched 1D Scaling}
\label{sec6:1Dscale}
We now examine the how the domain coarsening of the 1D Cahn--Hilliard equation scales with time. 
In~\cite{argentina2005coarsening} it is shown that, an average bubble, should scale in time as 
\begin{equation}
\langle S(t) \rangle \propto \log(t)
\label{eq6:1DscaleAvg}
\end{equation}
Thus in order to measure this quantity we must average across multiple independent simulations. 
We follow the computational methodology for solving batches of hyperdiffusion equations as presented in~\cite{gloster2019cupentbatch, gloster2019efficient} and combine it with the method for solving individual 1D Cahn--Hilliard equations presented above in Section~\ref{sec6:1Dscheme}.
\Acomment{We store each individual RHS of the system in interleaved format and the cuSten library~\cite{gloster2019custen} which was presented in Chapter~\ref{chapter:cuSten} is used to compute the required non--linear finite differences.}
Each system is given its own randomised initial condition $C(x, 0)$ where values are drawn from a uniform distribution between $-0.1$ and $0.1$.
The equations are then solved as a batch on a GPU, time--stepping till $T = 100$ and again we set $\gamma = 0.01$.
In order to measure $\langle S(t) \rangle$ we capture the quantity
\begin{equation}
S(t) = \frac{1}{1 - <C^2>}
\end{equation}
to measuring the coarsening rate~\cite{LennonAurore}.
This quantity is captured for each system at every time--step and then averaged across systems.

\begin{figure}
    \centering
    \begin{subfigure}[b]{0.49\textwidth}
        \centering
        \includegraphics[width=1.0\textwidth]{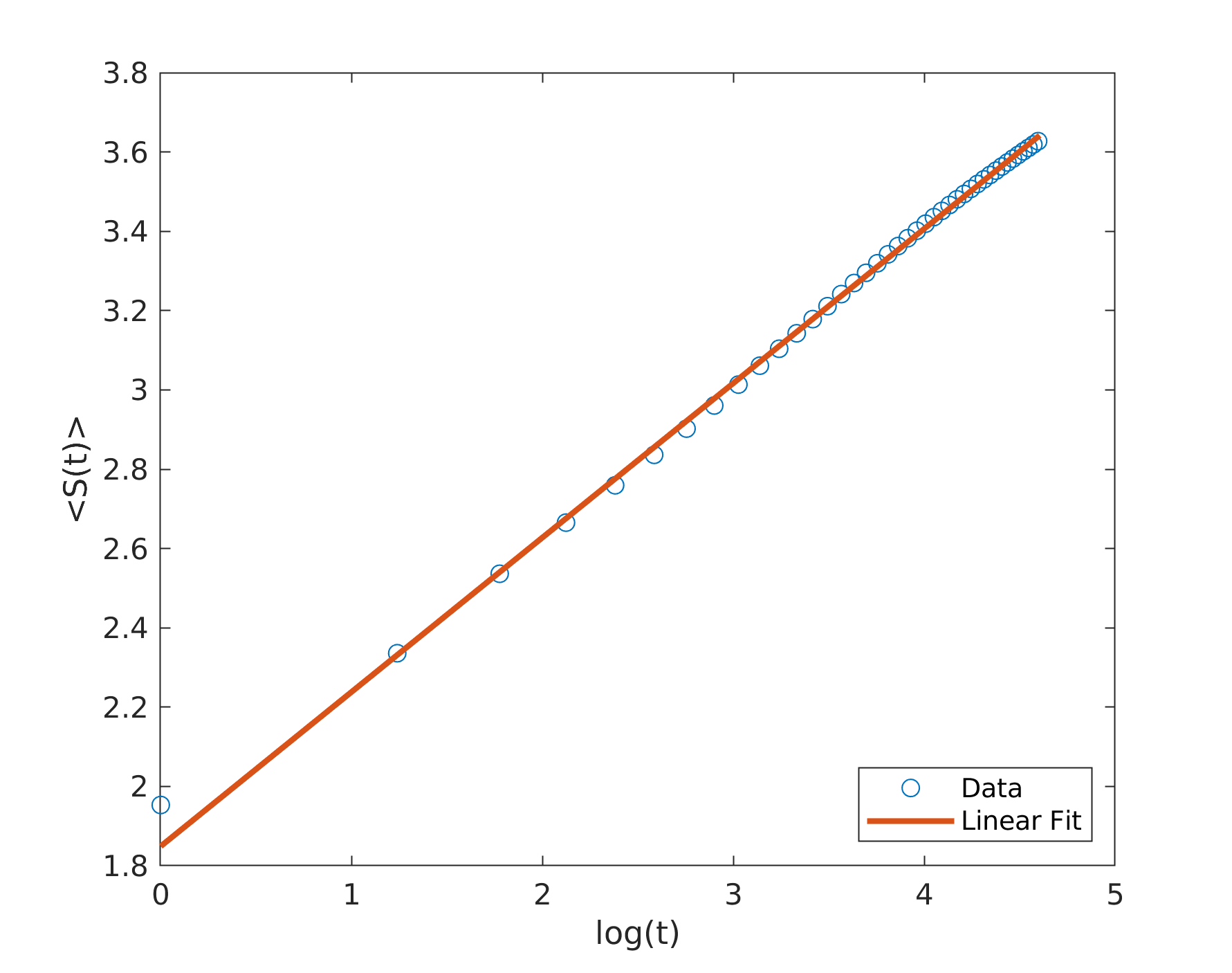}
    \caption{Scaling of domain size $2\pi$.}
    \label{fig6:1Dscale2piPlot}
    \end{subfigure}
    \hfill
    \begin{subfigure}[b]{0.49\textwidth}  
        \centering 
        \includegraphics[width=1.0\textwidth]{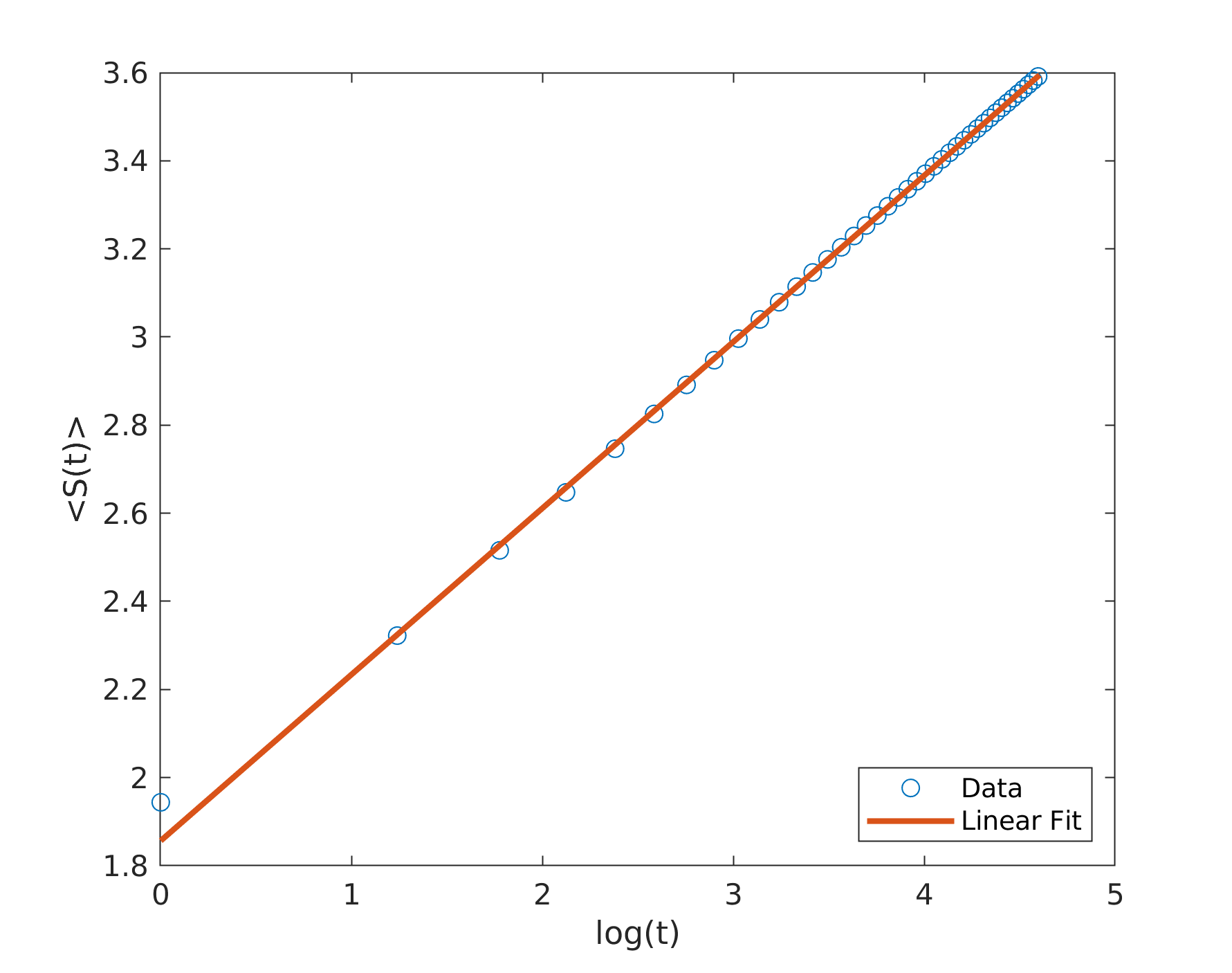}
    \caption{Scaling of domain size $4\pi$.}
    \label{fig6:1Dscale4piPlot}
    \end{subfigure}
        \caption{Plots showing the scaling in as a function of $\log{t}$.}
    \label{fig6:4pi1024}
\end{figure}

We plot the values for $\langle S(t) \rangle$ as a function of $\log(t)$ in Figures~\ref{fig6:1Dscale2piPlot} and~\ref{fig6:1Dscale4piPlot} for domains of size $2\pi$ and $4\pi$ respectively.
In both cases $\Delta x = 2 \pi / 256$ again to ensure $\Delta x \ll 2\pi \sqrt{\gamma}$ so that the interface is well resolved. 
In both cases we simulate $2^{20}$ independent systems in parallel on a single GPU in order to ensure the space is well \Acomment{sampled} and that we average out the statistic of interest as much as possible. 
The reader should note that not all of the data points are plotted in the graphs but were included in the analysis, this is done to ensure the graphs are readable.
In both cases \Acomment{we} fit lines with linear regression as we're testing that the quantities in equation~\eqref{eq6:1DscaleAvg} \Acomment{are} linearly proportional. 
For the $2\pi$ case in Figure~\ref{fig6:1Dscale2piPlot} we get a value of $r = 0.9989$ while for $4\pi$ in Figure~\ref{fig6:1Dscale4piPlot} we get a value of $r = 0.9996$.
Clearly these results confirm the theoretical findings in~\cite{argentina2005coarsening} and provide a good example of how a GPU can be used to efficiently solve a large batch of separate 1D equations.


\begin{figure}
    \centering
        \includegraphics[width=0.5\textwidth]{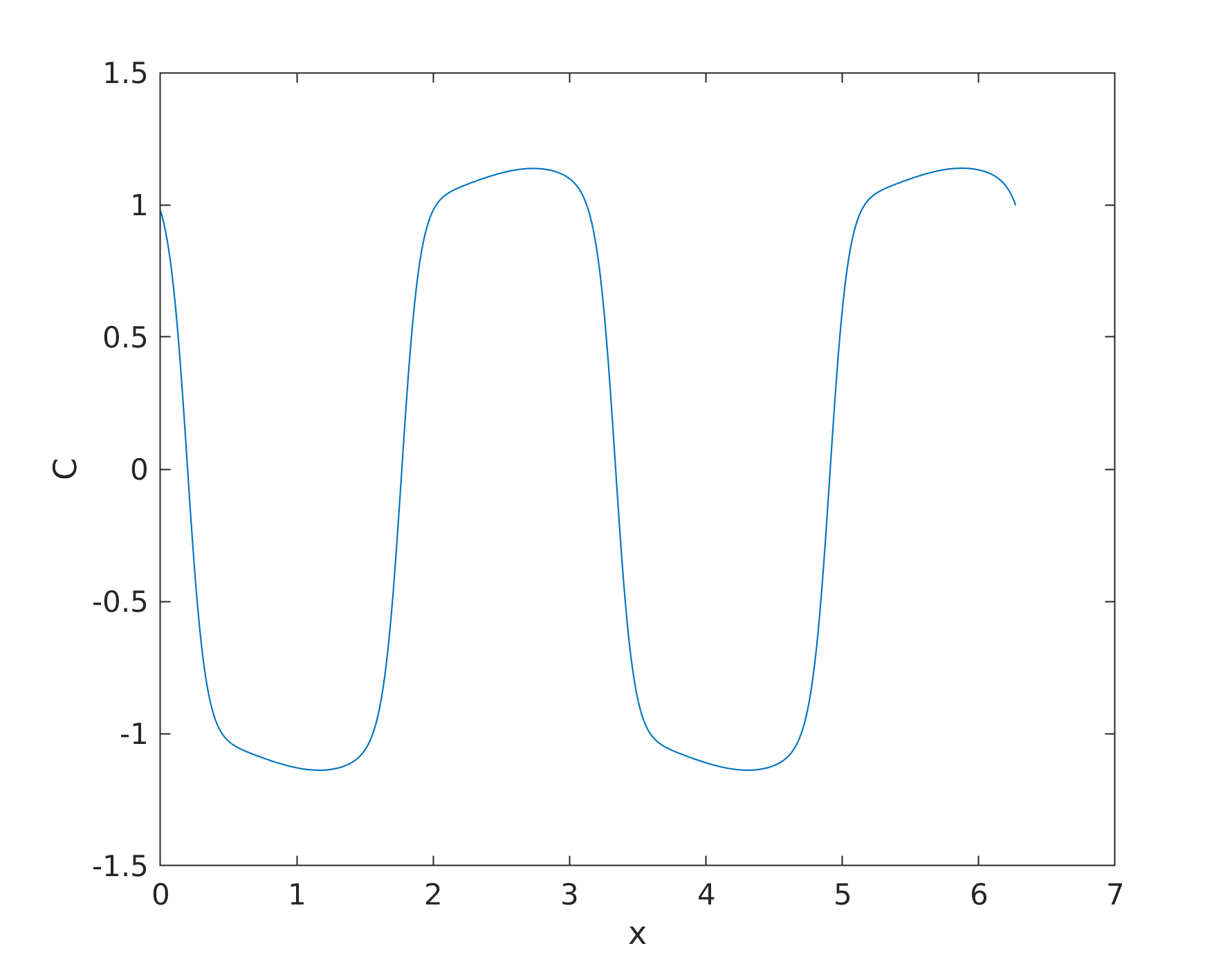}
        \caption{Plot of numerical solution to equation\eqref{eq6:1DcahnForcedTravel} showing regions of $C \approx \pm 1$.}
    \label{fig6:fullMinusPlus}
\end{figure}

\section{Forced 1D Cahn--Hilliard Flow Pattern Map}
\label{sec6:flowpattern}
In this section we focus on the 1D Cahn--Hilliard equation with an additional forcing term, specifically a travelling wave.
Thus we have a modified version of equation~\eqref{eq6:1Dcahn}, following~\cite{lennonWave} giving us the following expression
\begin{equation}
\frac{\partial C}{\partial t} = \frac{\partial^2}{\partial x^2} \left(C^3 - C - \gamma \frac{\partial^2 C}{\partial x^2} \right) + f_0 k \cos(k (x - vt))
\label{eq6:1DcahnForced}
\end{equation}
where $k$ is the forcing wave--number, $f_0 k$ the forcing amplitude and $v$ is the wave--speed.
We study this equation as \Acomment{a} simplification of the Ludwig--Sorret effect, this is where the concentration fluctuations in the Cahn--Hilliard equation are introduced by an external temperature gradient.
A common approach is to \Acomment{vary} parameters such as $k$, $v$ and $\langle C \rangle$ (where $\langle \cdot \rangle$ denotes a \Acomment{spatial} average) to produce a parameter study and generate flow--pattern maps~\cite{lennonWave}.

\begin{figure}
    \centering
        \includegraphics[width=0.5\textwidth]{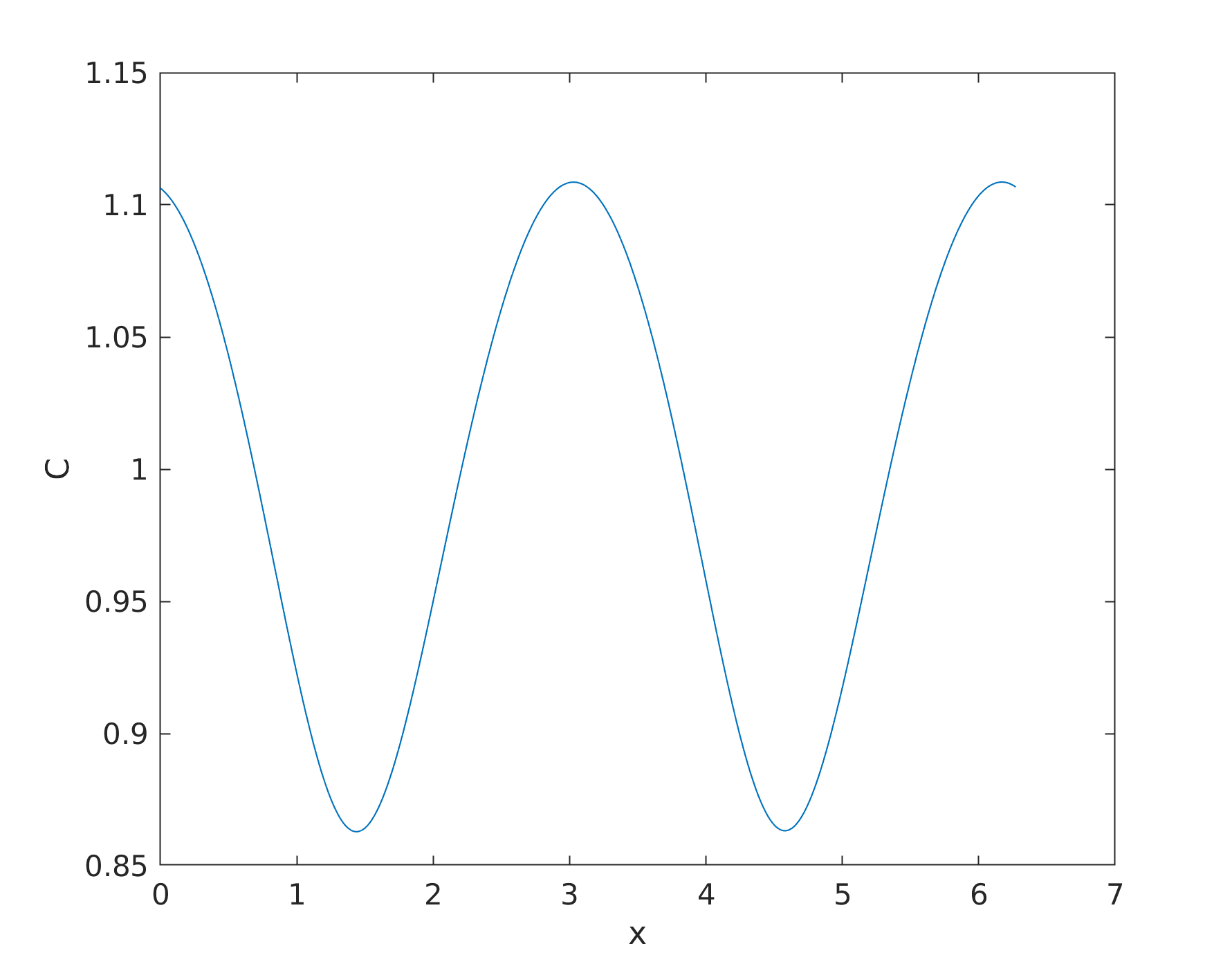}
        \caption{Plot of numerical solution to equation\eqref{eq6:1DcahnForcedTravel} showing regions oscillation around a mean value of $C \approx 1$.}
    \label{fig6:oscillation}
\end{figure}

In order to produce a good quality flow--pattern map with well--defined domains a large number of simulations is required followed by some `by--hand' techniques.
This approach is slow and inefficient as it leaves the researcher having to study each numerical solution to categorise the finding, particularly in cases where an analytic metric cannot be used to help determine the categorisation of a given solution in the flow--pattern map.
Another key drawback of this method is it can be hard to draw out the precise boundaries between two solution regions, particularly when there is a high number of simulations and a lack of an analytic expression to divide the parameter space.
We present in this section of the paper, first a method to produce the necessary large batches of simulation in parallel using a GPU, this allows for extremely large and diverse data--sets to be produced for the study, and then we present an application of k--means clustering to categorise the results automatically to produce the desired flow--pattern maps. 

Before proceeding any further we first need to rewrite this equation into the frame of the travelling wave, this ensures we have a consistent position relative to the wave when feeding it into our categorisation methods, the desired coordinate system is given by
\begin{equation}
\eta = x - vt
\end{equation}
Under this transformation equation~\eqref{eq6:1DcahnForced} becomes 
\begin{equation}
\frac{\partial C}{\partial t} - v \frac{\partial C}{\partial \eta} = D \frac{\partial^2}{\partial \eta^2}\left(C^3 - C - \gamma \frac{\partial^2 C}{\partial \eta^2} \right) + f_0 k \cos(k \eta)
\label{eq6:1DcahnForcedTravel}
\end{equation}
The forced 1D Cahn--Hilliard equation presented in equation~\eqref{eq6:1DcahnForcedTravel} is ideal for presenting this methodology as only three key solution types are present in the parameter space, we will label them in the same convention as~\cite{lennonWave}. 
The first possible solution type is one which consists of two regions of $C \approx \pm 1$ joined by a transition region, an example of this can be seen in Figure~\ref{fig6:1Dspacetime}, we refer to this solution type as A1.
In Figure~\ref{fig6:oscillation} we can see the second type of solution where there is an oscillation around some mean value of $C$, in this case $\approx 1$, we refer to this solution type as A2.
Any other solution in the parameter space will either be a combination of these two effects or a completely unforced standard Cahn--Hilliard solution, we label this solution as A0.
Thus we expect three regions in our parameter study to appear.
These three solution types provide a clear path for categorising the parameter space for the flow pattern map using k--means clustering, we introduce k-means clustering in Section~\ref{sec6:kmeans} and discuss the results in Section~\ref{sec6:kmeanResults}.
But first in Sections~\ref{sec6:schemeTravel} and~\ref{sec6:convergeTravel} we present the numerical scheme and convergence study for equation~\eqref{eq6:1DcahnForcedTravel}.

\begin{figure}
    \centering
        \includegraphics[width=0.5\textwidth]{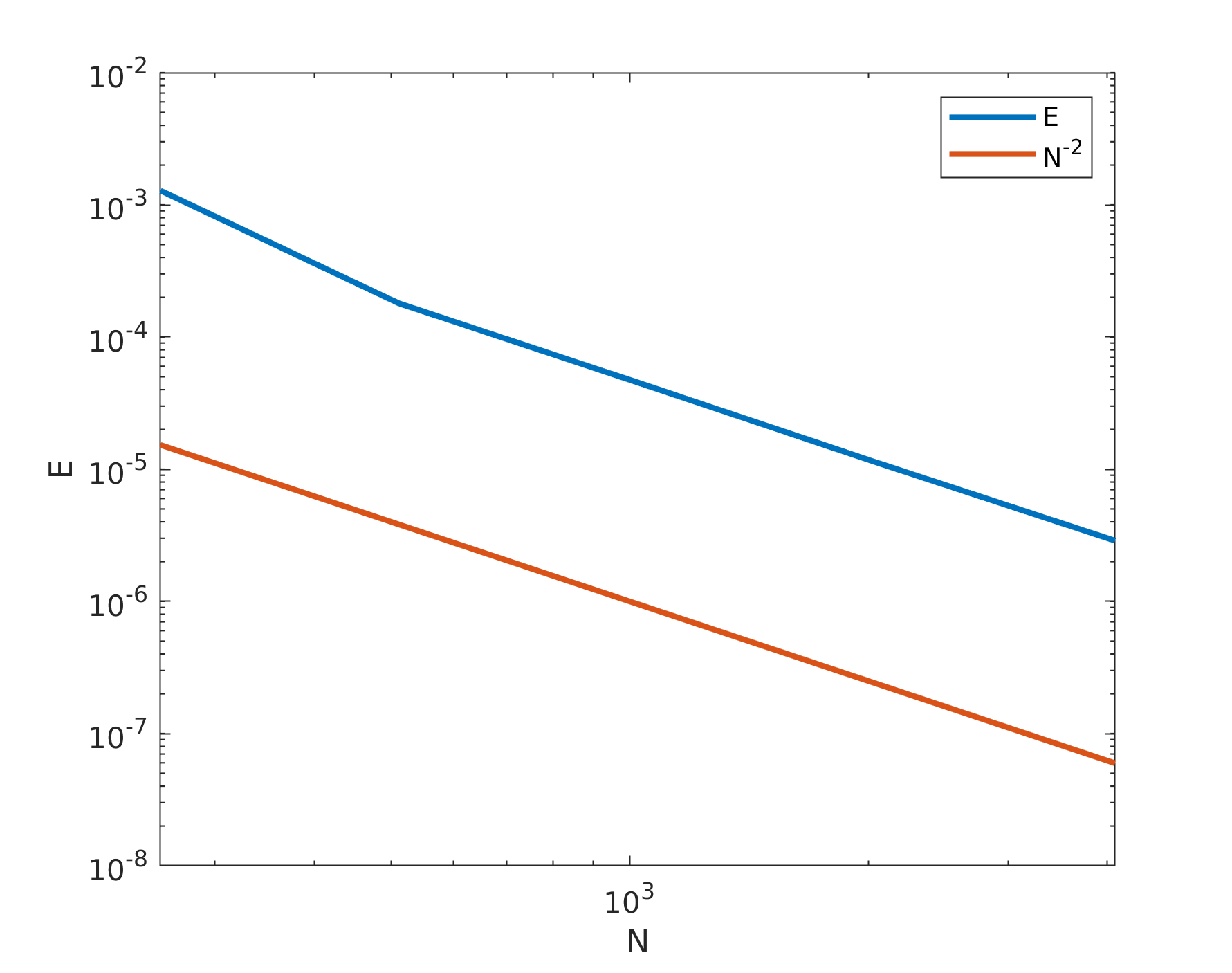}
        \caption{Plot showing the convergence of the forced Cahn--Hilliard scheme in equation~\ref{eq6:travellingNumerical}. The slope is $-2$ matching the spatial accuracy of the scheme.}
    \label{fig6:convergenceTravel1D}
\end{figure}

\subsection{Numerical Scheme}
\label{sec6:schemeTravel}
As previously we discretise the hyperdiffusion term implicitly and set $D = 1.0$ to give a version of equation~\eqref{eq6:1Dnumerical} with additional terms on the RHS
\begin{multline}
C^{n+1} + \Delta t \gamma \frac{\partial^4 C^{n+1}}{\partial \eta^4} = \\ C^n + \Delta t \left(\frac{\partial^2}{\partial \eta^2} \left(C^3 - C \right)^n + v \frac{\partial C^n}{\partial \eta} + f_0 k \cos(k \eta) \right)
\label{eq6:travellingNumerical}
\end{multline}
Thus the solution methodology with pentadiagonal inversions and central differences is the same as Section~\ref{sec6:1Dscheme} except we need to discretise the advection term $v \partial C / \partial \eta$.
\Acomment{In order to deal with this term we use a standard Hamilton--Jacobi WENO scheme presented in~\cite{fedkiwBook}.
We choose this scheme as it correctly differences the advection term in the direction of travel and is $O(\Delta x^5)$ accurate in smooth regions, this scheme has been implemented using the cuSten library~\cite{gloster2019custen}.}

\subsection{Convergence Study}
\label{sec6:convergeTravel}
We proceed with an identical convergence study to Section~\ref{sec6:1Dcov}, we have set the additional necessary parameters as follows $k = 1$, $v =0.5$ and $f_0 = 1.0$.
The convergence results are presented in Figure~\ref{fig6:convergenceTravel1D} and Table~\ref{table5:convergeTravel1D}.
In both cases we can see clear second order convergence as expected as we have been using at least second order accurate stencils in all of our calculations.

\begin{table}
\centering
\begin{tabular}{c|c|c} 
\hline
    {$N_x$} & {$E_N$} & {$log_2 (E_N / E_{2N})$} \\
    \hline
    256   & 0.0013                  & 2.8363    \\
    512   & $1.8 \times 10^{-4}$      & 1.9869    \\
    1024  & $4.54 \times 10^{-5}$  & 2.0086    \\
    2048  & $1.13 \times 10^{-5}$  & 1.9673    \\
    4096  & $2.88 \times 10^{-5}$  &           \\
\end{tabular}
\caption{Details of convergence study of semi-implicit scheme for solving the 1D forced Cahn--Hilliard equation.}
\label{table5:convergeTravel1D}
\end{table}

\subsection{K-Means Clustering}
\label{sec6:kmeans}
k--means clustering is a standard method for partitioning observations into clusters based on distance to the nearest centroid.
Distance in this sense is calculated using Euclidean distance in a space of suitable dimension for the dataset. 
In this work we use MATLAB's implementation of the k--means clustering algorithm which is itself based on Lloyd's algorithm~\cite{kmeans}.
This algorithm is built on two steps which alternate, the first step is to assign each observation to a cluster based on its Euclidean distance to each of the centroids, the assigned cluster is the closest centroid.
Then the second step is to calculate a new set of centroids by calculating the mean position of all the observations within each assigned cluster. 
These steps alternate until convergence has been reached.
The algorithm is initialised by randomly creating the desired number of centroid (3 in our case).

\subsection{Clustering Results}
\label{sec6:kmeanResults}
In this section we present the results of running batches of the Cahn--Hilliard equation with a travelling wave and categorising the solutions using the k--means algorithm.
The methodology will be to solve batches of equations where we fix the wave--velocity $v$ and wave--number $k$ for all systems within a given batch and then vary $\langle C \rangle$ and $f_0$ within the batch.
$\langle C \rangle$ will be varied between values of $0$ and $1.5$ and $f_0$ between $2$.
In order to accurately capture the travelling wave we set the domain size to $2\pi$ and set the number of points in the domain to $N = 512$, it was found that this high number of $N$ was required in order to accurately resolve the travelling wave at high values of $\langle C \rangle$ and $f_0$.
This choice of $N$ also gives us the maximum batch size of the simulations that can be run on a single GPU, in this study our GPU has 12GB of space, thus we determine that we can have a total number of 589824 simulations in each batch.
Thus the resolution of our flow pattern maps will be $768 \times 768$ dividing the ranges of $f_0$ and $\langle C \rangle$ appropriately.
The initial conditions in this section are drawn from a random uniform distribution of values between $-0.1$ and $0.1$ with $\gamma$ set to a value of $0.01$.
We simulate all of the systems up to a final time of $T = 50$.

\begin{figure}
    \centering
        \includegraphics[width=0.5\textwidth]{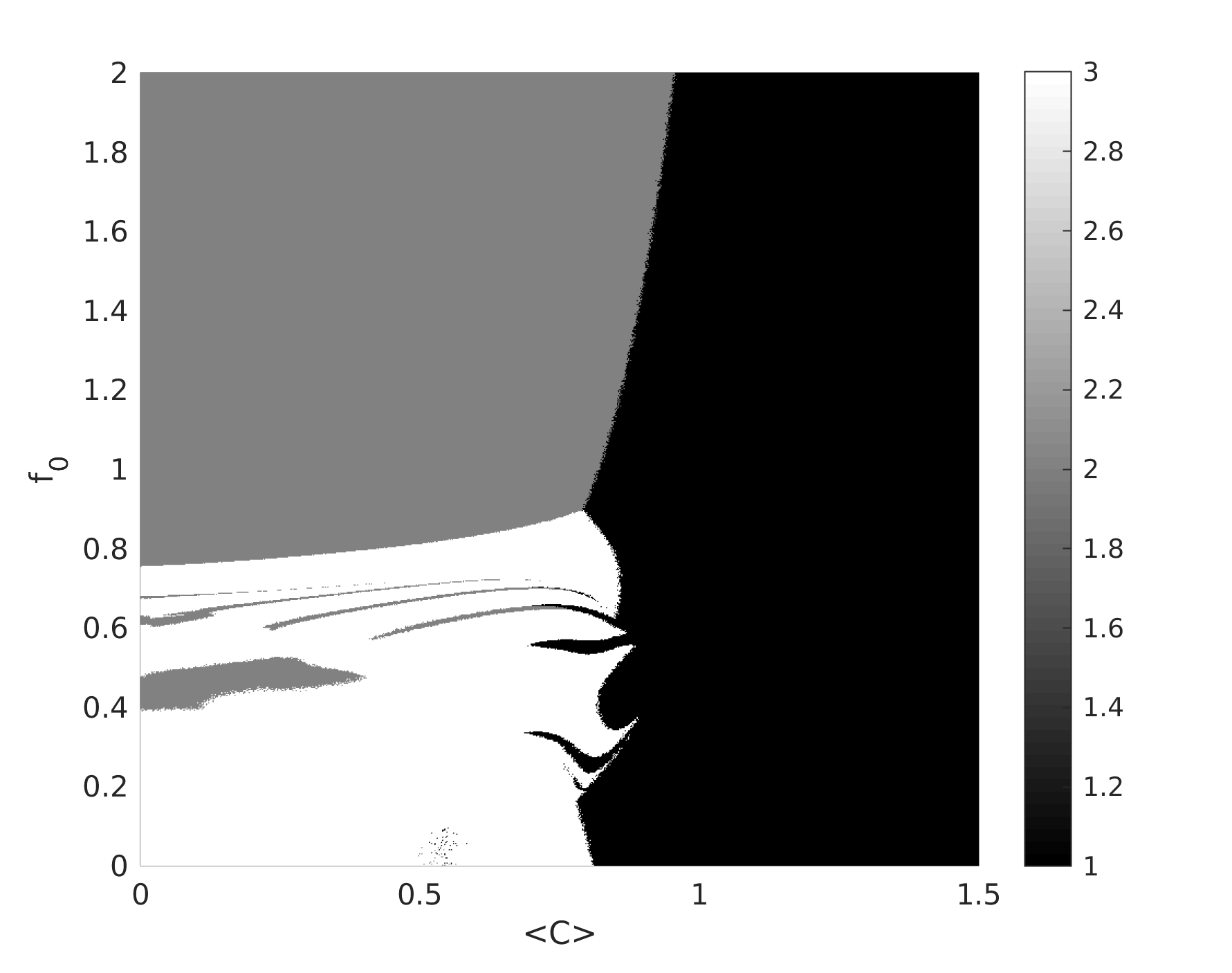}
        \caption{Flow--pattern map of equation~\eqref{eq6:travellingNumerical} with  $k = 1.0$ and $v = 0.5$, the white area corresponds to A0 solutions, grey to A1 solutions and black to A2 solutions.}
    \label{fig6:plot1}
\end{figure}

We begin by examining the case $k = 1.0$ and $v = 0.5$, the results of which are presented in Figure~\ref{fig6:plot1}.
In this plot we can see the clear extraction of regions of different solutions, the top left grey region corresponds to solution type A1 and the right side black region corresponds to solution type A2.
A mix of solutions can be seen in the bottom left white region of the domain corresponding to region A0 along with some areas of A1 and A2.
Between these two regions k--means has been able to pick out an exact boundary extending from $f_0 \approx 0.9$ to the top of the domain.
Some noise can be seen along this boundary where the k--means algorithm has struggled to differentiate between very similar solution types as the behaviour transitions,
This also explains the less well behaved bottom left region where the algorithm has been unable to deal with areas near the boundaries.

The results presented in Figure~\ref{fig6:plot1} are in--keeping with those presented in~\cite{lennonWave}; where the classification was performed by visual inspection.
Three regions in approximately the same locations were found, solution type A0 in the bottom left, A1 on the top left and then A2 on the right side of the parameter space. 
We now increase the value of v to see if the boundary between A0 and A1 moves higher up the $f_0$ axis.

\begin{figure}
    \centering
        \includegraphics[width=0.5\textwidth]{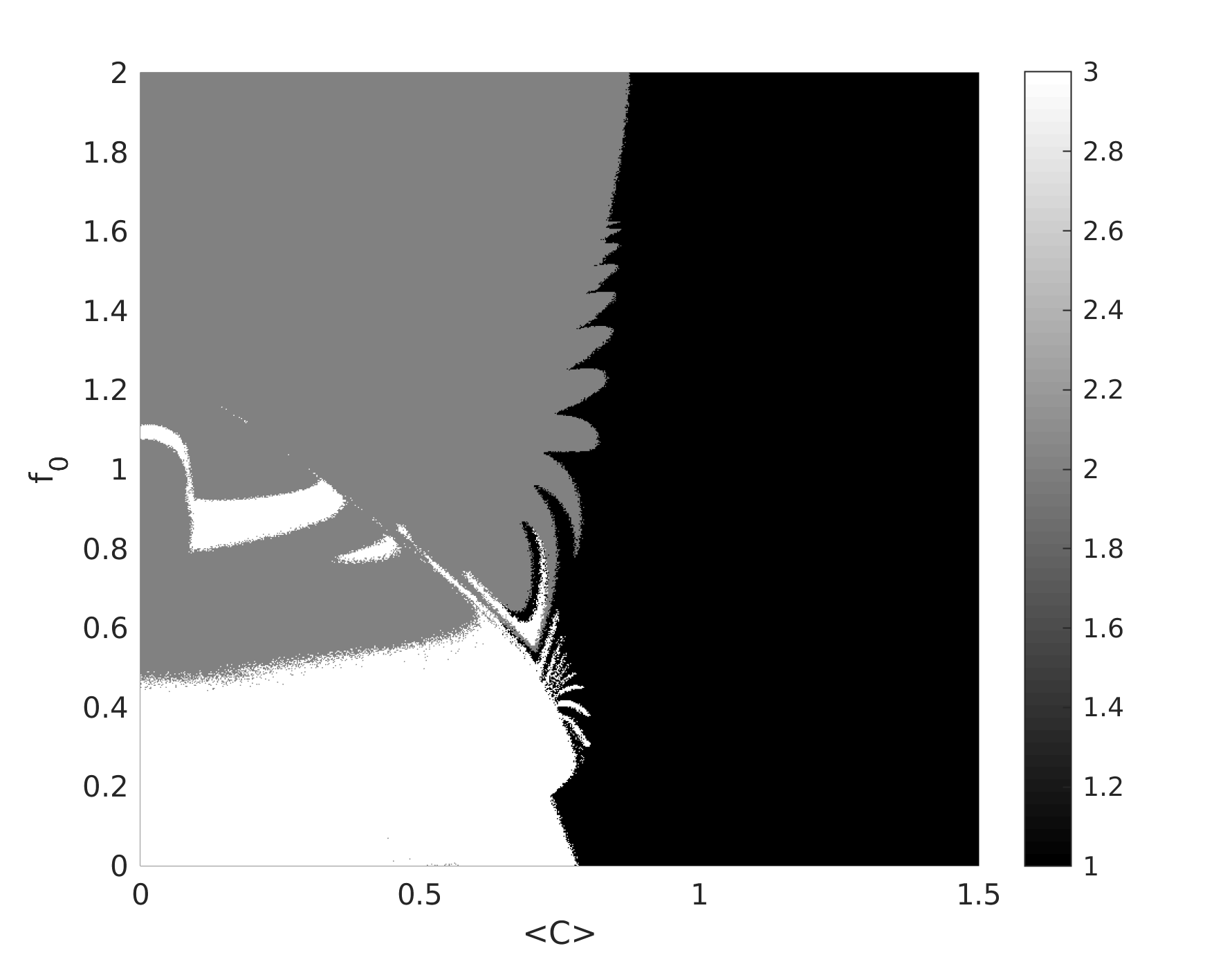}
        \caption{Flow--pattern map of equation~\eqref{eq6:travellingNumerical} with  $k = 1.0$ and $v = 1.0$, the white area corresponds to A0 solutions, grey to A1 solutions and black to A2 solutions.}
    \label{fig6:plot2}
\end{figure}

We now examine the case $k = 1.0$ and $v = 1.0$, the results of which are presented in Figure~\ref{fig6:plot2}.
The k--means algorithm has struggled here to extract clean boundaries between the A0 region and the other two. 
This can be attributed to the fact that these are hard to distinguish, even by eye.
We can see that the algorithm has successfully classified the regions A1 and A2 broadly quite well, with a less well defined boundary than previously.
Thus we can conclude that k--means is a useful approach for broadly classifying large parameter spaces in terms of flow pattern maps yet will struggle with more subtle transitions betweens regions.

Instead k--means could be used as a first step in building a model to classify simulation data for a flow pattern map.
\Acomment{Indeed one could take the well defined regions from a large number of simulations to then train a logistic regression classifier or neural network, these may produce better results as they may detect features that are not as obvious to the eye or representable using a Euclidean distance as in the k-means algorithm.} 
Also additional features beyond the raw solution could be included in the dataset such as a time series of the free energy or an increased number of time--steps of given solutions.


\section{Conclusion}
We have shown in this chapter how the methodologies for solving batches of 1D equations in previous chapters can be applied to confirm both theoretical results and used to produce data sets for flow pattern maps.
In particular the batch solving has enabled us to produce a large enough sample space to apply k--means clustering and explore its usefulness in producing flow pattern maps. 
In the following chapter we use our methodology to produce a large number of 2D simulations using GPUs in order to explore the scaling law of the 2D Cahn--Hilliard equation.

\lhead{\emph{Chapter 7}}  
\graphicspath{{chapter6/Figures/}}
\chapter{Statistical Study of Interface Growth in the Cahn--Hilliard Equation}
\label{chapter:statistical}
\Acomment{
In this chapter we examine the scaling of interface growth in the Cahn--Hilliard equation. 
Here we focus on interface growth only, so results including mergers of existing regions in the domain are omitted, how these mergers appear in the simulation data is discussed in Section~\ref{sec7:outliers}.
An extended version of the theory presented here, including a further discussion regarding a point--wise bound on $\beta$, and the same dataset including mergers was included in the paper "A large-scale statistical study of the coarsening rate in models of Ostwald-Ripening"~\cite{nraigh2019largescale}.
The specific work on a point--wise bound was conducted by my thesis supervisor Lennon {\'O} N{\'a}raigh and hence omitted from this thesis.
The work presented here and in the paper overlap in areas but each have distinct results presented. 
}

\section{Introduction}
\label{sec7:intro}
When a binary fluid in an initially well-mixed state experiences rapid cooling below a critical temperature, both phases spontaneously separate to form regions rich in the fluid's component parts.  These regions expand over time in a phenomenon known as coarsening.  
In the case of an asymmetric mixture where the volume fraction of one of the fluid components is very small (the `minority phase'), phase separation takes place via Ostwald Ripening -- in this case, `droplets' of the minority phase form, immersed in a matrix of the majority phase.  
The system rapidly reaches equilibrium in the matrix, in which case the late-stage evolution of the system is governed by the energetics of the interface and limited by inter--facial diffusion.  
In this late stage, there is the tendency for the large droplets of the minority phase to grow at the expense of the smaller droplets, which eventually vanish.  
This then leads to a decrease in the number density of droplets and an increase in the average size of the surviving droplets.  
The size of a typical droplet grows as $\langle R\rangle \sim t^{1/3}$, which is a manifestation of the coarsening phenomenon.
In this work, we revisit the theory of Ostwald Ripening in the case where there is a large-but-finite number of droplets present.  
There is ample physical motivation for such a study, as any real system will consist of only a finite number of droplets.  
The presence of only a finite number of droplets in our model system will alter the statistics of the coarsening.  
Understanding these statistical finite-size effects is the main motivation for this work.

There is a convenient asymptotic theory for Ostwald Ripening, valid in an infinitely large domain with infinitely many droplets, but formulated in such a way that the volume fraction occupied by the droplets is finite (and small).  
This theory was developed by Lifshitz and Slyzov in Reference~\cite{lifshitz1961kinetics} simultaneously, by Wagner in Reference~\cite{wagner1961}.  
The eponymous Lifshitz--Slyozov--Wagner (LSW) theory provides an analytical expression for the late-time drop--size distribution in Ostwald Ripening, denoted here by $p(r,t)$.  
A continuity equation for $p(r,t)$ is derived, where the probability flux depends on the droplet velocity -- the velocity is obtained by energetic arguments based on droplet inter--facial area.  
The continuity equation admits a self-similar solution $p(r,t)\propto f(x)$, where $x=r/\langle R\rangle$; the scaling $\langle R\rangle\propto t^{1/3}$ readily drops out of this calculation.  
It can be also noted that the functional form of $f(x)$  has compact support.  
This theory can be used as a starting-point for the present investigations: in Section~\ref{sec7:LSW} we formulate a droplet-population model for $N$ droplets in a dilute mixture, the limiting case of which ($N\rightarrow \infty$) is the LSW theory.
\Acomment{We note that the use of $N$ in this chapter is to denote numbers of bubbles/droplets, not matrix sizes as in previous chapters.}

The LSW theory is valid at late times effectively for very dilute systems.  
In this case, the LSW theory represents an approximate solution to so-called Mullins--Sekerka (MS) Dynamics, introduced first by Mullins and Sekerka to model particle growth in a supersaturated matrix~\cite{mullins1963morphological}, but then re--purposed as an effective model for Ostwald Ripening more generally~\cite{niethammer2008effective}.  
The MS dynamics describe the motion of extended regions $\{B_1,\cdots,B_N\}$ in a domain $\Omega$ and are expressible in terms of a generic chemical potential $\mu$:

\begin{subequations}
\begin{eqnarray}
\nabla^2\mu&=&0,\qquad \text{in }\Omega - \cup_{i=1}^N B_i,\\
\mu&=&\kappa,\qquad \text{in } \cup_{i=1}^N \partial B_i,\\
V&=&\left[\widehat{\bm{n}}\cdot\nabla \mu\right],\qquad \text{in } \cup_{i=1}^N \partial B_i.\label{eq7:ms3}
\end{eqnarray}
\label{eq7:ms}
\end{subequations}

Here, $\kappa$ denotes the mean \Acomment{interfacial} curvature, $V$ denotes the normal velocity of the interface, $\widehat{\bm{n}}$ denotes the normal vector to the interface, and $\left[\widehat{\bm{n}}\cdot\nabla \mu\right]$ denotes the jump in the normal derivative of the chemical potential across the interface.  
The fact that the interfaces move (with velocity $V$, via a mismatch in the chemical potential across the interfaces), means that this is a dynamical problem.  
\Acomment{LSW theory amounts to a solution of Equation~\eqref{eq7:ms} in the mean-field approximation, for spherical domains $B_i = R_i$, where $R_i$ is the radius of a given sphere.}

In practice, it is difficult to solve Equation~\eqref{eq7:ms} numerically, as it involves tracking a moving, disconnected interface, as it evolves to minimize the total system free energy.  
A viable theoretical and numerical alternative is to solve the Cahn--Hilliard Equation (equation~\eqref{eq2:ch}) instead with an initial condition of an asymmetric mixture.
In other words, a minority phase immersed in a matrix of the majority phase. 
Small `droplets' of the minority phase disappear only to be reabsorbed into larger droplets of the same phase in precisely the same process as Ostwald Ripening.  
Indeed, it has been shown by Pego~\cite{pego1989} that solutions of Equation~\eqref{eq2:ch} (under certain assumptions) converge to solutions of Equation~\eqref{eq7:ms}, in the limit as $\gamma\rightarrow 0$.  Hence, under the limit of a very dilute minority phase, solutions of the Cahn--Hilliard equation are characterized by LSW theory, in the limit as $\gamma\rightarrow 0$.  

In this section, we revisit the Cahn--Hilliard Equation~\eqref{eq2:ch}.  
We study the effect of a finite-sized domain $\Omega$ on the coarsening.  
Specifically, this is done by exploring the free-energy decay rate $\beta_\Omega=-(t/F)(\mathd F/\mathd t)$ emanating from Equation~\eqref{eq2:dFdt_ch}.  
Looking at the free-energy decay rate is equivalent to the coarsening rate, as there the system is statistically self-similar and admits only a single characteristic length scale at late times. Specifically, we look at the probability distribution function of $\beta_\Omega$.  
The effects of finite $\Omega$ are analogous to the effects of finite $N$ in the droplet-population model: we therefore use that model to provide a theoretical explanation for the $\beta_\Omega$ distribution in the Cahn--Hilliard equation.

The Cahn--Hilliard equation is valid for arbitrary configurations of the phases.  
Therefore, as a final study (Section~\ref{sec7:CH_symm}), we investigate finite-size effects and coarsening in the case of symmetric binary mixtures, where both phases are present in equal amounts.  
In such a scenario, the coarsening is characterized by an interconnected domain structure comprising both phases; the `typical length--scale' in this scenario is known to grow in time again as $t^{1/3}$.  
Using an ensemble of numerical simulations of Equation~\eqref{eq2:ch} (enabled by GPU computing), we characterize the coarsening phenomenon statistically, again using the coarsening rate $\beta_\Omega=-(t/F)(\mathd F/\mathd t)$ as a key random variable whose distribution we uncover.

This chapter is organized as follows.   
In Section~\ref{sec7:LSW} we  revisit  LSW theory: this is a convenient starting-point for our investigations.  We formulate a theory of finite-size effects in Ostwald Ripening via the introduction of a finite droplet-population model inspired by LSW theory.
In Section~\ref{sec7:methodology} we introduce the computational methodologies required for numerical simulations of the droplet-population model, as well as the necessary numerical methods for later computations.
In Section~\ref{sec7:CH} we solve an ensemble of Cahn--Hilliard Partial Differential Equations numerically with asymmetric mixtures, thereby developing a statistical picture of the decay rate of the free energy, $\beta_\Omega=-(t/F)(\mathd F/\mathd t)$, which can be compared back to the droplet-population model.  
In Section~\ref{sec7:CH_symm} we use the same approach to investigate the statistics of $\beta$ for a \textit{symmetric} Cahn--Hilliard mixture.  
In Section~\ref{sec7:cooke} we also extend the study to the Cahn--Hilliard--Cooke equation which is identical to the standard equation except the dynamics are driven by a noise term which represents a system driven by thermal fluctuations.
Discussion and final concluding remarks are presented in Section~\ref{sec7:disc}. 

\section{Theoretical Framework}
\label{sec7:LSW}
We begin by reviewing the mean-field solution of the Mullins--Sekerka dynamics~\eqref{eq7:ms}, valid for the case of a very dilute droplet population: here, the aim is to find a highly simplified expression for the chemical potential, which will be constant (in space) in the far field, and encode the effect of all other droplets on a particular droplet $B_i$ (with $i\in\{1,\cdots,N\}$).  
As such, we solve for $\mu$ with the following constraints: 
\begin{equation}
\mu\,\begin{cases} \text{is harmonic}, & \text{for} R_i\neq |\vecx-\vecx_i|\ll d,\\
                   =1/R_i,             & |\vecx-\vecx_i|=R_i,\\
                                     \approx  \overline{u}, & \text{for}R_i\ll |\vecx-\vecx_i|\ll d.
                                    \end{cases}
\label{eq7:bc}
\end{equation}
\Acomment{
Here, $d$ is the typical distance between droplets; this is assumed to be large in comparison with $R_i$ the radius of droplet $i$ -- this assumption is valid in the limit of very dilute systems.  
Also we define $\overline{u}$ as the mean field.
}
The fundamental solution is therefore given by
\[
\mu=\frac{a}{|\vecx-\vecx_i|}+b,
\]
where $a$ and $b$ are constants of integration.  
These are chosen so as to satisfy the boundary conditions~\eqref{eq7:bc}, hence
\[
\mu=\frac{1}{|\vecx-\vecx_i|}R_i\left(\frac{1}{R_i}-\overline{u}\right)+\overline{u}.
\]
We also compute 
\[
\left(\frac{\partial \mu}{\partial r}\right)_{|\vecx-\vecx_i|=R_i}=-\frac{1}{R_i}\left(\frac{1}{R_i}-\overline{u}\right).
\]
Using~\eqref{eq7:ms3} for spheres, this becomes
\Acomment{
\begin{equation}
\frac{\mathd R_i}{\mathd t}=H(R_i)\left(-\frac{1}{R_i^2}+\frac{\overline{u}}{R_i}\right)
\label{eq7:dRi}
\end{equation}
}
The Heaviside step function $H(R_i)$ is added as a pre--factor in Equation~\eqref{eq7:dRi} -- this is both a regularization and a book-keeping procedure to take account of particles whose radius shrinks to zero.
For the avoidance of doubt, we recall that $H(R_i)=0$ if $R_i\leq 0$ and $H(R_i)=1$ otherwise.

The value of the mean field $\overline{u}$ can now be obtained by imposing the constancy of the mass fraction, hence, the constancy of the total volume $\sum_{i=1}^N (4/3)\pi R_i^3$, hence
\begin{equation}
\overline{u}=\frac{\sum_{i=1}^N H(R_i)}{\sum_{i=1}^N H(R_i)R_i}.
\label{eq7:uval}
\end{equation}


\subsection{LSW theory}

We further review the LSW theory.  
This can be thought of as the limiting case of the mean-field MS theory where the number of droplets $N$ goes to infinity, while at the same time, the volume fraction
\[
\epsilon=\lim_{\substack{ N\rightarrow\infty\\|\Omega|\rightarrow\infty}}\left(\frac{\sum_{i=1}^N (4/3)\pi R_i^3}{|\Omega|}\right)
\] 
remains finite.  
\Acomment{The number density $P(r,t)$ of droplets of radius $r$ is introduced, such that $P(r,t)\mathd V$ is the number of droplets whose volume is in a range from $V$ to $V + \mathd V$.
Where $V = (4/3)\pi r^3$ and $\mathd V=r^2\mathd r\,\mathd \omega$  (here, $\mathd\omega$ is the differential solid-angle element)}.
Using standard conservation-type arguments, the evolution equation for $P(r,t)$ is just
\[
\frac{\partial P}{\partial t}+\nabla\cdot\left(\bm{v} P\right)=0,
\]
where $\bm{v}$ is the velocity of one of the droplets.  
The problem is radially symmetric, hence only the radial velocity is required.  
This is known from Equation~\eqref{eq7:dRi}, hence
\[
v_r(r,t)=-\frac{1}{r^2}+\frac{\overline{u}(t)}{r}.
\]
Using the expression for divergence in spherical polar coordinates, for a radially-symmetric configuration, the evolution equation for $P$ becomes:
\[
\frac{\partial P}{\partial t}+\frac{1}{r^2}\frac{\partial }{\partial r}\left(r^2 v_r P\right)=0.
\]
If we define
\[
P(r,t)r^2=p(r,t), \text{ such that }\int_0^\infty P(r,t)r^2\mathd r=\int_0^\infty p(r)\mathd r,
\]
then the required evolution equation is
\Acomment{
\begin{equation}
\frac{\partial p}{\partial t}+\frac{\partial }{\partial r}\left( v_r p\right) = 0,\qquad v_r(r,t)=-\frac{1}{r^2}+\frac{\overline{u}}{r}.
\label{eq7:fp}
\end{equation}
}
Hence, $p(r,t)$ has the interpretation as the number of droplets with a radius between $r$ and $r+\mathd r$, at time $t$.
In analogy with Equation~\eqref{eq7:uval}, Equation~\eqref{eq7:fp} is closed by requiring:
\[
\overline{u}=\frac{\int_0^\infty p(r,t)\mathd r}{\int_0^\infty rp(r,t)\mathd r}.
\]
We now seek a similarity solution of Equation~\eqref{eq7:fp}.  We write
\begin{equation}
p=t^a f(x),\qquad x=\frac{r}{ct^b}.
\label{eq7:sim}
\end{equation}
We fix $a$ in the first instance.  We use the fact that the volume fraction $\epsilon$ is constant, to compute 
\Acomment{
\begin{eqnarray*}
\epsilon&=&\frac{1}{|\Omega|}\iiint_{\Omega}(4\pi/3)r^3 P(r)\mathd V,\\
        &=&\frac{1}{|\Omega|}\iiint_{\Omega}(4\pi/3)r^3 p(r)\,\mathd r\,\mathd\omega,\\
                &=&\tfrac{1}{3}|\Omega|^{-1}(4\pi)^2\int_0^{R_{\mathrm{max}}} r^3 p(r)\mathd r.
\end{eqnarray*}
}
Here, $R_\mathrm{max}$ is a notional cutoff, with $R_{\mathrm{max}}\rightarrow\infty$ along with $|\Omega|\rightarrow\infty$, in such a way that $\epsilon$ remains finite.
Hence,
\[
\epsilon=\tfrac{1}{3}|\Omega|^{-1}(4\pi)^2 c^4 t^{a+4b}\int_0^{x_{\mathrm{max}}} x^3 f(x)\mathd x,
\]
where again, $x_\mathrm{max}\rightarrow\infty$ is a notional cut--off, chosen such that $\epsilon$ remains finite as $|\Omega|\rightarrow\infty$.
Thus, in order for $\epsilon$ to remain constant, it is required that $a=-4b$.
We now substitute the similarity solution~\eqref{eq7:sim} into Equation~\eqref{eq7:fp}.  After manipulations, we obtain:
\begin{equation}
-\tfrac{1}{3}c^3\left[3f(x)+ f'(x)x\right]+\frac{\partial}{\partial x}\left[\left(-\frac{1}{x^2}+\frac{\hat{u}}{x}\right)f\right]=0,\qquad
\hat{u}=\frac{\int_0^\infty f(x)\mathd x}{\int_0^\infty xf(x)\mathd x}.
\label{eq7:ls_sim}
\end{equation}
Following convention, we write $(1/3)c^3=\gamma_L$, to give:
\begin{equation}
-\gamma_L\left[3f(x)+ f'(x)x\right]+\frac{\partial}{\partial x}\left[\left(-\frac{1}{x^2}+\frac{\hat{u}}{x}\right)f\right]=0,\qquad
\hat{u}=\frac{\int_0^\infty f(x)\mathd x}{\int_0^\infty xf(x)\mathd x}.
\label{eq7:ls_sim1}
\end{equation}
This can then be integrated to give~\cite{bray2002theory}:
\begin{equation}
\ln [f(x)]=\int^x \frac{\mathd y}{y}\frac{2-y-3\gamma_L y^3}{\gamma_L y^3-y+1}.
\label{eq7:ls_sim2}
\end{equation}
Equation~\eqref{eq7:ls_sim1} gives a family of potential solutions, all parametrized by $\gamma_L$.  
The equation is also potentially without a normalizable solution with $f(x)\rightarrow 0$ as $x\rightarrow \infty$.  
These problems are solved by imposing two conditions on Equation~\eqref{eq7:ls_sim2}:
\begin{itemize}
\item The solution $f(x)$ should have compact support;
\item The value $\gamma_L=4/27$ must be selected.
\end{itemize}
The rationale for the second condition is related to the fixed points of Equation~\eqref{eq7:ls_sim} and was determined by Lifshitz and Slyzov~\cite{lifshitz1961kinetics} and summarized by Bray~\cite{bray2002theory}.
\Acomment{As such, the following general solution for $f(x)$ is found in for $D \geq 3$ dimensions (the results presented following this expression are for $D = 3$ to limit the solution to that of spheres as discussed above), by integration of Equation~\eqref{eq7:ls_sim2}:}
\begin{equation}
f(x)=\begin{cases} \text{Const.}\times x^2(3+x)^{-1-4D/9}\left(\tfrac{3}{2}-x\right)^{-2-5D/9}\exp\left(-\frac{D}{3-2x}\right),& 0\leq x<(3/2),\\
0,&\text{otherwise}.
\end{cases}
\label{eq7:fx_final}
\end{equation}
In LSW theory, the expected mean radius is computed as follows:
\begin{equation}
\langle R\rangle 
= \frac{\int_0^\infty r p(r)\,\mathd r}{\int_0^\infty p(r)\,\mathd r}
= c t^{1/3}\frac{\int_0^\infty x f(x)\,\mathd x}{\int_0^\infty f(x)\,\mathd x}
=(3\gamma_L t)^{1/3},
\label{eq7:R_av}
\end{equation}
where the last equation follows since $c=(3\gamma_L)^{1/3}$ and since the distribution $f(x)$ in Equation~\eqref{eq7:fx_final} has the property
$\left[\int x f(x)\mathd x\right]/\left[\int f(x)\mathd x\right]=1$.  
We note also that
\begin{equation}
\overline{u}=1/\langle R\rangle, \text{hence }\overline{u}=(3\gamma_L t)^{-1/3}.
\label{eq7:u_av}
\end{equation}


\subsection{LSW theory -- Droplet growth rates}


As a point of departure, we explore not only the distribution of droplet radii in Ostwald Ripening, but also, the distribution of droplet growth rates.
For this purpose, we introduce
\begin{equation}
\alpha_i= \frac{t}{R_i}\frac{\mathd R_i}{\mathd t}= \frac{t}{R_i}\left(-\frac{1}{R_i^2}+\frac{\overline{u}}{R_i}\right).
\label{eq7:alphai_def}
\end{equation}
Thus, either $\alpha_i\rightarrow 0$ as $t\rightarrow \infty$ (for evaporating drops), or $\alpha_i\rightarrow 1/3$ as $t\rightarrow \infty$ (for growing drops).  We now want to characterize the distribution of the $\alpha_i$'s, over an ensemble of droplets, which we denote by $p(\alpha,t)$:
\[
p(\alpha,t)=\text{Number of droplets with exponent between }\alpha\text{ and }\alpha+\mathd\alpha.
\]  
Formally, we can calculate $p(\alpha,t)$ using the standard change-of-variable formula for probability theory:
\[
p(\alpha)=p(r(\alpha),t)\left|\frac{\partial r}{\partial \alpha}\right|,
\]
where $\alpha$ and $r$ are connected by
\[
\alpha=\frac{t}{r}\left(-\frac{1}{r^2}+\frac{\overline{u}}{r}\right).
\]
As such, we have the following formal identity (the $t$-dependence is  suppressed for now):
%
%
\begin{equation}
p(\alpha)=p(r(\alpha))\left|\frac{\partial r}{\partial \alpha}\right|
         =p(r(\alpha))\left|\frac{\partial \alpha}{\partial r}\right|^{-1}.
\label{eq2:change}
\end{equation}
It is not straightforward to implement the substitutions in Equation~\eqref{eq2:change}, because $\alpha$ is not a monotone function of $r$ (e.g. Figure~\ref{fig7:r1} on the entire range of the function $\alpha(r)$).  
\Acomment{
Indeed, the derivative of $\alpha(r)$ changes sign at $r=(3/2)\overline{u}^{-1}$, i.e. $\partial \alpha/\partial r=0$ at $r=(3/2)\langle R\rangle$.  
}
Before solving this problem, we remark that the existence of the local maximum $\partial \alpha/\partial r=0$ gives rise to the following useful results:
\begin{proposition}
There is a maximum coarsening rate
\begin{equation}
\alpha_{max}=\tfrac{1}{3}\left(\frac{4}{9}\frac{t}{\langle R\rangle^3}\right),
\label{eq7:alpha_max1}
\end{equation}
where $\langle R\rangle$ is computed via Equation~\eqref{eq7:R_av}.
\end{proposition}
The proof of this statement by straightforward substitution of $r=(3/2)\langle R\rangle$ into the equation for $\alpha(r)$.  
Furthermore,
\begin{proposition}
In the regime where the self-similar distribution function~\eqref{eq7:fx_final} applies, the maximum coarsening rate simplifies:
\[
\alpha_{max}=1/3.
\]
\end{proposition}
This can be shown by direct computation, specifically by substituting $\langle R\rangle=(4t/9)^{1/3}$ into Equation~\eqref{eq7:alpha_max1}.

Having now established the existence of the maximum coarsening rate, it follows that that $\alpha(r)$ is non-monotonic, and hence, the formal change-of-variables law~\eqref{eq2:change} needs clarification.  
Therefore, to calculate $p(\alpha)$ properly, we refer to Figure~\ref{fig7:r1}. 
\begin{figure}
    \centering
        \includegraphics[width=0.6\textwidth]{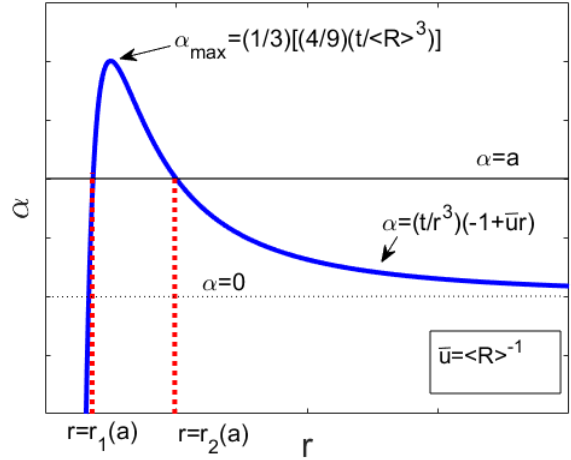}
        \caption{Definition sketch for the change of variable $\alpha=(t/r^3)(-1+\overline{u}r)$.}
    \label{fig7:r1}
\end{figure}
We look at a definite fixed value of $\alpha$, denoted by $a$.  
For $a>0$, we read off the definitions of $r_1(a)$ and $r_2(a)$ from the figure.  
We look at the cumulative probability function for $\alpha$,
\begin{eqnarray*}
F_\alpha(a)&=&\mathbb{P}(\alpha \leq a),\\
           &=&\mathbb{P}(r\leq r_1(a))+\mathbb{P}(r\geq r_2(a)),\\
                     &=&F_r(r_1(a))+\left[1-F_r(r_2(a))\right].
\end{eqnarray*}
We differentiate to compute the probability distribution function:
\begin{eqnarray*}
p_\alpha(a)&=&\frac{\mathd F_\alpha(a)}{\mathd a},\\
           &=&\frac{\partial F_r}{\partial r}\bigg|_{r_1(a)}\frac{\mathd r_1}{\mathd a}
                    -\frac{\partial F_r}{\partial r}\bigg|_{r_2(a)}\frac{\mathd r_2}{\mathd a},\\
                     &=&p_r(r_1(a))\frac{\mathd r_1}{\mathd a}-
                    p_r(r_2(a))\frac{\mathd r_2}{\mathd a},\\
                     &=&p_r(r_1(a))\frac{\mathd r_1}{\mathd a}+
                    p_r(r_2(a))\left|\frac{\mathd r_2}{\mathd a}\right|.
\end{eqnarray*}
Thus, the distribution of exponents $\alpha$ is established for $a>0$:
\begin{equation}
p_\alpha(a)=p_r(r_1(a))\left|\frac{\mathd r_1}{\mathd a}\right|+
                    p_r(r_2(a))\left|\frac{\mathd r_2}{\mathd a}\right|,\qquad a>0,
\label{eq2:change_of_var_final}
\end{equation}
where the first instance of $\left|\cdot\right|$ is added just to make the formula appear more symmetric.  
Referring back to Figure~\ref{fig7:r1}, at $a=0$, the two roots $r_1(a)$ and $r_2(a)$ coincide, and for $a<0$ only one root (denoted by $r_1(a)$ survives).  
As such, the following final form of $p_\alpha$ applies,
\begin{equation}
p_\alpha(a,t)=\begin{cases}p_r(r_1(a),t)\left|\frac{\mathd r_1}{\mathd a}\right|+
                    p_r(r_2(a),t)\left|\frac{\mathd r_2}{\mathd a}\right|,& a>0,\\
                     p_r(r_1(a),t)\left|\frac{\mathd r_1}{\mathd a}\right|,&a\leq 0.
                    \end{cases}
\label{eq2:change_of_var_final_final}
\end{equation}
where we have restored the time-dependence of the distributions.

In the case where $p_r(r,t)$ satisfies the LSW distribution~\eqref{eq7:fx_final}, it is possible to compute the corresponding growth-rate distribution $p_\alpha(a,t)$.  
This  is shown in Figure~\ref{fig7:p_alpha}.  
\Acomment{It is verified that in this instance, the distribution of growth rates is time-independent.}
\begin{figure}
    \centering
        \includegraphics[width=0.6\textwidth]{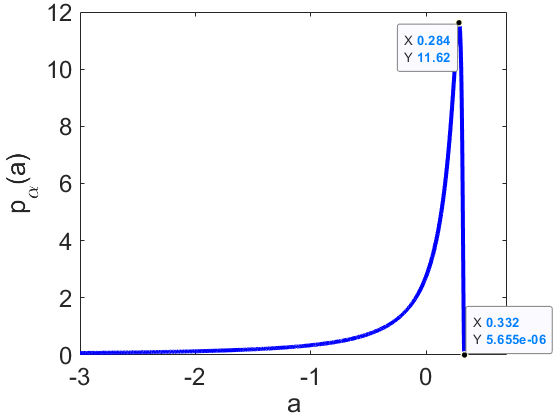}
        \caption{The PDF $p_\alpha(a)$ (with a slight abuse of notation on the axis labels in the figure).  The distribution is time-independent so the notation $p_\alpha(a,t)$ can be replaced with $p_\alpha(a)$.}
    \label{fig7:p_alpha}
\end{figure}
The distribution $p_\alpha(a)$ goes to zero at $a=1/3$, \Acomment{corresponding to the fact that this is the maximum possible growth rate of a bubble in the system.}
Otherwise, the distribution is strongly peaked at $a\approx 0.284$, corresponding to the mode or the most probable coarsening rate in the problem.  
Notably, the distribution is strongly skewed to the left, with a long tail of negative growth rates extending to $a=-\infty$.  
The negative growth rates correspond to the evaporating droplets.  
The maximum growth rate corresponds to a winner-takes-all scenario where a single droplet is growing to the maximal extent possible, at the expense of all other droplets in the system.  
As such, the distribution in Figure~\ref{fig7:p_alpha} makes sense physically.


\subsection{System energy}
\label{sec7:systemEnergy}
We further introduce an energy function to characterize the dynamics of Equation~\eqref{eq7:dRi}.  
Implicit in that equation is the presence of a surface-tension coefficient which is set to $1/2$.  
As such, the energy of a single droplet corresponds to its surface energy, and is given by $(1/2)\times (4\pi R_i^2)$.  
The total surface energy is obtained by summation over all droplets.  
\Acomment{
However, the proper surface energy contains the constraint which enforces $\sum_{i=1}^N (4\pi /3)R_i^3=V_0$, where $V_0$ is constant.}  
Thus, the surface energy reads:
\begin{equation}
F= \sum_{i=1}^N \tfrac{1}{2} R_i^2-\lambda\left(\sum_{i=1}^3 \tfrac{1}{3} R_i^3-\frac{V_0}{4\pi}\right),
\end{equation}
where $\lambda$ is the possibly time-dependent Lagrange multiplier which enforces the constancy of $\sum_{i=1}^3 (4\pi/3)R_i^3$, and where we have omitted an overall factor of $4\pi$ in the definition of $F$ -- this is done for convenience. 
\begin{equation}
\frac{\mathd F}{\mathd t}=\sum_{i=1}^N \left(R_i-\lambda R_i^2\right)\dot R_i+(\mathd \lambda/\mathd t)\left(\sum_{i=1}^3 \tfrac{1}{3} R_i^3-\frac{V_0}{4\pi}\right)
\end{equation}
The last term proportional to $\mathd \lambda/\mathd t$ vanishes on enforcing the constraint on $\sum_{i=1}^3 (4\pi/3)R_i^3$.  
Thus,
\begin{equation}
\frac{\mathd F}{\mathd t}
\stackrel{\text{Eq.~\eqref{eq7:dRi}}}{=}\sum_{i=1}^N H(R_i)\left(R_i-\lambda R_i^2\right)\left(-\frac{1}{R_i^2}+\frac{\overline{u}}{R_i}\right)
\stackrel{\lambda=\overline{u}}{=}-\sum_{i=1}^N \frac{H(R_i)}{R_i}\left(1-\overline{u}R_i\right)^2,
\label{eq7:dFdt}
\end{equation}
hence $\mathd F/\mathd t\leq 0$.
Here, the equation $\lambda=\overline{u}$ can be made, since $\lambda$ and $\overline{u}$ are associated with the same constraint.  
Using the identity $\lambda=\overline{u}$, we can write   $\partial F/\partial R_i=R_i-\overline{u} R_i^2$, and hence, from Equation~\eqref{eq7:dRi},

\begin{equation}
\frac{\mathd R_i}{\mathd t}=-m(R_i)\frac{\partial F}{\partial R_i},\qquad m(R_i)=H(R_i)R_i^3.
\label{eq7:dRi_m}
\end{equation}

Thus, the dynamics of the droplets take the form of a gradient flow, with mobility $m(R_i)=H(R_i)R_i^3$.  Furthermore, we can therefore write

\begin{equation}
\frac{\mathd F}{\mathd t}=-\sum_{i=1}^N m(R_i)\left(\frac{\partial F}{\partial R_i}\right)^2,
\label{eq7:dFdt1}
\end{equation}
which makes the relation $\mathd F/\mathd t\leq 0$ more manifest.
In analogy to the growth rate $\alpha_i$ (Equation~\eqref{eq7:alphai_def}) for individual droplets, we introduce an energy decay rate, applicable to the entire system of $N$ droplets:

\begin{equation}
\beta_N=-\frac{t}{F}\frac{\mathd F}{\mathd t}.
\label{eq7:betaN}
\end{equation}

In Section~\ref{sec7:drop_pop}, the system of equations~\eqref{eq7:dRi}--\eqref{eq7:uval} will be solved repeatedly for a fixed number of droplets $N$; each simulation will have different random initial conditions.  
In this way, a probability distribution function $p_\beta(b,t)$ will be constructed, such that $p_\beta(b,t)$ is the probability that a given simulation will produce a decay rate $\beta_N$ in the range from $b$ to $b+\mathd b$, at time $t$.  
The probability distribution function will be explored numerically.  
The distribution of $\beta_N$ can however already be extracted straightforwardly in the LSW limit with $N\rightarrow\infty$:

\begin{proposition}
At late times, $p_\beta(b,t)\rightarrow \delta(b-(1/3))$, for the LSW limit with $N\rightarrow \infty$ and $\epsilon$ small but finite.
\end{proposition}

\begin{proof}
Once the volume-constraint $\sum_{i=1}^N (4\pi/3)R_i^3=V_0$ has been implemented, the energy is just $F=(1/2)\sum_{i=1}^N R_i^2$. 
In the LSW limit, this can be computed explicitly, via Equation~\eqref{eq7:sim}

\begin{equation}
F=
\tfrac{1}{2}\int_0^\infty r^2 p(r,t)\mathd r.
\label{eq7:F_expected}
\end{equation}

We therefore have:
\[
F=\tfrac{1}{2}c^3 t^{3b+a}\int_0^\infty x^2 f(x)\mathd x,\qquad \text{as }t\rightarrow\infty.
\]
The late-time limit is required here as the LSW theory is valid only asymptotically, as $t\rightarrow\infty$.  
We also use $a=-4b$, hence
\[
F=\tfrac{1}{2}c^3 t^{-b}\int_0^\infty x^2 f(x)\mathd x.
\]
We now use $b=1/3$ to conclude that $F\propto t^{-1/3}$, and hence, $\beta=-(t/F)(\mathd F/\mathd t)=1/3$.  
Thus, $\beta$ takes only a single value in the LSW theory, hence $p_\beta(b,t)=\delta(b-(1/3))$ as $t\rightarrow\infty$.
\end{proof}


\section{Methodology}
\label{sec7:methodology}
For the purposes of producing the statistical results required in this work we rely on two methods, parallelisation of batches of solutions of 1D coupled ODEs in MATLAB and parallelisation of batches of 2D Cahn--Hilliard equations using a GPU.

\subsection{ODE Methodology}
We begin by discussing the implementation of the ODEs for the LSW scaling theory, we are interested in solving the following set of ODEs:
\begin{subequations}
\begin{equation}
\frac{\mathd R_i}{\mathd t}=H(R_i)\left(-\frac{1}{R_i^2}+\frac{\overline{u}}{R_i}\right),\qquad i=1,\cdots,N,\qquad t>0,
\label{eq7:dRi1}
\end{equation}
where $H(\cdot)$ is the Heaviside step function, and where $\overline{u}$ is a constraint which forces $\sum_{i=1}^N (4/3)R_i^3=\mathrm{Const.}$, hence $\sum_{i=1}^N R_i^2(\mathd R_i/\mathd t)=0$, hence finally,
\begin{equation}
\overline{u}=\frac{\sum_{i=1}^N H(R_i)}{\sum_{i=1}^N H(R_i)R_i}.
\end{equation}
Equation~\eqref{eq7:dRi1} is solved with the initial condition
\begin{equation}
R_i(t=0)=r_i,
\end{equation}%
\label{eq7:dRi1_model}%
\end{subequations}
where $r_i$ is a random variable drawn from a uniform distribution between $0$ and $1$.  
The use of the Heaviside step function also allows us to regularise the ODE, in the numerical scheme we implement the step function as $H(R_i + \epsilon)$ where we take $\epsilon = 1 \times 10^{-3}$.
This regularisation correctly reduces the radius of any small bubble to $0$ and allows us to deal with the coordinate singularity present in equation~\eqref{eq7:dRi1}.
The system is solved using MATLAB's \Acomment{standard ode45 solver}, the systems are solved in parallel with one system per CPU core. 
A batch $100$ of independent systems for varying sizes of $N$ is then used to extract a $\beta_{i,N}(t)$ from each simulation where $i$ denotes a simulation index and $\beta_N$ is given by equation~\eqref{eq7:betaN}.
Exact simulation details are results are given in Section~\ref{sec7:drop_pop}.

\subsection{2D Cahn--Hilliard Methodology}
For the studies of the Cahn--Hilliard equation we solve~\eqref{eq2:ch} on a uniform periodic $2D$ grid using a finite difference scheme with an ADI method for time-stepping, specifically with the hyperdiffusion operator solved implicitly~\cite{gloster2019custen, gloster2019cupentbatch, gloster2019efficient}, details can also be found in Chapter~\ref{chapter:cuSten} of this thesis.
The initial conditions are randomised using a uniform distribution with an average depending on the desired mixture, symmetric and asymmetric, of the simulations, a representative value of $\gamma = 0.01$ is used in the simulations along with a final simulation time of $T = 100$, the time--steps are of size $0.1 \Delta x$ to ensure stability.
This numerical scheme is then implemented on a GPU using CUDA and we take $\Delta x = 2 \pi / 256$ to ensure $\Delta x < 2 \pi \sqrt{\gamma}$ to resolve the bubble interfaces.

Batches of 2D simulations are then run using a CUDA MPS server which allows multiple simulations, each with its own process, to be run on a given GPU at the same time.
From these simulations we then recover the necessary statistics such that a statistical picture of the growth rate $\beta = -(t/N)(\mathd F/\mathd t)$ can be built up.
For the case of the Cahn--Hilliard equation the free energy necessary to study this system is given by equation~\eqref{eq2:dFdt_ch}, we will extract the necessary time derivative numerically.
In each case we produce a batch of $1024$ simulations for domains of size $2\pi \times 2\pi$, $4\pi \times 4\pi$, $8\pi \times 8\pi$ and $16\pi \times 16\pi$.
The results for asymmetric mixtures are presented in Section~\ref{sec7:CH}, for symmetric mixtures in Section~\ref{sec7:CH_symm} and finally the Cahn--Hilliard--Cooke equation in Section~\ref{sec7:cooke}.
In all cases we limit the statistical windows from these simulations to $10 < t < 100$ as we are interested in the period of the simulation covering late stage coarsening of the domain but before the final steady state appears, all histograms presented will have had their areas normalised to $1$.
We also only sample for values of $\beta < 1.0$, this is to remove statistical outliers caused by the merging of separate bubbles or interconnected regions, further discussion of these outliers in the context of the simulation results and literature is given in the following Sub-Section~\ref{sec7:outliers}.


\begin{figure}
    \centering
\begin{subfigure}[b]{0.49\textwidth}
    \centering
        \includegraphics[width=1.0\textwidth]{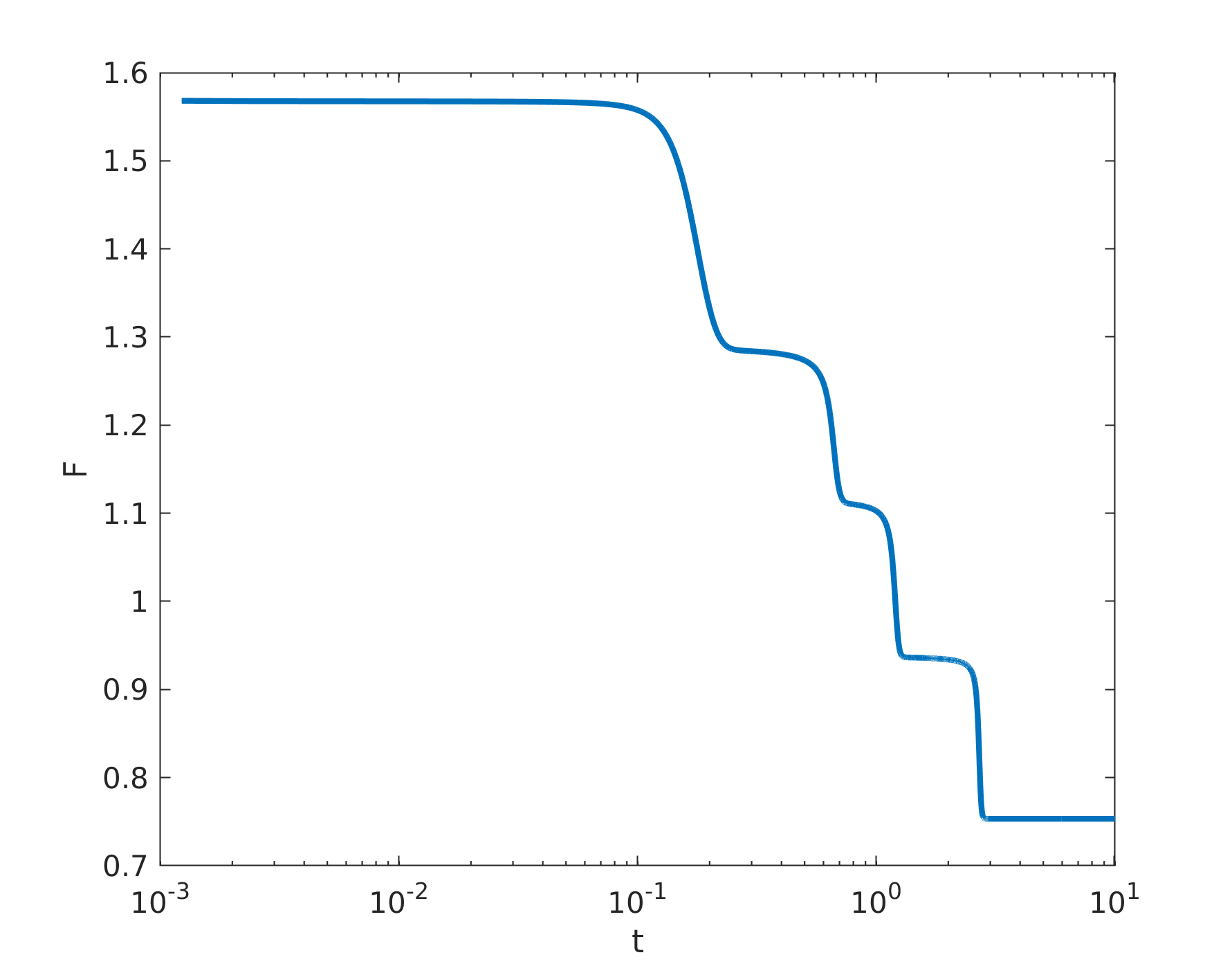}
        \caption{Free energy $F$}
    \label{fig7:1Dcoarsening}
\end{subfigure}
    \hfill
\begin{subfigure}[b]{0.49\textwidth}
    \centering
        \includegraphics[width=1.0\textwidth]{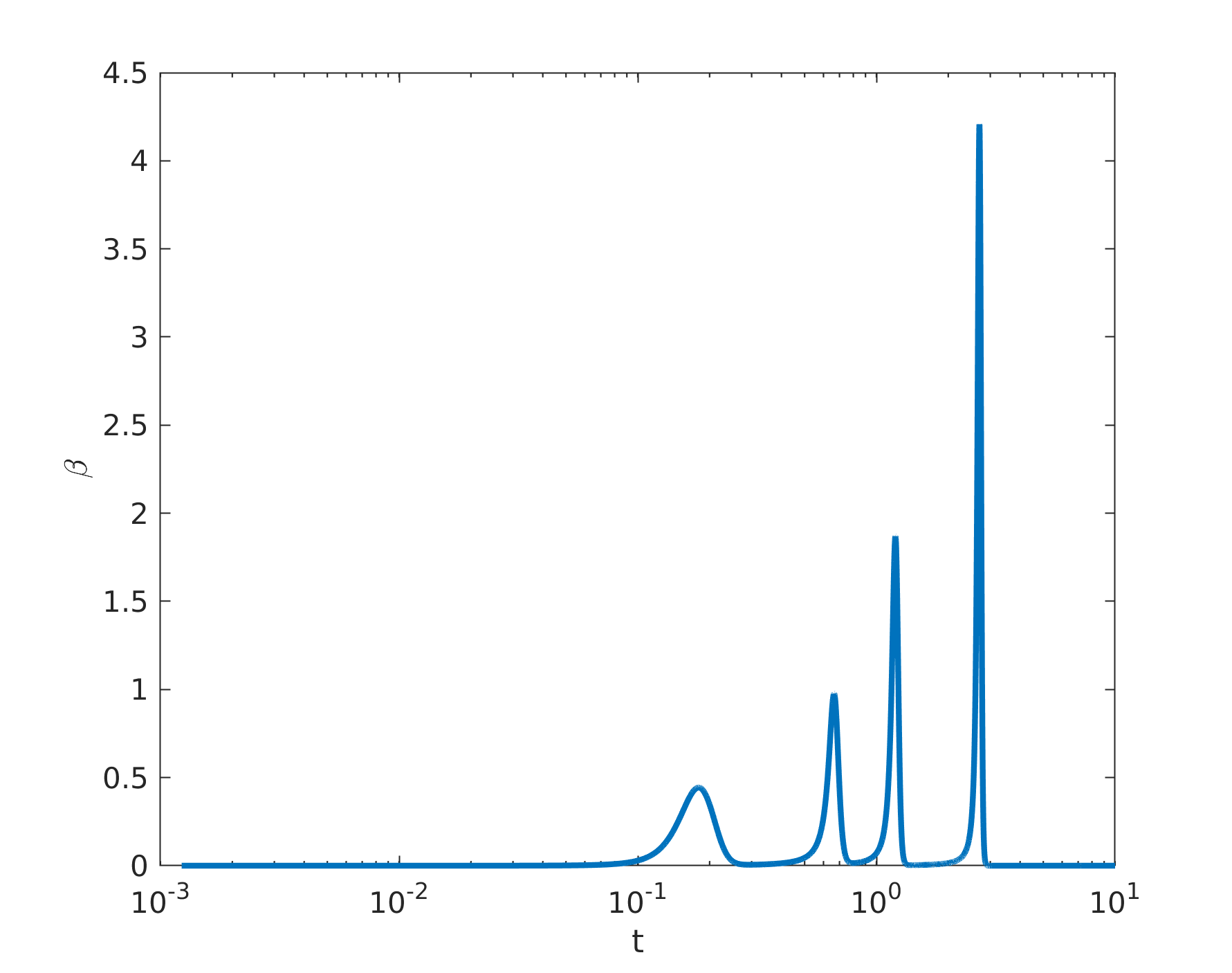}
        \caption{Growth rate $\beta$}
    \label{fig7:beta1D}
\end{subfigure}
        \caption{Plots of (a) the free energy of the 1D Cahn--Hilliard equation and (b) the corresponding value of $\beta$.}
    \label{fig7:dynamics1D}
\end{figure}

\subsection{Merging Regions and Statistical Outliers}
\label{sec7:outliers}
We begin this discussion by briefly examining the evolution of $F$ in a 1D Cahn--Hilliard system, an example of which can be seen in Figure~\ref{fig7:1Dcoarsening}.
It can be seen that a steady state of dynamics exists for intervals of time broken by rapid drops in the value of the free energy, these drops correspond to the sudden merging of two regions to form one larger region.
Plotting the corresponding values of $\beta$ for this data in Figure~\ref{fig7:beta1D} we see corresponding peaks, far greater in amplitude than $1/3$.
Thus when two regions merge the value of $\beta$ spikes and forms a value that can be considered an outlier.

The exact same feature occurs in 2D simulations where two bubbles or previously unconnected regions suddenly join and merge to form larger bubbles/regions.
\Acomment{We can see this behaviour in Figure~\ref{fig7:dynamics2D} where we have extracted one simulation from the asymmetric mixture Cahn--Hilliard equation $2\pi \times 2\pi$ batch.}
Thus we're correct to omit large outliers from our statistics discussed throughout this chapter, the theory from LSW and our new stochastic model account for steady growth with no merging bubbles/regions yet we now have a clear physical understanding as to why they occur.

\begin{figure}
    \centering
\begin{subfigure}[b]{0.49\textwidth}
    \centering
        \includegraphics[width=1.0\textwidth]{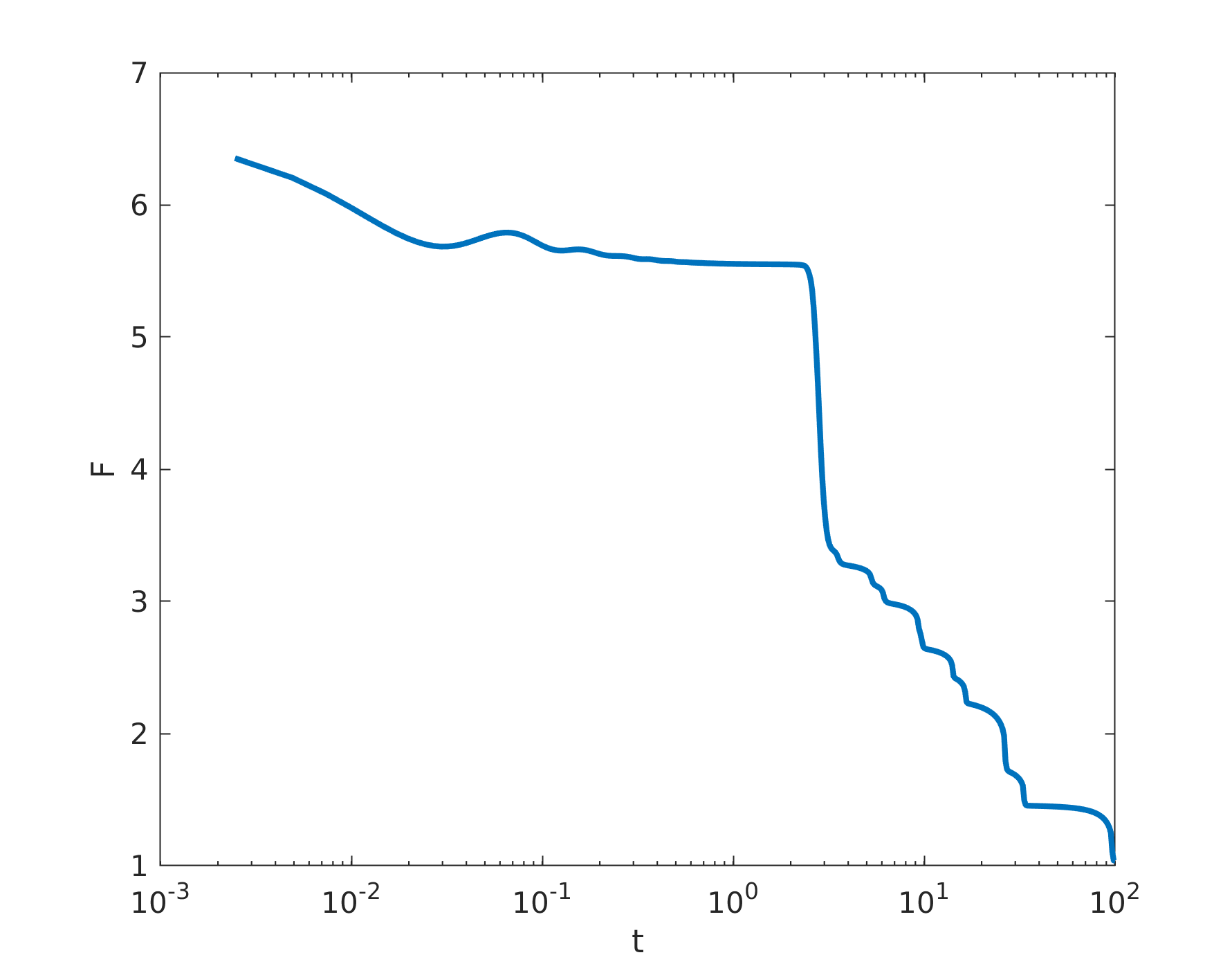}
        \caption{Free energy $F$}
    \label{fig7:2DFreeEng}
\end{subfigure}
    \hfill
\begin{subfigure}[b]{0.49\textwidth}
    \centering
        \includegraphics[width=1.0\textwidth]{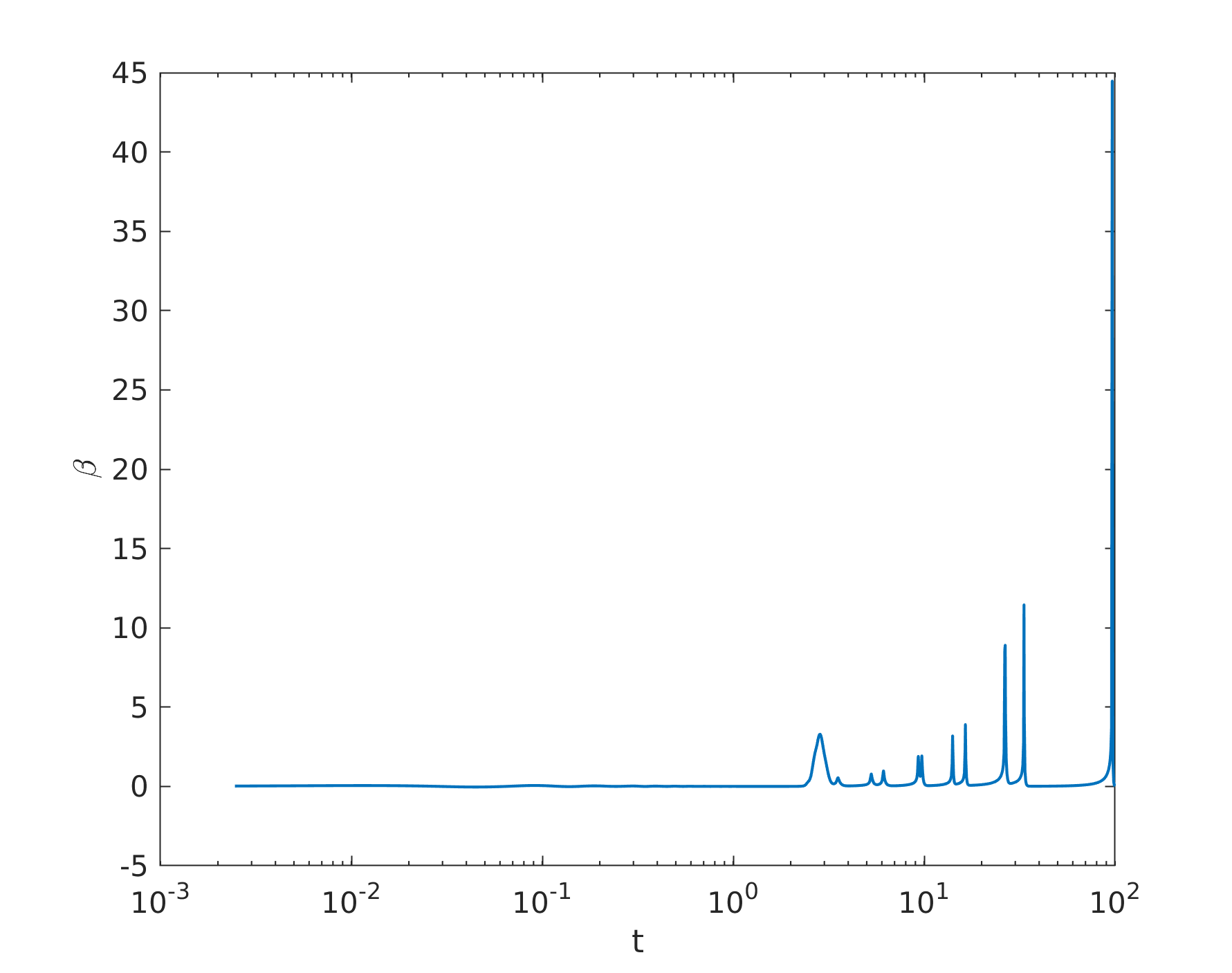}
        \caption{Growth rate $\beta$}
    \label{fig7:2Dbeta}
\end{subfigure}
        \caption{Plots of (a) the free energy of the 2D Cahn--Hilliard equation with asymmetric mixture and (b) the corresponding value of $\beta$.}
    \label{fig7:dynamics2D}
\end{figure}


\begin{figure}
    \centering
\begin{subfigure}[b]{0.49\textwidth}
    \centering
        \includegraphics[width=1.0\textwidth]{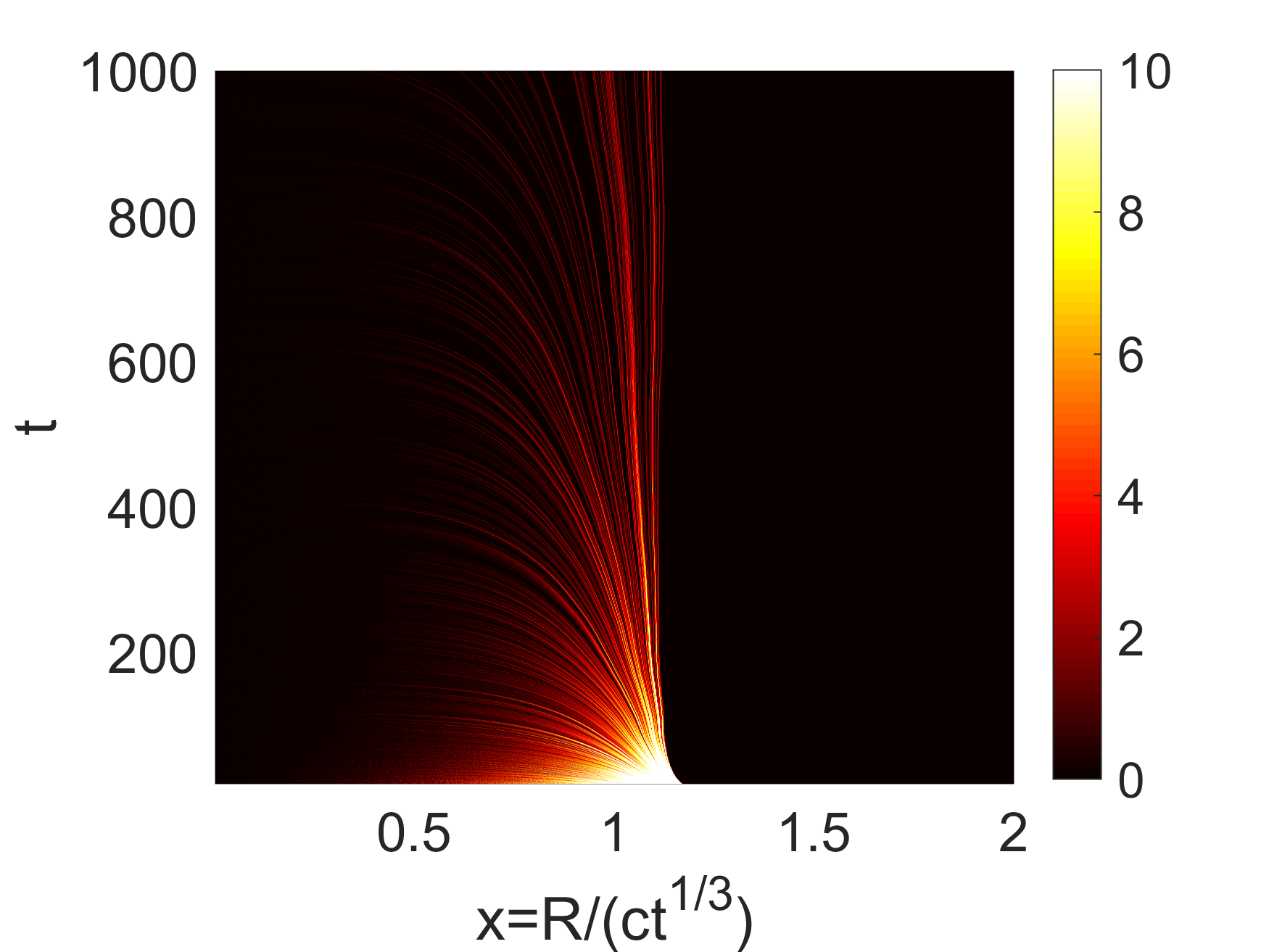}
        \caption{\Acomment{Distribution of bubble--radii with time.}}
    \label{fig7:histogramR}
\end{subfigure}
    \hfill
\begin{subfigure}[b]{0.49\textwidth}
    \centering
        \includegraphics[width=1.0\textwidth]{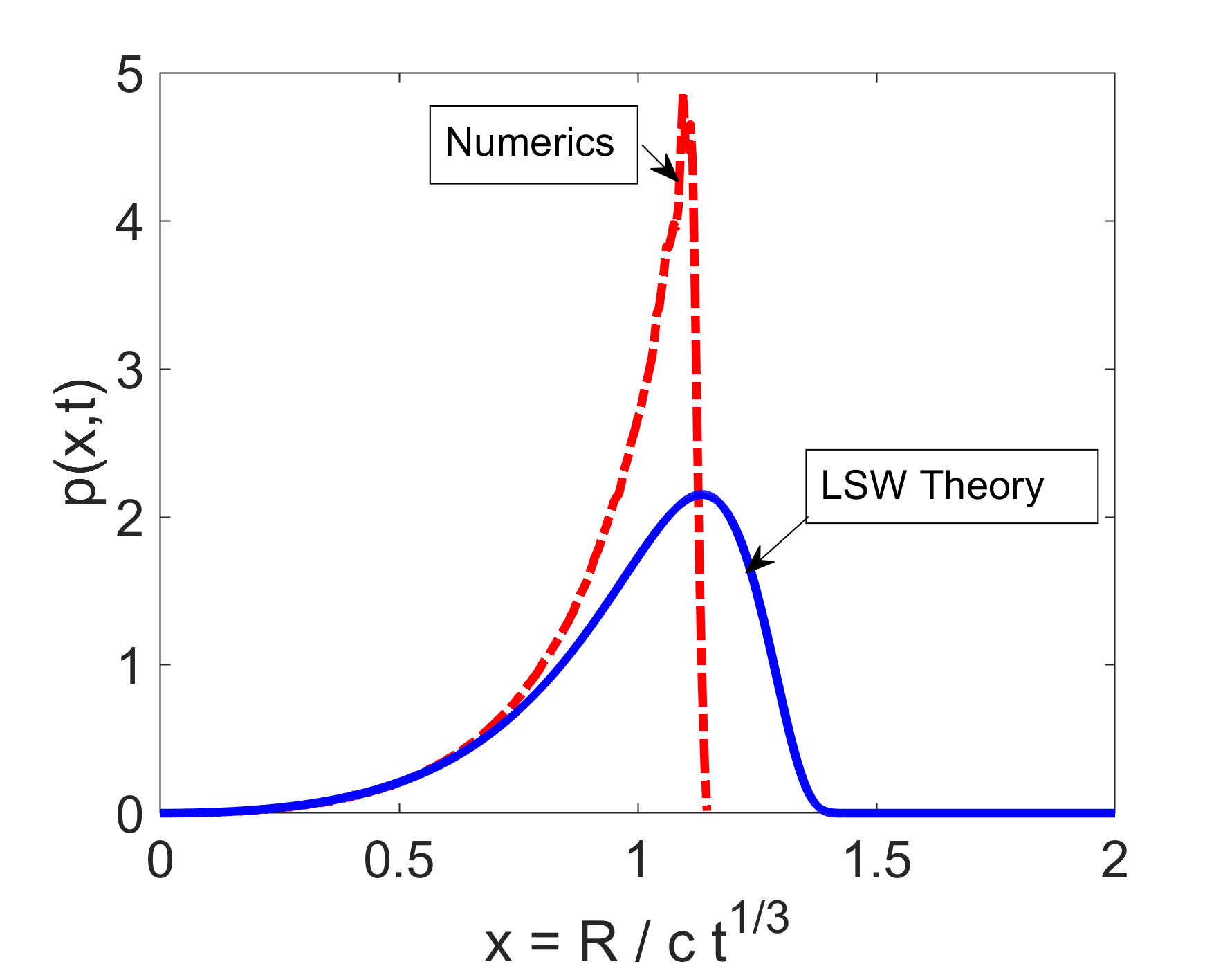}
        \caption{$N = 100,000$ vs. Theory.}
    \label{fig7:bubbleRadVsTheory}
\end{subfigure}
    \vskip\baselineskip
\begin{subfigure}[b]{0.49\textwidth}
    \centering
        \includegraphics[width=1.0\textwidth]{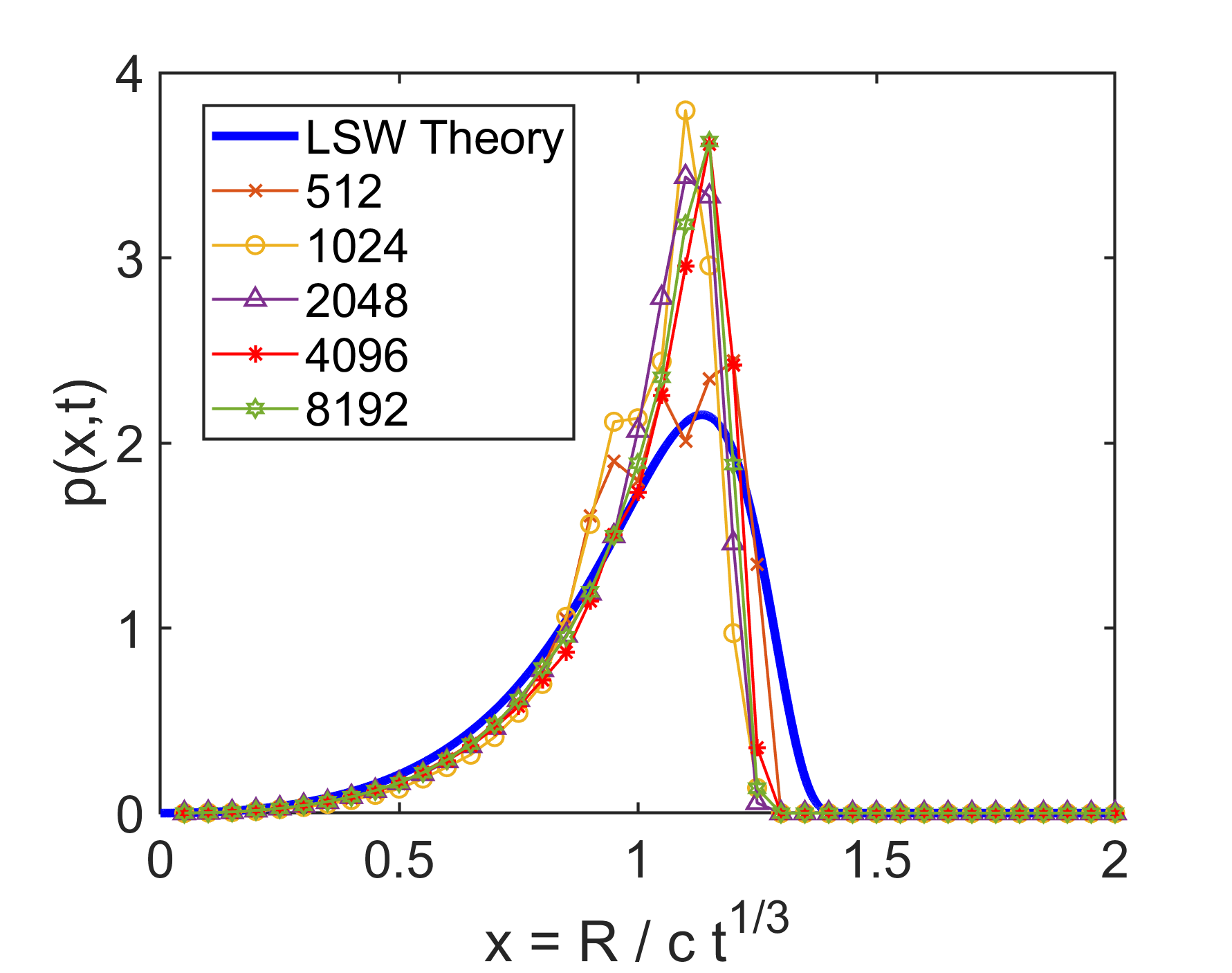}
        \caption{$x$ for various $N$ vs. Theory.}
    \label{fig7:radConverge}
\end{subfigure}
        \caption{Plots examining the distribution of bubble radii in time and comparisons with theory, in both cases the distribution have been normalised to have an area of $1$: (a) \Acomment{Plot of distribution of bubble radii in the modified coordinate system $x = R / c t^{1/3}$ as a function of time} (b) Cumulative histogram of the late time--steps, the presence of the effect of the finite number of bubbles is clearly seen with the truncation of the distribution before the possible maximum and the good match on the lower tail for small bubbles unaffected by finite size effects. (c) The truncation of large bubble sizes is clearer in this plot, we can see convergence of the tail for increasing $N$ with a peak falling before the theoretical prediction for the largest value of $x$.}
    \label{fig7:bubbleRad}
\end{figure}

\section{Droplet population model -- numerical results}
\label{sec7:drop_pop}
We begin with a direct comparison of the droplet radius distribution from the simulations to the theory given by equation~\eqref{eq7:fx_final}, the results for this discussion can be seen in Figure~\ref{fig7:bubbleRad}.
We present a \Acomment{distribution--time} plot of the distribution of the quantity $x = R(t) / (c t^{1/3})$ using data recovered from simulation of equation~\eqref{eq7:dRi1_model} in Figure~\ref{fig7:histogramR}.
In this case we have set $N = 100,000$ for the initial bubble count and simulated the ODE up to \Acomment{$T = 1000$.  
We} present a snapshot of this simulation and omit initial times as we want to examine only late stage bubble growth while equally we omit very late times in order to not skew our statistics by regions where there is a low number of remaining bubbles. 
We can see a clear movement towards a steady state distribution.
\Acomment{Due} to the finite number of the bubble of the domain we can also see that the very largest of bubbles which should have a size of $x \approx 1.5$ have not formed.

Taking this \Acomment{distribution--time} plot and plotting a cumulative histogram of late stage time--steps, in this case we take $40 < t < 200$, allows us to compare directly with the bubble distribution from equation~\eqref{eq7:fx_final}, this can be seen in Figure~\ref{fig7:bubbleRadVsTheory}.
In this plot the effects of the finite number of the bubbles is clearer, with the maximum bubble \Acomment{size cut} off long before the theoretical maximum value of $x$.
A higher peak at lower values of $x$ also accounts for this.
There is a good match for small bubble radii (small $x$) between theory and numerics, this is expected as larger bubbles are only formed by consumption of smaller bubbles in the domain and there is only a finite number of smaller bubbles available to produce the large bubbles.
\Acomment{Smaller} bubbles themselves and their subsequent distribution are not \Acomment{affected} by this behaviour.
This behaviour is similarly seen in Figure~\ref{fig7:radConverge} where good convergence to the left tail can be seen with increasing $N$ yet we see the histogram falling short of the possible maximum $x$ predicted by the theory.

It is of interest to look into the lack of agreement between the numerics and the LSW theory in Figure~\ref{fig7:bubbleRadVsTheory}.  
The number of droplets $N$ present initially in the simulation can be ruled out as the cause of the disagreement: the dependence of the cumulative histogram on $N$ is shown in Figure~\ref{fig7:radConverge}, there is little or no difference between all of the considered $N$ values.
Therefore, the cause of the disagreement in Figure~\ref{fig7:bubbleRadVsTheory} can be attributed to the shape of the initial drop--size distribution: the initial drop--size distribution is compactly supported (the uniform distribution with initial radii between $0$ and $1$); however, this distribution is not smooth at the points where it touches down to zero.  
\Acomment{
In addition, further criteria must be satisfied, these are known as the weak selection rules~\cite{giron1998weak}, and these are not satisfied by our chosen initial distribution. 
Hence convergence to the LSW statistics is not guaranteed as the late time behaviour may become non-self-similar~\cite{niethammer1999non}.
An overview of this behaviour can also be found in Reference~\cite{mielke2006analysis}.
It can be emphasized that in other works on droplets (e.g. Reference~\cite{yao1993theory}), the initial drop--size distribution was carefully selected such that late-time convergence to LSW statistics was obtained.  
From these results, the convergence to the LSW statistics is demonstrated not to be robust.} 

\begin{figure}
    \centering
\begin{subfigure}[b]{0.49\textwidth}
    \centering
        \includegraphics[width=1.0\textwidth]{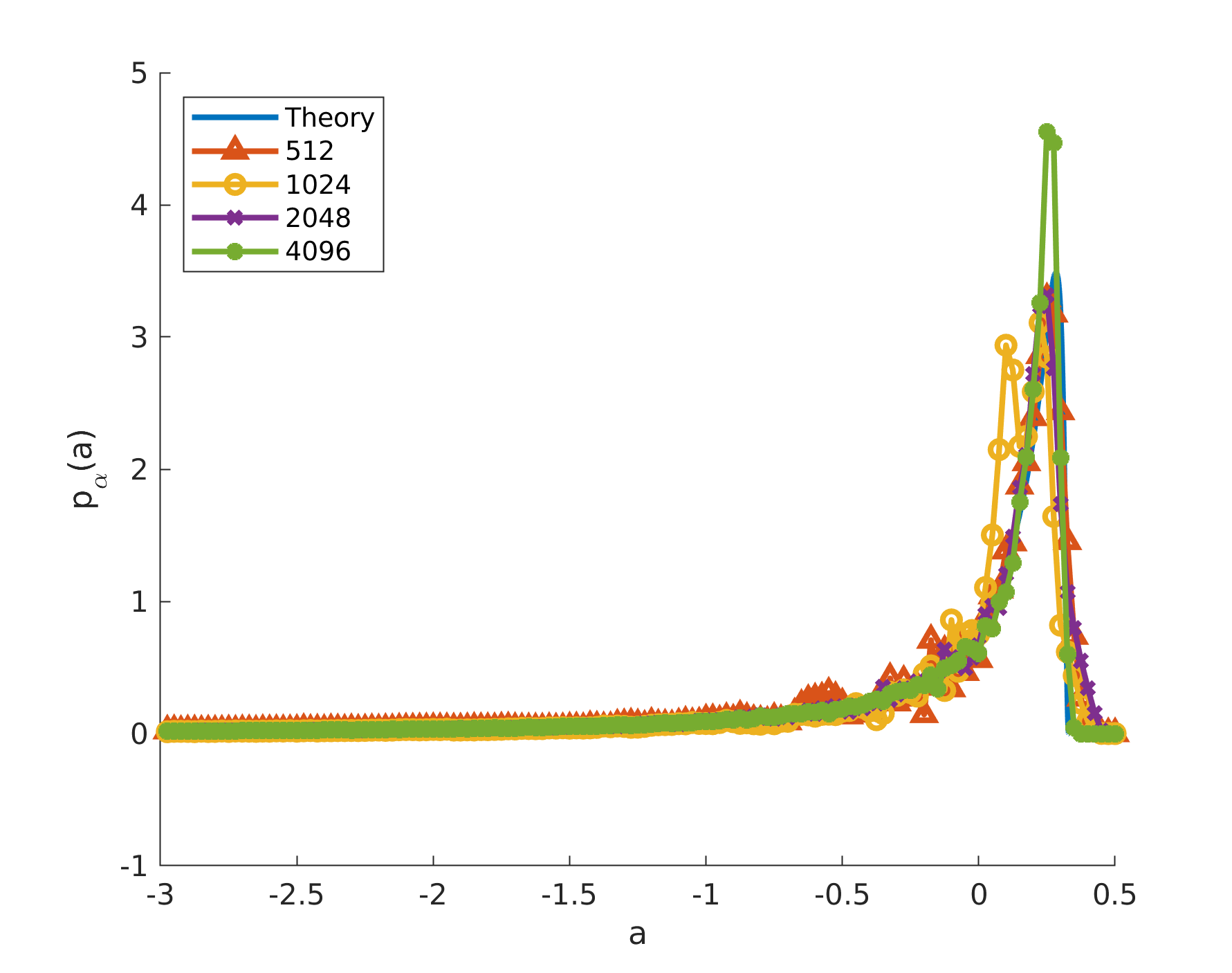}
        \caption{Convergence to theory.}
    \label{fig7:alphaConverge}
\end{subfigure}
    \hfill
\begin{subfigure}[b]{0.49\textwidth}
    \centering
        \includegraphics[width=1.0\textwidth]{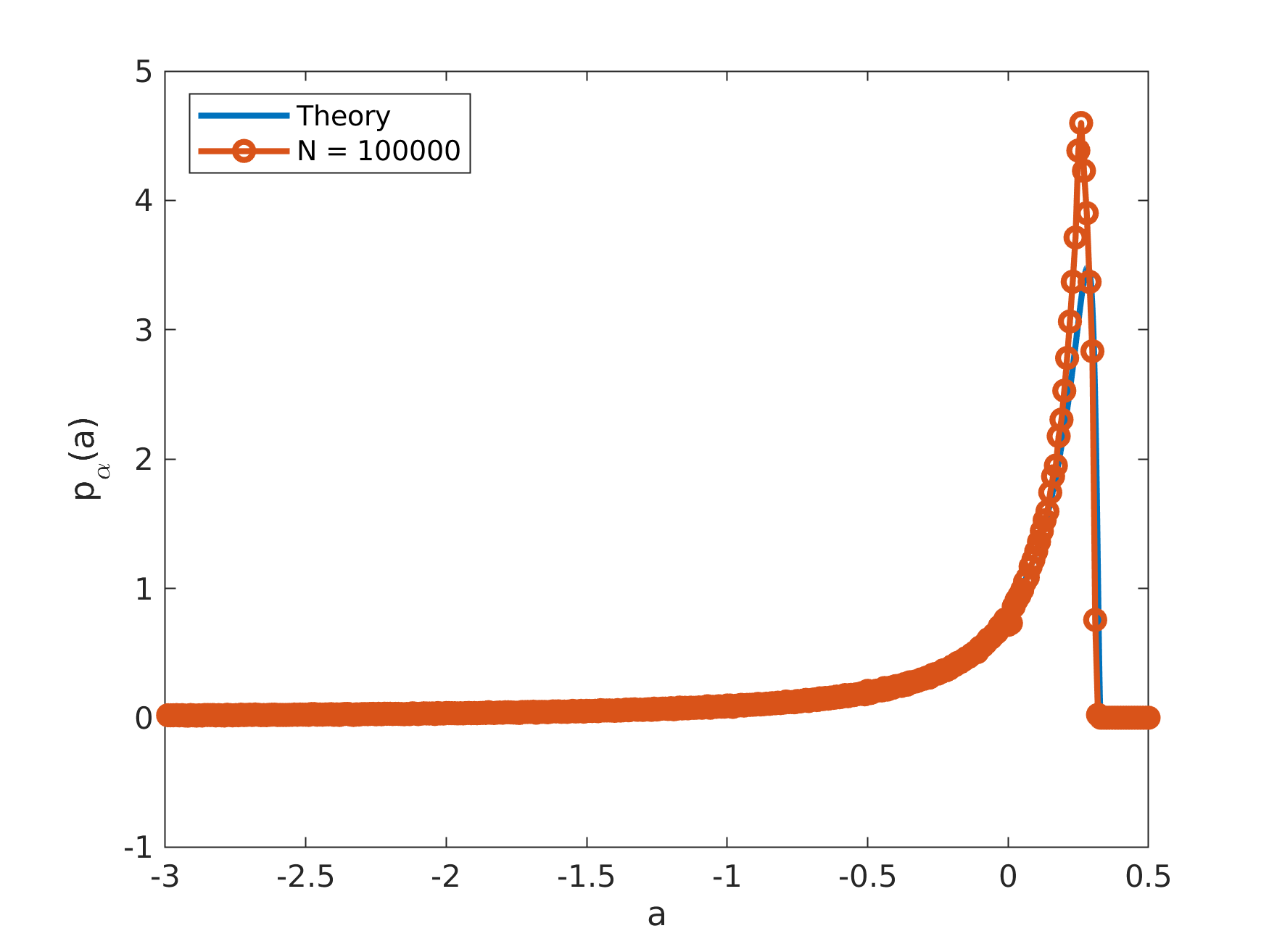}
        \caption{$N = 100,000$ vs. Theory.}
    \label{fig7:n100000}
\end{subfigure}
        \caption{\Acomment{Simulations of domains with different initial values of $N$ and showing how the distributions of the bubble growth rates in each domain do not converge to the theory as $N$ increases due to the choice of initial condition. All distributions normalised to have area of $1$: (a) Plot showing how the numerics become better with increasing $N$, the lack of matching around the peak is a key indication of the finite number of bubbles in the domain as there is not an infinite supply of bubbles. (b) A comparison between a relatively large $N = 100,000$ with theory, a clear improvement can be seen but still there is an overshoot at the peak. This can be attributed to the initial choice of the distribution of droplet radii and its lack of adherence to the weak selection criteria.}}
    \label{fig7:alphaTheory}
\end{figure}

Focusing now the convergence of the numerics, with increasing $N$, to the distribution of droplet growth rates $\alpha$ given in equation~\eqref{eq2:change_of_var_final_final}.
We use equation~\eqref{eq7:alphai_def} to recover $\alpha$ for each bubble in a given system and then extract the distribution of $\alpha$ in these systems by taking a histogram over the time range $10 < t < 40$.
In Figure~\ref{fig7:alphaConverge} we can see the convergence of the histograms to the theoretical distribution given by equation~\eqref{eq2:change_of_var_final_final} as $N$ increases.
The peaks of the distribution slowly converge to the theoretical peak and are still short of the maximum possible values, again due to the finite number of bubbles present in the ODE simulations.
As before we see a good match at smaller growth rates corresponding to the smaller droplets which are less affected by the finite number of droplets in the system.
For definiteness a plot comparing $\alpha$ for $N =100,000$ with theory can be seen in Figure~\ref{fig7:n100000}, a significantly better match can be seen between theory and numerics.
The match is again good for smaller growth rates, but we can see mismatches particularly at the peak.
\Acomment{We can attribute the mismatches seen here again to the choice of the initial drop--size distribution and its lack of conformity with the weak selection criteria.}

We now move to examine $\beta$. 
\Acomment{We have shown} earlier in Section~\ref{sec7:systemEnergy} in a system with infinite $N$ that $p_\beta(\beta,t)=\delta(\beta-(1/3))$ as $t\rightarrow\infty$, we compare this with numerical results for the distribution of $\beta$ in Figure~\ref{fig7:betaVaryN}.
\Acomment{A growth and narrowing of the distribution can be seen as $N$ gets larger appearing to converge slowly to the expected theoretical prediction of a $\delta$ distribution.}
The histograms here are collected across $100$ simulations for each value of $N$, as discussed in Section~\ref{sec7:methodology}, with a sampling window $1 < t < 40$, at larger times the bubble numbers again become too low for the statistics to be included over the range of simulated $N$.
Indeed at very large $t$ one bubble forms in the domain, and $\beta \rightarrow 0$ as there are no bubbles left to consume.
Slow convergence to the predicted $\delta(\beta-(1/3))$ distribution can be seen in Figure~\ref{fig7:betaMax} where the values for $\beta$ where the peaks occur in Figure~\ref{fig7:betaVaryN} are plotted as a function of $N$, even with $\log$ axes the convergence levels out with increasing $N$.
This points to extremely slow convergence to the desired $\delta$ distribution emphasising the effect the finite size of $N$ has on the distributions.

\begin{figure}
    \centering
\begin{subfigure}[b]{0.49\textwidth}
    \centering
        \includegraphics[width=1.0\textwidth]{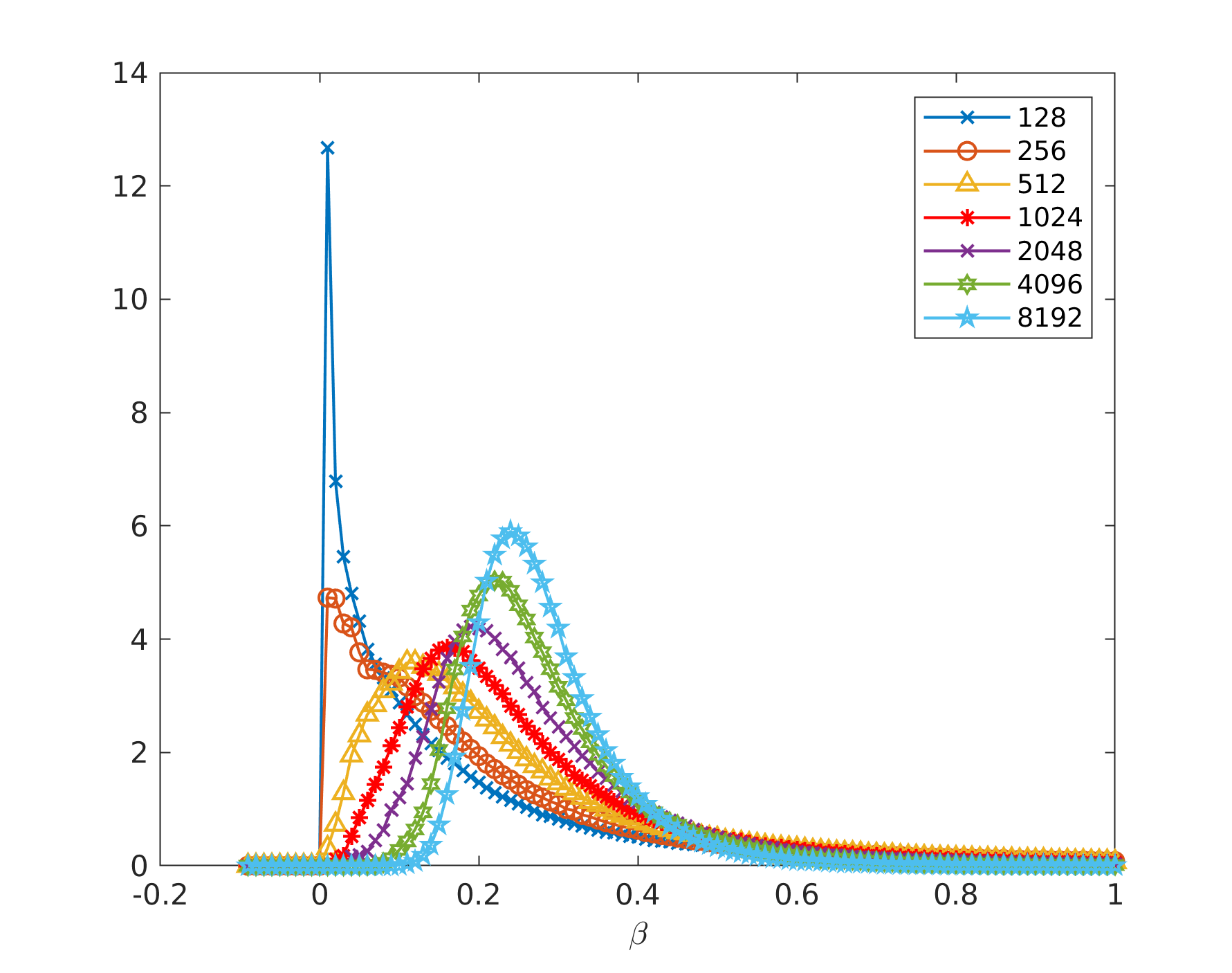}
        \caption{$\beta$ distribution with increasing $N$.}
    \label{fig7:betaVaryN}
\end{subfigure}
    \hfill
\begin{subfigure}[b]{0.49\textwidth}
    \centering
        \includegraphics[width=1.0\textwidth]{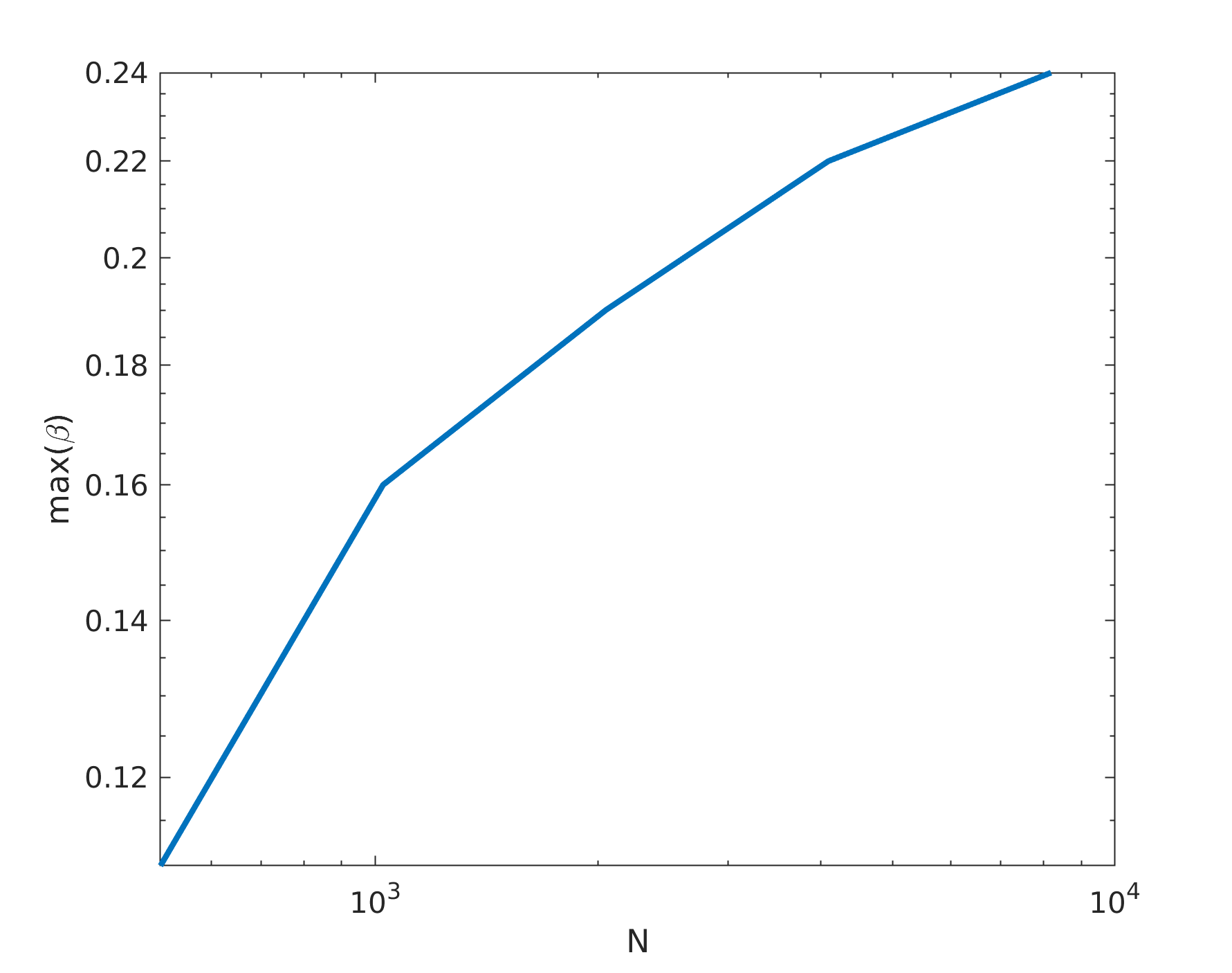}
        \caption{$\max(\beta)$ with increasing $N$.}
    \label{fig7:betaMax}
\end{subfigure}
        \caption{Plots showing how the distribution of $\beta$ develops with increasing $N$, the histograms have been normalised to have area $1$: (a) We can see in this plot that as $N$ increases that the distribution moves to the left and grows in height, in particular the smearing of the distribution due to finite bubble numbers can be seen. (b) Plot showing how the maximum point of the distribution develops with increasing $N$, the levelling out of the growth gives an indication of how slowly the distribution will converge to a $\delta$ function.}
    \label{fig7:betaConverge}
\end{figure}

In summary we have shown in this section that due to the finite number of bubbles in a numerical, and indeed physical, domain leads to a smearing out of the statistics predicted by LSW theory. 
In particular the radii of the bubbles do not grow to the largest size possible as seen in Figure~\ref{fig7:bubbleRad} due to the lack of available bubbles to consume.
The distributions of smaller bubbles match well with the theory, these are unaffected by finite size effects.
The growth rates for these bubbles also match well as shown in Figure~\ref{fig7:alphaTheory}.
Finally a slow convergence to the true growth rate of the free energy was shown in Figure~\ref{fig7:betaConverge}.

We will now focus in particular on the free energy growth rate as it is here we can make connections with the Cahn--Hilliard equation, it is a common quantity we can measure and compare to the LSW theory.
Instead of a finite number of bubbles though we will have a finite domain size.
The key feature of comparison will be how the finite size of the domain smears the distribution of $\beta$ away from a $\delta$ function into a strictly positive, skewed and statistically stationary distribution of possible values around a mean.
\Acomment{It should be emphasised that while the droplet dynamics we have discussed here for LSW theory are for a three dimensional system and the underlying connections to the Cahn--Hilliard equation are defined in three dimensions there exists an analogous quantitative theory for two dimensions~\cite{rogers1989numerical}. 
The outcome of that quantitative theory is again a self--similar distribution whose form is similar to that explored above.  
Therefore, the results we have discussed here for three dimensions will carry over in a suitable qualitative sense to asymmetric mixtures presented in the following section.}

We begin in the next section with asymmetric mixtures of the fluids in the Cahn--Hilliard equation which are of direct applicability as Ostwald Ripening occurs in these cases, in the section following this we then present a statistical model for symmetric mixtures and their distributions which can then be compared with simulation results.
In both cases we relate our findings back to those in this section showing how finite domain/number effects smear the theoretical distributions.


\begin{figure}
    \centering
    \begin{subfigure}[b]{0.49\textwidth}
        \centering
        \includegraphics[width=1.0\textwidth]{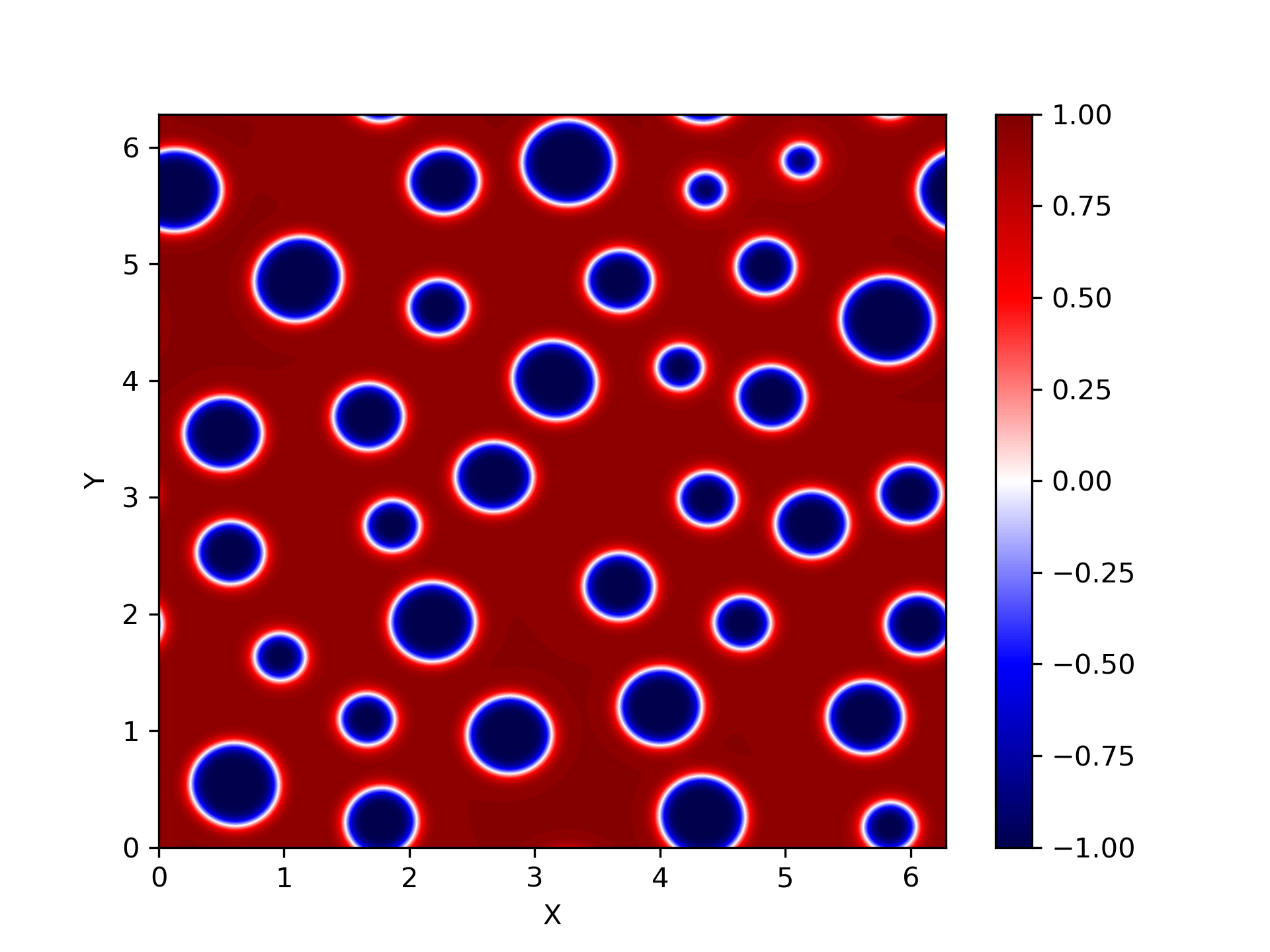}
    \caption{$t = 6.1359231515$}
    \label{fig7:contour24oswald}
    \end{subfigure}
    \hfill
    \begin{subfigure}[b]{0.49\textwidth}  
        \centering
        \includegraphics[width=1.0\textwidth]{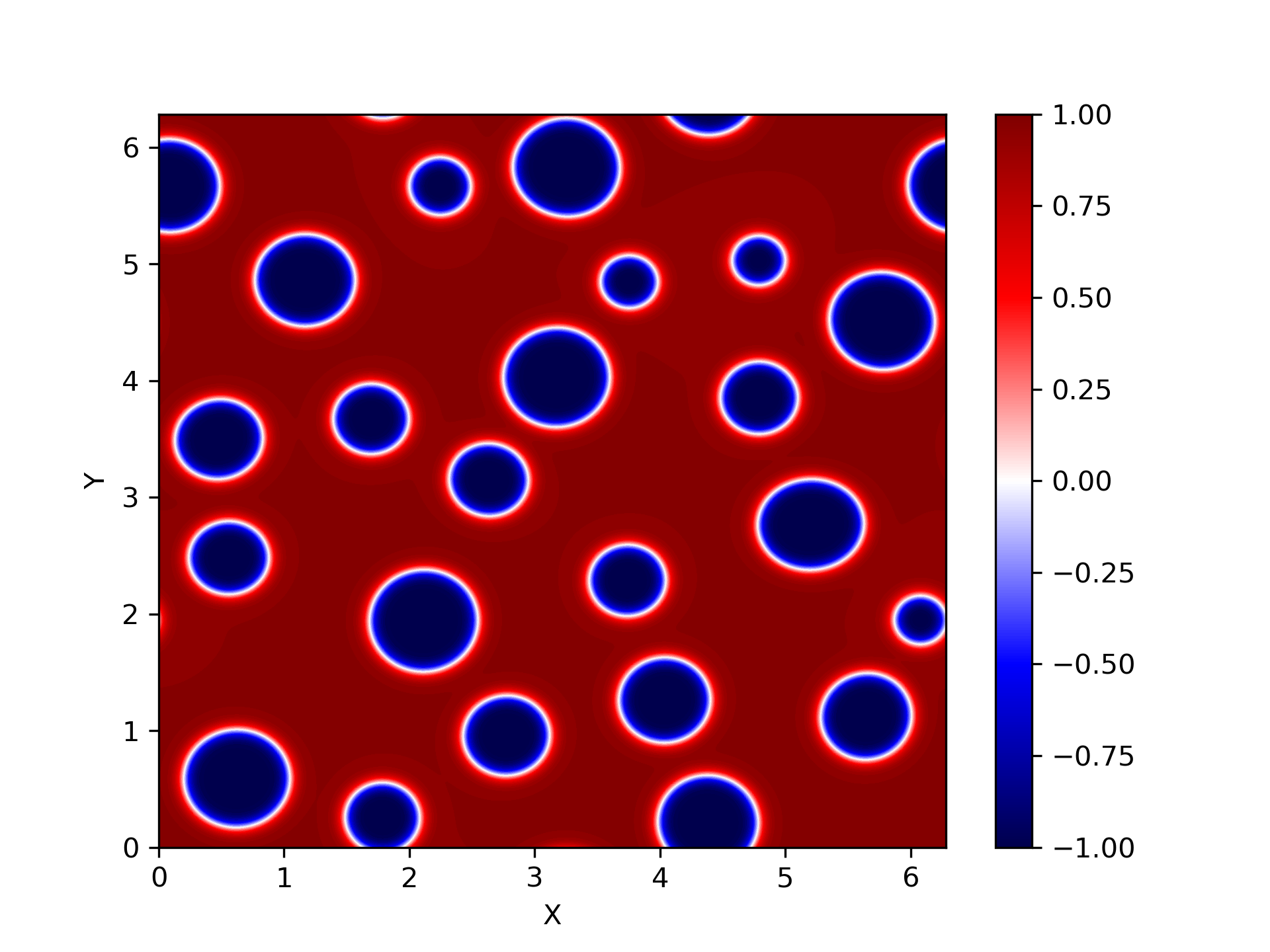}
    \caption{$t = 12.5172832291$}
    \label{fig7:contour50oswald}
    \end{subfigure}
    \vskip\baselineskip
    \begin{subfigure}[b]{0.49\textwidth}   
        \centering
        \includegraphics[width=1.0\textwidth]{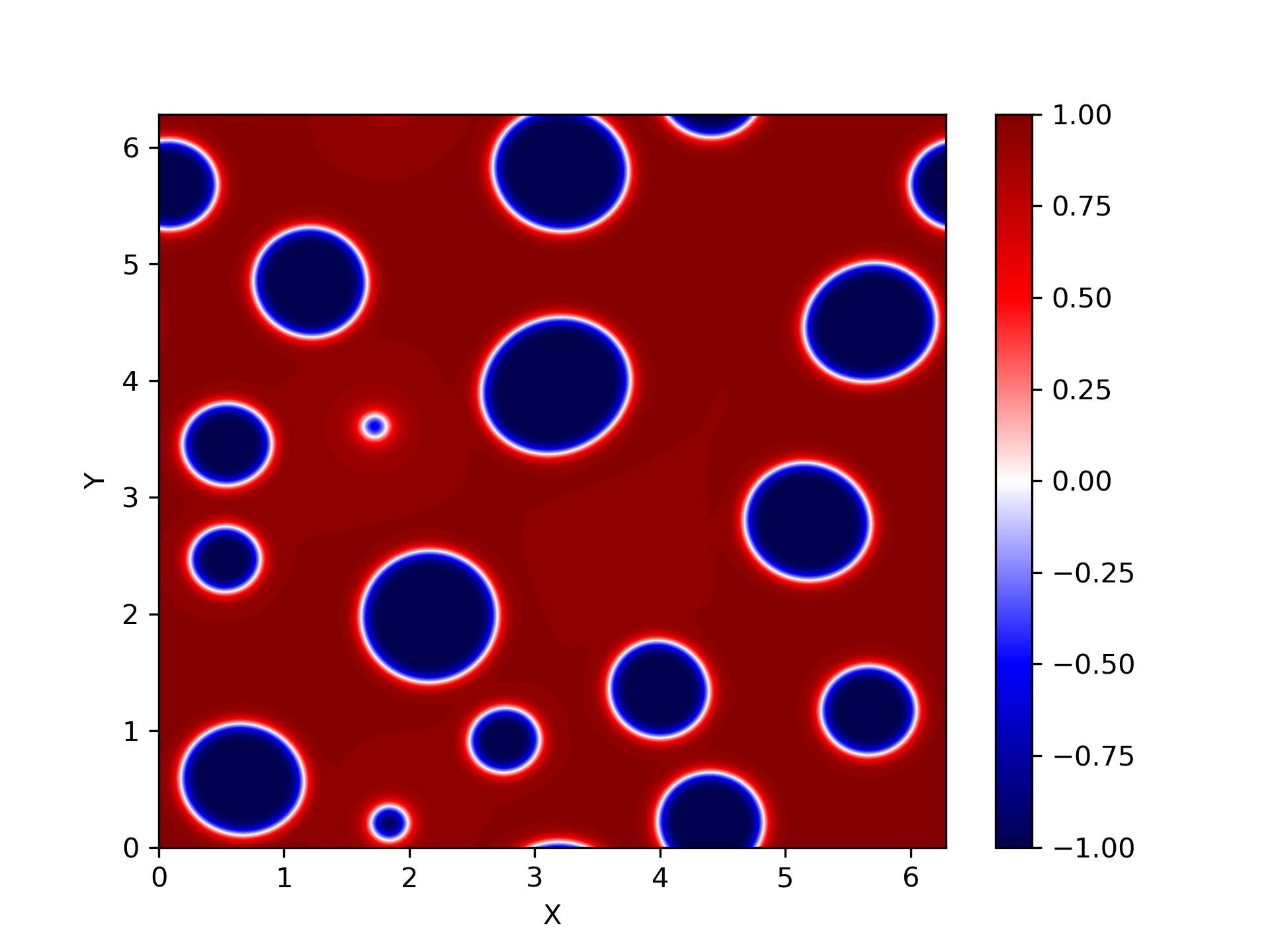}
    \caption{$t = 24.7891295322$}
    \label{fig7:contour100oswald}
    \end{subfigure}
    \begin{subfigure}[b]{0.49\textwidth}   
        \centering
        \includegraphics[width=1.0\textwidth]{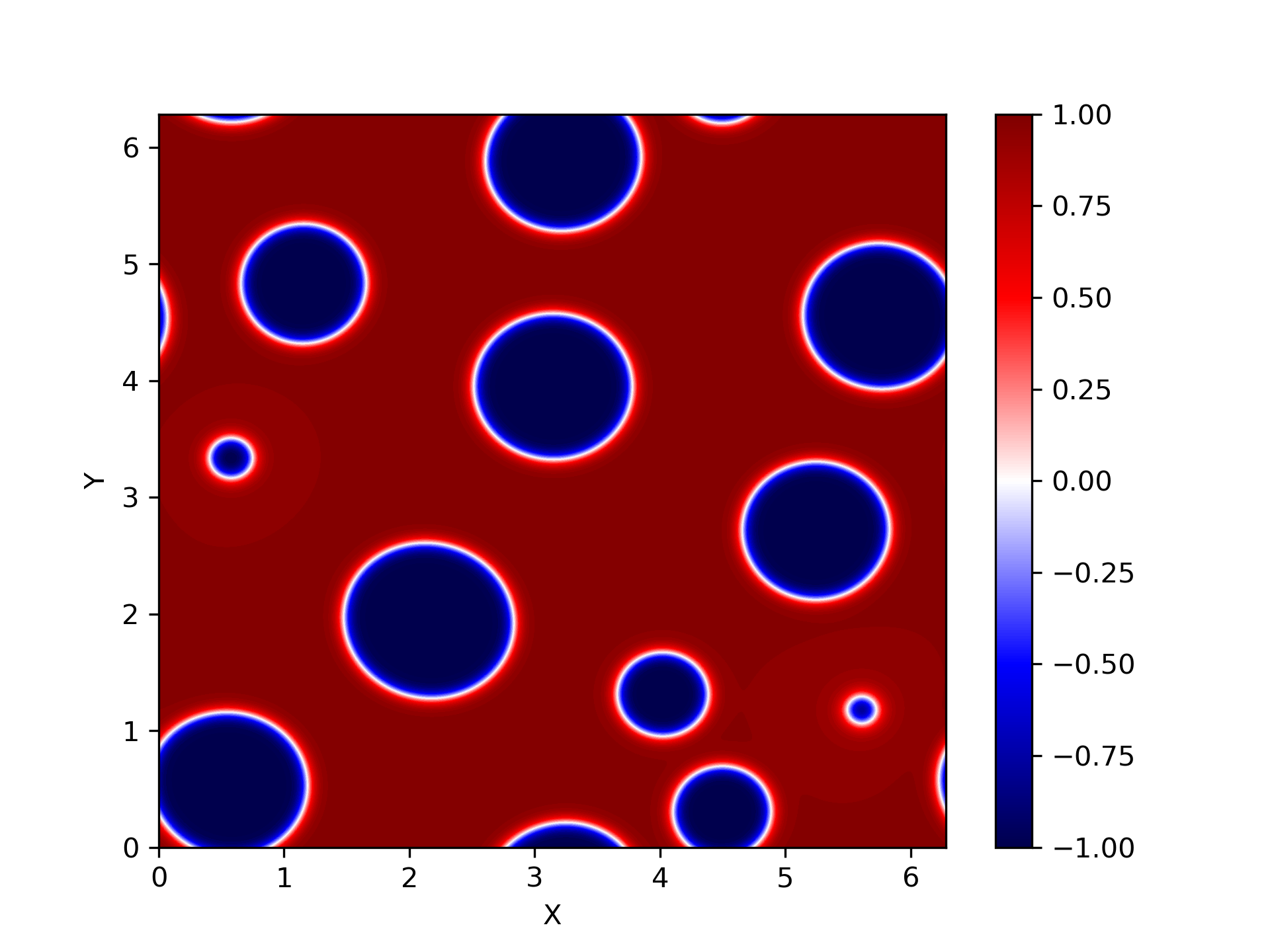}
    \caption{$t = 49.3328221384$}
    \label{fig7:contour200oswald}
    \end{subfigure}
        \caption{Contour plots of the Cahn--Hilliard equation with asymmetric mixture showing Ostwald Ripening on a domain of size $4\pi \times 4\pi$.}
    \label{fig7:oswaldRipening}
\end{figure}

\section{Cahn--Hilliard Equation -- Asymmetric Mixtures}
\label{sec7:CH}
In this section we apply the second methodology discussed in Section~\ref{sec7:methodology}, here we solve batches of the Cahn--Hilliard equation using the ADI method on a GPU, each simulation is initialised with a randomised asymmetric mixture.
In order to achieve an asymmetric mixture we initialise the simulations using a uniform distribution of values between $0.4$ and $0.6$, an example evolution of one such simulation showing the presence of Ostwald Ripening is shown in Figure~\ref{fig7:oswaldRipening}.
It can be seen in this illustrative example how the domain develops from a series of small bubbles into just a few large bubbles, eventually the domain will be characterised by the presence of just one bubble which will cease growing and thus the system will enter a steady state, the size of the final bubble is dictated by the finite domain in which it is contained \Acomment{and the initial distribution.}

\begin{figure}
    \centering
    \begin{subfigure}[b]{0.49\textwidth}
        \centering
        \includegraphics[width=1.0\textwidth]{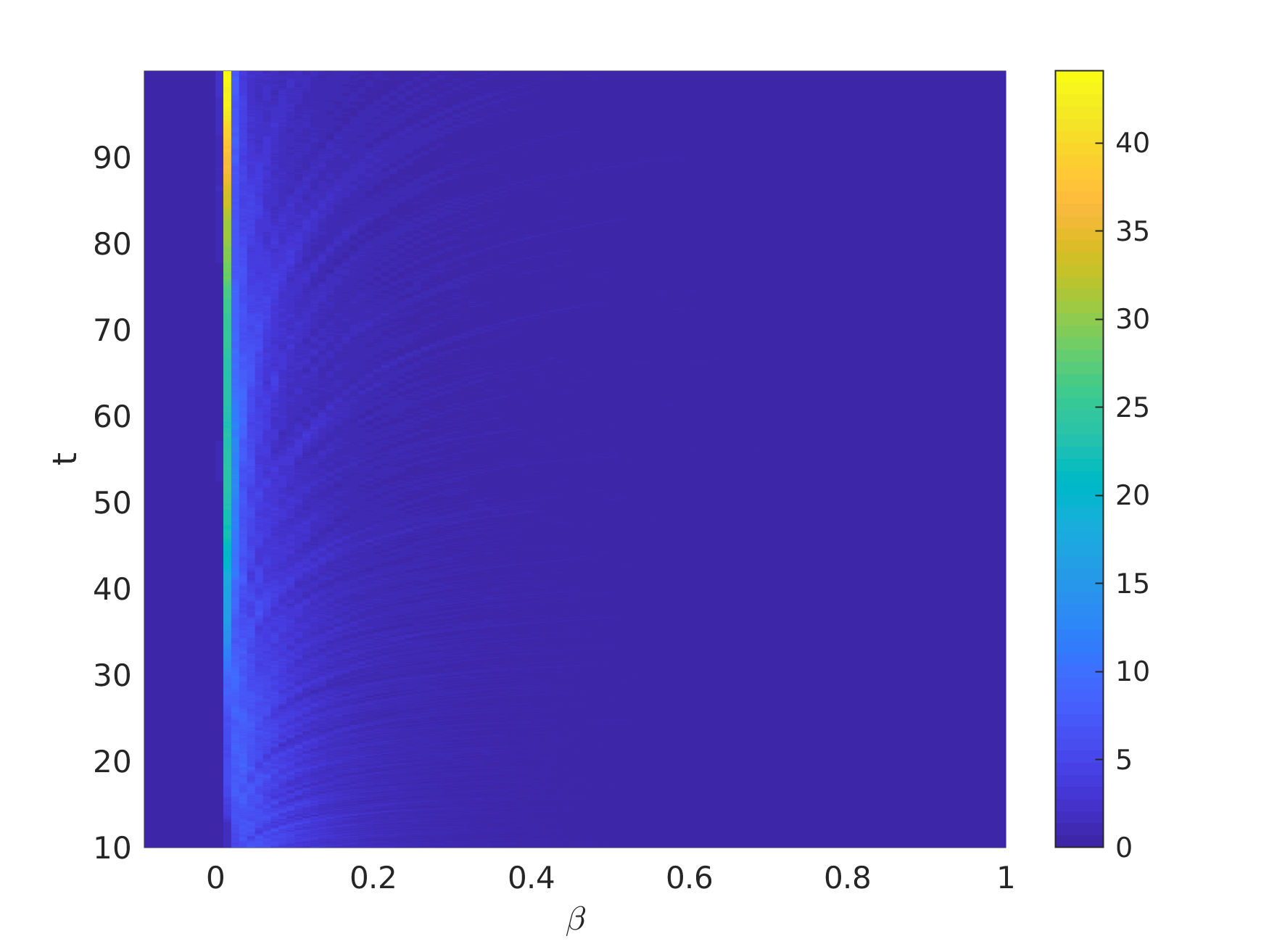}
    \caption{$2\pi \times 2\pi$}
    \label{fig7:spacetime2pi256Avg}
    \end{subfigure}
    \hfill
    \begin{subfigure}[b]{0.49\textwidth}
        \centering
        \includegraphics[width=1.0\textwidth]{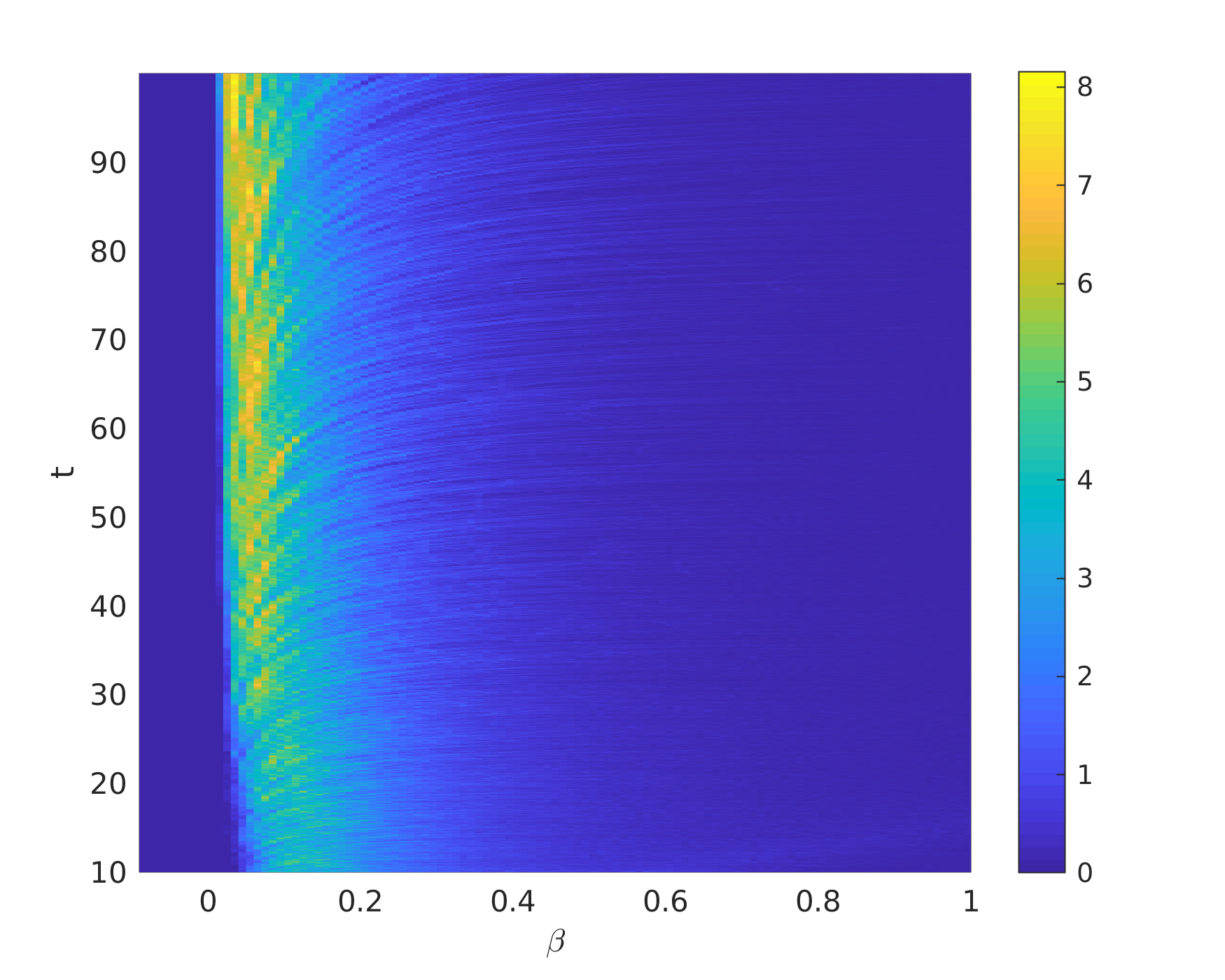}
    \caption{$4\pi \times 4\pi$}
    \label{fig7:spacetime4pi512Avg}
    \end{subfigure}
    \vskip\baselineskip
    \begin{subfigure}[b]{0.49\textwidth}
        \centering
        \includegraphics[width=1.0\textwidth]{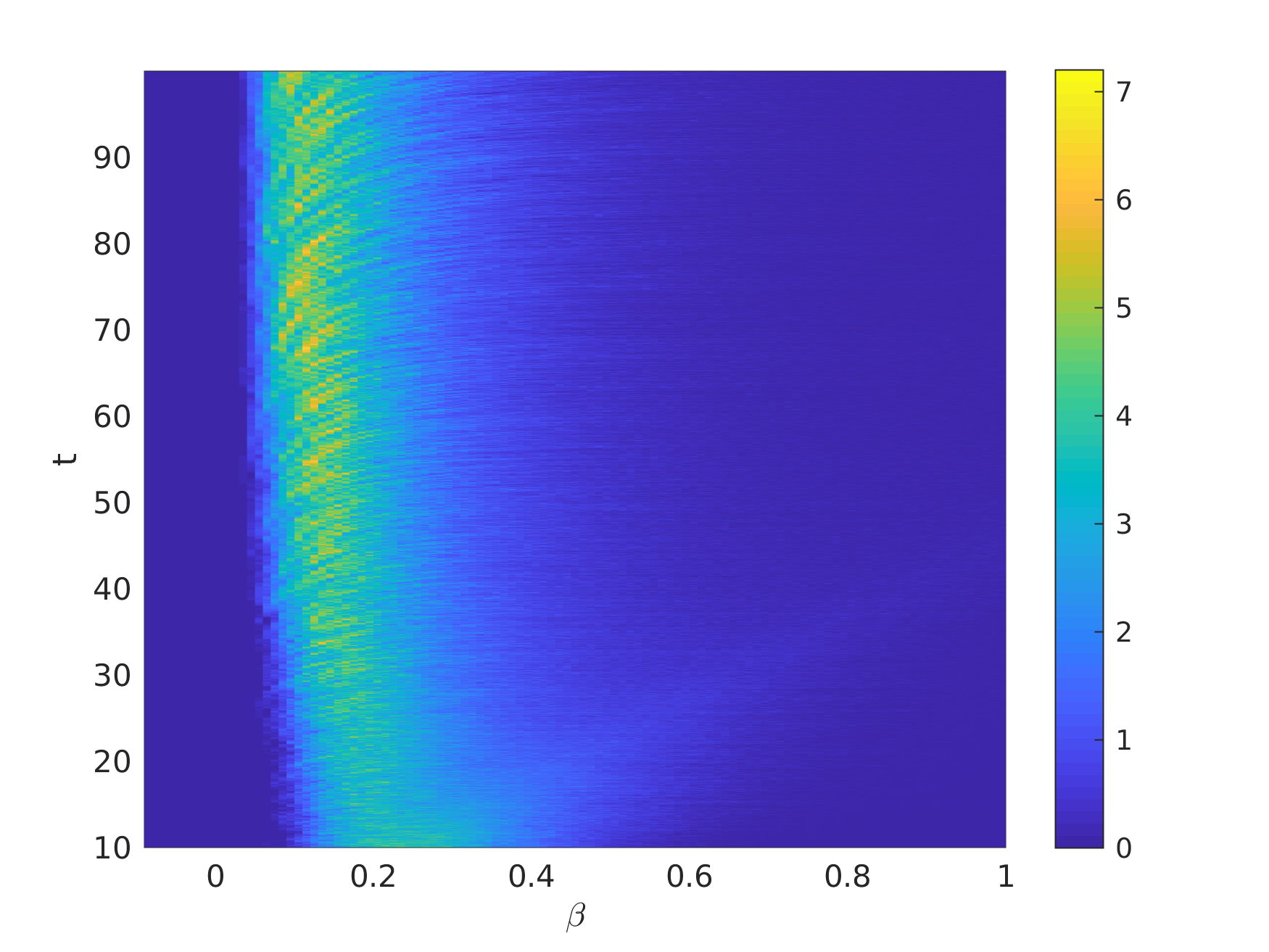}
    \caption{$8\pi \times 8\pi$}
    \label{fig7:spacetime8pi1024Avg}
    \end{subfigure}
    \begin{subfigure}[b]{0.49\textwidth}
        \centering
        \includegraphics[width=1.0\textwidth]{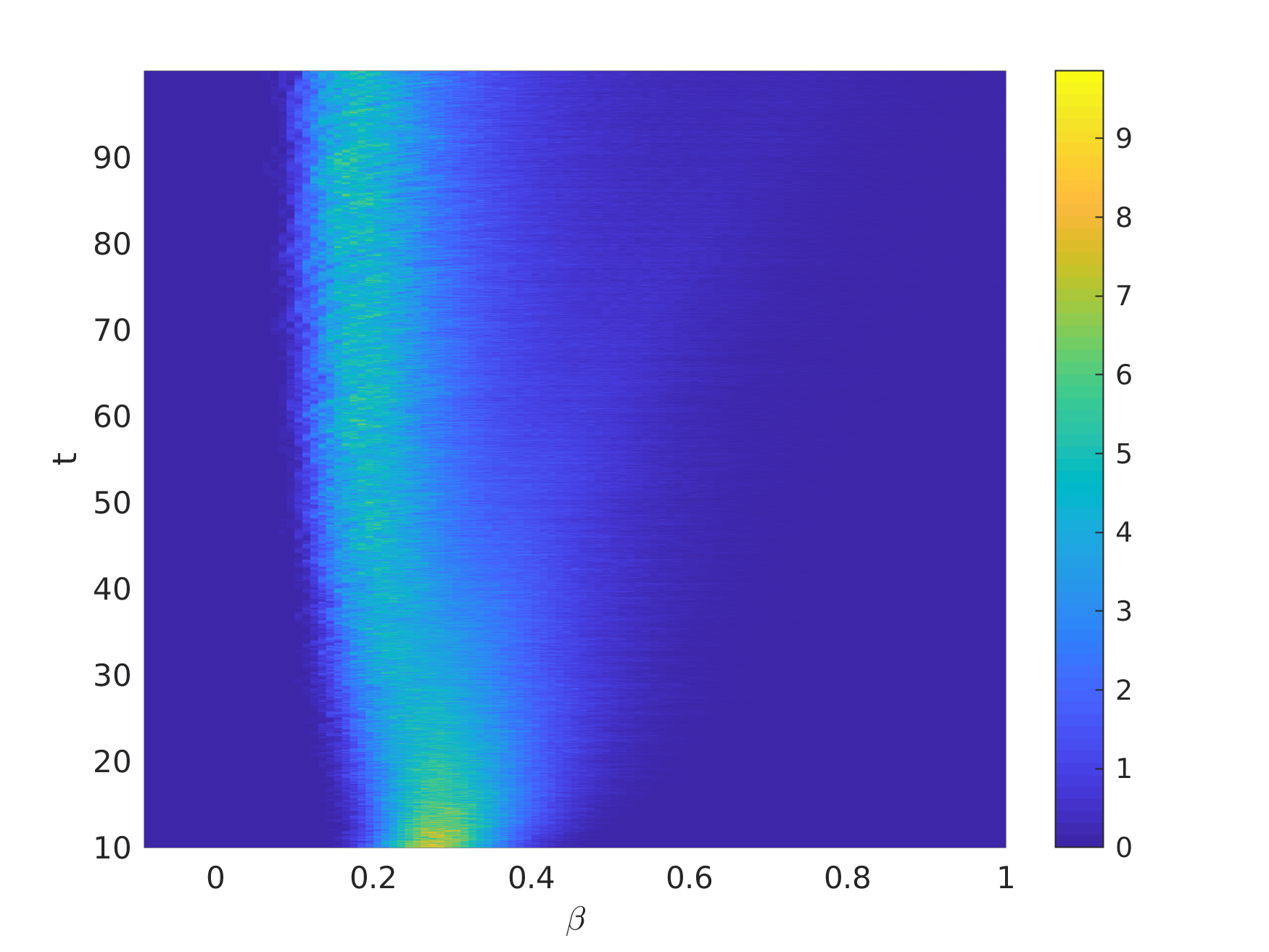}
    \caption{$16\pi \times 16\pi$}
    \label{fig7:spacetime16pi2048Avg}
    \end{subfigure}
        \caption{\Acomment{Distributions of $\beta$ as a function of time for the Cahn--Hilliard equation with an asymmetric mixture over varying domain sizes.}}
    \label{fig7:spacetimeAsymmetric}
\end{figure}

\Acomment{We begin our study of the batch results for an asymmetric mixture by examining the plots of the distribution of $\beta$ as a function of time, these results can be seen in Figure~\ref{fig7:spacetimeAsymmetric}.}
For the $2\pi \times 2\pi$ domain shown in Figure~\ref{fig7:spacetime2pi256Avg} the presence of finite size effects on the domain can clearly be seen, as time progresses most of the simulations for this case have reached a steady state, dominated by a few large bubbles which has ceased growing and thus the distribution is centred around $0$ almost entirely.
As the domain size increases we can see the development of steady states around values of $\beta$ that are non--zero, indeed in the $16\pi \times 16\pi$ domain shown in Figure~\ref{fig7:spacetimeAsymmetric} we can see how the distribution is beginning to centre itself closer to $1/3$.
Based on the development of these results we can expect to see this trend of the distribution moving towards $1/3$ continue as the domain size is increased further.

\begin{figure}
    \centering
    \begin{subfigure}[b]{0.49\textwidth}
        \centering
        \includegraphics[width=1.0\textwidth]{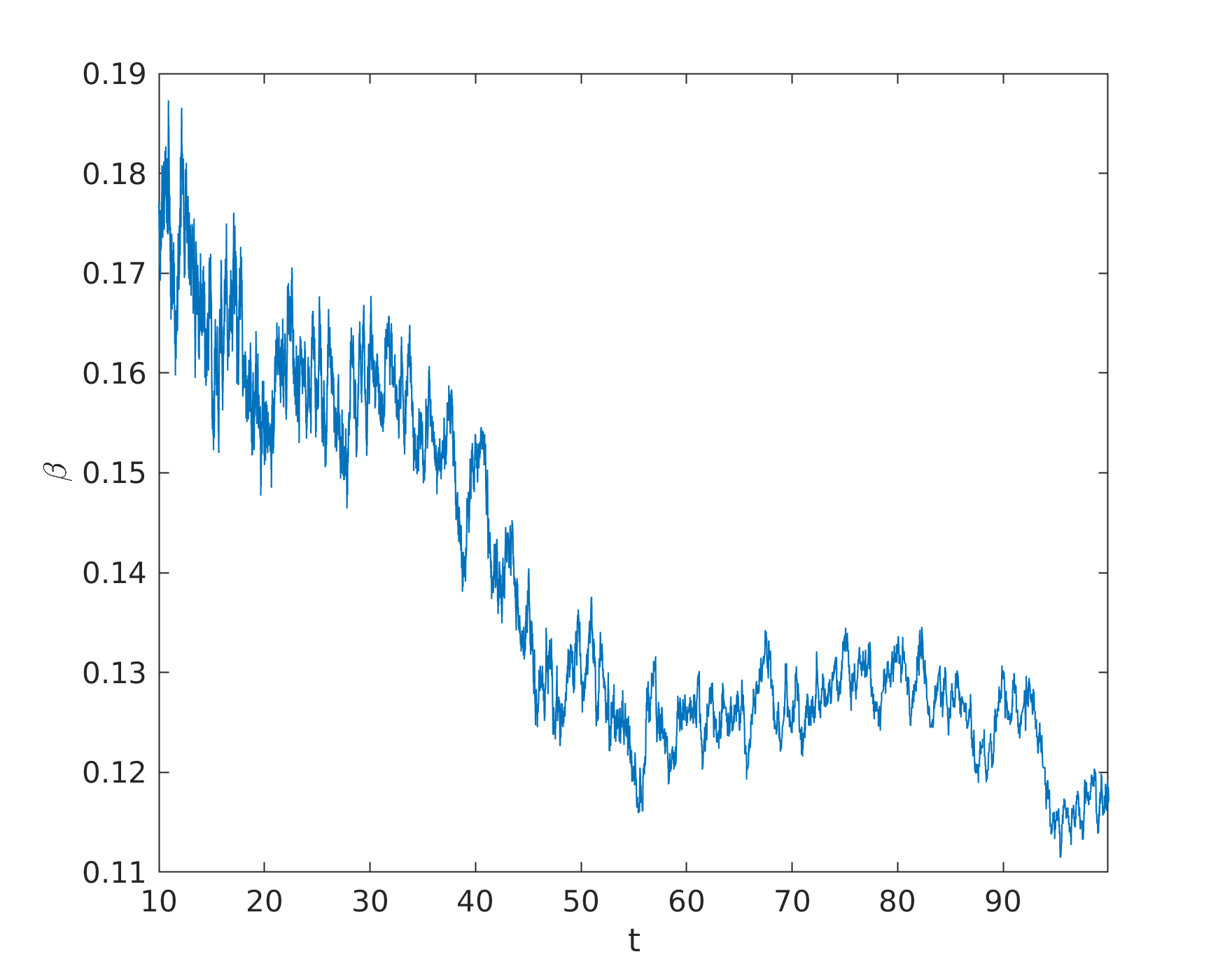}
    \caption{$2\pi \times 2\pi$}
    \label{fig7:mean2pi256Avg}
    \end{subfigure}
    \hfill
    \begin{subfigure}[b]{0.49\textwidth}
        \centering
        \includegraphics[width=1.0\textwidth]{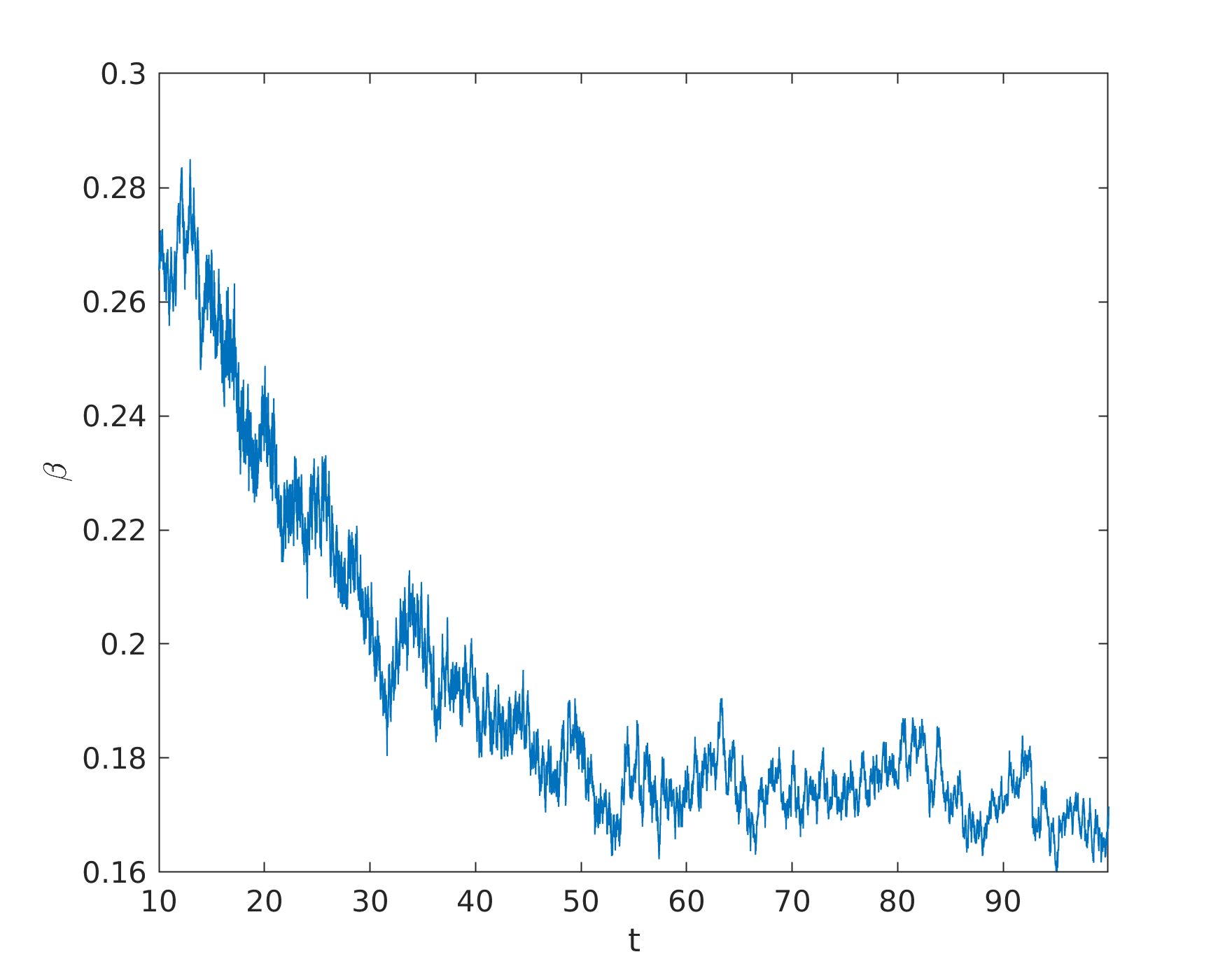}
    \caption{$4\pi \times 4\pi$}
    \label{fig7:mean4pi512Avg}
    \end{subfigure}
    \vskip\baselineskip
    \begin{subfigure}[b]{0.49\textwidth}
        \centering
        \includegraphics[width=1.0\textwidth]{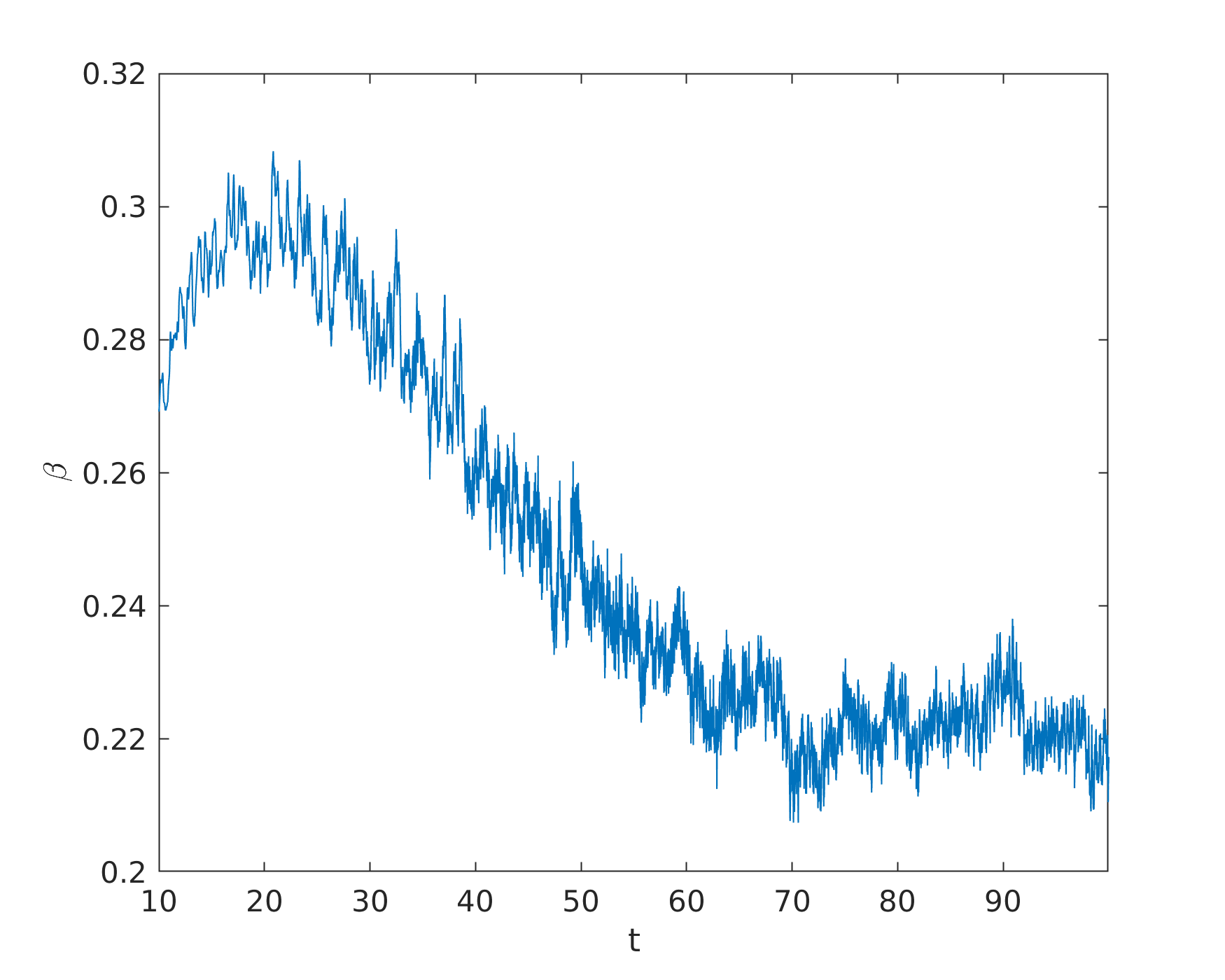}
    \caption{$8\pi \times 8\pi$}
    \label{fig7:mean8pi1024Avg}
    \end{subfigure}
    \begin{subfigure}[b]{0.49\textwidth}
        \centering
        \includegraphics[width=1.0\textwidth]{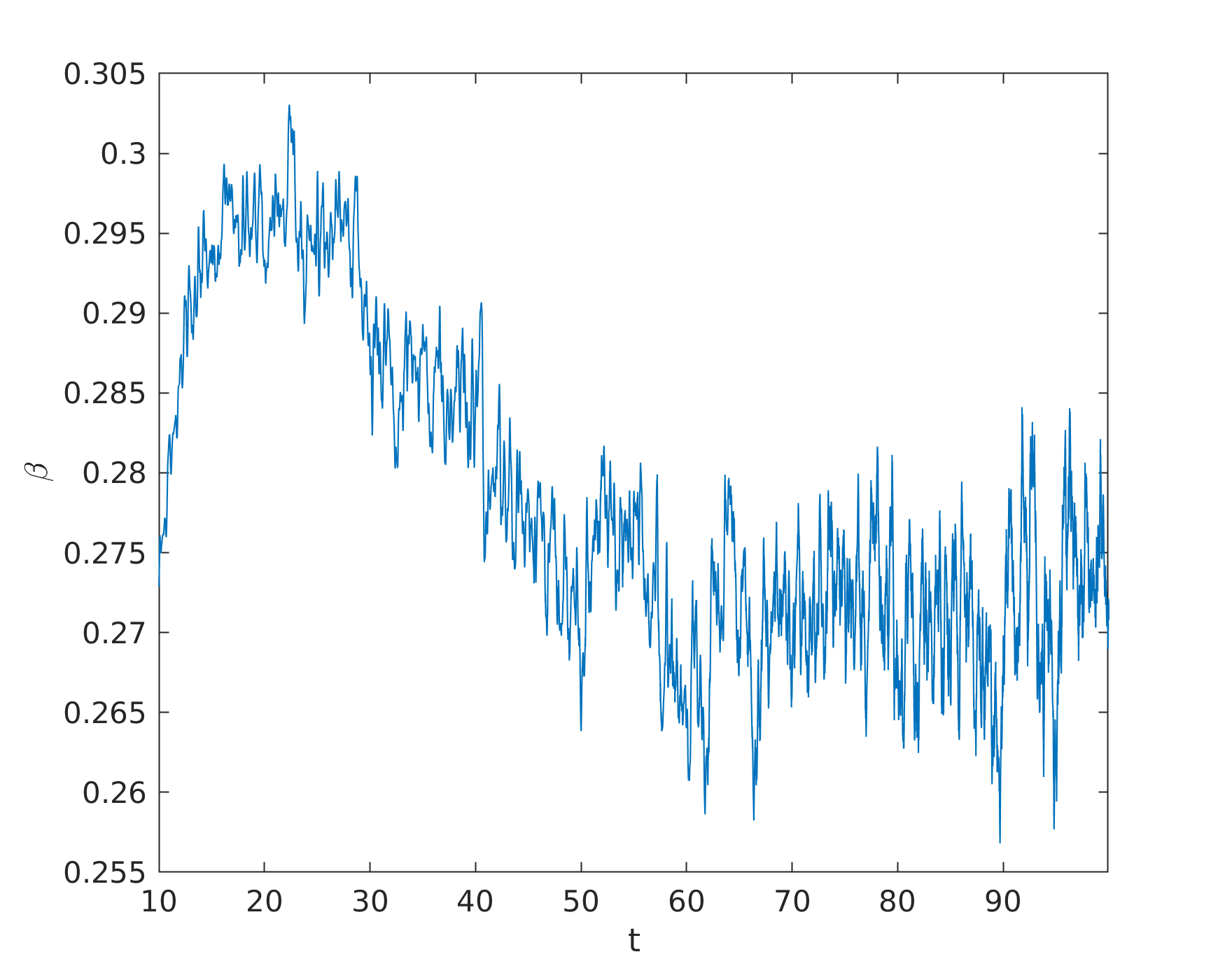}
    \caption{$16\pi \times 16\pi$}
    \label{fig7:mean16pi2048Avg}
    \end{subfigure}
        \caption{Plots of the mean value of $\beta$ as a function of time for the Cahn--Hilliard equation with an asymmetric mixture with varying domain sizes.}
    \label{fig7:meanAsymmetric}
\end{figure}

In the next set of Figures~\ref{fig7:meanAsymmetric} we can see the mean values of $\beta$ as a function of time.
Again the presence in each of finite size effects can be seen as the values fall below the LSW prediction of $1/3$, there are not enough bubbles present in the domain to drive the average growth rate to its possible value of $1/3$.
The finite domain size effect on the results is again seen in the $2\pi \times 2\pi$ domain as the mean of the distribution drops significantly to values around $0$, reinforcing how the bubbles in these domains have stopped growing due to the finite numbers present in the finite domain. 
In all cases we can see the development of a steady state at late times as the mean settles into an oscillation around a given value.
We can also see that the mean value around which this steady state is occurring is increasing with domain size towards the predicted $1/3$ value for an infinite domain.

\begin{figure}
    \centering
    \begin{subfigure}[b]{0.49\textwidth}
        \centering
        \includegraphics[width=1.0\textwidth]{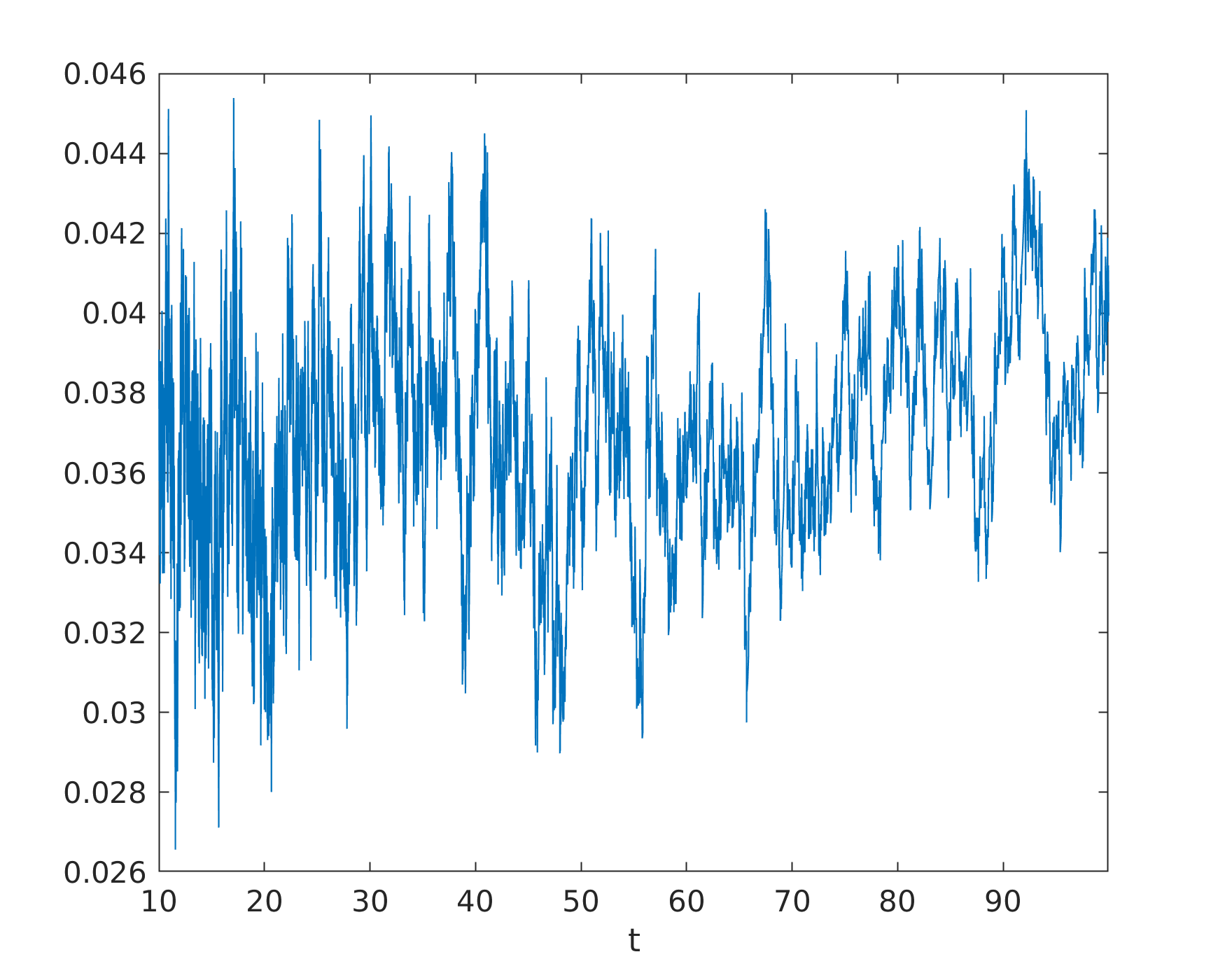}
    \caption{$2\pi \times 2\pi$}
    \label{fig7:variance2pi256Avg}
    \end{subfigure}
    \hfill
    \begin{subfigure}[b]{0.49\textwidth}
        \centering
        \includegraphics[width=1.0\textwidth]{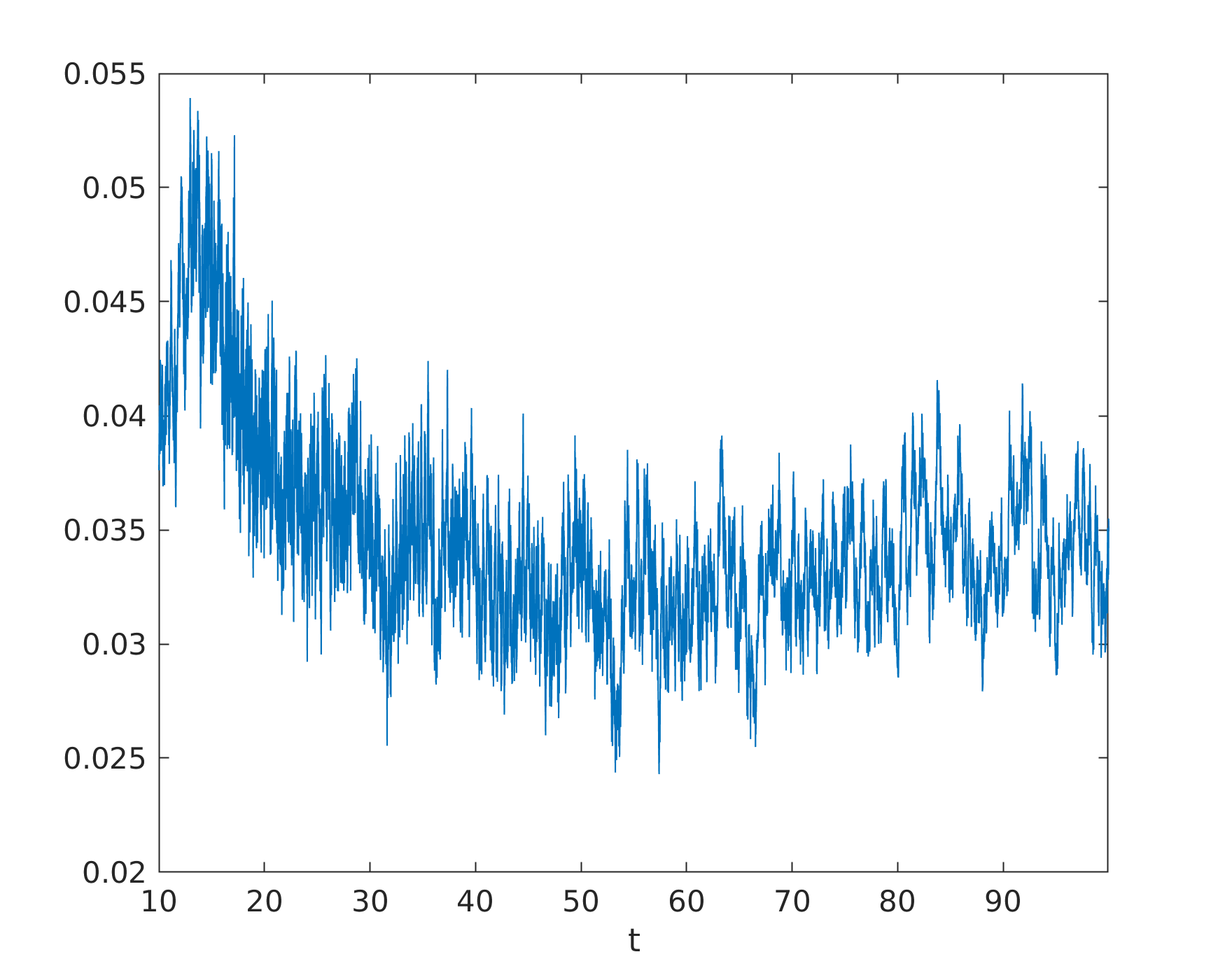}
    \caption{$4\pi \times 4\pi$}
    \label{fig7:variance4pi512Avg}
    \end{subfigure}
    \vskip\baselineskip
    \begin{subfigure}[b]{0.49\textwidth}
        \centering
        \includegraphics[width=1.0\textwidth]{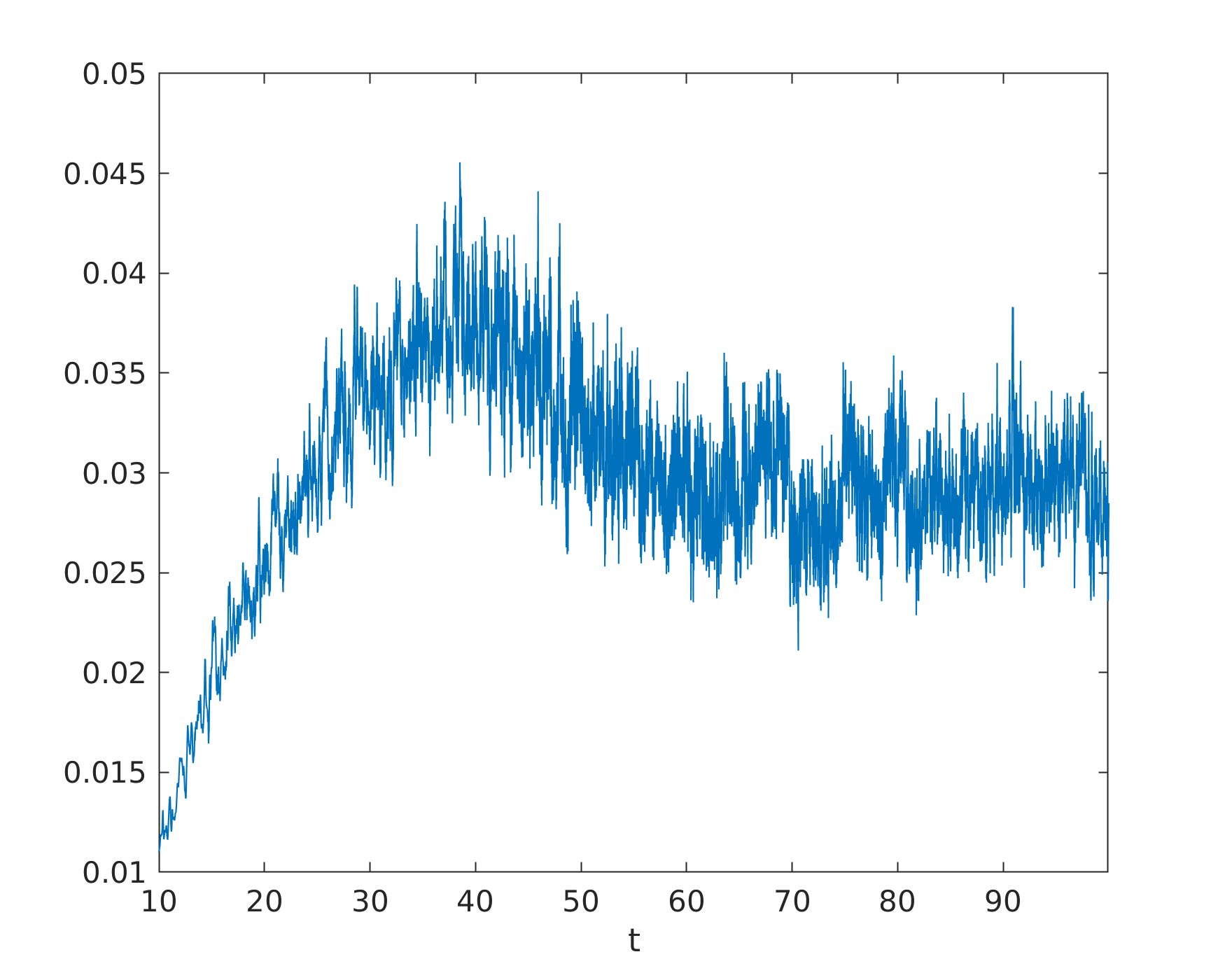}
    \caption{$8\pi \times 8\pi$}
    \label{fig7:variance8pi1024Avg}
    \end{subfigure}
    \begin{subfigure}[b]{0.49\textwidth}
        \centering
        \includegraphics[width=1.0\textwidth]{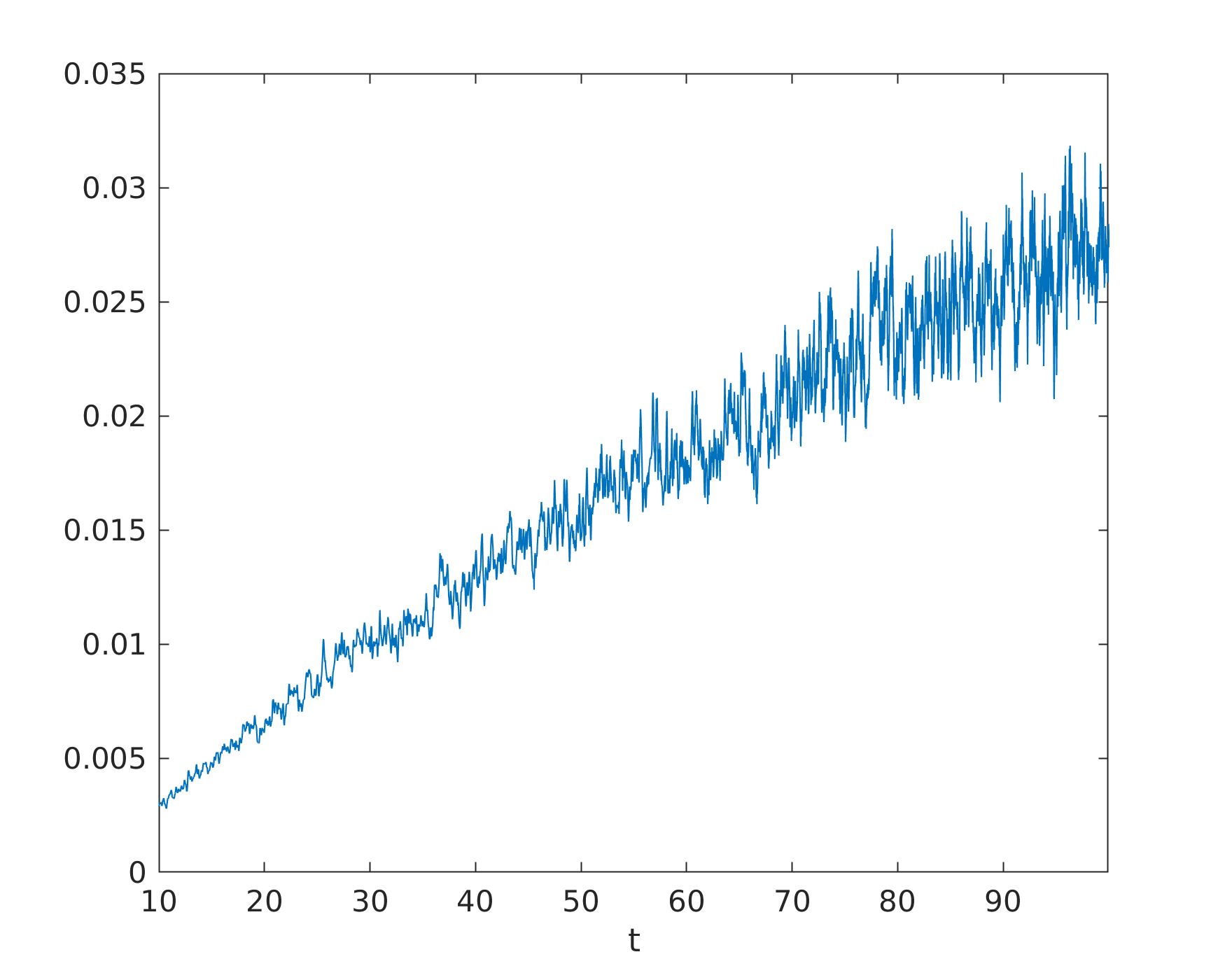}
    \caption{$16\pi \times 16\pi$}
    \label{fig7:variance16pi2048Avg}
    \end{subfigure}
        \caption{Plots of the variance of $\beta$ as a function of time for the Cahn--Hilliard equation with an asymmetric mixture over varying domain sizes.}
    \label{fig7:varianceAsymmetric}
\end{figure}

The variance of $\beta$ as a function of time for different domain sizes is shown in Figure~\ref{fig7:varianceAsymmetric}.
For the case of $2\pi \times 2\pi$ shown in Figure~\ref{fig7:variance2pi256Avg} we can see the presence of a steady state as the value oscillates around $\approx 0.37$, this is in--keeping with the steady position of the distribution in the \Acomment{distribution--time} plot shown previously in Figure~\ref{fig7:spacetime2pi256Avg}
For domains of size $4\pi \times 4\pi$ and $8\pi \times 8\pi$ we can see variance settling to statistically stationary values at later times while for a $16\pi \times 16\pi$ domain we can see that the variance is still growing which then, based on the smaller domain sizes, will settle to a steady state at a later time.
The time dependence of the variance in all cases, certainly at early times, suggests that the diffusion here is anomalous before settling down to a constant rate of diffusion during the statically stationary portion of the domain evolution.

\begin{figure}
    \centering
    \begin{subfigure}[b]{0.49\textwidth}
        \centering
        \includegraphics[width=1.0\textwidth]{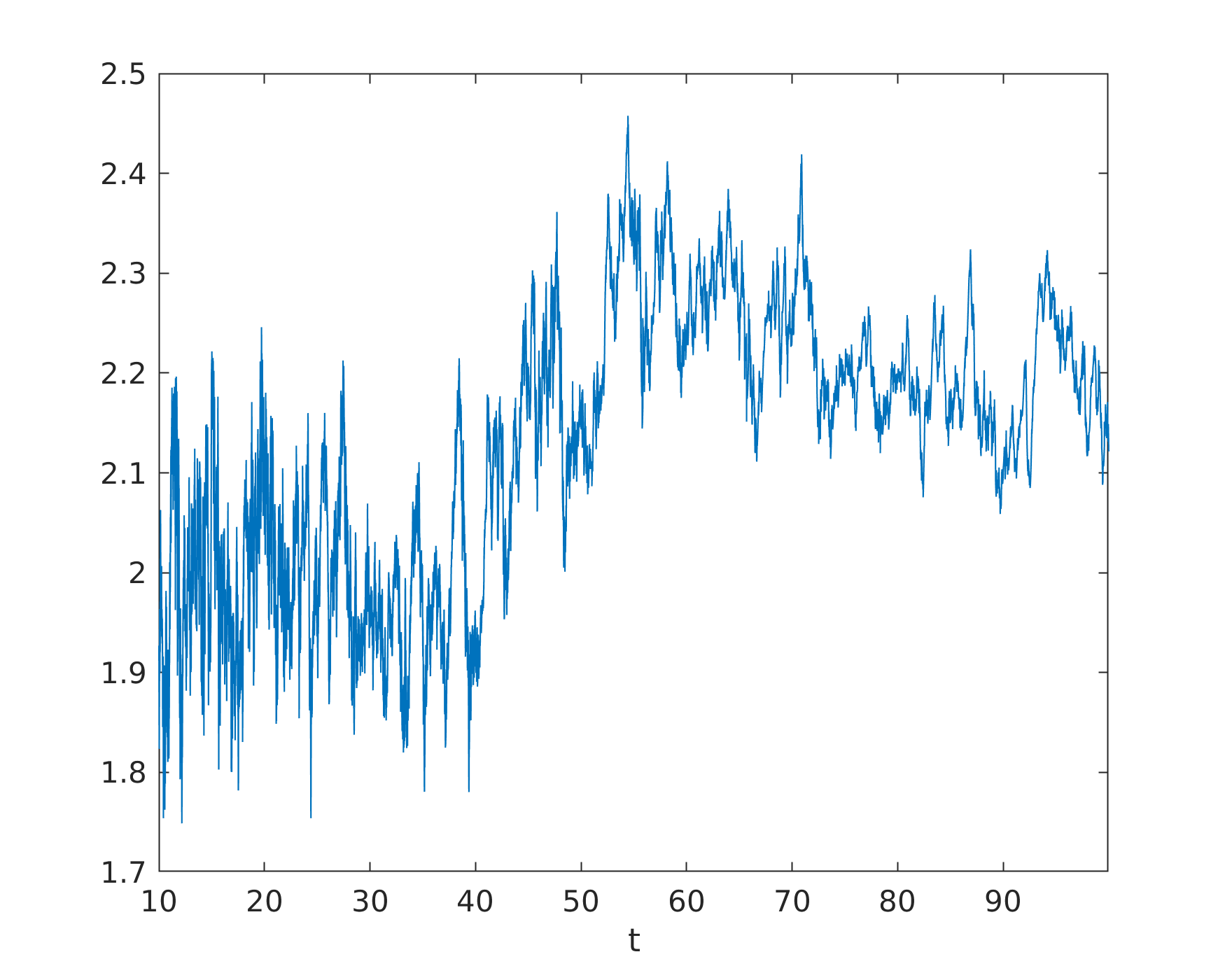}
    \caption{$2\pi \times 2\pi$}
    \label{fig7:skew2pi256Avg}
    \end{subfigure}
    \hfill
    \begin{subfigure}[b]{0.49\textwidth}
        \centering
        \includegraphics[width=1.0\textwidth]{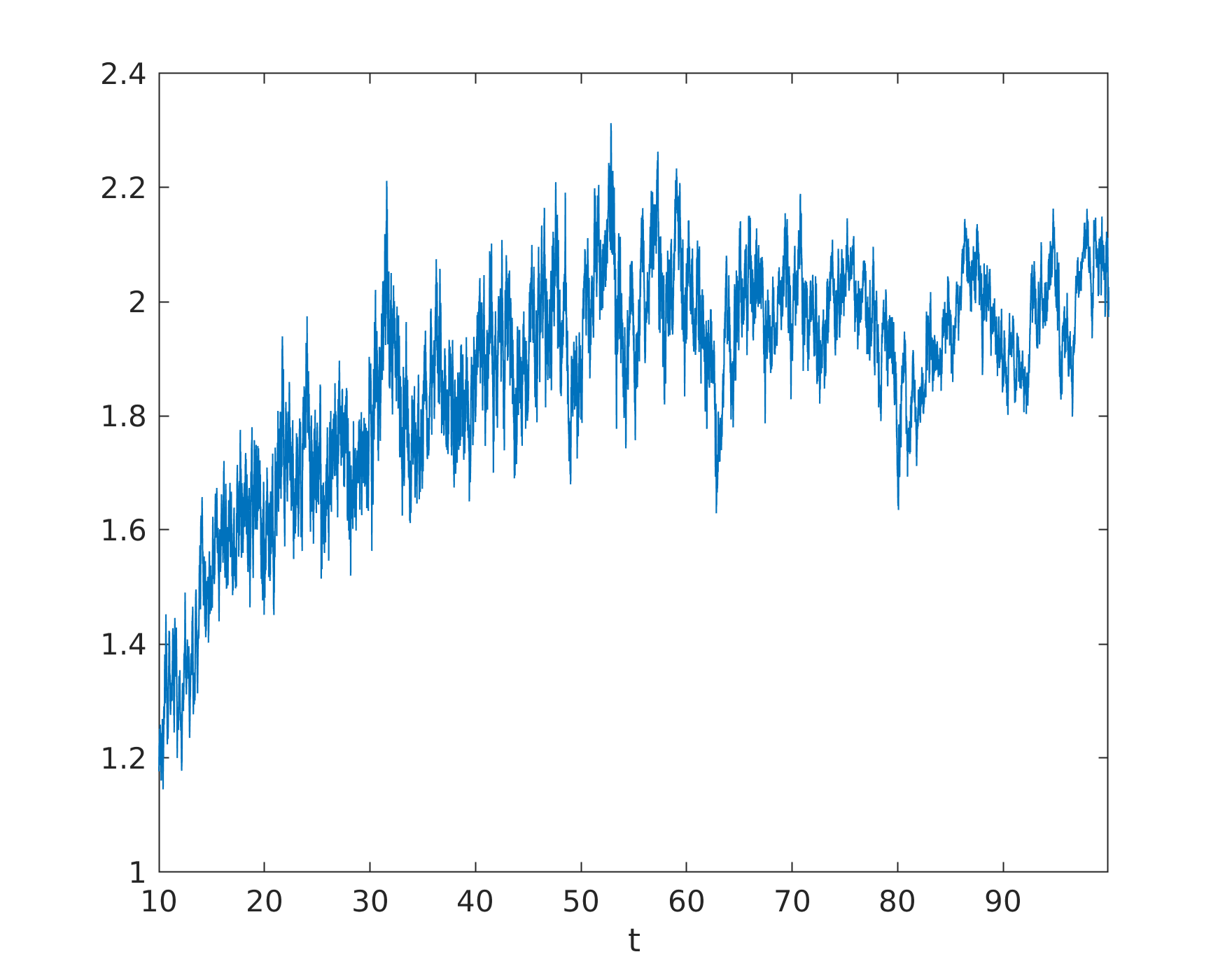}
    \caption{$4\pi \times 4\pi$}
    \label{fig7:skew4pi512Avg}
    \end{subfigure}
    \vskip\baselineskip
    \begin{subfigure}[b]{0.49\textwidth}
        \centering
        \includegraphics[width=1.0\textwidth]{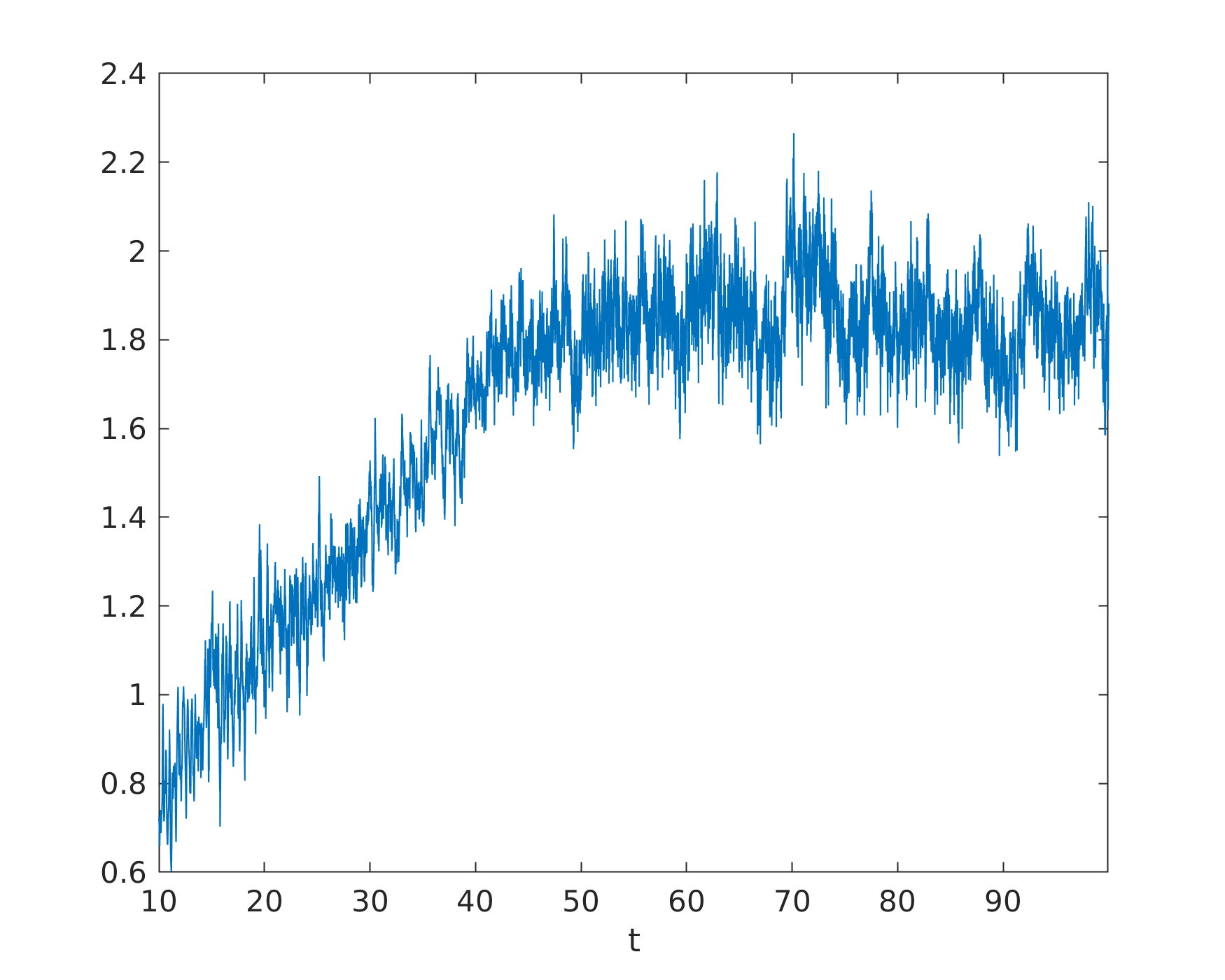}
    \caption{$8\pi \times 8\pi$}
    \label{fig7:skew8pi1024Avg}
    \end{subfigure}
    \begin{subfigure}[b]{0.49\textwidth}
        \centering
        \includegraphics[width=1.0\textwidth]{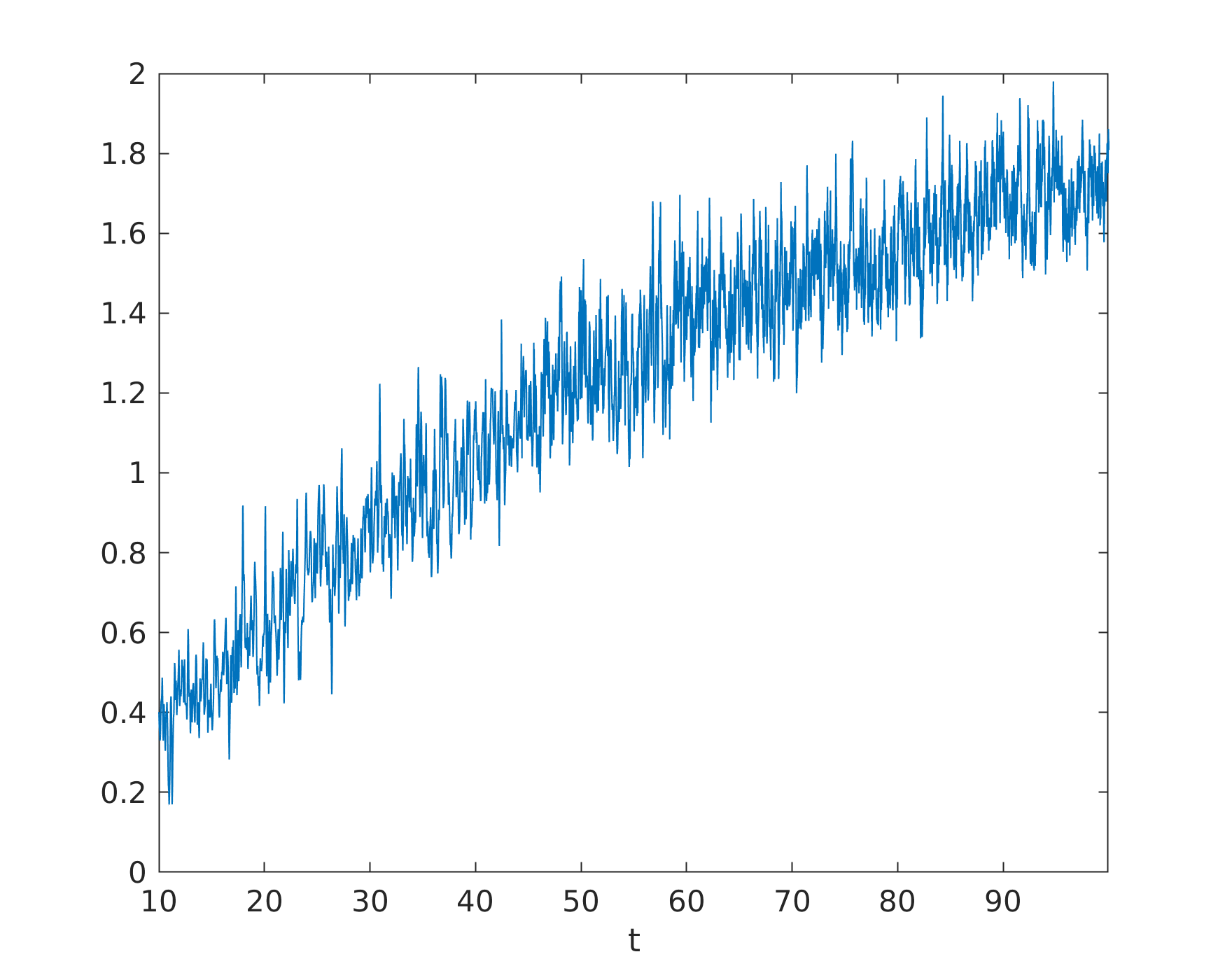}
    \caption{$16\pi \times 16\pi$}
    \label{fig7:skew16pi2048Avg}
    \end{subfigure}
        \caption{Plots of the skewness of $\beta$ as a function of time for the Cahn--Hilliard equation with an asymmetric mixture over varying domain sizes.}
    \label{fig7:skewAsymmetric}
\end{figure}

Next we examine the skewness of the distributions for $\beta$ in Figure~\ref{fig7:skewAsymmetric}. 
It is very clear in all cases that the skew is positive in all cases, this is in--keeping with the physics of the problems where the overall domain length scale should always be increasing and thus the values of $\beta$ should be positive and skewed away from negative values.
Similar findings to what we saw above with the results for variance are also the case here where a steady state develops in each of the domains and the time at which it occurs is dependent on the domain size. 
The $2\pi \times 2\pi$ domain is already in its steady regime at the times shown, $4\pi$ and $8\pi$ settle into theirs as time progresses and $16\pi$ has not quite reached its steady state within the simulation time for this set of results.

\begin{figure}
    \centering
        \includegraphics[width=0.6\textwidth]{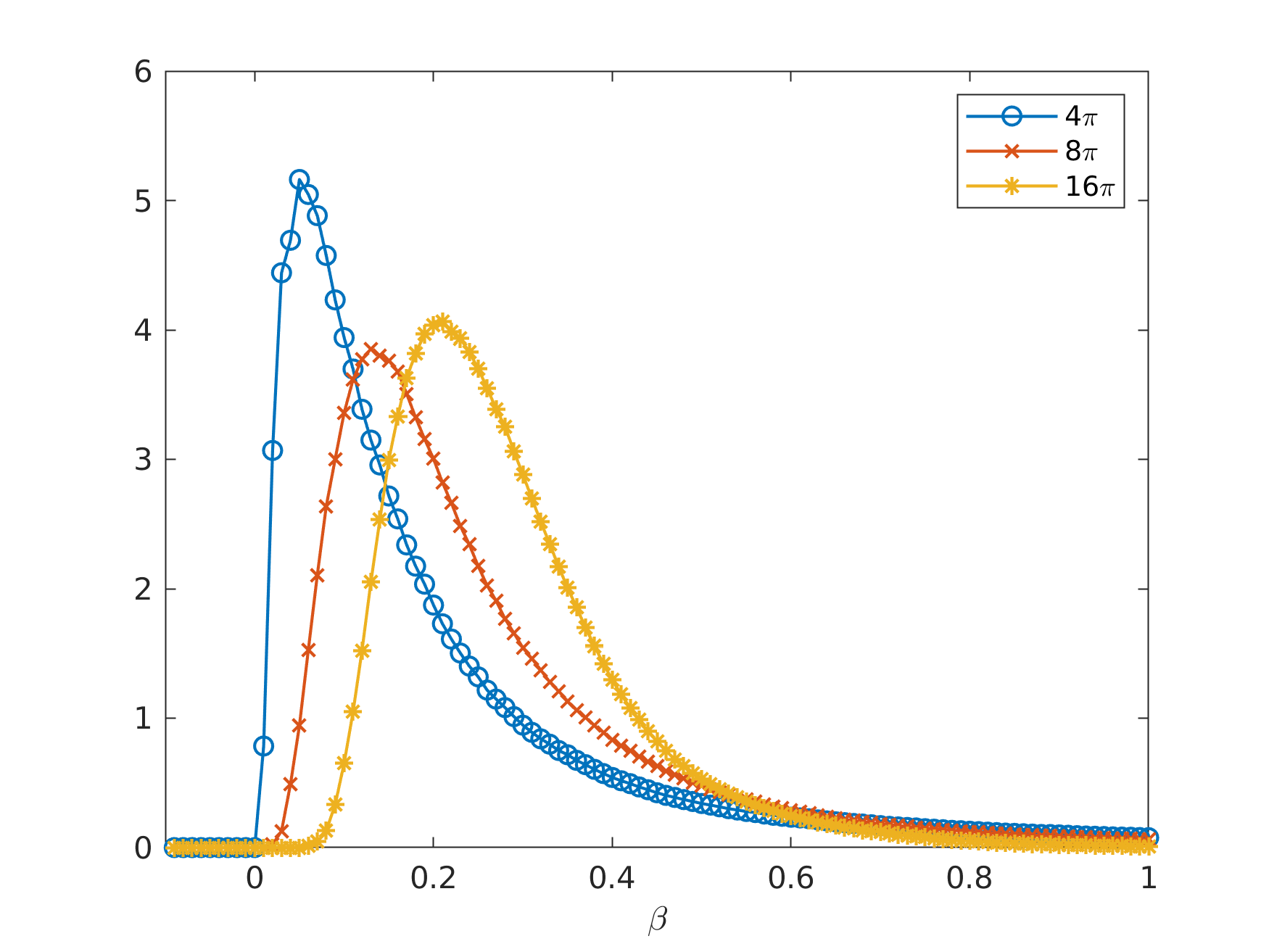}
        \caption{Plot showing the evolution of the distribution of $\beta$ with increasing domain size for an asymmetric mixture, we can see a clear similarity between these results and those extracted from the ODE model in Figure~\ref{fig7:betaVaryN}. We omit $2\pi$ from this plot as it is too heavily skewed towards $0$, making it too hard to see the trend for the other values due to its magnitude.}
    \label{fig7:asymmetricBeta}
\end{figure}

In Figure~\ref{fig7:asymmetricBeta} we present histograms for $\beta$ sampled over the complete time range $10 < t < 100$ presented in this section.
We plot these histograms for each domain size, $2\pi$ is omitted as the distribution is too heavily centred around $0$ driving the peak too high to be able to see the desired trend of the other results.
As the domain size is increased the distribution moves away from $0$ and closer to $1/3$. 
We can also see a growth in the height when moving from $8\pi$ to $16\pi$, this is expected from our study of LSW theory in previous sections; as the number of bubbles in the domain is increased the distribution should approach a $\delta$ function at $\beta = 1/3$.
The trend and behaviour of $\beta$ is identical here to the results presented for the ODE simulation shown in Figure~\ref{fig7:betaVaryN}.
This shows consistency between our ODE model for bubble growth and the full Cahn--Hilliard equation with an asymmetric mixture, both of which are models for Ostwald Ripening. 

Finally we summarise our findings for the Cahn--Hilliard equation with asymmetric mixture.
As the domain size increases the distribution of $\beta$ moves away from $0$ and towards $1/3$ with the overall peak height increasing as it moves across, this matches our finite bubble ODE statistics and discussion from Section~\ref{sec7:drop_pop}.
The point in time at which the distributions \Acomment{become} steady is domain size dependent, the larger the domain the longer it takes for this steady behaviour to appear, this is physical as more nucleation and initial bubble mergers occur in larger domains as there is greater space for bubble to form.
\Acomment{Thus the positions of the distributions and their steadiness are all finite domain size effects, as the domain becomes larger and larger the distribution of $\beta$ appears to develop into a steady distribution around $1/3$. 
With a big enough domain this distribution will continue to narrow (as shown by the drop in variance with increasing domain size as seen in Figure~\ref{fig7:varianceAsymmetric}) and grow in height to form a $\delta$ function at $\beta = 1/3$ as predicted by LSW theory for Ostwald ripening which a model for the Cahn--Hilliard equation with an asymmetric mixture.}
The finite size of the domain and thus the finite number of bubbles present in a asymmetric Cahn--Hilliard mixture effectively smear the LSW theory predicted $\delta$ distribution for an infinite number of bubbles/domain size to a distribution like those presented in Figure~\ref{fig7:asymmetricBeta}.


\begin{figure}
    \centering
    \begin{subfigure}[b]{0.49\textwidth}
        \centering
        \includegraphics[width=1.0\textwidth]{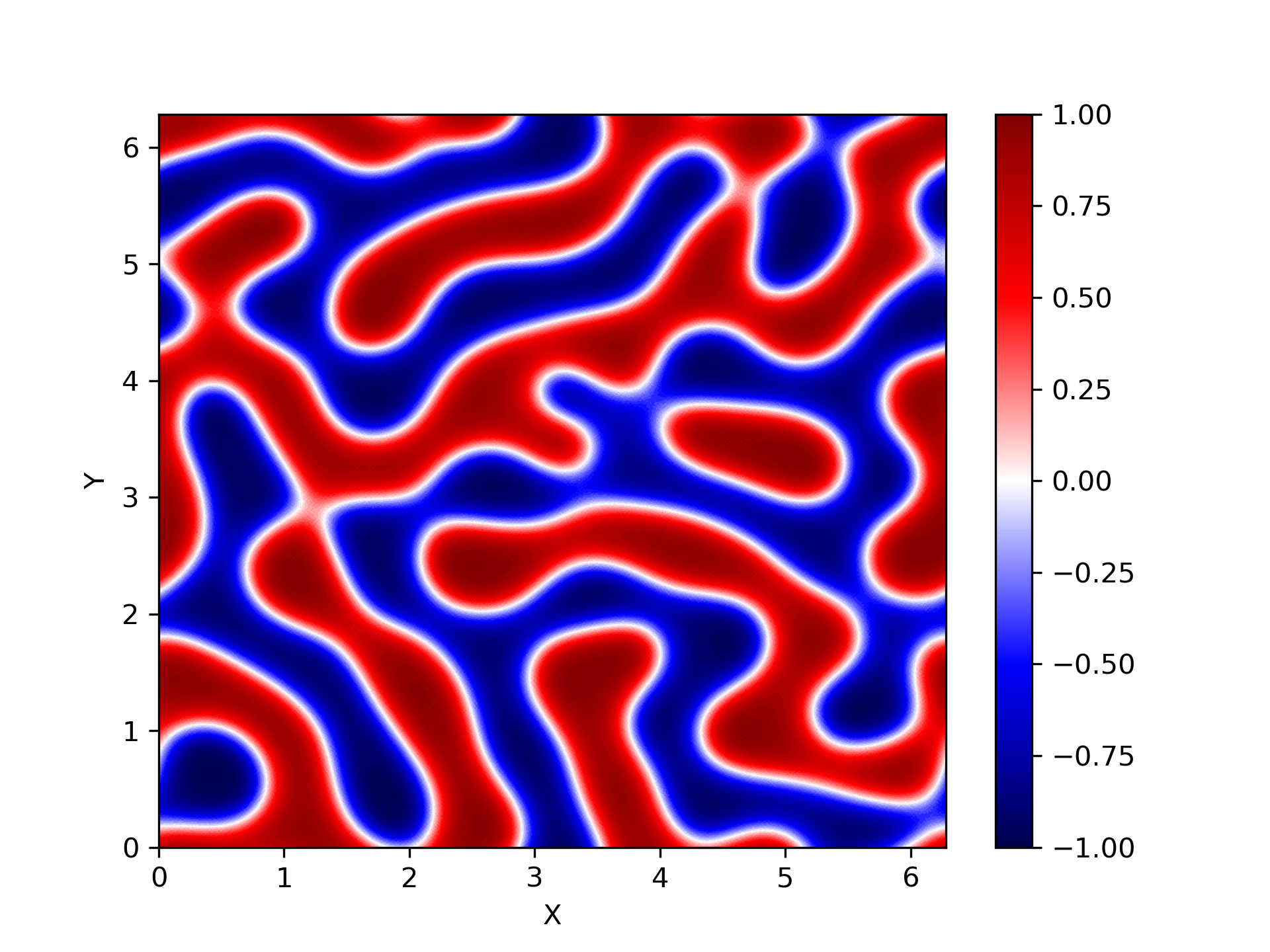}
    \caption{$t = 0.6135923152$}
    \label{fig7:contour4cahn}
    \end{subfigure}
    \hfill
    \begin{subfigure}[b]{0.49\textwidth}  
        \centering
        \includegraphics[width=1.0\textwidth]{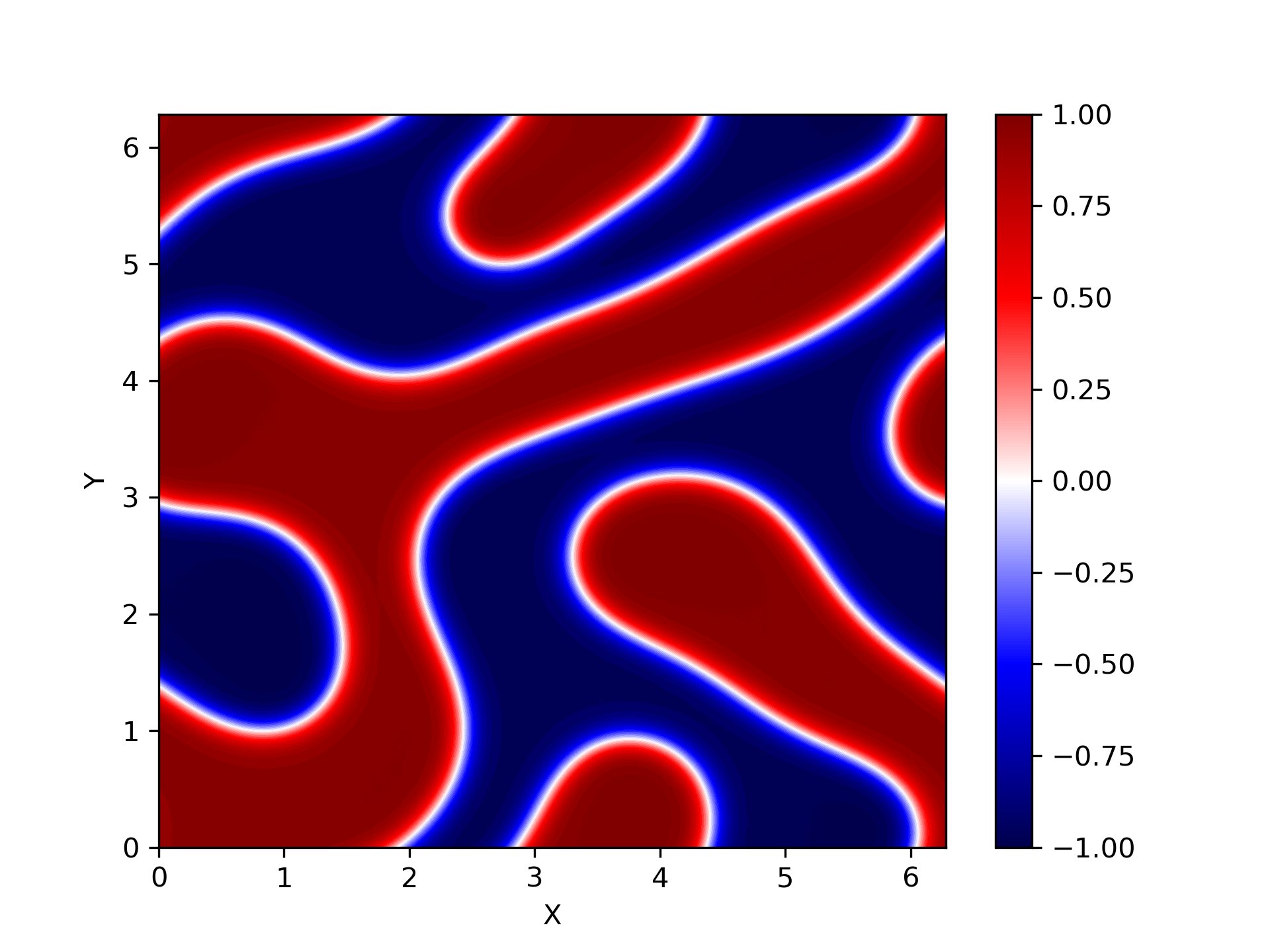}
    \caption{$t = 3.8042723540$}
    \label{fig7:contour30cahn}
    \end{subfigure}
    \vskip\baselineskip
    \begin{subfigure}[b]{0.49\textwidth}   
        \centering
        \includegraphics[width=1.0\textwidth]{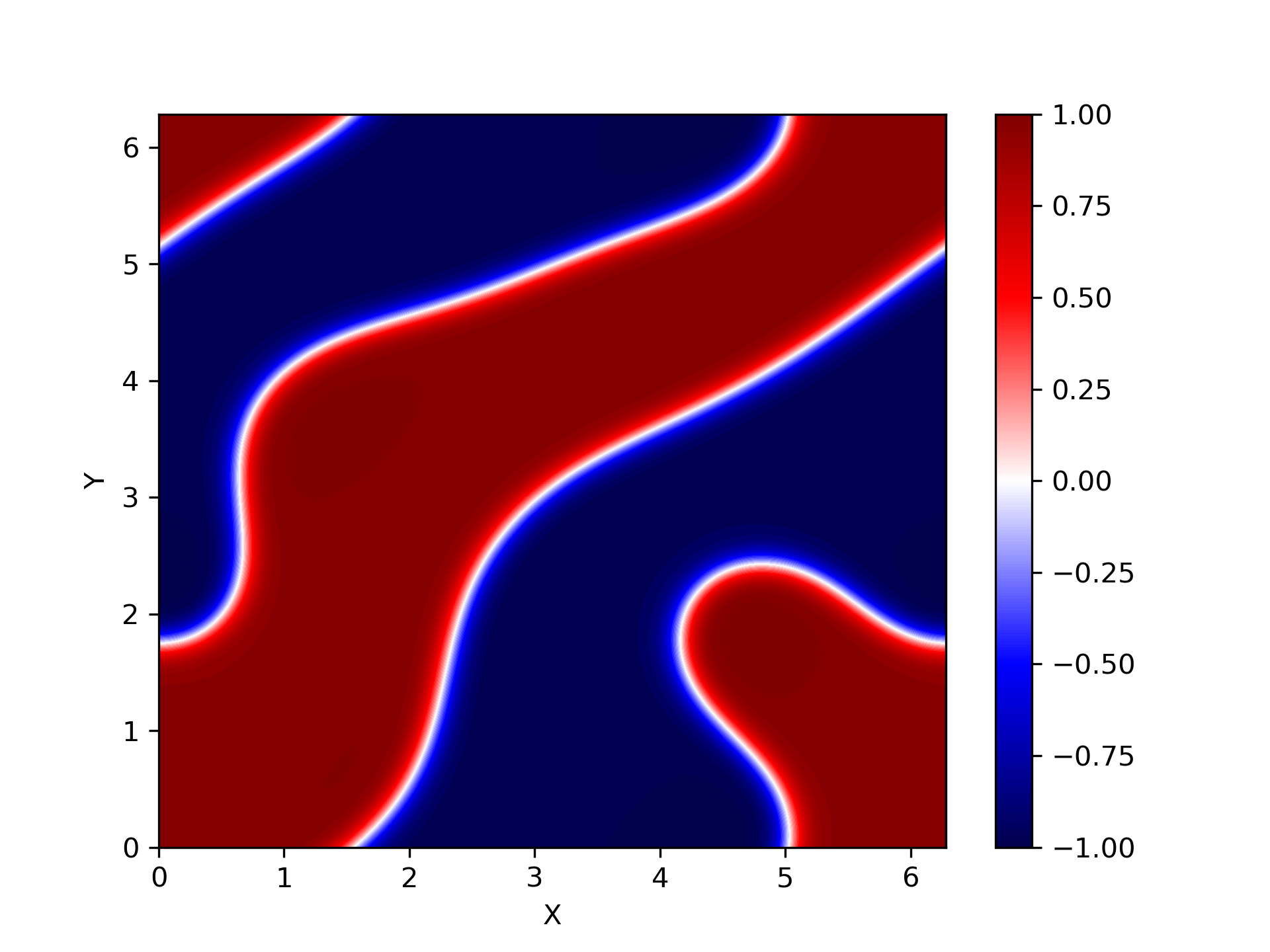}
    \caption{$t = 14.8489340267$}
    \label{fig7:contour120cahn}
    \end{subfigure}
    \begin{subfigure}[b]{0.49\textwidth}   
        \centering
        \includegraphics[width=1.0\textwidth]{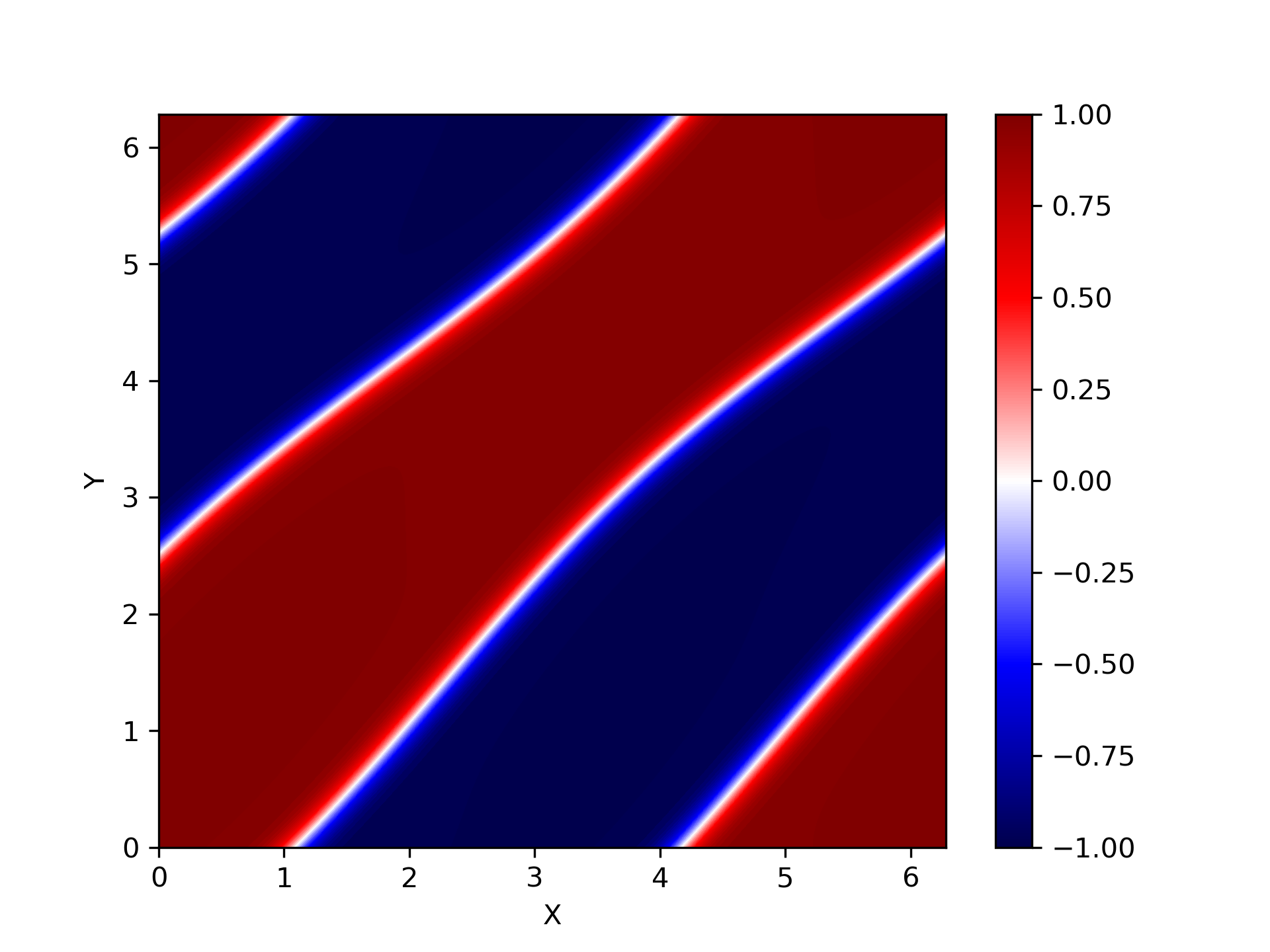}
    \caption{$t = 61.4819499785$}
    \label{fig7:contour500cahn}
    \end{subfigure}
        \caption{\Acomment{Contour plots of the Cahn--Hilliard equation with asymmetric mixture showing interconnected regions on a domain of size $2\pi \times 2\pi$ develop as a function of time. We can see the finite size features beginning to form in the final panel where the domain has reduced to two large binary regions (one blue, one red).}}
    \label{fig7:interconnected}
\end{figure}

\section{Cahn--Hilliard Equation -- Symmetric Mixtures}
\label{sec7:CH_symm}
The Cahn--Hilliard equation initialised with a symmetric mixture develops into a domain of interconnected regions, an example of a simulation initialised with a symmetric mixture showing the development of these interconnected regions can be seen in Figure~\ref{fig7:interconnected}.
These interconnected regions coarsen over time and eventually coalesce to form two distinct regions of $C = \pm 1$ joined together by an interface whose width is characterised by $\gamma$, once the two regions form the domain begins to reach a steady state and growth stops.
The length of time taken for this to occur and the size of the two domains is dictated by the finite size of the domain in which the dynamics takes place.
Again in this section we are applying the methodology described in the second part of Section~\ref{sec7:methodology}.
The initial conditions are all randomised and drawn from a uniform distributions between $-0.1$ and $0.1$ to ensure a symmetric mixture develops.

\begin{figure}
    \centering
    \begin{subfigure}[b]{0.49\textwidth}
        \centering
        \includegraphics[width=1.0\textwidth]{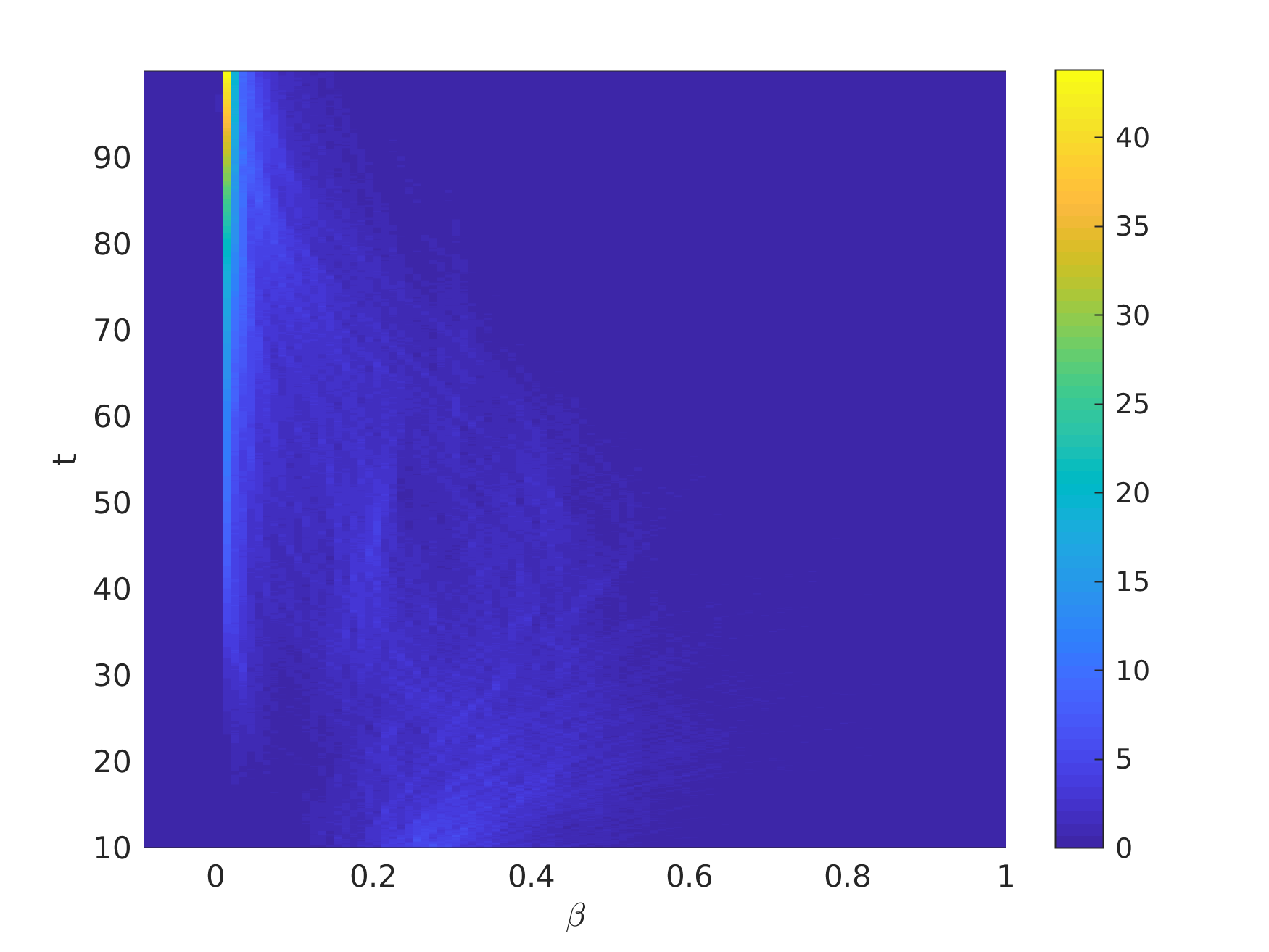}
    \caption{$2\pi \times 2\pi$}
    \label{fig7:spacetime2pi256}
    \end{subfigure}
    \hfill
    \begin{subfigure}[b]{0.49\textwidth}
        \centering
        \includegraphics[width=1.0\textwidth]{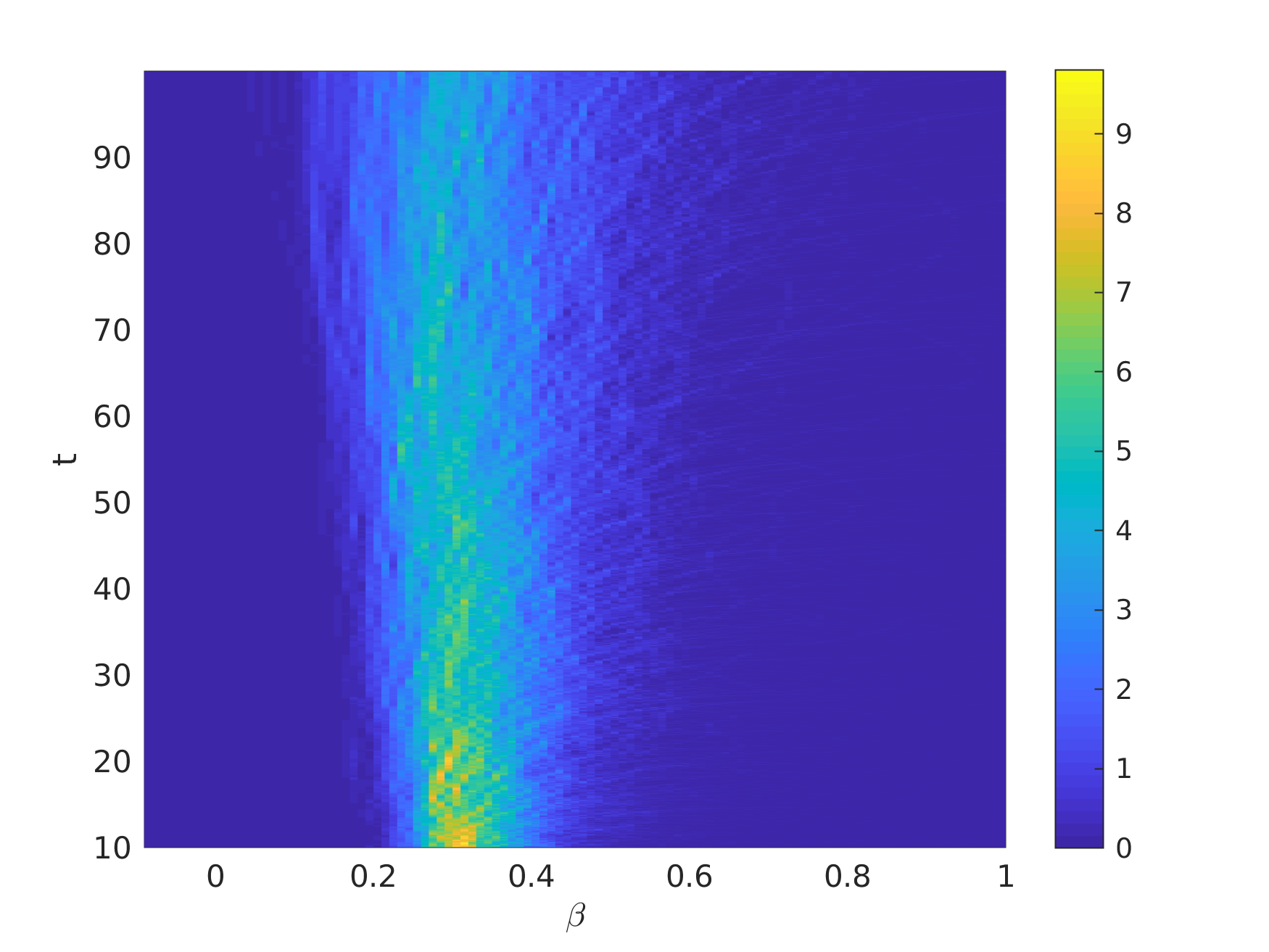}
    \caption{$4\pi \times 4\pi$}
    \label{fig7:spacetime4pi512}
    \end{subfigure}
    \vskip\baselineskip
    \begin{subfigure}[b]{0.49\textwidth}
        \centering
        \includegraphics[width=1.0\textwidth]{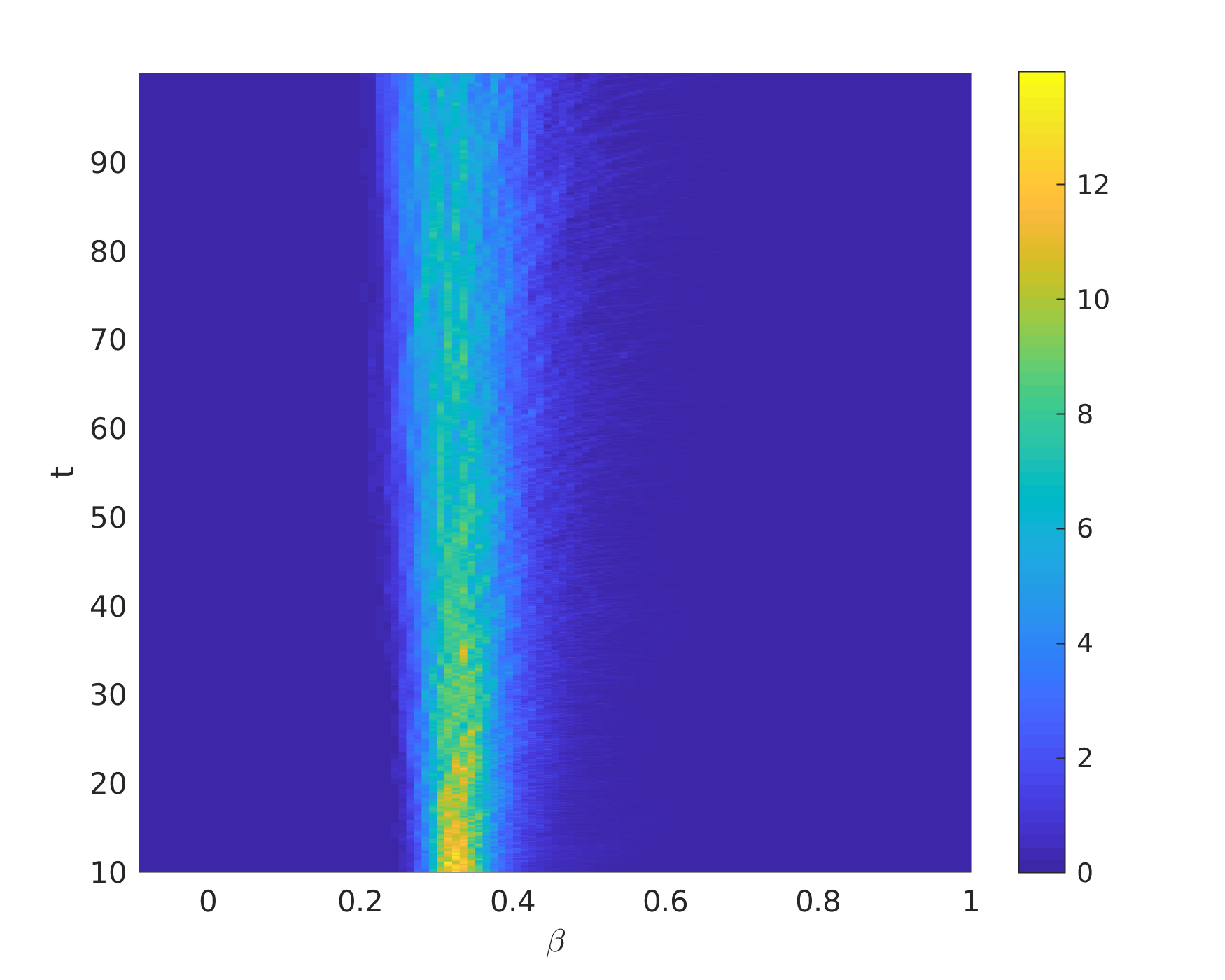}
    \caption{$8\pi \times 8\pi$}
    \label{fig7:spacetime8pi1024}
    \end{subfigure}
    \begin{subfigure}[b]{0.49\textwidth}
        \centering
        \includegraphics[width=1.0\textwidth]{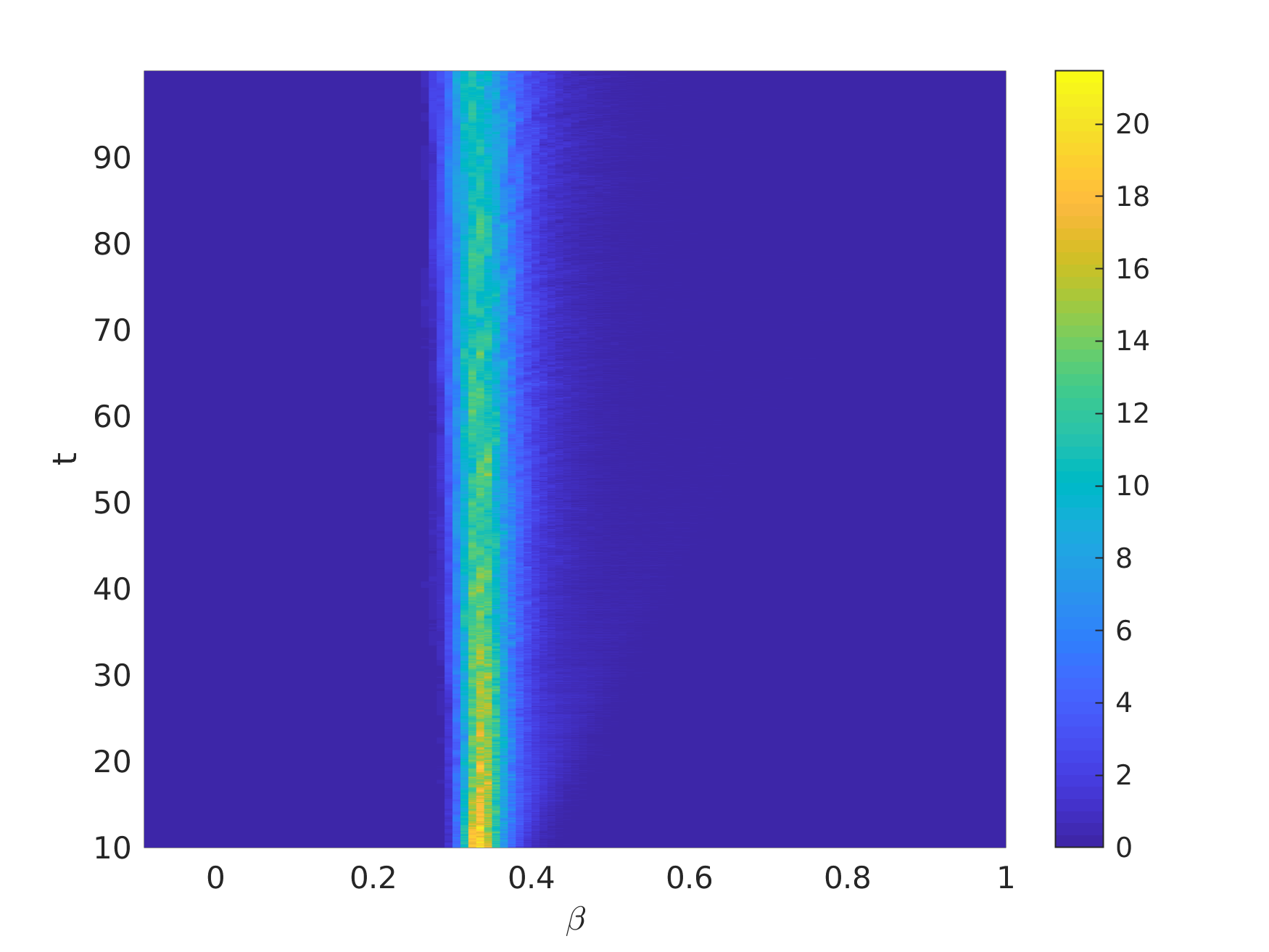}
    \caption{$16\pi \times 16\pi$}
    \label{fig7:spacetime16pi2048}
    \end{subfigure}
        \caption{\Acomment{Plots of the distribution of $\beta$ as a function of time for the Cahn--Hilliard equation with a symmetric mixture with varying domain sizes.}}
    \label{fig7:spacetimeSymmetric}
\end{figure}

\Acomment{We begin by studying the distribution--time plots of the simulations in Figure~\ref{fig7:spacetimeSymmetric}.}
The effect of the finite size of the domain is particularly prevalent in a $2\pi \times 2\pi$ domain where most of the values for $\beta$ go to $0$ as the domains stop growing due to hitting late stage, steady state, finite size effects such as those developing in Figure~\ref{fig7:contour500cahn}.
It can be seen clearly that as the domain size is increased the distribution moves towards the expected value of $1/3$ and narrows.
The narrowing of the distributions and the growth of the peaks is emphasised in the growth of the maximum values in the \Acomment{distribution--time} plots, captured in the colour--bars where between $4\pi$ and $16\pi$ the peak increases by a factor of $\approx 2$. 
We expect this trend to continue as the domain size is increased.

\begin{figure}
    \centering
    \begin{subfigure}[b]{0.49\textwidth}
        \centering
        \includegraphics[width=1.0\textwidth]{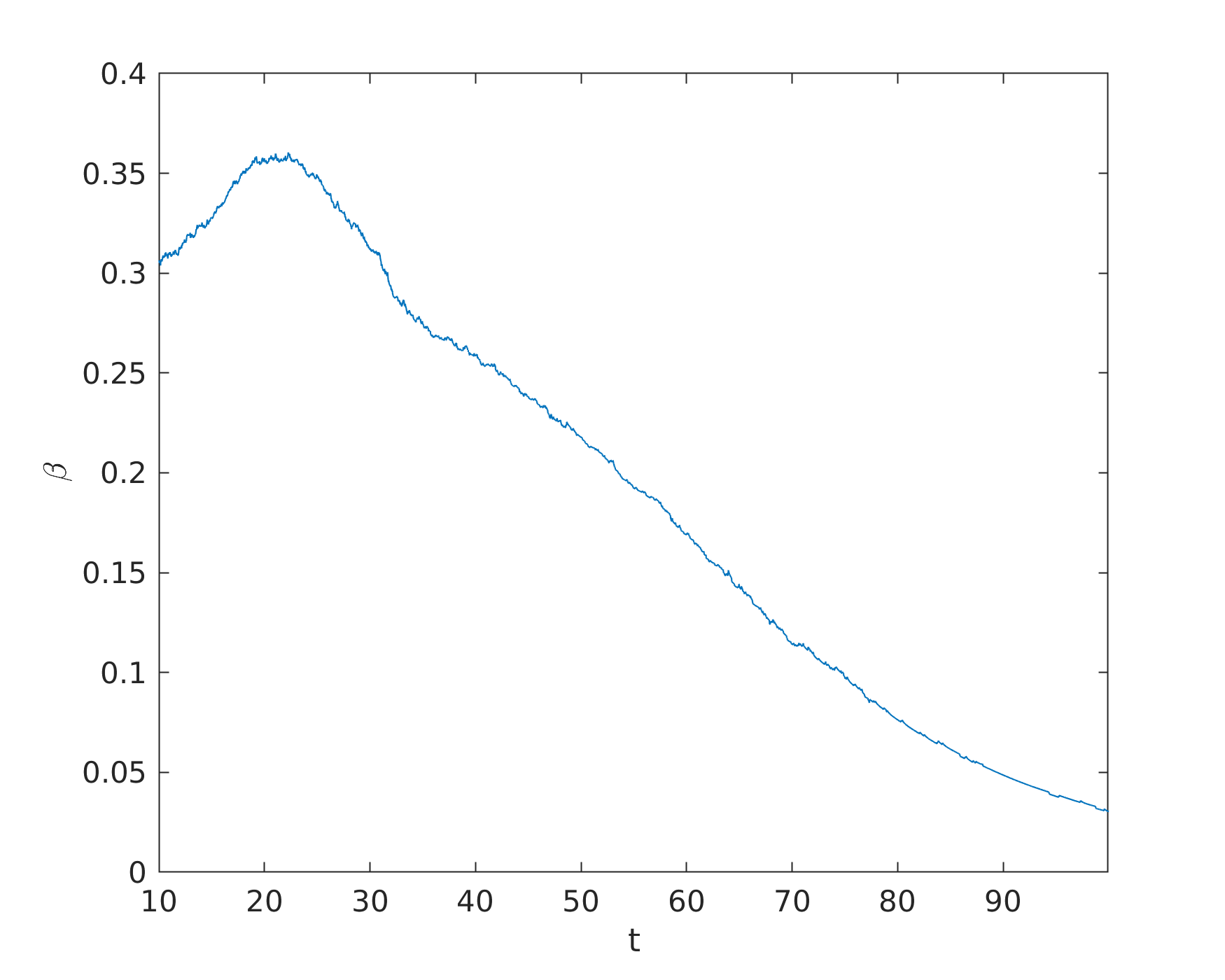}
    \caption{$2\pi \times 2\pi$}
    \label{fig7:mean2pi256}
    \end{subfigure}
    \hfill
    \begin{subfigure}[b]{0.49\textwidth}
        \centering
        \includegraphics[width=1.0\textwidth]{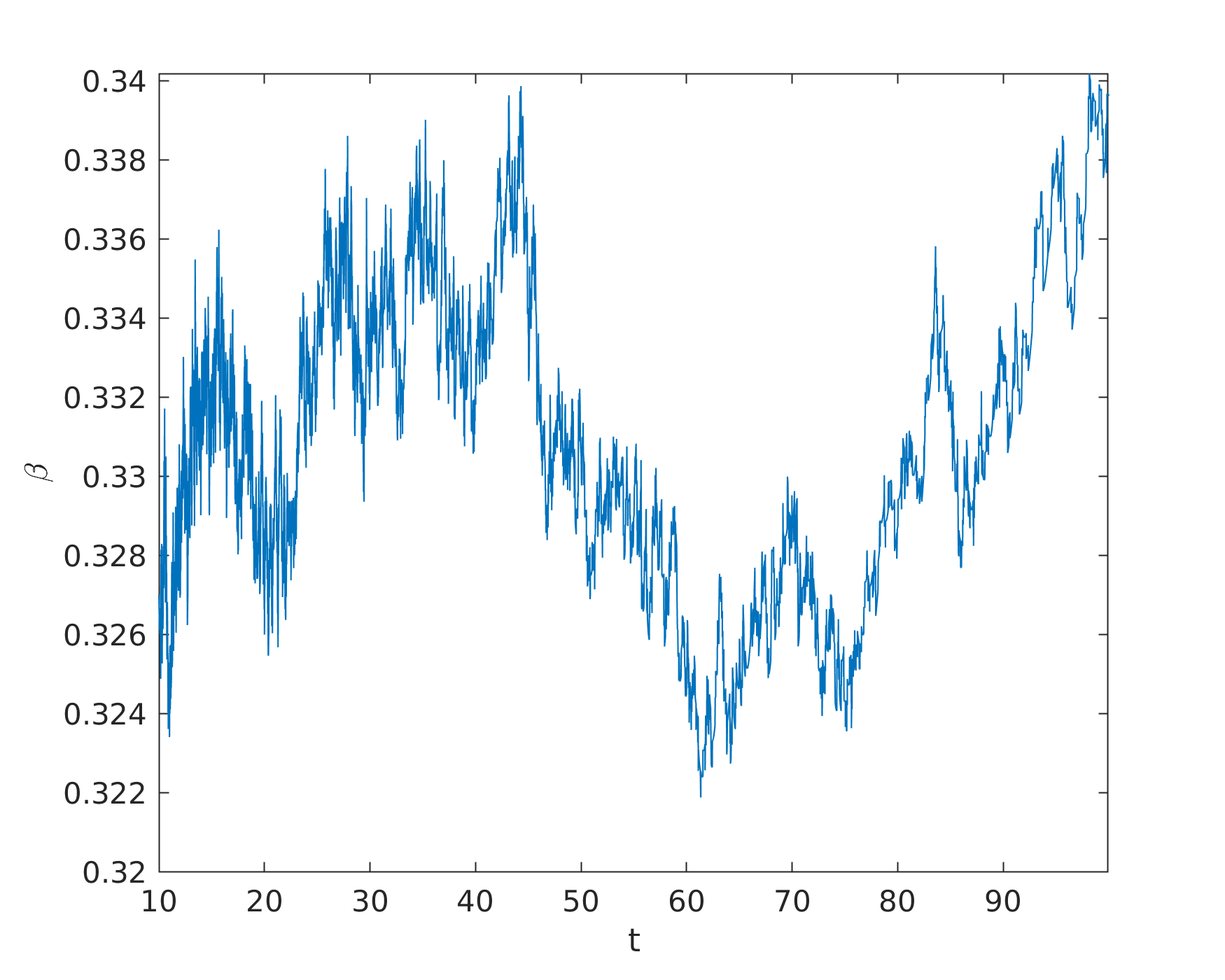}
    \caption{$4\pi \times 4\pi$}
    \label{fig7:mean4pi512}
    \end{subfigure}
    \vskip\baselineskip
    \begin{subfigure}[b]{0.49\textwidth}
        \centering
        \includegraphics[width=1.0\textwidth]{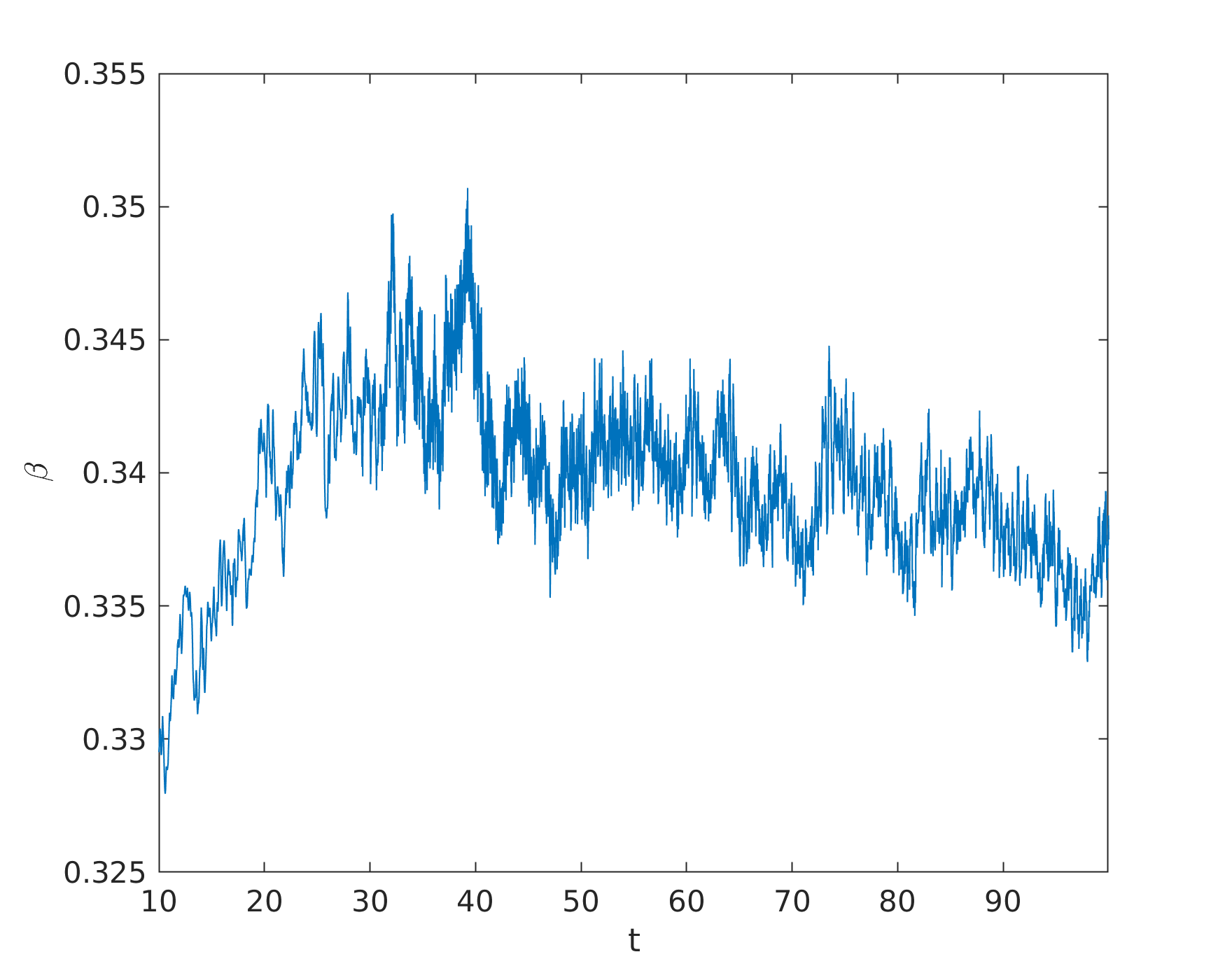}
    \caption{$8\pi \times 8\pi$}
    \label{fig7:mean8pi1024}
    \end{subfigure}
    \begin{subfigure}[b]{0.49\textwidth}
        \centering
        \includegraphics[width=1.0\textwidth]{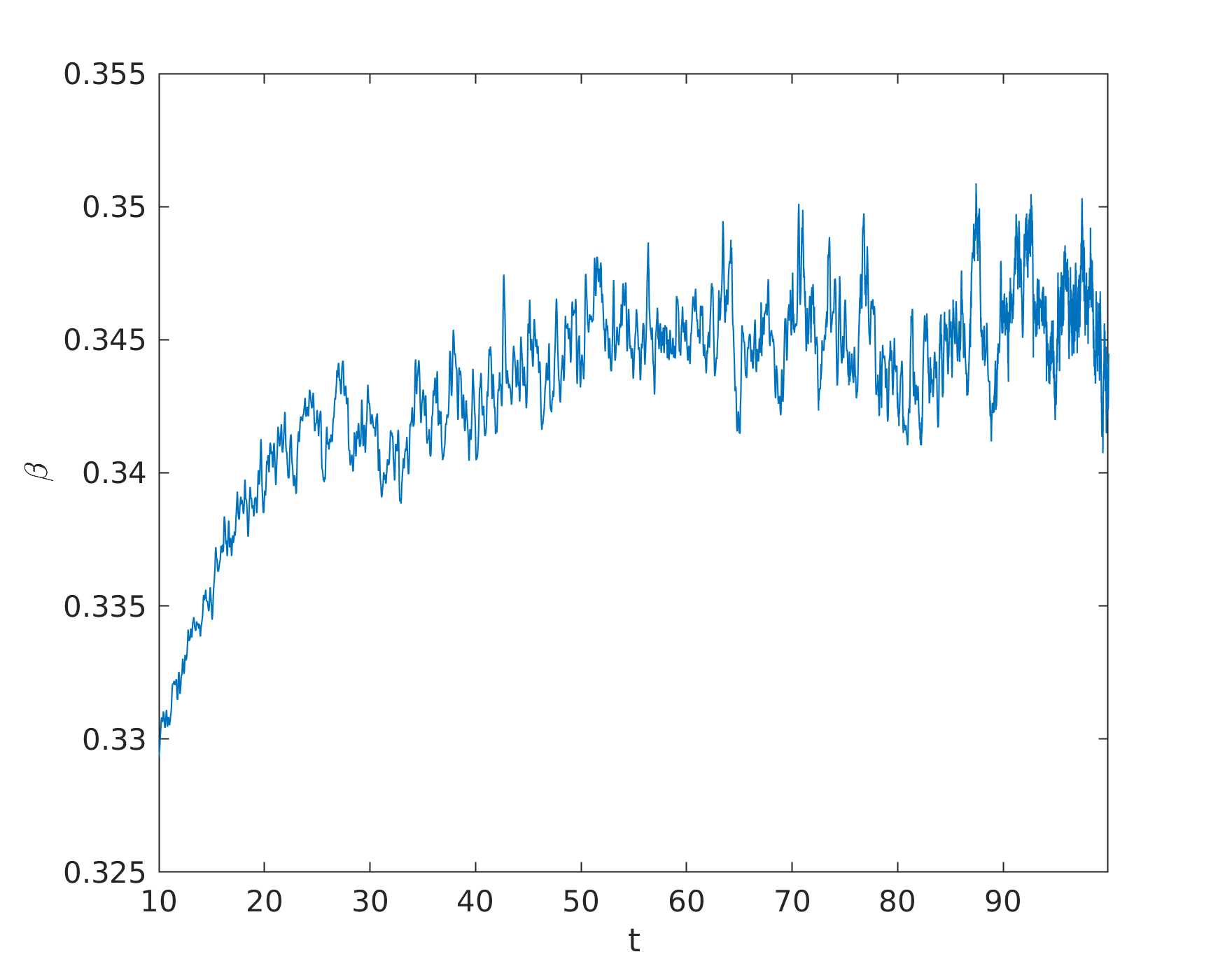}
    \caption{$16\pi \times 16\pi$}
    \label{fig7:mean16pi2048}
    \end{subfigure}
        \caption{Plots of the mean value of $\beta$ as a function of time for the Cahn--Hilliard equation with a symmetric mixture over varying domain sizes.}
    \label{fig7:meanSymmetric}
\end{figure}

Our discussion of how the domain of size $2\pi \times 2\pi$ has its growth rate reduced to a distribution around $0$ are again confirmed in Figure~\ref{fig7:meanSymmetric}.
The values of the mean rapidly drops off in time, for the larger domains we can see much clearer trends around $1/3$ and above.
The values above $1/3$ are due to the positive skew which we shall discuss later in this section.
As the domain size is increased the various plots of the mean also make clear how the distribution of $\beta$ becomes steadier, levelling out at large values of $t$.
The steadiness of the distribution improves with increasing domain size.
Thus not only are the position and height functions of the finite size effects but also the steadiness of the distribution.

\begin{figure}
    \centering
    \begin{subfigure}[b]{0.49\textwidth}
        \centering
        \includegraphics[width=1.0\textwidth]{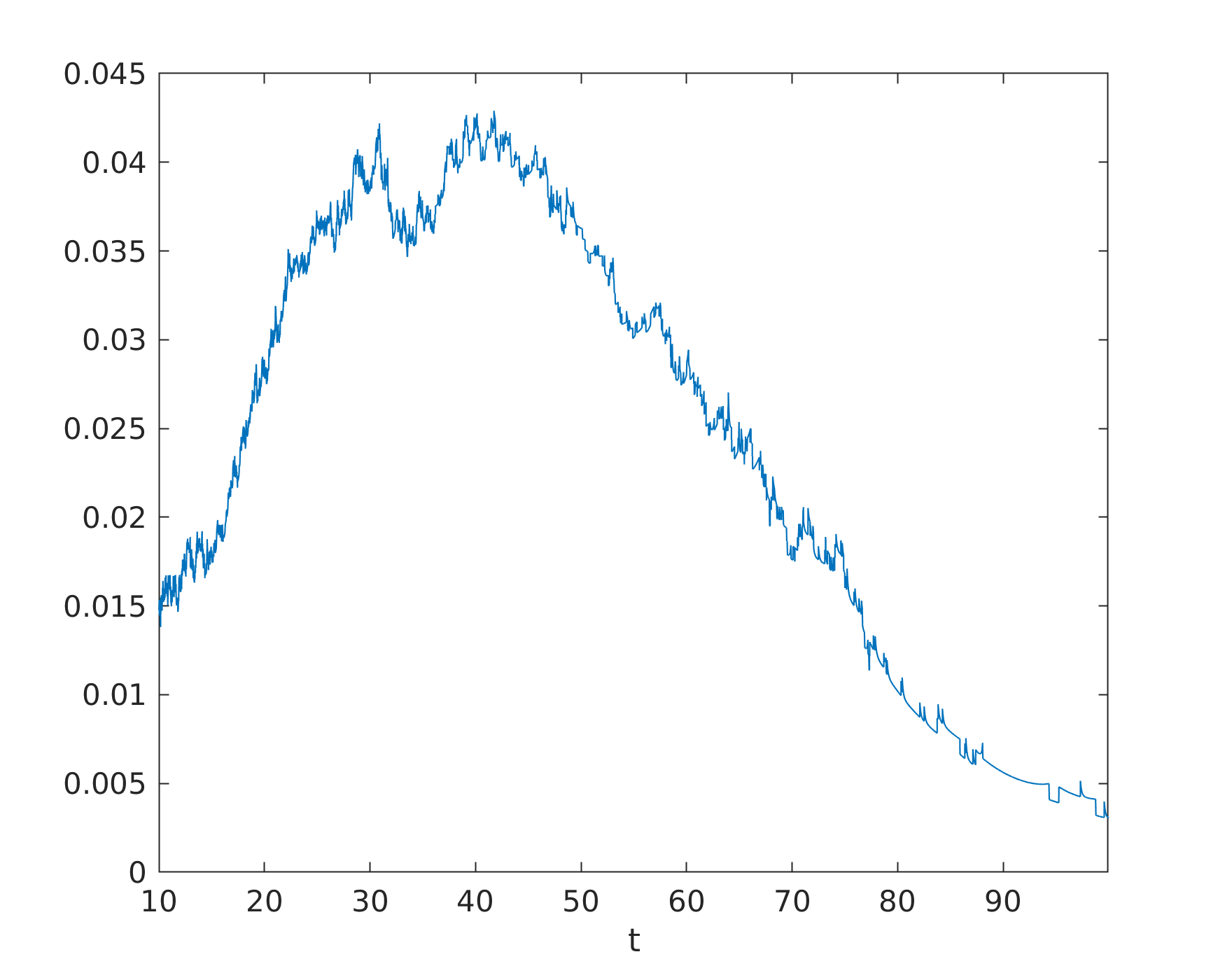}
    \caption{$2\pi \times 2\pi$}
    \label{fig7:variance2pi256}
    \end{subfigure}
    \hfill
    \begin{subfigure}[b]{0.49\textwidth}
        \centering
        \includegraphics[width=1.0\textwidth]{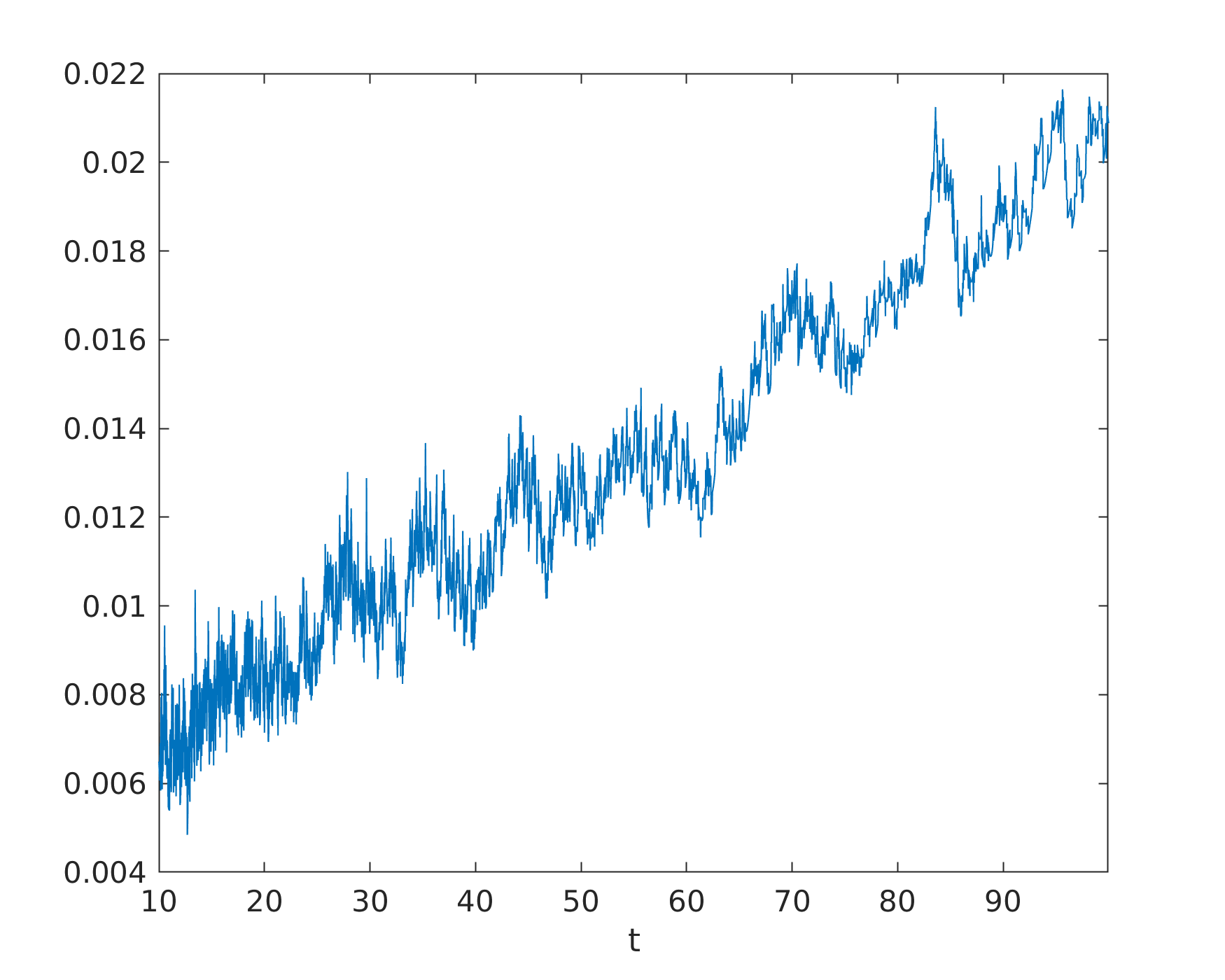}
    \caption{$4\pi \times 4\pi$}
    \label{fig7:variance4pi512}
    \end{subfigure}
    \vskip\baselineskip
    \begin{subfigure}[b]{0.49\textwidth}
        \centering
        \includegraphics[width=1.0\textwidth]{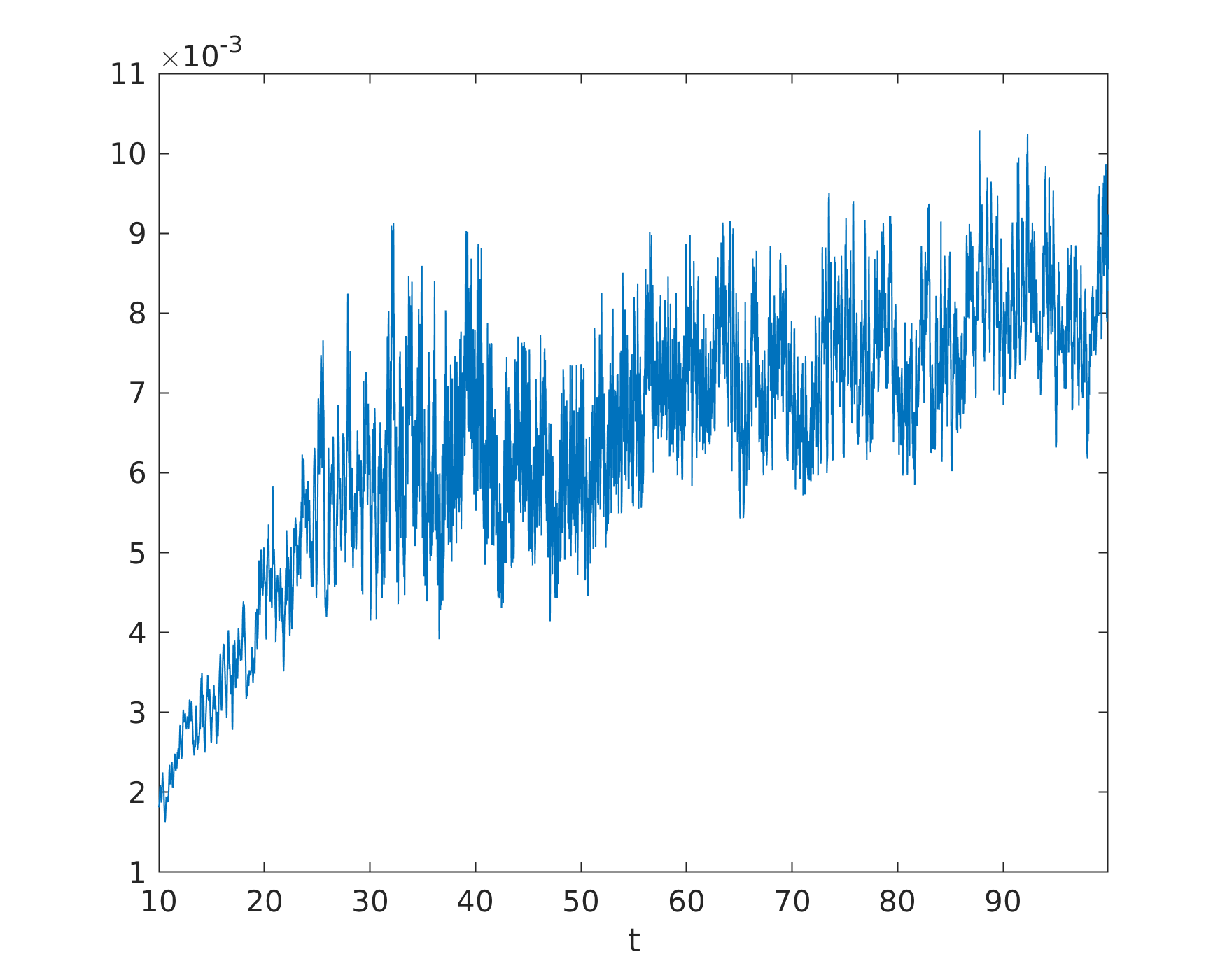}
    \caption{$8\pi \times 8\pi$}
    \label{fig7:variance8pi1024}
    \end{subfigure}
    \begin{subfigure}[b]{0.49\textwidth}
        \centering
        \includegraphics[width=1.0\textwidth]{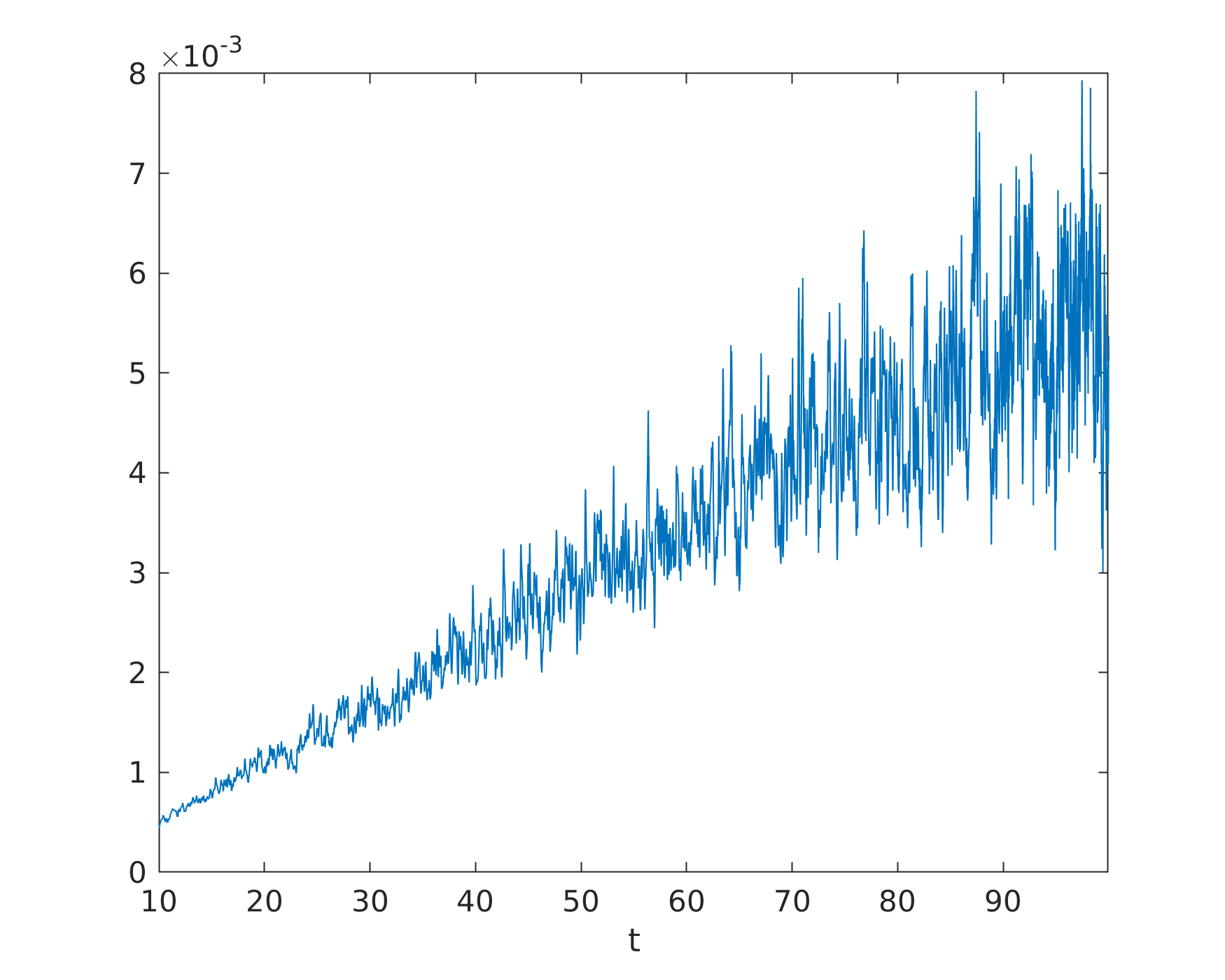}
    \caption{$16\pi \times 16\pi$}
    \label{fig7:variance16pi2048}
    \end{subfigure}
        \caption{Plots of the variance of $\beta$ as a function of time for the Cahn--Hilliard equation with a symmetric mixture over varying domain sizes.}
    \label{fig7:varianceSymmetric}
\end{figure}

Plot of the variance of $\beta$ for the case of a symmetric Cahn--Hilliard mixture can be seen in Figure~\ref{fig7:varianceSymmetric}.
It is clear in all of these plots that the variance is time dependent, particularly at early times, thus we can say that the underlying dynamics of the system is driven by anomalous diffusion. 
The development of a steady state with respect to the variance can be seen in the $8\pi \times 8\pi$ plot in Figure~\ref{fig7:variance8pi1024}, this steady state suggests that at late times the diffusion is constant rather than anomalous. 
We can see that the $16\pi \times 16 \pi$ domain has yet to reach this steady state, while the two smaller domains $2\pi$ and $4\pi$ are experiencing greater movement in variance due to experiencing considerably earlier stage finite size effects in some of their simulations.
This is particularity clear in Figure~\ref{fig7:spacetime4pi512} where the width of the distribution is extremely wide due a diverse range of simulation behaviours, some approaching finite size effects like those in Figure~\ref{fig7:interconnected} while others are still growing.
A broad range of behaviours is also evident in the earlier development of the $2\pi$ case shown in Figure~\ref{fig7:spacetime2pi256}.

\begin{figure}
    \centering
    \begin{subfigure}[b]{0.49\textwidth}
        \centering
        \includegraphics[width=1.0\textwidth]{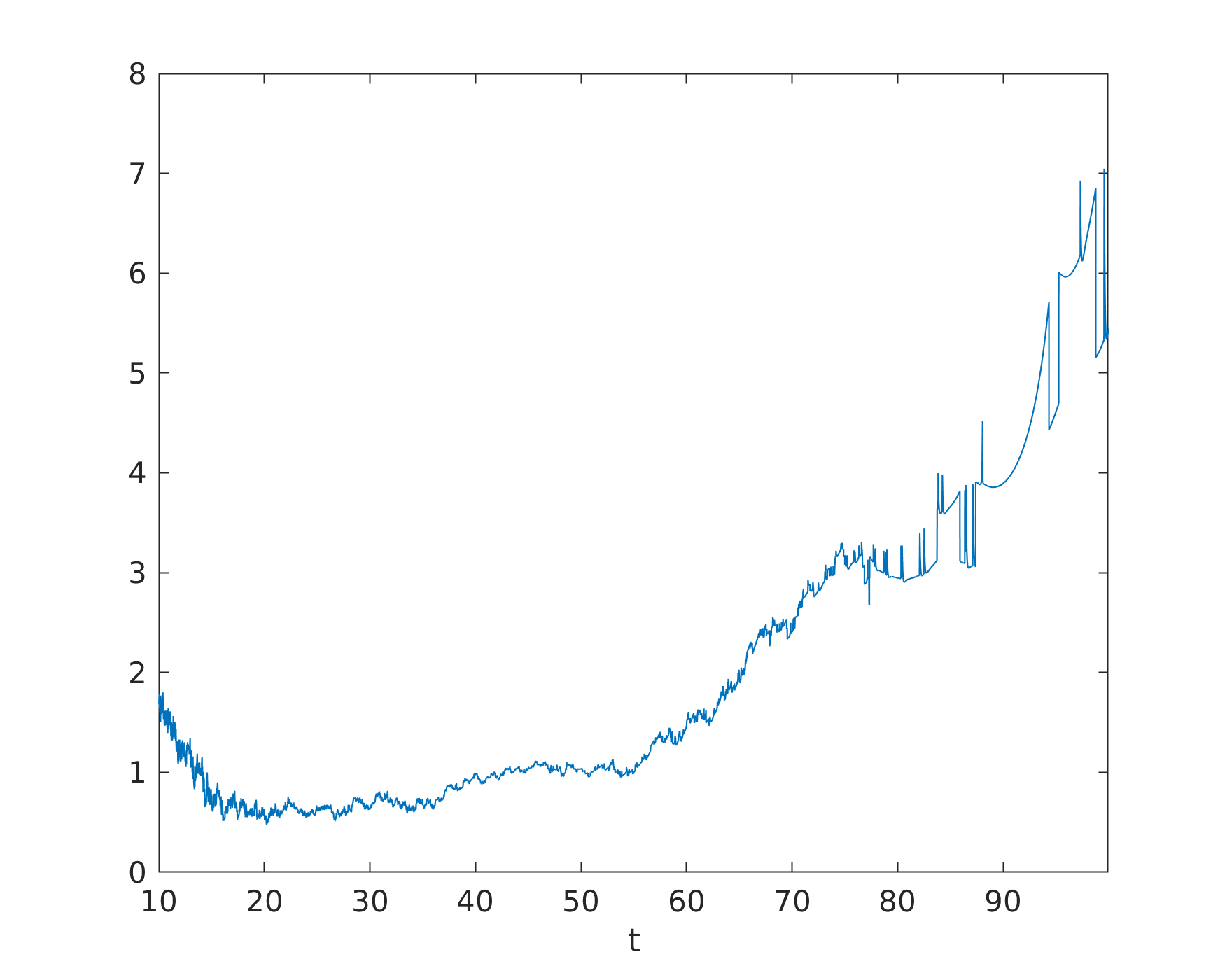}
    \caption{$2\pi \times 2\pi$}
    \label{fig7:skew2pi256}
    \end{subfigure}
    \hfill
    \begin{subfigure}[b]{0.49\textwidth}
        \centering
        \includegraphics[width=1.0\textwidth]{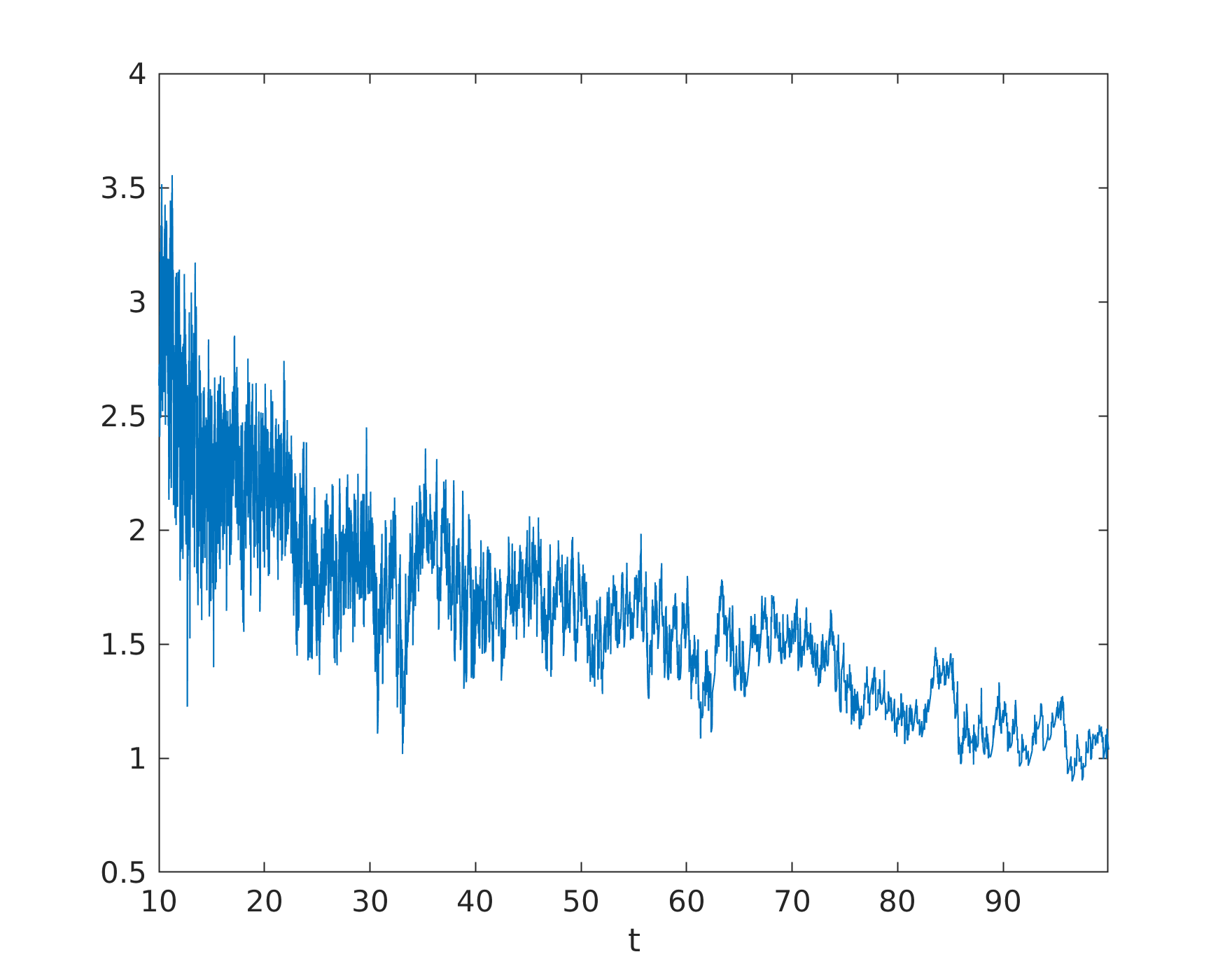}
    \caption{$4\pi \times 4\pi$}
    \label{fig7:skew4pi512}
    \end{subfigure}
    \vskip\baselineskip
    \begin{subfigure}[b]{0.49\textwidth}
        \centering
        \includegraphics[width=1.0\textwidth]{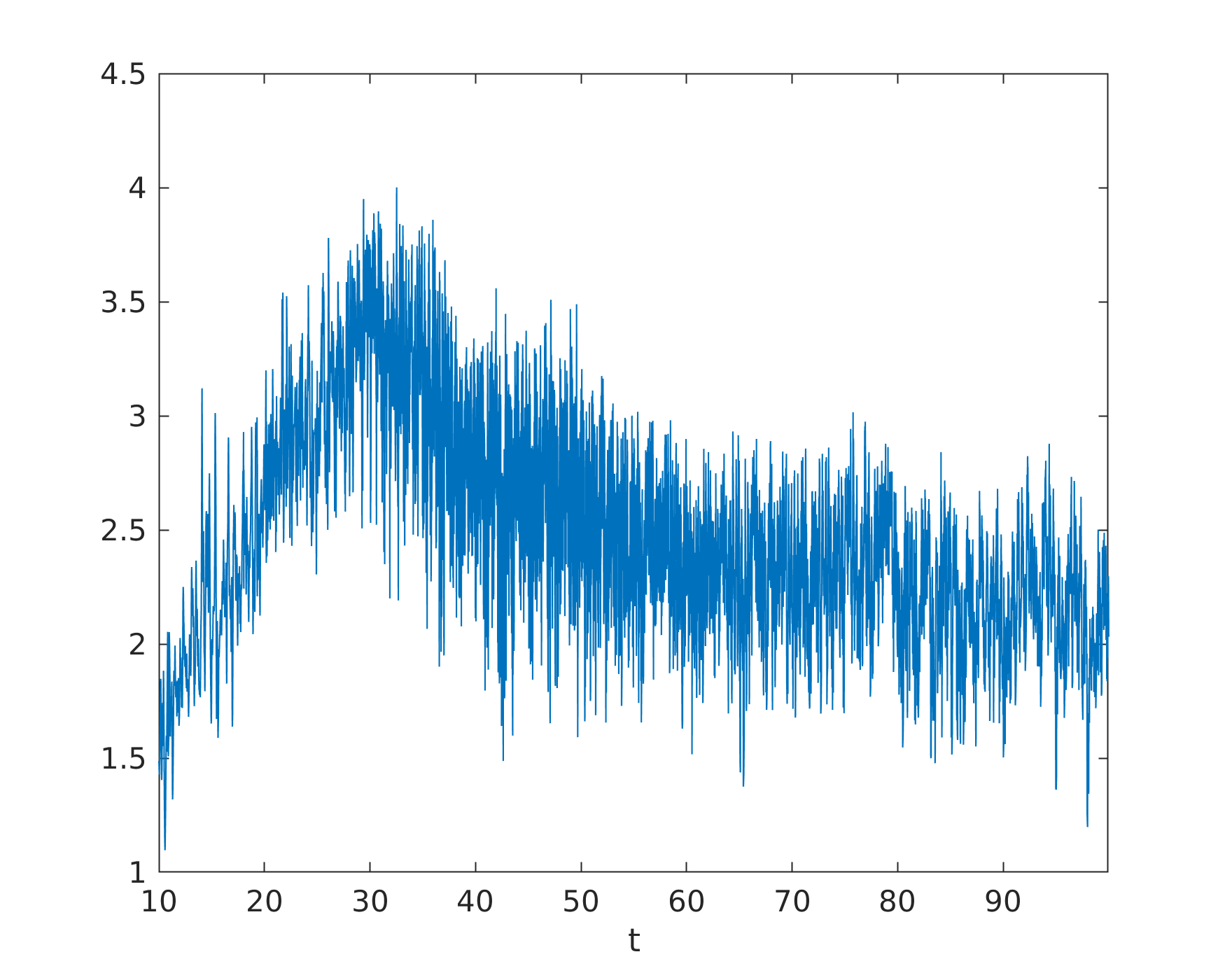}
    \caption{$8\pi \times 8\pi$}
    \label{fig7:skew8pi1024}
    \end{subfigure}
    \begin{subfigure}[b]{0.49\textwidth}
        \centering
        \includegraphics[width=1.0\textwidth]{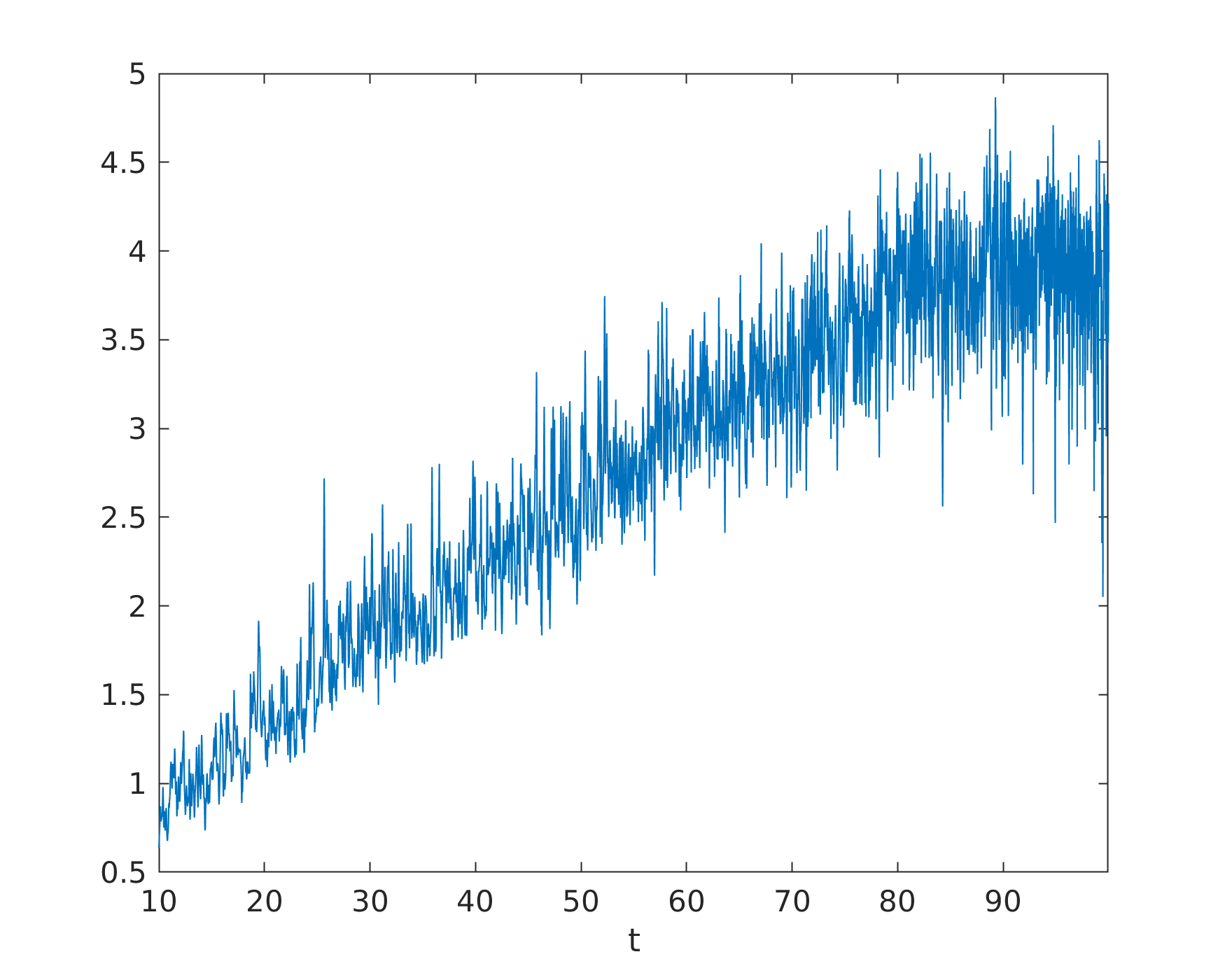}
    \caption{$16\pi \times 16\pi$}
    \label{fig7:skew16pi2048}
    \end{subfigure}
        \caption{Plots of the skewness of $\beta$ as a function of time for the Cahn--Hilliard equation with a symmetric mixture over varying domain sizes.}
    \label{fig7:skewSymmetric}
\end{figure}

In Figure~\ref{fig7:skewSymmetric} we can see the skewness of the distributions of $\beta$ as a function of time.
It is clear in all cases that the distribution is indeed positively skewed, this ensures that no negative growth rates appear which is an expected result, there should be no reduction is domain length scales in any of the simulations.
Like the variance the skewness is time dependent and again we can see the development of a steady state in the $8\pi \times 8\pi$ case shown in Figure~\ref{fig7:skew8pi1024}.
Again the $2\pi$ and $4\pi$ results are dominated by early stage finite size effects and the $16\pi$ domain has yet to reach its steady phase of growth.

\begin{figure}
    \centering
        \includegraphics[width=0.6\textwidth]{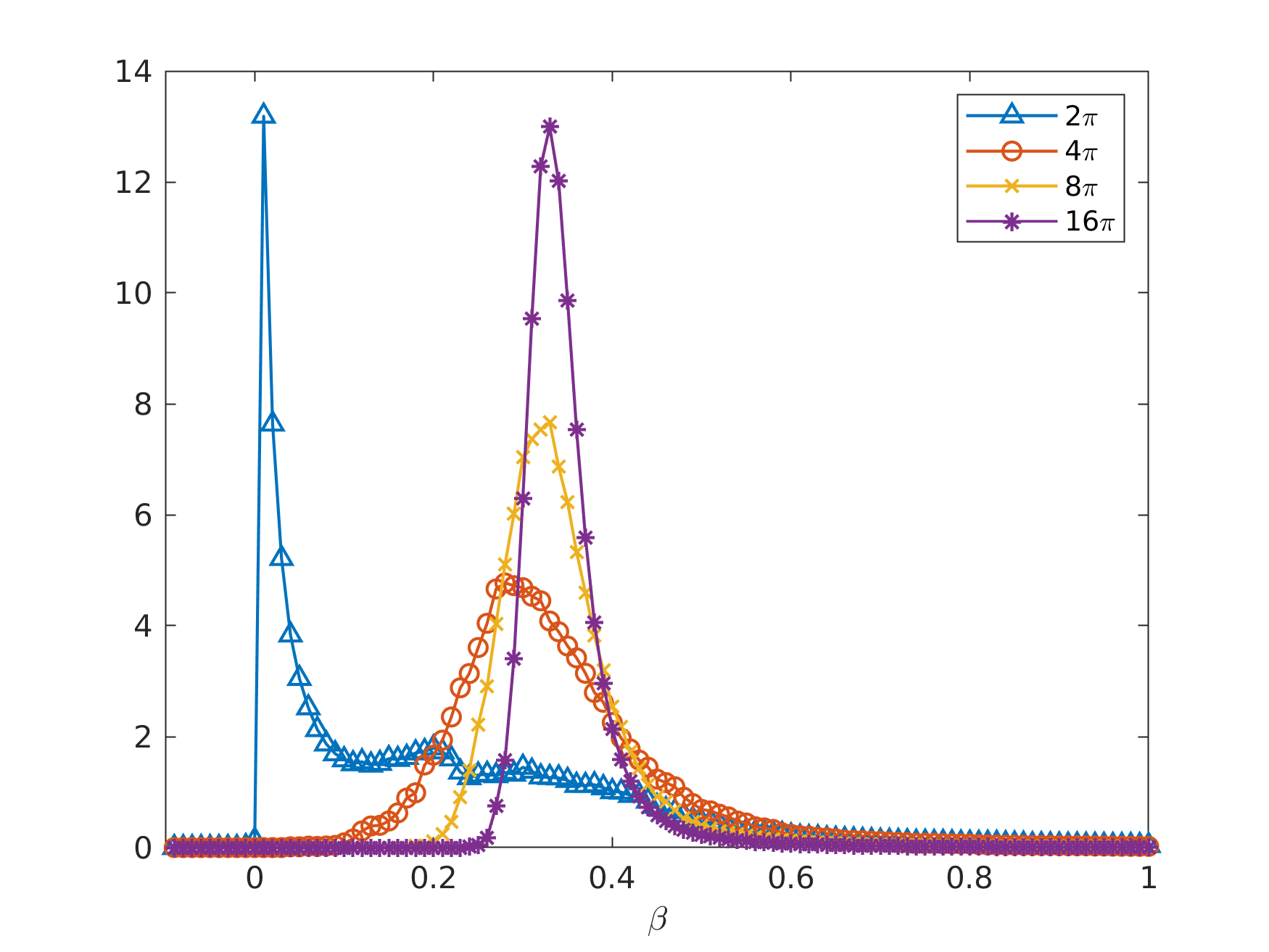}
        \caption{Plot showing the evolution of the distribution of $\beta$ for the Cahn--Hilliard equation with increasing domain size for a symmetric mixture, we can see a clear similarity between these results and those extracted from the ODE model in Figure~\ref{fig7:betaVaryN}.}
    \label{fig7:symmetricBeta}
\end{figure}

Finally in Figure~\ref{fig7:symmetricBeta} we plot the histograms of all values of $\beta$ presented in this section for various domain sizes, the histograms are sampled across all values of $\beta$ between $10 < t < 100$ shown in this section.
We can see a clear comparison between the results presented here and those for the ODE simulations of the bubble model presented in Figure~\ref{fig7:betaVaryN}.
The smallest domain/number of bubbles has a distribution centred at $0$ and then as the number of bubbles/size of the domain is increased the distribution moves toward $1/3$.
In this study of the symmetric Cahn--Hilliard we can see a clear trend of increasing height of the distribution and narrowing around $1/3$ as domain size increases, this results agrees with our ODE solutions and the predictions of LSW theory.
\Acomment{It is indicated here that as the domain size increases the distribution will indeed approach a $\delta$ function at $\beta = 1/3$.
Thus we have shown that the LSW theory and results we developed for Ostwald Ripening on a finite domain and its predictions of the behaviour of $\beta$ can be extended to symmetric system with interconnected domains.}

Summarising our results for the evolution of symmetric mixtures simulated with the Cahn--Hilliard equation we have shown that our results for how $\beta$ develops with increasing domain size are in--keeping with the results for $\beta$ when increasing bubble numbers in the ODE model.
The positions and steadiness of the distributions are dependent on the finite domain size and in particular the steadiness improves with increasing domain size. 
The finite size of the domain leads to a smearing of the $\delta$ function at $1/3$ predicted by LSW theory and in the limit of an infinite domain they will indeed approach the LSW predictions.


\begin{figure}
    \centering
    \begin{subfigure}[b]{0.49\textwidth}
        \centering
        \includegraphics[width=1.0\textwidth]{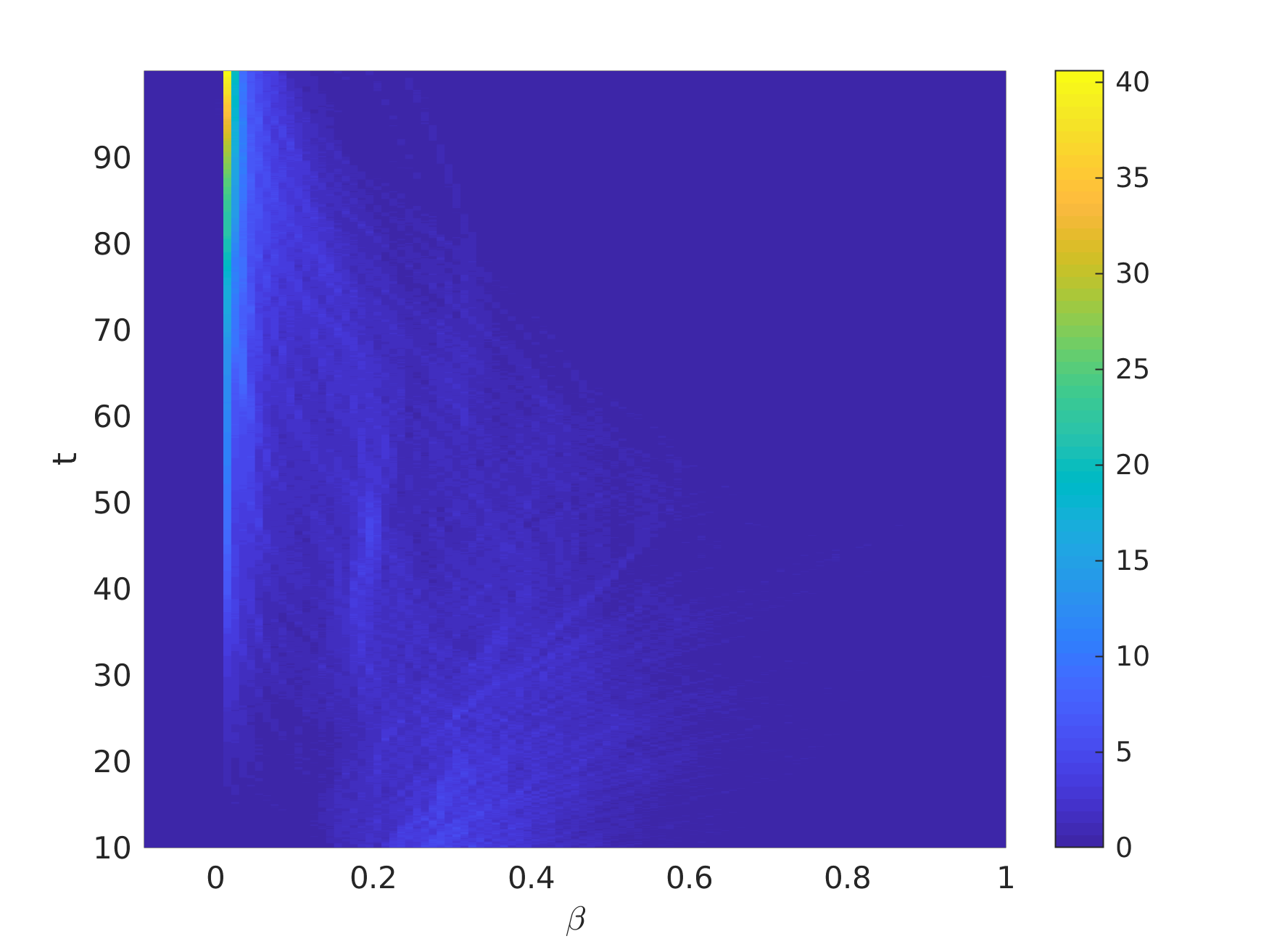}
    \caption{$2\pi \times 2\pi$}
    \label{fig7:spacetime2pi256Cooke}
    \end{subfigure}
    \hfill
    \begin{subfigure}[b]{0.49\textwidth}
        \centering
        \includegraphics[width=1.0\textwidth]{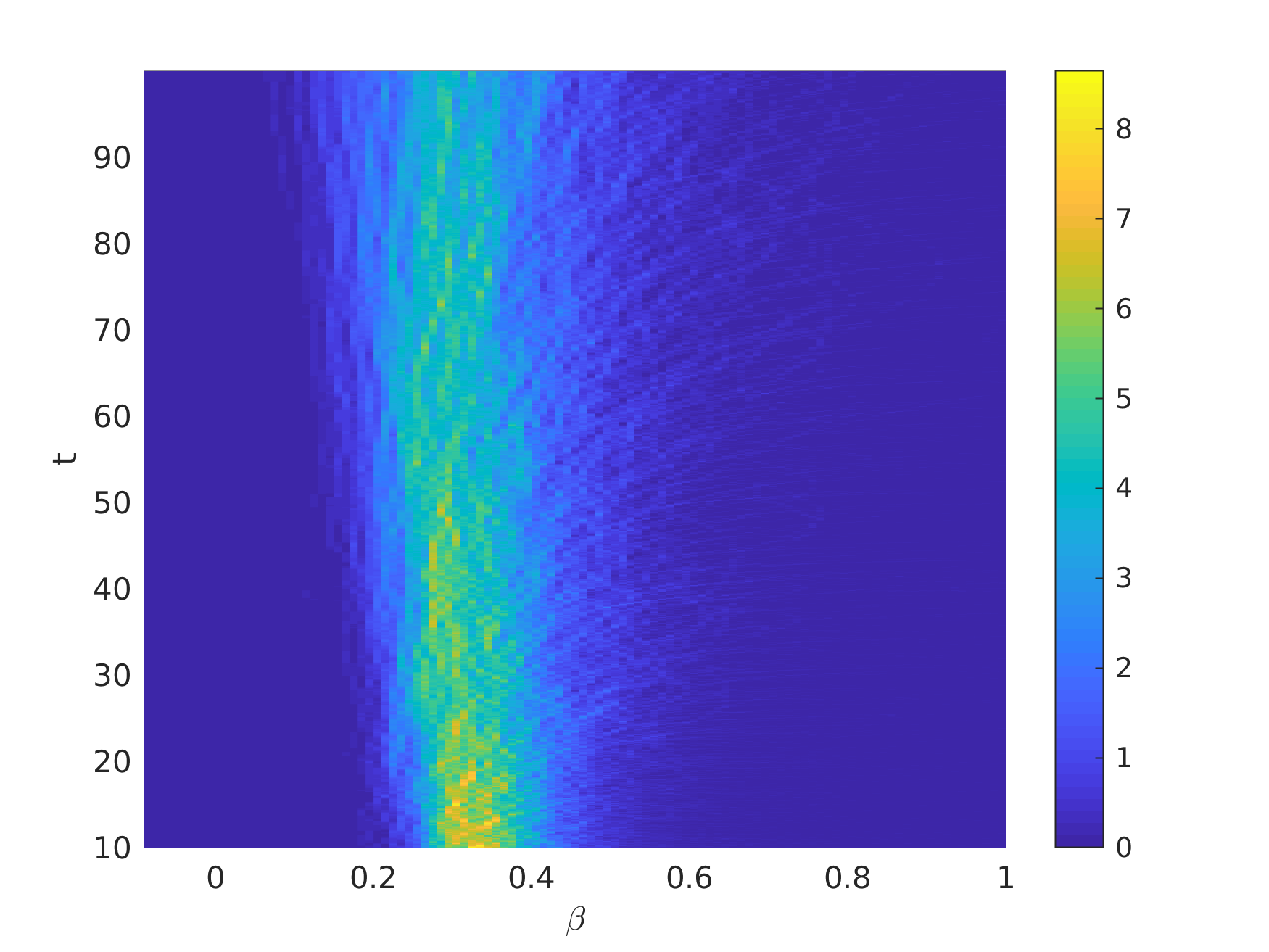}
    \caption{$4\pi \times 4\pi$}
    \label{fig7:spacetime4pi512Cooke}
    \end{subfigure}
    \vskip\baselineskip
    \begin{subfigure}[b]{0.49\textwidth}
        \centering
        \includegraphics[width=1.0\textwidth]{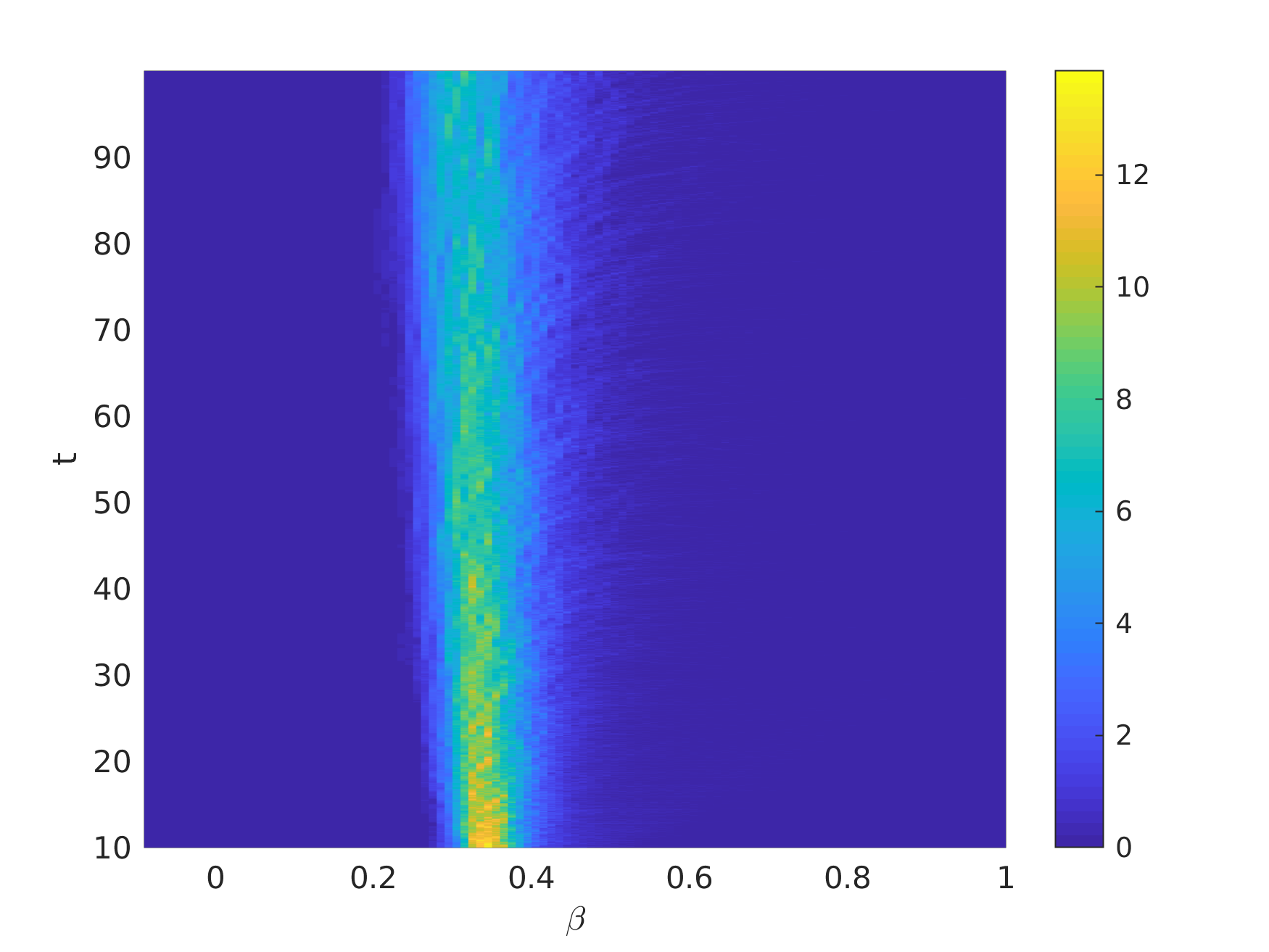}
    \caption{$8\pi \times 8\pi$}
    \label{fig7:spacetime8pi1024Cooke}
    \end{subfigure}
    \begin{subfigure}[b]{0.49\textwidth}
        \centering
        \includegraphics[width=1.0\textwidth]{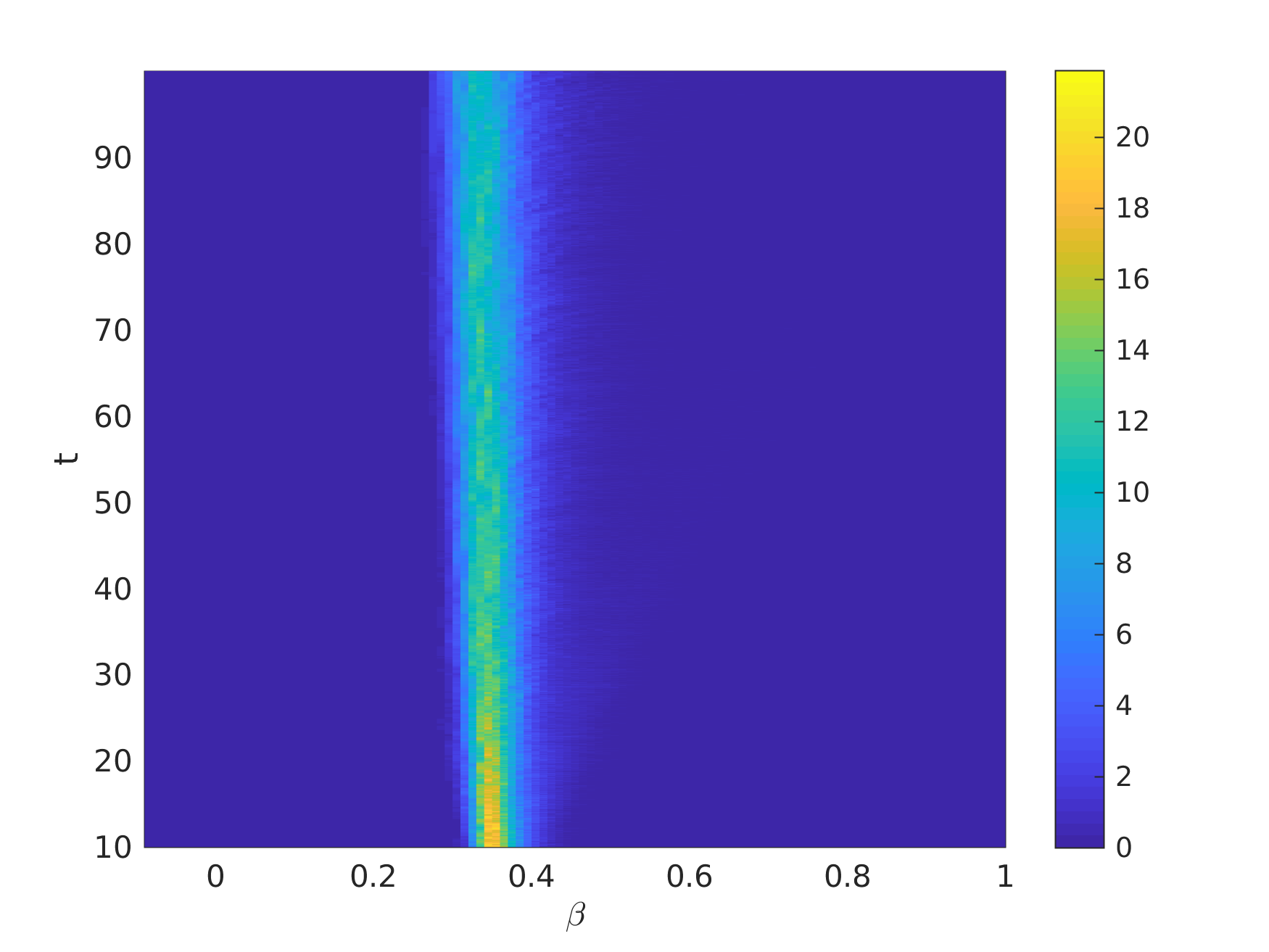}
    \caption{$16\pi \times 16\pi$}
    \label{fig7:spacetime16pi2048Cooke}
    \end{subfigure}
        \caption{\Acomment{Distribution--time} plots of $\beta$ for the Cahn--Hilliard--Cook equation with a symmetric mixture with varying domain sizes.}
    \label{fig7:spacetimeCooke}
\end{figure}

\begin{figure}
    \centering
    \begin{subfigure}[b]{0.49\textwidth}
        \centering
        \includegraphics[width=1.0\textwidth]{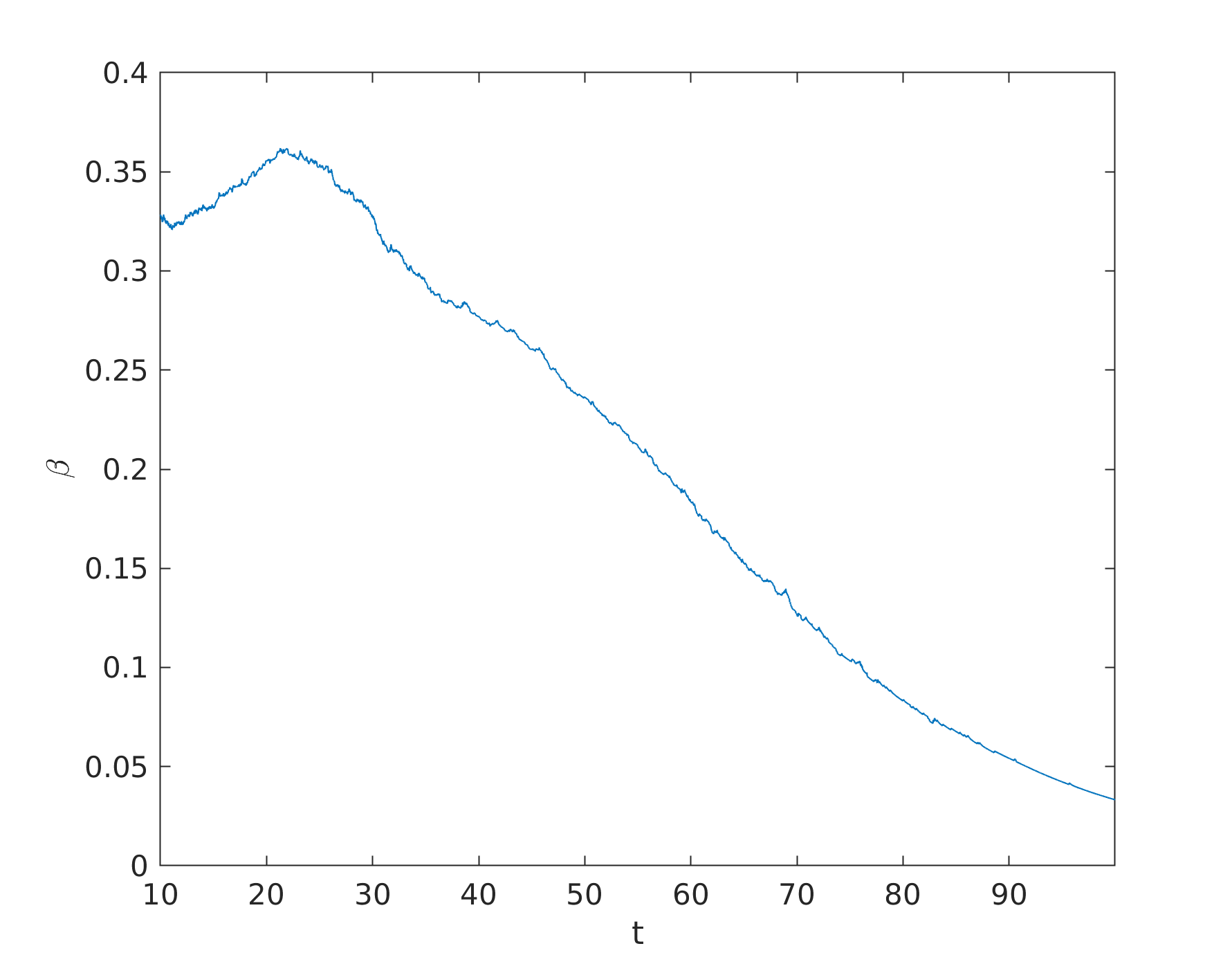}
    \caption{$2\pi \times 2\pi$}
    \label{fig7:mean2pi256Cooke}
    \end{subfigure}
    \hfill
    \begin{subfigure}[b]{0.49\textwidth}
        \centering
        \includegraphics[width=1.0\textwidth]{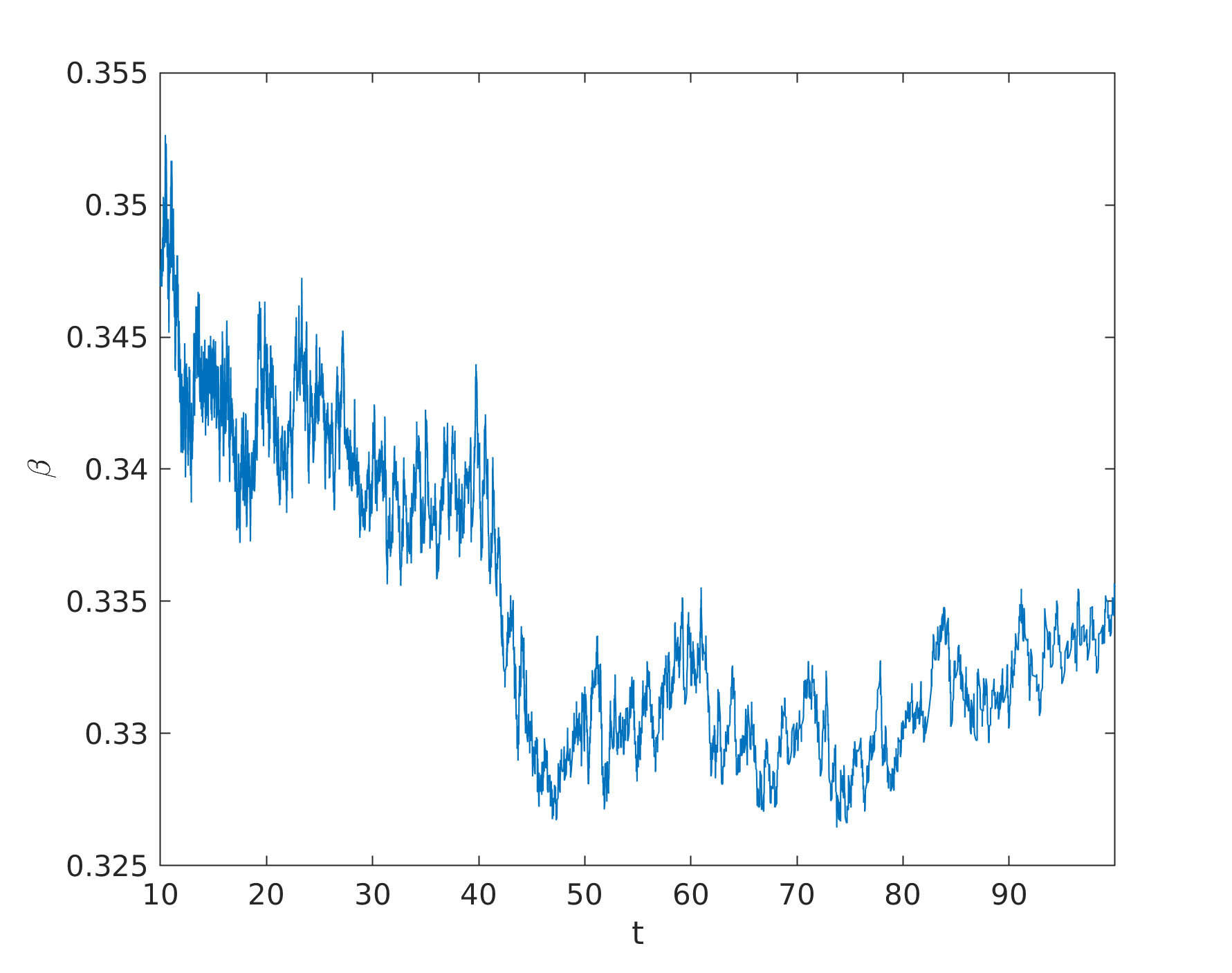}
    \caption{$4\pi \times 4\pi$}
    \label{fig7:mean4pi512Cooke}
    \end{subfigure}
    \vskip\baselineskip
    \begin{subfigure}[b]{0.49\textwidth}
        \centering
        \includegraphics[width=1.0\textwidth]{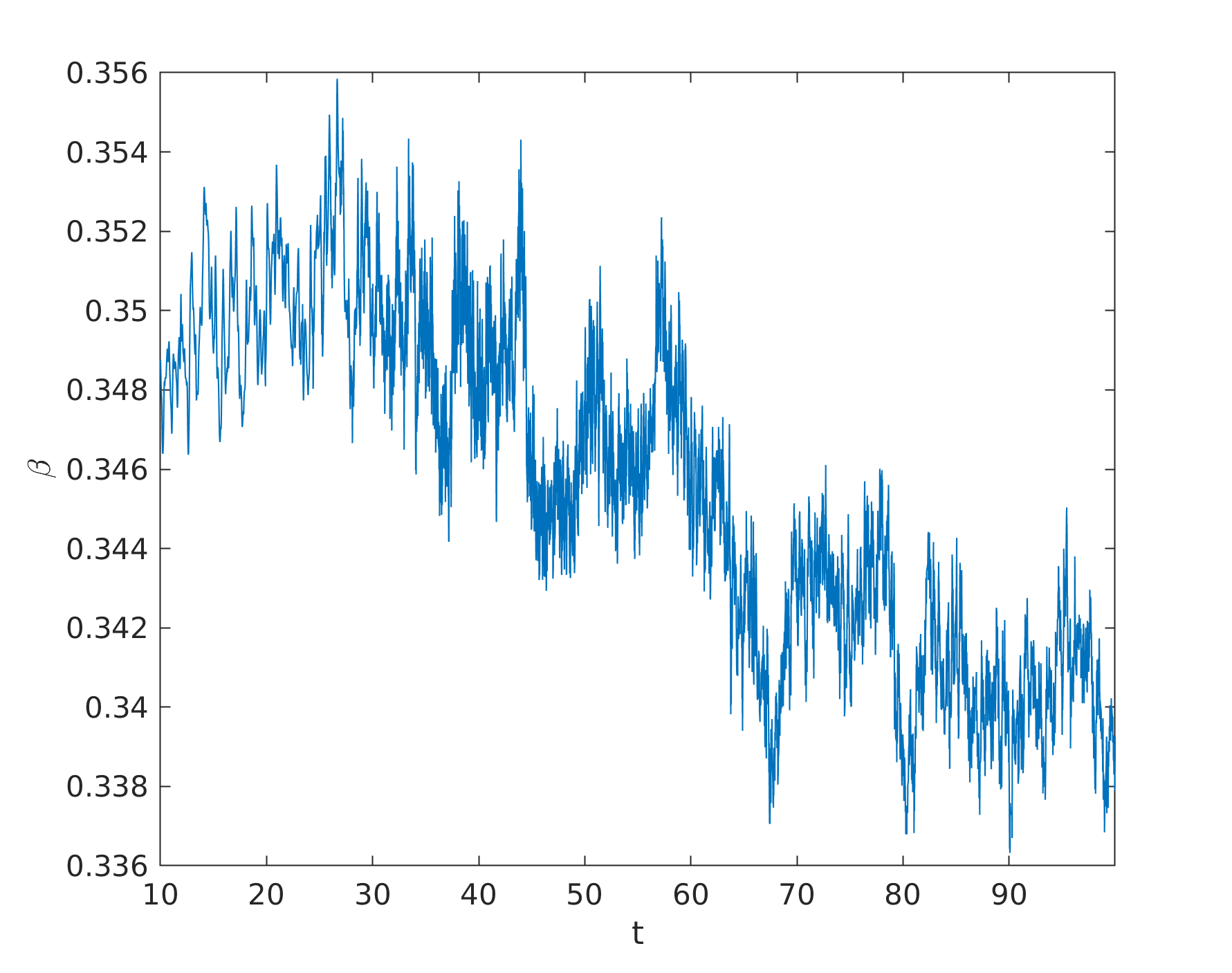}
    \caption{$8\pi \times 8\pi$}
    \label{fig7:mean8pi1024Cooke}
    \end{subfigure}
    \begin{subfigure}[b]{0.49\textwidth}
        \centering
        \includegraphics[width=1.0\textwidth]{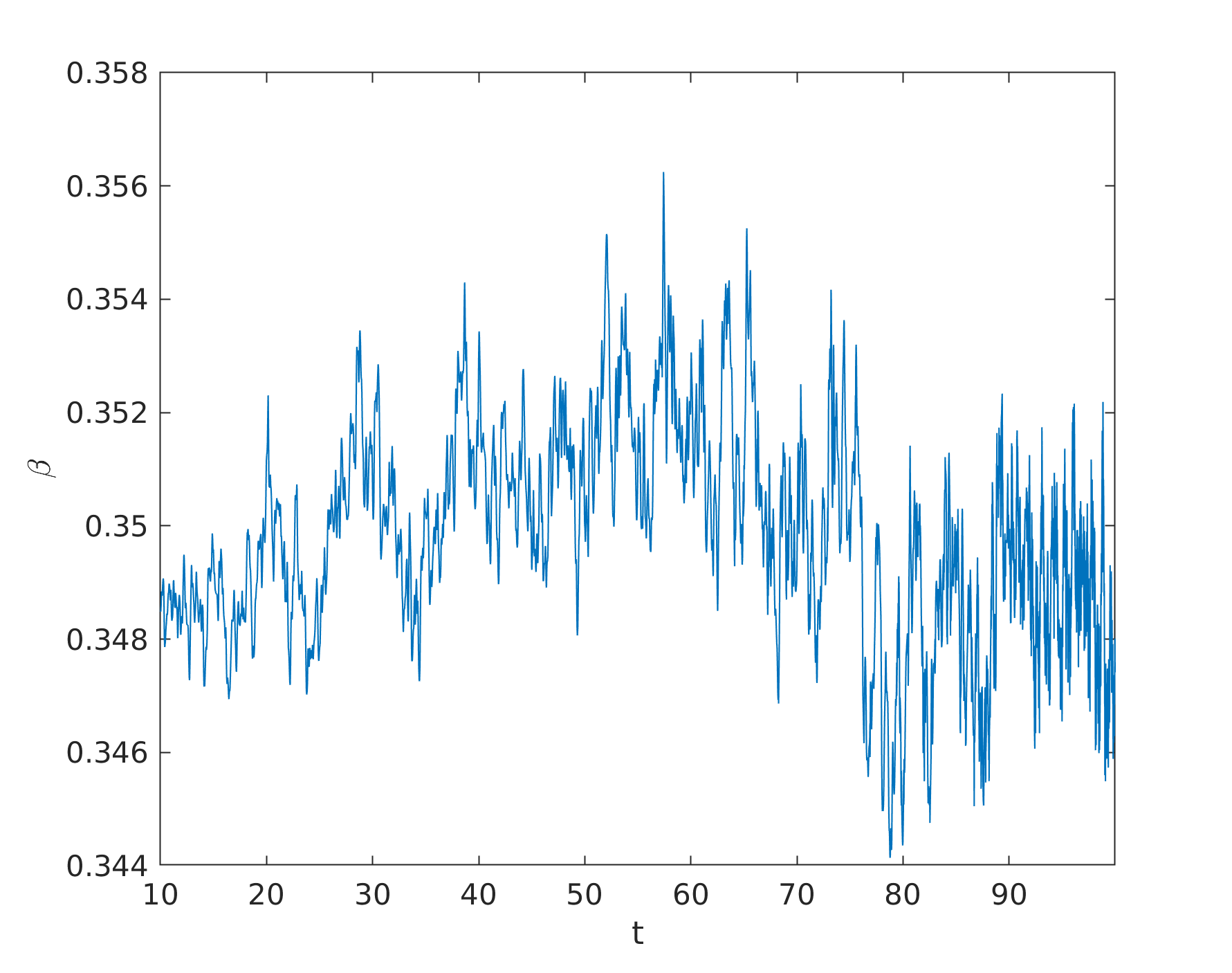}
    \caption{$16\pi \times 16\pi$}
    \label{fig7:mean16pi2048Cooke}
    \end{subfigure}
        \caption{Plots of the mean value of $\beta$ as a function of time for the Cahn--Hilliard--Cook equation with a symmetric mixture over varying domain sizes.}
    \label{fig7:meanCooke}
\end{figure}

\section{Cahn--Hilliard--Cook Equation}
\label{sec7:cooke}
We now extend our analysis to the Cahn--Hilliard--Cook equation, this equation is essentially the same as in equation~\eqref{eq2:ch} except we now have an additional forcing term 
\begin{equation}
\frac{\partial C}{\partial t} = D \nabla^2  \left(C^3-C-\gamma\nabla^2 C\right) + \eta(x, t)
\end{equation}
Here $\eta(x,t)$, the forcing, is a Gaussian noise term to capture the effect of thermal fluctuations within the mixture.
The expectation value of the noise is 
\begin{equation}
\left< \eta(x,t) \right> = 0
\end{equation}
while it is correlated in space by the expression
\begin{equation}
\left< \eta(x_1,t_1) \eta(x_2,t_2) \right> = - \sigma M \delta(t_1 - t_2) \Delta \delta(x_1 - x_2)
\end{equation}
where $\sqrt{\sigma}$ is the intensity of the thermal fluctuation and $M$ is the mobility which we set to $1$.
These conditions imply no overall drift force, a partial correlation of noise in space but no correlation in time and finally the Laplacian operator guarantees the forcing satisfies the overall conservation law at play in the Cahn--Hilliard equation~\cite{cookParallel}.
The Cahn--Hilliard--Cook equation allows us to set the initial condition to $0$ everywhere and repeat our previous analysis except we allow the coarsening be introduced to the system by thermal fluctuations rather than an artificial initial condition, a more physically realistic setting. 
We extend the same GPU methodology and scheme to the Cahn--Hilliard--Cook equation as we have in Section~\ref{sec7:methodology}.
In order to numerically deal with the $\eta(x,t)$ term we follow the methodology as presented in~\cite{cookParallel, cookNuc}, thus the expression of $\eta$ becomes
\begin{equation}
\eta_{i,j} = \sqrt{\frac{\sigma}{\Delta x^2 \Delta t}} \nabla \cdot {\bf{\rho_{i,j}}}
\end{equation}
Where ${\bf{\rho_{i,j}}}$ is a vector field of dimension $2$ made up of numbers drawn from a $N(0,1)$ distribution, this is achieved in practice by drawing numbers from a uniform distribution and using the Box--Muller transformation to populate the vector field at each time step. For the purposes of this chapter we set the value of $\sigma = 1 \times 10^{-14}$ to ensure the forcing is small scale and does not dominate the results, the initial condition is $0$ everywhere and all other parameters are the same as those in the standard Cahn--Hilliard case. The development of the interconnected regions of the domain is the same as those in Figure~\ref{fig7:interconnected} with the dynamics producing a symmetric mixture, we now proceed immediately to the statistical analysis.

\begin{figure}
    \centering
    \begin{subfigure}[b]{0.49\textwidth}
        \centering
        \includegraphics[width=1.0\textwidth]{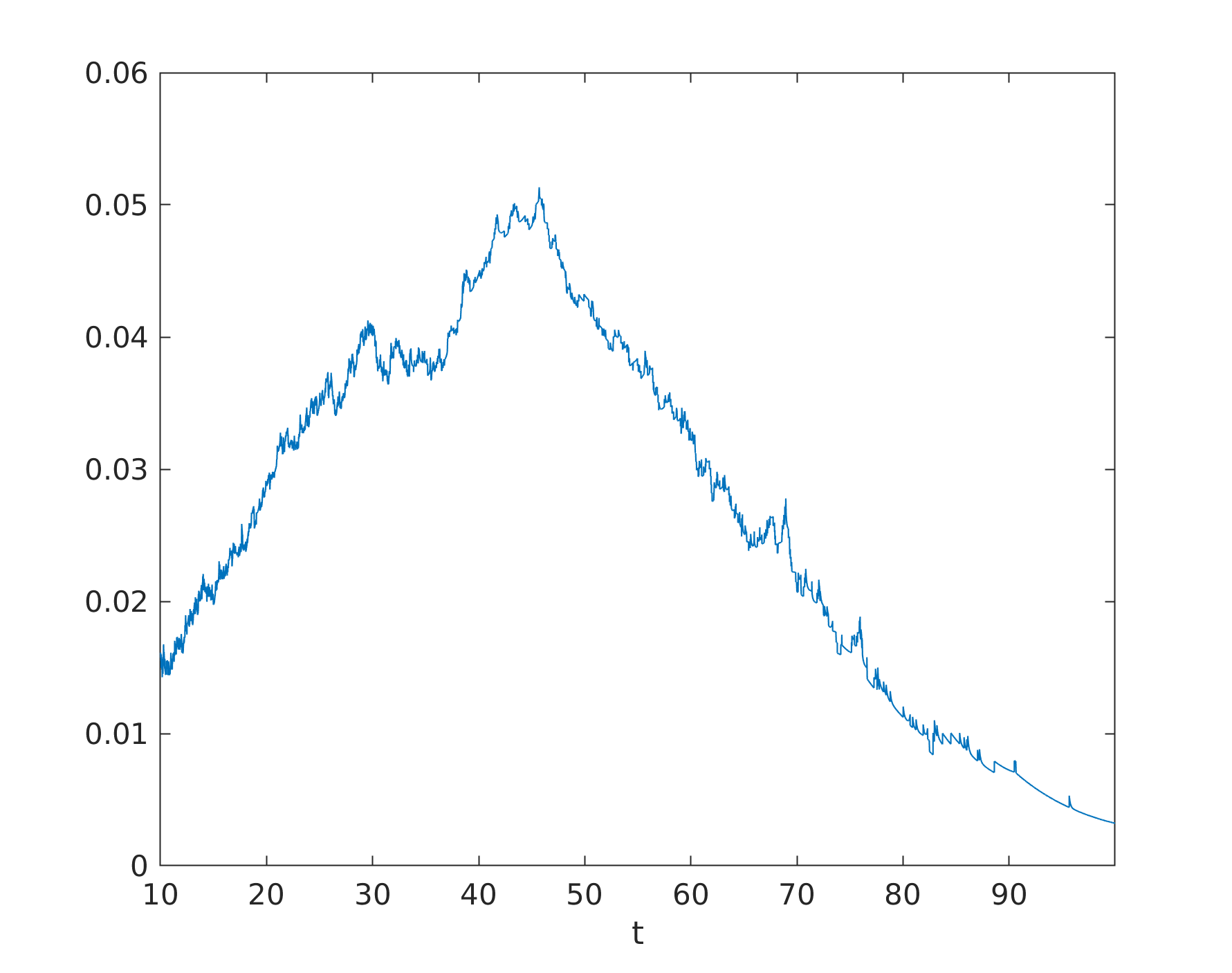}
    \caption{$2\pi \times 2\pi$}
    \label{fig7:variance2pi256Cooke}
    \end{subfigure}
    \hfill
    \begin{subfigure}[b]{0.49\textwidth}
        \centering
        \includegraphics[width=1.0\textwidth]{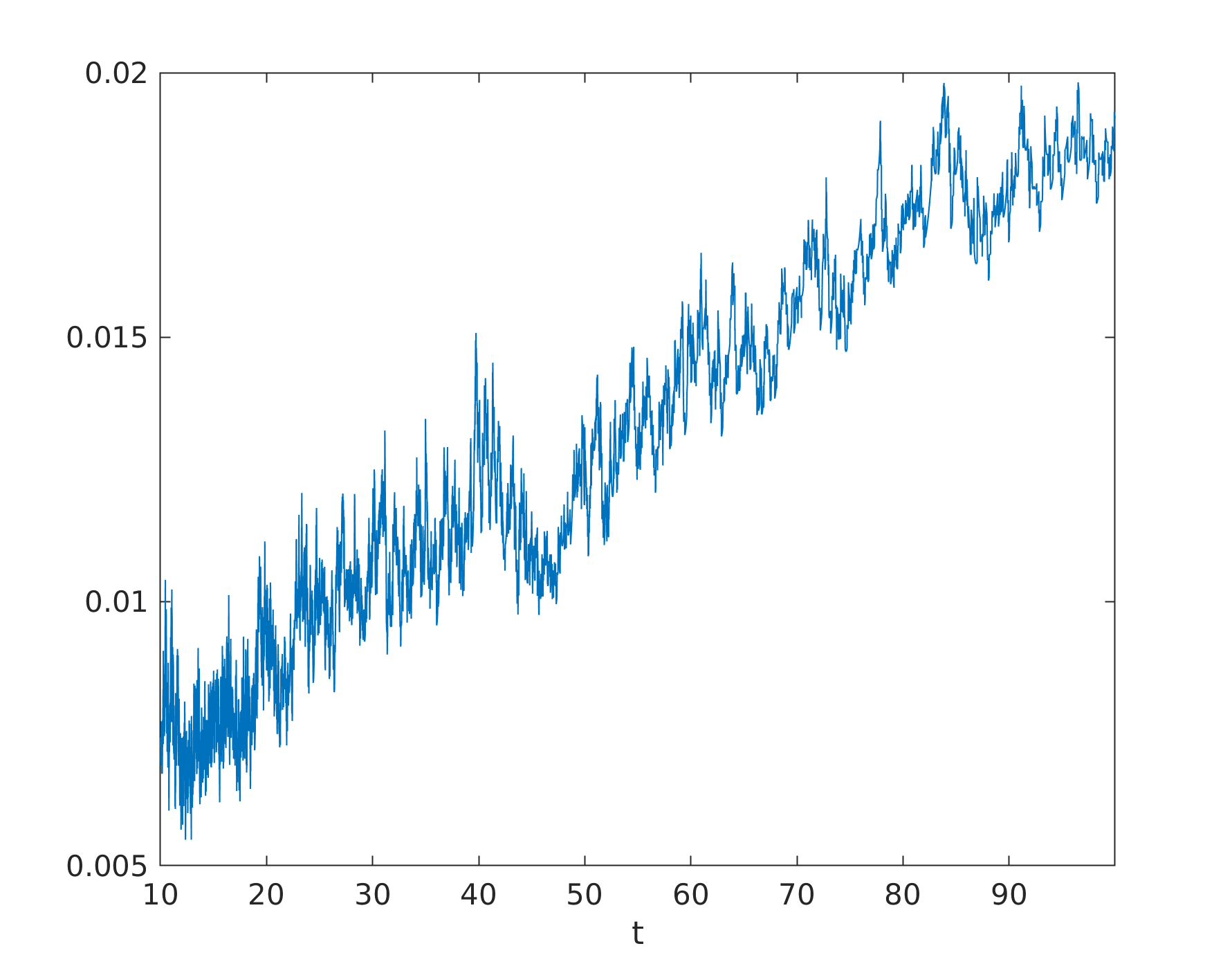}
    \caption{$4\pi \times 4\pi$}
    \label{fig7:variance4pi512Cooke}
    \end{subfigure}
    \vskip\baselineskip
    \begin{subfigure}[b]{0.49\textwidth}
        \centering
        \includegraphics[width=1.0\textwidth]{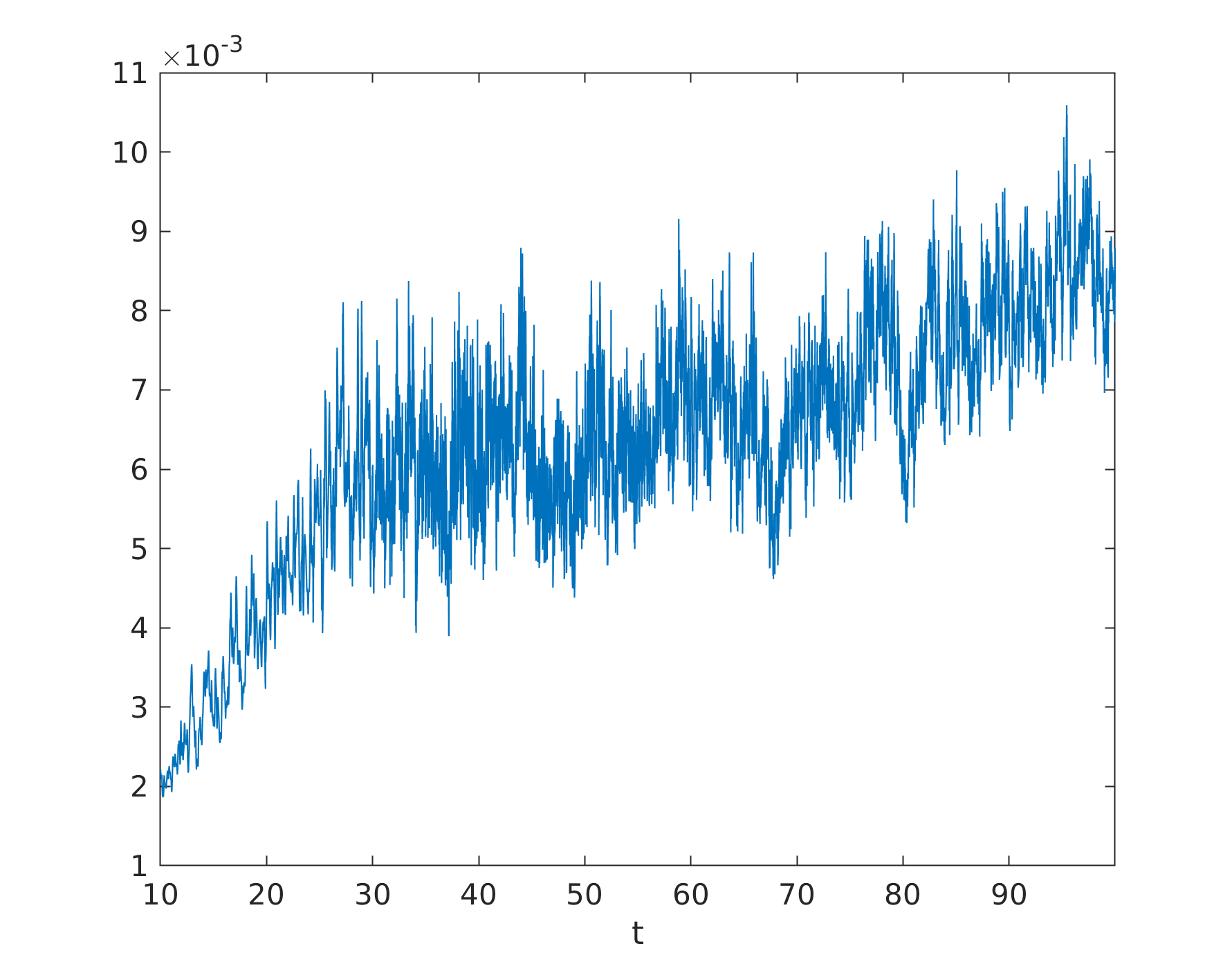}
    \caption{$8\pi \times 8\pi$}
    \label{fig7:variance8pi1024Cooke}
    \end{subfigure}
    \begin{subfigure}[b]{0.49\textwidth}
        \centering
        \includegraphics[width=1.0\textwidth]{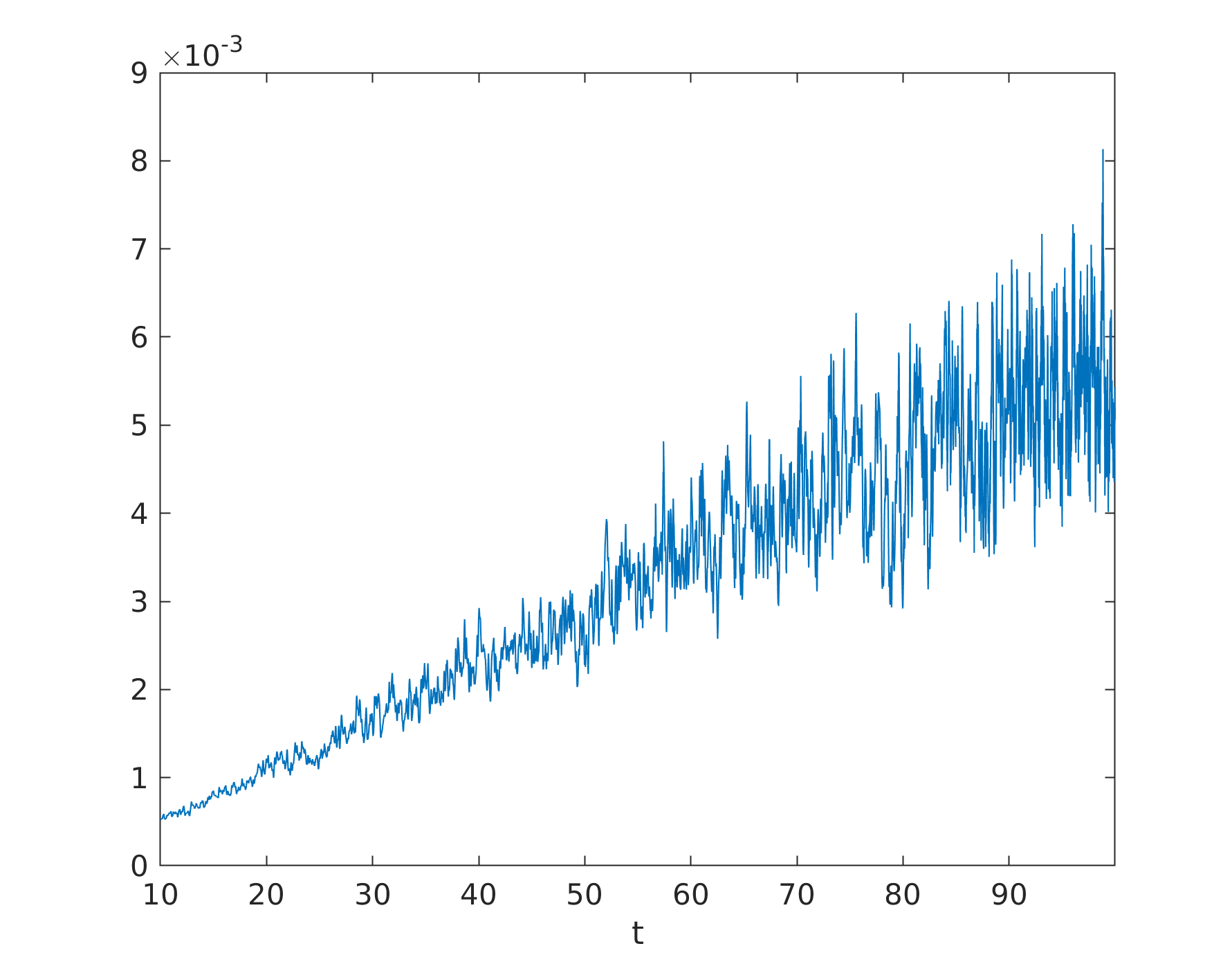}
    \caption{$16\pi \times 16\pi$}
    \label{fig7:variance16pi2048Cooke}
    \end{subfigure}
        \caption{Plots of the variance of $\beta$ as a function of time for the Cahn--Hilliard--Cook equation with a symmetric mixture over varying domain sizes.}
    \label{fig7:varianceCooke}
\end{figure}

We begin by studying the \Acomment{distribution--time} plots of the simulations in Figure~\ref{fig7:spacetimeCooke}.
The effect of the finite size of the domain can be seen particularly the $2\pi \times 2\pi$ domain where the majority of the simulations experience their growth rate $\beta$ going to $0$ as the domains stop growing due to hitting late stage, steady state, finite size effects such as those developing in Figure~\ref{fig7:contour500cahn}.
It can be seen clearly that as the domain size is increased the distribution moves towards the expected value of $\beta = 1/3$ and the distribution narrows.
The narrowing of the distributions and the growth of the peaks is emphasised in the growth of the maximum values in the \Acomment{distribution--time} plots, captured in the colour--bars where between $4\pi$ and $16\pi$ the peak increases by a factor of $\approx 2.5$. 
We expect this trend to continue as the domain size is increased.

We can see again how the domain of size $2\pi \times 2\pi$ has its growth rate reduced to a distribution around $0$ in Figure~\ref{fig7:meanCooke}.
The values of the mean rapidly drops off in time, for the larger domains we can see much clearer trends around $1/3$ and above. 
The values above $1/3$ are due to the positive skew which is discussed below.
As the domain size is increased the various plots of the mean also make clear how the distribution of $\beta$ becomes steadier, levelling out at large values of $t$.
The steadiness of the distribution improves with increasing domain size.
Thus not only are the position and height functions of the finite size effects but also the steadiness of the distribution.

\begin{figure}
    \centering
    \begin{subfigure}[b]{0.49\textwidth}
        \centering
        \includegraphics[width=1.0\textwidth]{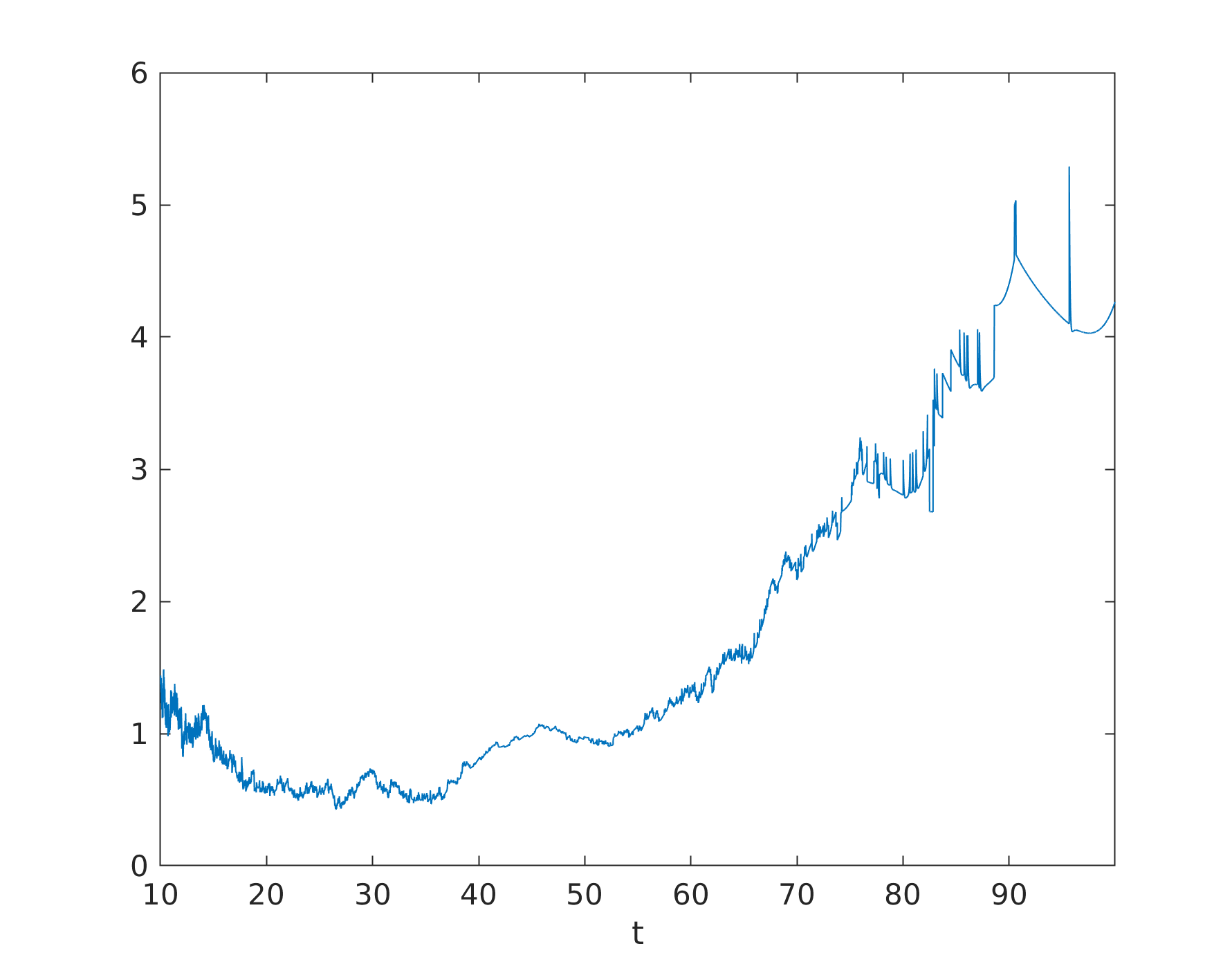}
    \caption{$2\pi \times 2\pi$}
    \label{fig7:skew2pi256Cooke}
    \end{subfigure}
    \hfill
    \begin{subfigure}[b]{0.49\textwidth}
        \centering
        \includegraphics[width=1.0\textwidth]{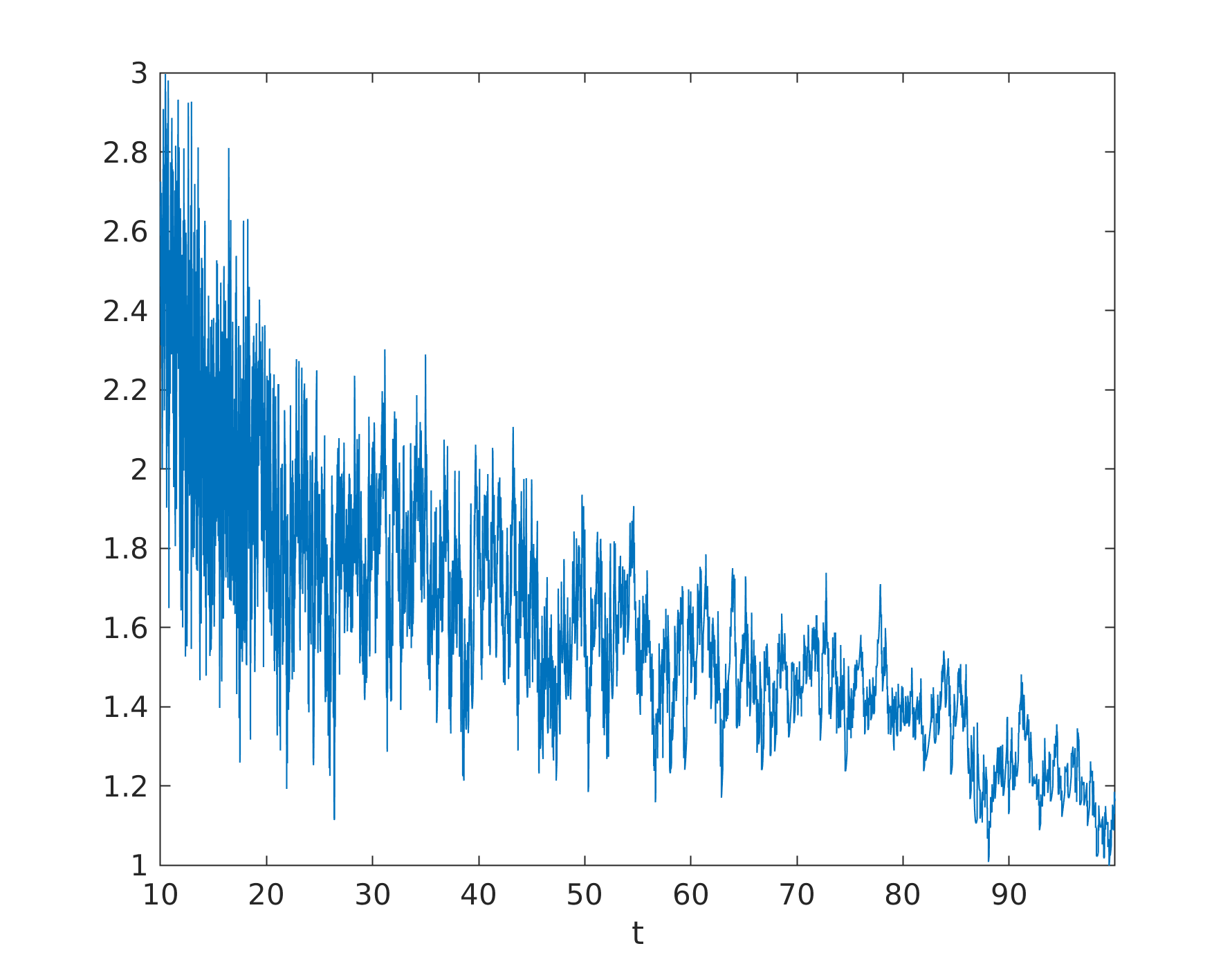}
    \caption{$4\pi \times 4\pi$}
    \label{fig7:skew4pi512Cooke}
    \end{subfigure}
    \vskip\baselineskip
    \begin{subfigure}[b]{0.49\textwidth}
        \centering
        \includegraphics[width=1.0\textwidth]{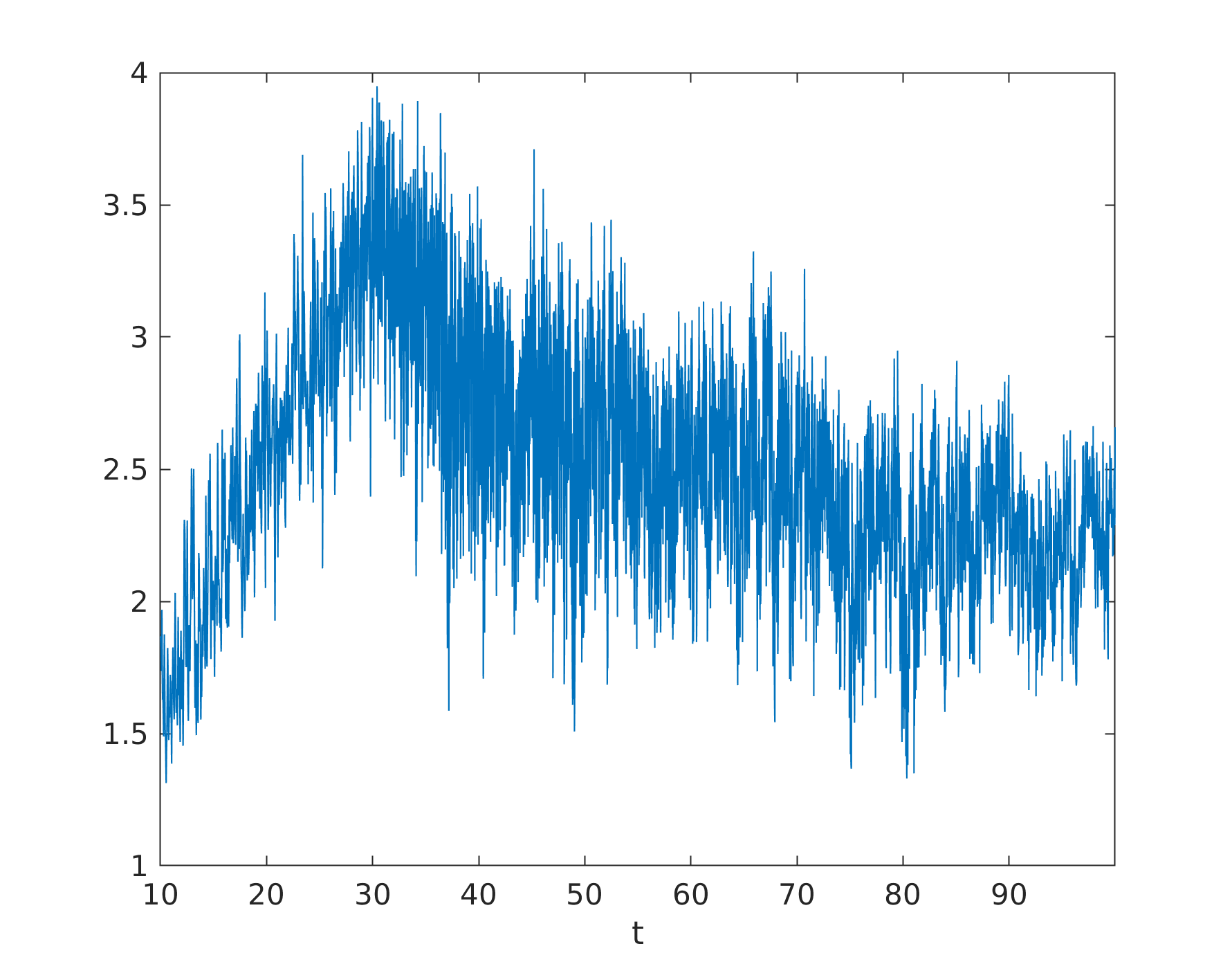}
    \caption{$8\pi \times 8\pi$}
    \label{fig7:skew8pi1024Cooke}
    \end{subfigure}
    \begin{subfigure}[b]{0.49\textwidth}
        \centering
        \includegraphics[width=1.0\textwidth]{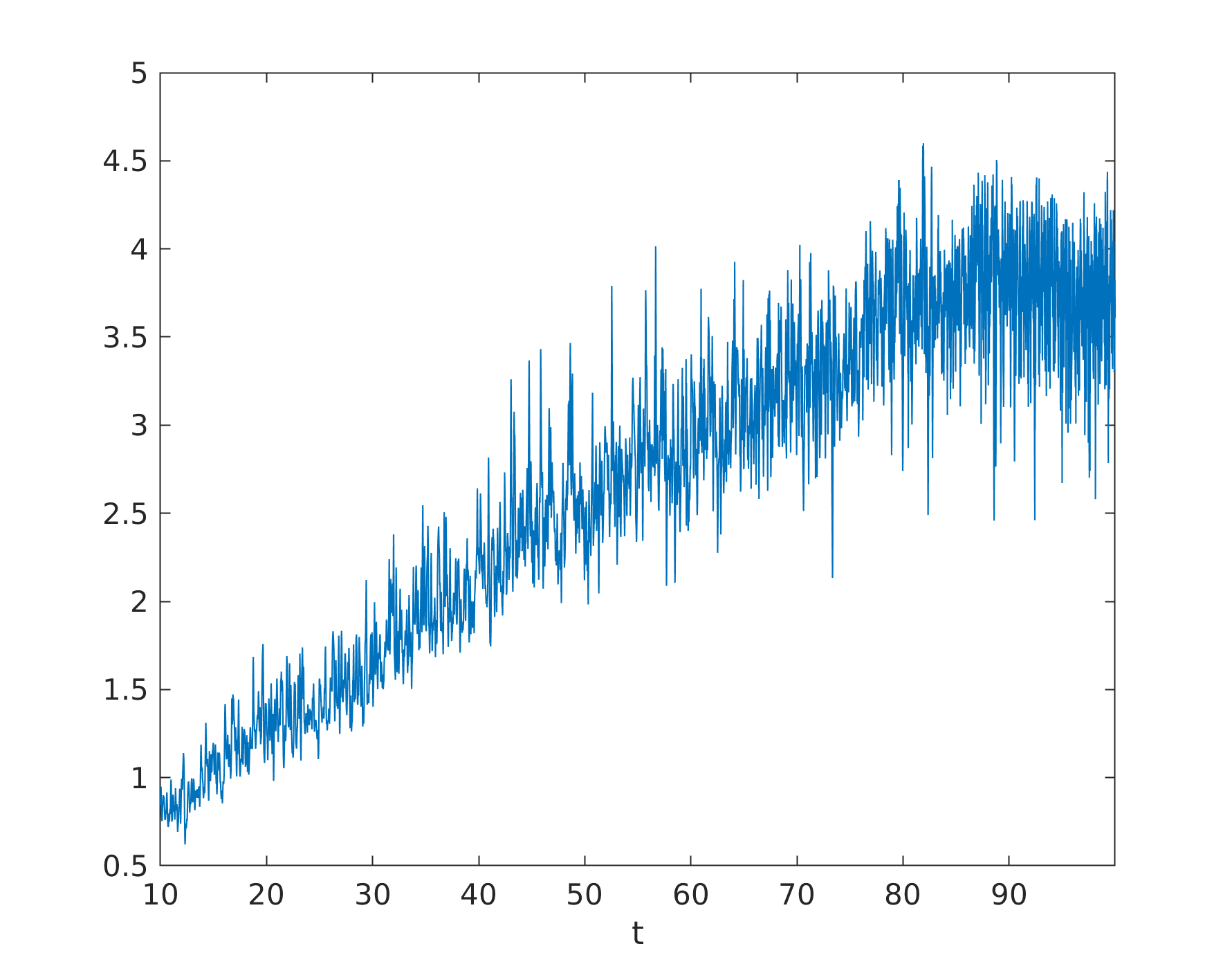}
    \caption{$16\pi \times 16\pi$}
    \label{fig7:skew16pi2048Cooke}
    \end{subfigure}
        \caption{Plots of the skewness of $\beta$ as a function of time for the Cahn--Hilliard--Cook equation with a symmetric mixture over varying domain sizes.}
    \label{fig7:skewCook}
\end{figure}

The variance of $\beta$ for the case of a symmetric Cahn--Hilliard--Cook mixture can be seen in Figure~\ref{fig7:varianceCooke}.
The time dependence of the variance is clearly evident in these plots, particularly at early times, thus we can say that the underlying dynamics of the system is driven by anomalous diffusion. 
The development of a steady state with respect to the variance can be seen in the $8\pi \times 8\pi$ plot in Figure~\ref{fig7:variance8pi1024Cooke}, this steady state suggests that at late times the diffusion is constant rather than anomalous. 
We can see that the $16\pi \times 16 \pi$ domain has yet to reach this steady state, while the two smaller domains $2\pi$ and $4\pi$ are experiencing greater movement in variance due to experiencing considerably earlier stage finite size effects in some of their simulations.
This is particularity clear in Figure~\ref{fig7:spacetime4pi512Cooke} where the width of the distribution is extremely wide due a diverse range of simulation behaviours, some approaching finite size effects like those in Figure~\ref{fig7:interconnected} while others are still growing.
A broad range of behaviours is also evident in the earlier development of the $2\pi$ case shown in Figure~\ref{fig7:spacetime2pi256Cooke}.

\begin{figure}
    \centering
        \includegraphics[width=0.6\textwidth]{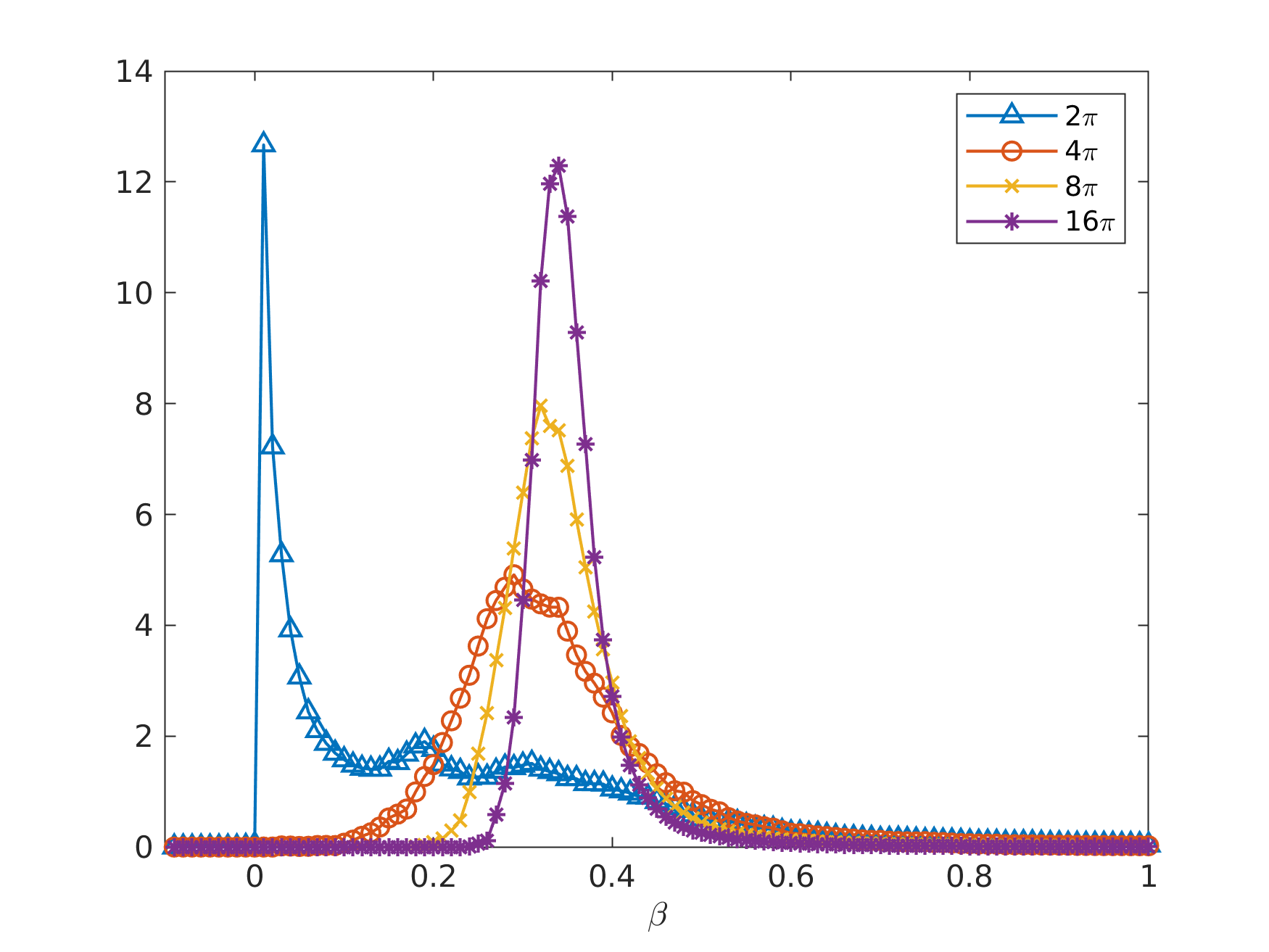}
        \caption{Plot showing the evolution of the distribution of $\beta$ for the Cahn--Hilliard--Cook equation with increasing domain size, we can see a clear similarity between these results and those extracted from the ODE model in Figure~\ref{fig7:betaVaryN}.}
    \label{fig7:cookeBeta}
\end{figure}

In Figure~\ref{fig7:skewSymmetric} we present the skewness of $\beta$ as a function of time.
It is clear in all cases that the distribution is indeed positively skewed, this ensures that no negative growth rates appear which is an expected result, there should be no reduction is domain length scales in any of the simulations.
Like the variance the skewness is time dependent and again we can see the development of a steady state in the $8\pi \times 8\pi$ case shown in Figure~\ref{fig7:skew8pi1024}.
Again the $2\pi$ and $4\pi$ results are dominated by early stage finite size effects and the $16\pi$ domain has yet to reach its steady phase of growth.

Finally in Figure~\ref{fig7:symmetricBeta} we plot the histograms of all values of $\beta$ presented in this section for various domain sizes, the histograms are sampled across all values of $\beta$ between $10 < t < 100$ shown in this section.
We can see a clear similarity between the results presented here and those for the ODE simulations of the bubble model presented in Figure~\ref{fig7:betaVaryN}.
The smallest domain/number of bubbles has a distribution centred at $0$ and then as the number of bubbles/size of the domain is increased the distribution moves toward $1/3$.
In this study of the Cahn--Hilliard--Cook equation we can see a clear trend of increasing height of the distribution and narrowing around $1/3$ as domain size increases.
\Acomment{
It again evident here, as in the symmetric Cahn--Hilliard results that as the domain size increases the distribution should approach a $\delta$ function at $\beta = 1/3$ as predicted by our LSW results for Ostwald ripening on a finite domain.
}

Summarising our results for the evolution of symmetric mixtures simulated with the \Acomment{Cahn--Hilliard--Cook} equation we have shown how $\beta$ develops with increasing domain size is in--keeping with the results for $\beta$ when increasing bubble numbers in the ODE model.
The positions and steadiness of the distributions are dependent on the finite domain size and in particular the steadiness improves with increasing domain size. 
The finite size of the domain leads to a smearing of the $\delta$ function at $1/3$ predicted by LSW theory and in the limit of an infinite domain they will indeed approach the LSW predictions.


\section{Discussion and Conclusions}
\label{sec7:disc}
\Acomment{In this section we first present a brief statistical convergence study to show that our statistical results in the previous sections are independent of grid resolution.
We then present a unified discussion to provide} a cohesive picture of the Cahn--Hilliard equation results versus the predictions of LSW Theory and the ODE simulation results and how finite bubble numbers/finite domain sizes effect the $t^{1/3}$ scaling result.
Following this the derivation of a stochastic model is presented which captures the dynamics of the results and provides a unified theory to explain the results presented throughout this chapter.
Finally we present some concluding remarks and potential future avenues of work.

\begin{figure}
    \centering
        \includegraphics[width=0.6\textwidth]{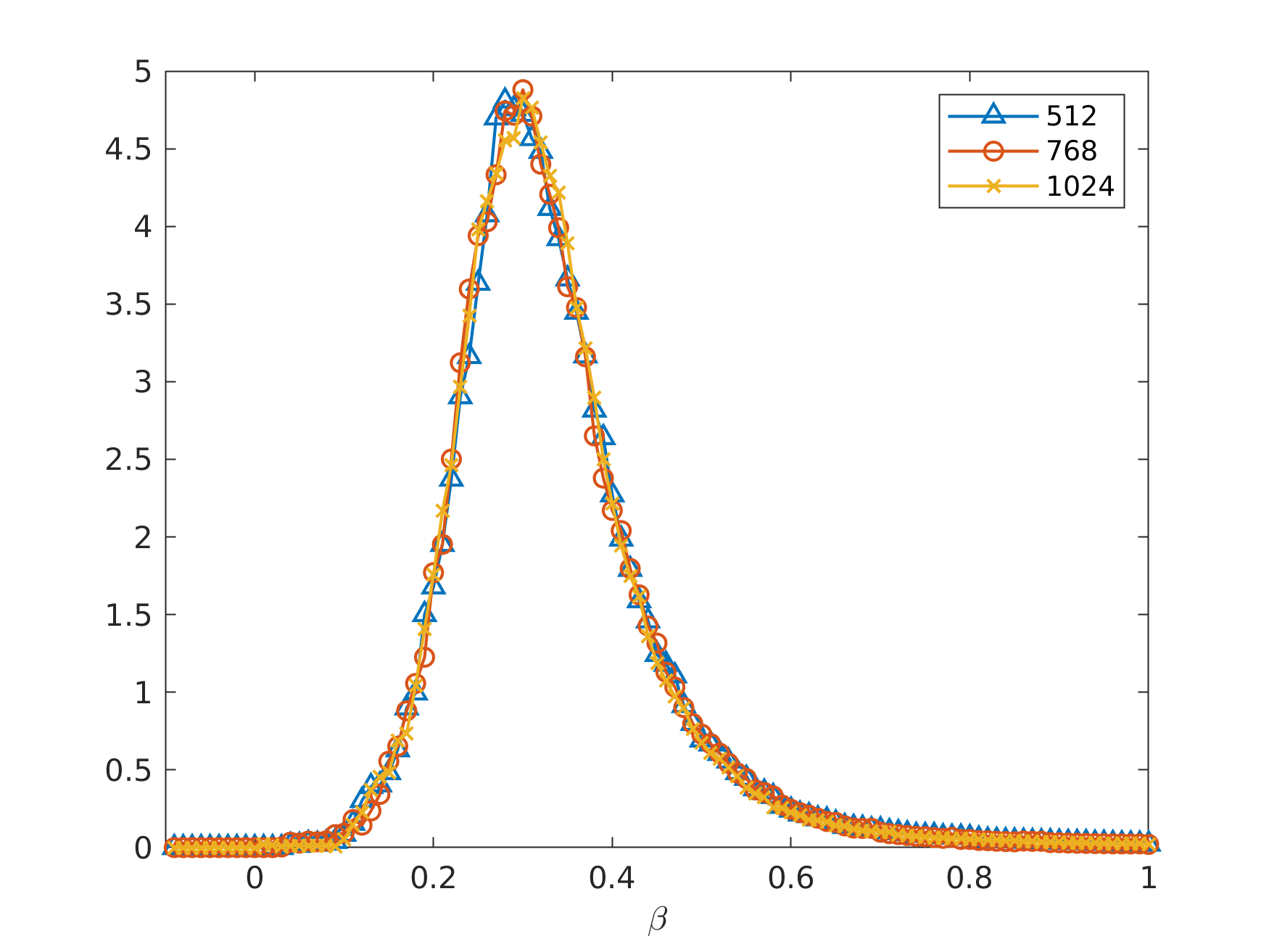}
        \caption{\Acomment{Plot showing the distribution of $\beta$ for a range of domain resolutions on a $4\pi \times 4 \pi$ domain. The initial condition was a symmetric mixture. We can see from this plot that the statistical distribution of $\beta$ due to finite size effects is independent of grid resolution.}}
    \label{fig7:betaConvergeNumerical}
\end{figure}

\Acomment{
\subsection{Grid independence of statistical picture}
In order to show grid independence of the statistics presented we examine a batch of simulations with symmetric mixture on a $4\pi \times 4\pi$ domain.
The batches of simulations are completed with resolutions of $512 \times 512$, $768 \times 768$ and $1024 \times 1024$ points.
In Figure~\ref{fig7:betaConvergeNumerical} we can see the results of the distribution of $\beta$ for $10 < t < 100$. 
It is clear from this plot that the distribution of $\beta$ due to finite size effects is independent of grid resolution and thus is not a numerical feature.
We now summarise the results presented in this chapter and then provide a physical model to explain this statistical behaviour.
}

\subsection{Summary of Results: Cahn--Hilliard vs. LSW Theory vs. ODE Simulations}
Throughout our four independent studies of coarsening phenomena presented in Sections~\ref{sec7:drop_pop},~\ref{sec7:CH},~\ref{sec7:CH_symm} and~\ref{sec7:cooke} we witnessed many repeated features and trends for the distributions of $\beta$.
It is this repetitive nature that allows us to now draw concrete universal conclusions to all of these results.
The presence of a finite domain/finite number of bubbles, for which we now refer for this discussion to as finite size effects, has a clear and precise impact on the distribution of $\beta$.
The finite size smears the $\delta$ function at $\beta = 1/3$, predicted by LSW theory, into a distribution which peaks at a value of between $0$ and $1/3$.
This distribution is always positively skewed, in keeping with the physics that the domain is always coarsening, thus the length--scales are increasing.
The distribution has a variance that is \Acomment{time--dependent}, thus the underlying dynamics is driven by anomalous diffusion, in other words a non--constant diffusion coefficient.
As the finite size of the domain is made larger, the distribution will centre with a peak at $1/3$ and this peak will grow and narrow to yield the predicted $\delta$ function.
Finally the steadiness of the distribution improves with increased domain sizes, this is reflected in the behaviour of the various statistical moments captured in this study.

Now that we have established a clear picture of how $\beta$ behaves in a finite domain we now present a unified model capturing the behaviour of Ostwald Ripening, the Cahn--Hilliard equation with symmetric and asymmetric mixtures and the Cahn--Hilliard--Cooke equation.
The theory relies on ideas similar to that of Kolmolgorov and his famous 1941 theory of turbulence~\cite{kol1941} and his lesser known 1962 extensions to that theory~\cite{kol1962}.
Kolmogorov's 1941 theory provides a universal scaling law for regions of flow that are sufficiently small when compared to the global size of the domain.
So in other words in a large flow there should appear subregions which themselves are 'typical' in a statistical sense. 
Here `typical' meaning statistically stationary, homogeneous and isotropic with $E(k) \propto k^{-5/3}$.
But in reality any single arbitrary region has its own scaling which depends on local features of the flow.
Kolmogorov states this in his 1962 theory where he revises his arguments from 1941 and states that $E(k) \propto k^{-5/3}$ is asymptotically valid~\cite{kol1962, kraich1974} and in reality for a given subregion of the flow $E(k) \propto k^{-5/3 - \mu}$ where $\mu$ depends on the specific local non-linear interactions within the Navier-Stokes equation. 
The dependence of a given scaling of a subregion on local values also means that this distribution is not a feature of the central limit theorem.

It is this idea we extend to the area concerned, a typical solution of the Cahn--Hilliard equation will not exhibit a scaling of exactly $t^{1/3}$ and instead scale like $t^{\beta}$.
But where $\beta$ is a value drawn from distribution with the properties expressed above, a distribution with a mean of $1/3$, a time dependent variance and a positive skew.
In addition to this, the distribution will narrow with increasing domain size, this behaviour is much like taking more and more averages of homogeneous regions in the Kolomogorov theory, the larger the domain the more typical it is and thus the more \Acomment{likely} it is for the scaling to have a value of $\beta = 1/3$.
We now present a theory to extract such a model.

\subsection{Stochastic Model for Symmetric Mixture Scaling}
\label{sec7:stochastic}
We examine the growth of an interface in a generic Cahn--Hilliard system, on the interface between two phases of the mixture we have that
\begin{equation}
C(\vecx_I, t) = 0
\end{equation}
In order to study how this interface develops we take the time derivative to give
\begin{equation}
\left.\frac{\partial C}{\partial t}\right|_{\vecx_i} + \frac{\partial \vecx_I}{\partial t} \cdot \nabla C = 0
\label{eq7:interTimeDer}
\end{equation}
From here the normal interface velocity an be identified as $V = \frac{\partial \vecx_I}{\partial t} \cdot \hat{\vecn}$ and we can rearrange equation.~\eqref{eq7:interTimeDer} to give
\begin{equation}
V = \left. - \frac{1}{|\nabla C|} \frac{\partial C}{\partial t} \right|_{\vecx_I}
\end{equation}
Approximating $|\nabla C|$ using jump in concentration between the two phases across an interface of width $l$ characterised by $\sqrt{\gamma}$ gives us the expression
\begin{equation}
|\nabla C| \approx \frac{C_+ - C_-}{l}
\end{equation}
Similarly approximating \Acomment{the} Laplacian term at the interface with
\begin{equation}
\nabla^2 \mu \approx \frac{\hat{\vecn} \cdot \nabla \mu_+  - \hat{\vecn} \cdot \nabla \mu_-}{l}
\end{equation}
allows us to substitute equation~\eqref{eq2:ch} in our expression for $V$ and then recast the interface velocity in terms of Mullins-Sekerka Dynamics
\begin{equation}
V = - \frac{[\hat{\vecn} \cdot \nabla \mu]}{[C]}
\label{eq7:vel}
\end{equation}
where $[A] = A_+ - A_-$ is the jump the value of quantity $A$ across an interface.
Rigorous application of Mullins-Sekerka Dynamics to the Cahn-Hilliard equation by Pego in~\cite{pego1989} yielded the same result as equation~\eqref{eq7:vel}.
This was achieved using expansions of order $\gamma t$ to examine the time scale of interface migration.
This examination additionally led to a relation between chemical potential at the interface and the curvature of the interface $\kappa$

\begin{equation}
\mu_I = - \frac{S \kappa}{[C]}
\label{eq7:pego}
\end{equation}
where $S$ is the surface tension coefficient, equation~\eqref{eq7:vel} can then be recovered from further analysis. 
Physically these results tells us that interface motion is due to diffusion away from the interface caused by a mismatch in mass flux. 
We can now use these results to recover the standard $O(t^{1/3})$ scaling result of the Cahn--Hilliard equation. 
We relate the interface velocity to the rate of change of an arbitrary bubble radius $R_b$, where the radius is defined in the sense of the signed distance function to the interface when dealing with \Acomment{a} symmetric mixture or an actual bubble radius when dealing with Ostwald Ripening 
\begin{equation}
V = \frac{d R_b}{dt}
\end{equation}
The jump in normal derivative across the interface can also be related to the bubble radius by 
\begin{equation}
[\hat{\vecn} \cdot \nabla \mu] \approx \frac{\mu_I}{R_b}
\end{equation}
Combining these along with Eq.~\eqref{eq7:vel} and Eq.\eqref{eq7:pego} gives
\begin{equation}
\frac{d R_b}{d t} = \Gamma \frac{1}{R_b^2}
\end{equation}
where $\Gamma = S / [C]^2$ is a constant and we have also used that $R_b = 1 / \kappa$.
Integrating and allowing $t \rightarrow \infty$ the standard bubble scaling result for the Cahn-Hilliard equation is recovered
\begin{equation}
R_b \propto t^{\frac{1}{3}}
\label{eq7:StandardScale}
\end{equation}

Taking the value of $\Gamma$ above as a constant is a idealisation, actual bubbles experience a local value which changes in time.
This local value is where connection with Kolomogorov can be seen as considering it to be non--constant will  lead to a distribution of $\beta$ around a value of $1/3$.
We denote the desired local value as $\Gamma_t$. 
So now for the bubble radius we have a new version of equation~\eqref{eq7:StandardScale} where substituting in $\Gamma_t$ gives
\begin{equation}
R_b = \Gamma_t^{\frac{1}{3}} t^{\frac{1}{3}}
\end{equation}
Taking the difference of this value between $t$ and $t + \Delta t$ we get
\begin{equation}
R_b(t + \Delta t) - R_b(t) = \Gamma_{t+\Delta t}^{\frac{1}{3}} (t + \Delta t)^{\frac{1}{3}} - \Gamma_t^{\frac{1}{3}} t^{\frac{1}{3}}
\label{eq7:diffRb}
\end{equation}
$\Gamma_{t + \Delta t}^{\frac{1}{3}}$ can be eliminated using
\begin{equation} 
\Gamma_{t + \Delta t}^{\frac{1}{3}} - \Gamma_{t}^{\frac{1}{3}} = \sigma \Gamma_{t}^{\frac{1}{3}} \Delta w_t
\end{equation}
where $\Delta w_t$ is a generic stochastic process and $\sigma$ is the variance of this process.
Substituting into Eq.~\eqref{eq7:diffRb} yields 
\begin{equation}
\Delta R_b = R_b \left( \frac{\Delta t}{3t} + \sigma \Delta w_t\right)
\end{equation}
which can then be used to find a stochastic differential equation describing the scaling of a bubble in time due to local dependency of $\Gamma$
\begin{equation}
\frac{t}{R_b} \frac{dR_b}{dt} = \frac{1}{3} + t \sigma \xi_t
\label{eq7:stochmodel}
\end{equation}
where $\xi_t = \frac{\Delta w_t}{\Delta t}$.
This is exactly an equation for $\beta$
\begin{equation}
\beta = \frac{1}{3} + t \sigma \xi_t
\label{eq7:betaDConst}
\end{equation}
Thus we have recovered a model for $\beta$ that captures the desired behaviour the distributions found in Sections~\ref{sec7:drop_pop},~\ref{sec7:CH},~\ref{sec7:CH_symm} and~\ref{sec7:cooke}.
We remark on the time dependence can be seen on the RHS of our recovered expression, this is in--keeping with the expression for the scaling from LSW theory discussed in Section~\ref{sec7:LSW} where a $t$ dependence also appears.
The time dependence also ensures that the variance $\sigma$ is time--dependent, matching our numerical results for the variance which was indeed time dependent.
Another remark we make is that as the expression $t \sigma \xi_t$ goes to $0$, as expected in an infinite domain, we indeed recover exactly the LSW theory result of $\beta = 1/3$.
The positive skew of the expression is ensured by the fact that we have extracted it by manipulating $\Gamma$ which itself depends on the surface tension and the square of the jump in $C$ across an interface, both strictly positive quantities.

Finally we note \Acomment{that in previous studies~\cite{latestagecoarsePRE, spectralCahn} where} the value of $\beta$ has been determined by fitting lines to the equation $R_b(t)^3 = mt + c$ \Acomment{there are errors} on the fitted line. 
Using the model we have developed in equation~\eqref{eq7:betaDConst} these error can now be attributed to the underlying dynamics of the Cahn--Hilliard equation and are not some sampling or numerical error.
We also note that it is not surprising that a line \Acomment{fit would} recover the value of $\beta = 1/3$ as the process of fitting a line in some sense captures the average trend of a piece of data, just as averaging our model will yield $\beta = 1/3$.

\subsection{Conclusions}
We conclude the chapter by summarising our results.
It has been shown that there does indeed exist a distribution for $\beta$ at finite domain sizes, whether that is a finite number of bubbles in the case of Ostwald ripening or finite domain sizes in the case of the symmetric Cahn--Hilliard equation.
The results also extend to the Cahn--Hilliard--Cooke equation.
Specifically also we have shown that in the limit of infinite domains we will get a distribution of $\delta(\beta - 1/3)$ for interface growth.
Finally we have presented a stochastic model that captures this behaviour of $\beta$.

\lhead{\emph{Conclusions}}  
\chapter{Conclusions}
\label{chapter:con}
We conclude the thesis with a brief overview of the research conducted throughout and then present a short discussion on potential future work carrying on from this thesis.

\section{Summary}
In this thesis we have explored two major topics, GPU computing for PDEs and then using this methodology to study the Cahn--Hilliard equation.
With regards the GPU computing a library to \Acomment{apply} finite difference stencils to batched 1D and 2D equations has been developed.
In addition to this we have developed several versions of pentadiagonal solvers which take advantage of various aspects of both the GPU and the systems of equations being solved, this allowed us to outperform the state of the art algorithm in the cuSPARSE library in each case.

With regards the Cahn--Hilliard equation we have presented two results with regards the 1D equation.
We have numerically confirmed the 1D theoretical result that on average the length--scale is proportional to $\log(t)$.
Then we presented a methodology for producing and classifying a large dataset of 1D solutions to produce flow--pattern maps, a methodology that can be extended to further PDEs.
While in 2D we have examined the distribution of $\beta$ presenting a new picture of how this parameter behaves in finite domain sizes, we concluded this discussion by presenting a stochastic model for the distribution.

\section{Future Work}
With regards future extensions of work conducted in this thesis the following is a non--exhaustive list of potential areas of work.

\begin{itemize}
\item The extension of the cuSten library to more general data types such as \codeword{float} and \codeword{int}. 
Indeed it is intended that the next release will cover the former by introducing function templates to deal with the generalisation. 

\item Further extension of the library to 3D stencils.
This will present new challenges with how to optimally access the data.

\item Optimisation of the library to use warp shuffles to perform the core stencil operations, limiting the number of reads of shared memory.
As these are register operations they should produce increased performance.

\item Further extend the cuPentBatch solver and its tridiagonal equivalent to allow for variable vector lengths within a batch and only store exactly a single copy of each LHS matrix needed. 
So should there exist only 6 say LHS matrices within the batch but 100000 RHS matrices then the RHS matrices could be associated with the correct LHS.
This would bring the development of these functions to their natural conclusion with optimisation for both data storage and allowing for generalised matrix dimensions within a batch.

\item Apply and further develop the methodologies for flow pattern maps within the study of two--phased flow and look for applications for studying large parameter spaces in more general PDEs.

\item Unify the statistical picture of the Cahn--Hilliard equation with similar statical studies of the \Acomment{Kuramoto--Sivashinsky (KS)} and \Acomment{Kardar--Parisi--Zhang (KPZ)} equations. 
We note here the connections between the Convective--Cahn--Hilliard equation and the KS equation~\cite{golovin2001convective} and also a recent work showing the KS equation to be part of the KPZ universality class~\cite{roy2019one}.
While potentially aspirational, a picture of a common behaviour underlying the Cahn--Hilliard, KS and KPZ equations is beginning to emerge, particularly when tackled from a statistical physics stand--point. 
Future work would involve further exploring this common behaviour both numerically and theoretically, hopefully leading to a unification of the areas under one universal theory.
\end{itemize}


\addtocontents{toc}{\vspace{2em}} 

\appendix 
 
\lhead{\emph{Appendix}}

\addtocontents{toc}{\vspace{2em}}  
\backmatter

\label{Bibliography}
\lhead{\emph{Bibliography}}  
\bibliographystyle{unsrtnat} 
\bibliography{bibliography/bibliography.bib}  

\end{document}